\documentclass{amsart}

\usepackage{epsfig}
\usepackage{amsmath}
\usepackage{amssymb}
\usepackage{amscd}
\usepackage{graphicx}
\usepackage{color}

\newtheorem{thm}{Theorem}[section]
\newtheorem{lem}[thm]{Lemma}
\newtheorem{prop}[thm]{Proposition}

\newtheorem{exam}[thm]{Example}

\newtheorem{ques}[thm]{Question}
\newtheorem{cor}[thm]{Corollary}
\newtheorem{defn}[thm]{Definition}
\newtheorem{rem}[thm]{Remark}
\newtheorem{claim}[thm]{Claim}

\def \N {\mathbb N}

\def \Z {\mathbb Z}
\def \R {\mathbb R}

\numberwithin{equation}{section}

\begin{document}

\title{local entropy theory of a random dynamical system}

\author{Anthony H. Dooley and Guohua Zhang}

\address{\vskip 2pt \hskip -12pt Anthony H. Dooley}

\address{\hskip -12pt School of Mathematics and Statistics, University of New South Wales, Sydney, NSW 2052, Australia}

\email{a.dooley@unsw.edu.au}

\address{\vskip 2pt \hskip -12pt Guohua Zhang}

\address{\hskip -12pt School of Mathematical Sciences and LMNS, Fudan University, Shanghai 200433, China and}

\address{\hskip -12pt School of Mathematics and Statistics, University of New South Wales, Sydney, NSW 2052, Australia}

\email{chiaths.zhang@gmail.com}



\begin{abstract}
In this paper we extend the notion of a continuous bundle random dynamical system to the setting where the action of $\R$ or $\N$ is replaced by  the action of an infinite countable discrete amenable group.

Given such a system, and a monotone sub-additive invariant family of random continuous functions, we introduce the concept of local fiber topological pressure and establish an associated variational principle, relating it to measure-theoretic entropy. We also discuss some variants of this variational principle.

We introduce both topological and measure-theoretic entropy tuples for continuous bundle random dynamical systems, and apply our variational principles to obtain a relationship between these of entropy tuples. Finally, we give applications of these results to general topological dynamical systems, recovering and extending many recent results in local entropy theory.
\end{abstract}

\maketitle

\markboth{A. H. Dooley and G. H. Zhang}{local entropy theory of a random dynamical system}

\tableofcontents

\newpage

\section{Introduction}

In the early 1990s Blanchard introduced the concept of entropy pairs to search for satisfactory topological
analogues of Kolmogorov systems \cite{B1, B2}. Stimulated by these two papers, local entropy theory for continuous actions of a
countable amenable group on compact metric spaces developed rapidly during the last
two decades, see
 \cite{BGH, BHMMR, BL, DYZ, G, GW4, HMRY, HY, HYZ1, HYZ2, HYZ, R, YZ}. It has been studied for countable sofic group actions in \cite{Zhangsofic} by the second author. For more details of the area, see for example  Glasner's book chapter \cite[Chapther 19]{G1} or the nice survey \cite{GY} by Glasner and Ye.
 Observe that, as shown by \cite[Chapter 19]{G1} and \cite{GY} (and references therein), a detailed analysis of the local properties of entropy provides
additional insight into the related global properties, and local properties of entropy can help us to draw conclusions for global properties.

The foundations of the theory of amenable group
actions were set up in the pioneering paper \cite{OW} by
Ornstein and Weiss, and further developed by Rudolph and Weiss \cite{RW} and Danilenko \cite{D}. See also Benjy Weiss' lovely survey article \cite{We}.
Global entropy theory for amenable group actions has also been discussed by Moulin Ollagnier \cite{MO}.
Other related aspects were discussed in \cite{DP, DZ, GTW, OP1, OP2, ST, WZ}. The connection between local entropy and combinatorial independence  across orbits of sets in dynamical systems was studied systematically by Kerr and Li in \cite{KL, KeLi1} for amenable group actions and in \cite{KLsofic} for sofic group actions, and has been discussed by Chung and Li in \cite{CHFL} for amenable group actions on compact groups by automorphisms.

\medskip

\emph{Our principal aim in this article is to extend the local theory of entropy to the setting of random dynamical systems of countable amenable group actions.}
To date, most discussions of random dynamical systems have concerned $\R$-actions, $\Z$-actions or $\Z_+$-actions. Furthermore, to the best of our knowledge, there has been little discussion of the local theory. In slightly more precise terms, we aim to make a systematic study of the local entropy theory of a continuous bundle random dynamical system over an infinite countable discrete amenable group.

In the setting of random dynamical systems, rather than considering iterations of just one map, we study the successive application of different transformations chosen at random. The basic framework was established by Ulam and von Neumann \cite{UN} and later by Kakutani \cite{Kaku} in proofs of the random ergodic theorem. Since the 1980s, mainly because of stochastic flows arising as solutions of stochastic differential equations, interest in the ergodic theory of random transformations has grown \cite{Arnold, Bog1, Bog2, BG, CDF, KK, K0, K1, KL1, KW, LL, Liu2, Liu, LQ, ZC}.
  It was shown in \cite{Bog1} that the cornerstone for the entropy theory of random transformations is the Abramov-Rokhlin mixed entropy of the fiber of a skew-product transformation (cf \cite{Ar}). Our main result, Theorem \ref{1007141414} establishes a variational principle for local topological pressure in this setting.

\medskip

In the local entropy theory of dynamical systems as studied in \cite[Chapter 19]{G1}, \cite{GY} (and references therein) and \cite{HYZ}, most significant results involving entropy pairs have been obtained using
measure-theoretic techniques and a local variational principle initiated by \cite{BGH}.

Let $G$ be an infinite countable discrete amenable group acting on a compact metric space $X$. Let $\mathcal{V}$ be a finite open cover of the space $X$, and $\nu$ a $G$-invariant Borel probability measure on $X$. Denote by $h_{\text{top}} (G, \mathcal{V})$ and $h_\nu (G, \mathcal{V})$ the topological entropy and measure-theoretic $\nu$-entropy of $\mathcal{V}$, respectively. In \cite{HYZ} Huang, Ye and the second author of the paper proved the following version of local variational principle \cite[Theorem 5.1]{HYZ}:
\begin{equation} \label{1208061906}
h_{\text{top}} (G, \mathcal{V})= \max_{\nu\in \mathcal{P} (X, G)} h_\nu (G, \mathcal{V}),
\end{equation}
where $\mathcal{P} (X, G)$ denotes the set of all $G$-invariant Borel probability measures $\nu$ on $X$. Subsequently, \eqref{1208061906} was generalized by Liang and Yan \cite[Corollary 1.2]{LY}, recovering the global variational principle \cite[Variational Principle 5.2.7]{MO} by Moulin Ollagnier. They showed that for each real-valued continuous function $f$ over $X$,
\begin{equation} \label{1208061916}
 P (f, \mathcal{V})= \max_{\nu\in \mathcal{P} (X, G)} [h_\nu (G, \mathcal{V})+ \int_X f (x) d \nu (x)],
\end{equation}
where $P (f, \mathcal{V})$ denotes the topological $\mathcal{V}$-pressure of $f$. We recover $h_{\text{top}} (G, \mathcal{V})$ when $f$ is the constant zero function.

Remark that, in the local theory of entropy of dynamical systems, many variants of \eqref{1208061906} and \eqref{1208061916} have been discussed by \cite{BGH, CFH, GW4, HY, HYZ1, HYi, R, Z-DCDS},
either for a $\Z$-action on compact metric spaces or for a factor map between topological $\Z$-actions.

\medskip

 Let the family $\mathbf{F}$, associated with $\mathcal{E}\in \mathcal{F}\times \mathcal{B}_X$, be a continuous bundle random dynamical system over a measure-preserving $G$-action $(\Omega,
\mathcal{F}, \mathbb{P}, G)$, where: $G$ is an infinite countable discrete amenable group, $(\Omega, \mathcal{F}, \mathbb{P})$ is a Lebesgue space, and $X$ is a compact metric space associated with Borel $\sigma$-algebra $\mathcal{B}_X$.

In our process of building local entropy theory for $\mathbf{F}$, the first and most important step is to prove a local variational principle similar to that given by equations \eqref{1208061906} and \eqref{1208061916}.

   More precisely, let $\mathcal{U}$ be a finite random open cover, $f$ a random continuous function and $\mu\in \mathcal{P}_\mathbb{P} (\mathcal{E}, G)$, where $\mathcal{P}_\mathbb{P} (\mathcal{E}, G)$ denotes the set of all $G$-invariant probability measures on $\mathcal{E}$ having the marginal $\mathbb{P}$ over $\Omega$.
    Denote by
$P_\mathcal{E} (f, \mathcal{U}, \mathbf{F})$ and $P_\mathcal{E} (f, \mathbf{F})$ the fiber topological $f$-pressure of $\mathbf{F}$ with respect to $\mathcal{U}$ and fiber topological $f$-pressure of $\mathbf{F}$, respectively. Denote by $h^{(r)}_\mu (\mathbf{F}, \mathcal{U})$ and $h^{(r)}_\mu (\mathbf{F})$
the $\mu$-fiber entropy of $\mathbf{F}$ with respect to $\mathcal{U}$ and $\mu$-fiber entropy of $\mathbf{F}$, respectively.

We introduce the property of \emph{factor good} for finite random open covers, and obtain a local variational principle which may be stated as follows:
\begin{equation} \label{1208062326}
P_\mathcal{E} (f, \mathcal{U}, \mathbf{F})= \max_{\mu\in \mathcal{P}_\mathbb{P} (\mathcal{E}, G)} [h_\mu^{(r)} (\mathbf{F}, \mathcal{U})+ \int_\mathcal{E} f (\omega, x) d \mu (\omega, x)]
\end{equation}
provided that $\mathcal{U}$ is factor good. We show in Theorem \ref{1007212202} and Theorem \ref{cover} that many interesting finite random open covers are factor good.

By taking the supremum over all finite random open covers which are factor good, and using \eqref{1208062326} one obtains:
\begin{equation} \label{1208062328}
P_\mathcal{E} (f, \mathbf{F})= \sup_{\mu\in \mathcal{P}_\mathbb{P} (\mathcal{E}, G)} [h_\mu^{(r)} (\mathbf{F})+ \int_\mathcal{E} f (\omega, x) d \mu (\omega, x)],
\end{equation}
which is exactly  Kifer's \cite[Proposition 2.2]{K1} in the special case where  $G= \Z$.
Note that by Remark \ref{1103031512}, if the underlying $G$-action $(\Omega,
\mathcal{F}, \mathbb{P}, G)$ is trivial, i.e. $\Omega$ is a singleton, then the equation \eqref{1208062326} becomes \eqref{1208061906} and \eqref{1208061916}, and the equation \eqref{1208062328} becomes \cite[Variational Principle 5.2.7]{MO}, respectively.

In fact, we prove our main result Theorem \ref{1007141414} in the more general setting given by a
monotone sub-additive invariant family $\mathbf{D}$ of random continuous functions.
   Denote by
$P_\mathcal{E} (\mathbf{D}, \mathcal{U}, \mathbf{F})$ and $P_\mathcal{E} (\mathbf{D}, \mathbf{F})$ the fiber topological $\mathbf{D}$-pressure of $\mathbf{F}$ with respect to $\mathcal{U}$ and fiber topological $\mathbf{D}$-pressure of $\mathbf{F}$, respectively. Theorem \ref{1007141414} states that: if, in addition, the family $\mathbf{D}$ satisfies the assumption $(\spadesuit)$ (cf \S \ref{variational principle concerning pressure}), then
\begin{equation} \label{1208071002}
P_\mathcal{E} (\mathbf{D}, \mathcal{U}, \mathbf{F})= \max_{\mu\in \mathcal{P}_\mathbb{P} (\mathcal{E}, G)} [h_\mu^{(r)} (\mathbf{F}, \mathcal{U})+ \mu (\mathbf{D})]
\end{equation}
for factor good $\mathcal{U}$, and finally
\begin{equation} \label{1208071003}
P_\mathcal{E} (\mathbf{D}, \mathbf{F})= \sup_{\mu\in \mathcal{P}_\mathbb{P} (\mathcal{E}, G)} [h_\mu^{(r)} (\mathbf{F})+ \mu (\mathbf{D})].
\end{equation}
As shown by \eqref{1208011732} and \eqref{1207292302}, equations \eqref{1208071002} and \eqref{1208071003} contain \eqref{1208062326} and \eqref{1208062328}, respectively. We explore further assumption $(\spadesuit)$ in \S \ref{assumption} and \S \ref{special}. It turns out to be quite natural for countable amenable groups in the following sense:
the assumption $(\spadesuit)$ always holds if, in addition, either the family $\mathbf{D}$ is strongly sub-additive (cf Proposition \ref{1008300004}) or the group $G$ is abelian (cf Proposition \ref{1008292115}).

With the above variational principles, we are able to introduce both topological and measure-theoretic entropy tuples for a continuous bundle random dynamical system, and build a variational relationship between these two kinds of entropy tuples.

It is known (\S \S \ref{1208042024}) that the setting of a factor map between topological dynamical systems is in fact equivalent to a special kind of continuous bundle random dynamical systems. Thus, we can apply the above results to study general topological dynamical systems.
For example, in \S \S
\ref{1208042011} we show that, using \eqref{1207291438} and \eqref{1207291439}, variants of Theorem \ref{1007141414}, one can obtain \cite[Theorem 2.1]{LW}, the main result of \cite{LW} by Ledrappier and Walters.

If \S \S \ref{inner}, we may apply Theorem \ref{1007141414} to generalize the Inner Variational Principle \cite[Theorem 4]{DS} of Downarowicz and Serafin to arbitrary amenable group actions and any finite open cover (cf Theorem \ref{1010201200}). Theorem \ref{1010201200}
has also been used to set up symbolic extension theory for amenable group actions by Downarowicz and the second author of the paper \cite{DownZ} .

Moreover, our results on entropy tuples of a continuous bundle random dynamical systems, enable us to study entropy tuples for a topological dynamical systems, recovering many recent results in the local entropy theory of $\Z$-actions (cf \cite{B2, BHMMR, G, G1, GY, HY, HYZ2}) and of infinite countable discrete amenable group actions (cf \cite{HYZ}).

The ideas in the proofs of Propositions \ref{1008300004} and \ref{1008292115} have been used by Golodets and the authors of the paper to obtain analogues of Kingman's sub-additive ergodic theorem for countable amenable groups (\cite{DGZ}).

\medskip

The paper consists of three parts and is organized as follows.

\medskip

The first part gives some preliminaries: on infinite countable discrete amenable groups following \cite{MO, OW, WZ, We}, on general measurable dynamical systems of amenable group actions,
 and on continuous bundle random dynamical systems of an amenable group action extending the case of $\Z$ by \cite{Bog1, K1, KL1, Liu}.
In addition to recalling known results, this part contains some new results: firstly, a convergence result (Proposition \ref{p1006172118}) for infinite countable discrete amenable groups extending \cite[Theorem 5.9]{We} (the difference between Moulin Ollagnier's Proposition \ref{1102111944} and our Proposition \ref{p1006172118} is seen in Example \ref{1104071212}); secondly, a relative Pinsker formula for a measurable dynamical system with an amenable group action (discussed in \cite{GTW} in the case where the state space is a Lebesgue space), see Theorem \ref{1007031233} and Remark \ref{1007-21545}; thirdly, an improved understanding of the local entropy theory of measurable dynamical systems, see Theorem \ref{1006301434} and Question \ref{1102041455}.

In the second part we present and prove our main results.
More precisely, in \S \ref{fourth}, we take a continuous bundle random dynamical system of an infinite countable discrete amenable group action and a monotone sub-additive invariant family of random continuous functions, and follow the ideas of \cite{CFH, HYi, R, Z-DCDS}  to introduce and discuss the local fiber topological pressure for a finite random open cover.  Then in \S \ref{factor good} we introduce and discuss the concept of \emph{factor excellent} and \emph{good} covers,
which assumptions are needed for our main result, Theorem \ref{1007141414}. We show in Theorem \ref{1007212202} and Theorem \ref{cover} that many interesting finite random open covers are factor good. In \S \ref{variational principle concerning pressure} we state Theorem \ref{1007141414}, and give some comments and direct applications.Then, in
\S \ref{seventh} we present the proof of Theorem \ref{1007141414} following the ideas of \cite{HYZ1, HYZ, Mis, Z-Thesis, Z-DCDS}.

For Theorem \ref{1007141414}, we need to assume a condition, which we call $(\spadesuit)$ on the family of random continuous functions: this is discussed in detail in \S \ref{assumption}.
In
\S \ref{special} we
discuss the special case of Theorem \ref{1007141414} for amenable groups admitting a tiling F\o lner sequence, and prove that assumption $(\spadesuit)$ always holds if the group is abelian.
 The proof of Theorem \ref{1007141414} is for finite random open covers. Inspired by Kifer's work \cite[\S 1]{K1}, in \S \ref{ninth} we generalize Theorem \ref{1007141414} to countable random open covers.

The last part of the paper is devoted to applications of the local variational principle established in Part \ref{skdj}.
 In
\S \ref{entropy tuple},
 following the line of local entropy theory (cf \cite[Chapter 19]{G1} or \cite{GY}), we introduce and discuss both topological and measure-theoretic entropy tuples for a continuous bundle random dynamical system, and establish a variational relationship between them.
Finally, in \S \ref{factor}
we apply the results obtained in the previous sections to the setting of a general topological dynamical system, incorporating and extending many recent results in the local entropy theory \cite{B2, BHMMR, G, G1, GY, HY, HYZ1, HYZ2, HYZ}, as well as establishing (Theorems \ref{1010201200} and \ref{1008301614}) some new variational principles concerning the entropy of a topological dynamical system.
We should emphasize that, by the results of \cite{DownZ},  Theorem \ref{1010201200} is important for building the symbolic extension theory of amenable group actions.

\section*{Acknowledgements}

The authors would like to thank Professors Wen Huang, Xiangdong Ye, Kening Lu, Hanfeng Li and Benjy Weiss for useful discussions during the preparation of this manuscript. We also thank the referee for many important comments that have resulted in substantial improvements to this paper.

\medskip

We gratefully acknowledge the support of the Australian Research
Council.

The second author was also supported by FANEDD (No. 201018), NSFC (No.
10801035 and No. 11271078) and a grant from Chinese Ministry of Education (No. 200802461004).

\newpage

\part{Preliminaries} \label{preli}

Denote by $\Z, \Z_+, \N, \R, \R_+$ and $\R_{> 0}$ the set of all integers, non-negative integers, positive integers, real numbers, non-negative real numbers and positive real numbers, respectively.

In this part, we give some preliminaries, including: infinite countable discrete amenable groups, measurable dynamical systems, and continuous bundle random dynamical systems.

\section{Infinite countable discrete amenable groups} \label{first}

In this section, we recall the principal results from \cite{MO, OW, WZ, We} and obtain a new convergence result Proposition \ref{p1006172118} for infinite countable discrete amenable groups. As shown by Remark \ref{1207221549}, Proposition \ref{p1006172118} strengthens \cite[Theorem 5.9]{We} proved by Benjy Weiss.
The difference between Moulin Ollagnier's Proposition \ref{1102111944} and Proposition \ref{p1006172118} is demonstrated by Example \ref{1104071212}; the two results are different even in the setting of an infinite countable discrete amenable group admitting a tiling F\o lner sequence.

The principal convergence results (Proposition \ref{1006122129}, Proposition \ref{1102111944} and Proposition \ref{p1006172118}) are crucial for the introduction and discussion of local fiber topological pressure of a continuous bundle random dynamical system in Part \ref{skdj}.

\medskip

Let $G$ be an infinite countable discrete group and denote by $e_G$ the identity of $G$. Denote by $\mathcal{F}_G$ the
set of all non-empty finite subsets of $G$.

$G$ is called \emph{amenable}, if
for each $K\in \mathcal{F}_G$ and any $\delta> 0$ there exists $F\in \mathcal{F}_G$
such that
$|F\Delta K F|< \delta |F|,$
where $|\bullet|$ is the counting measure of the set $\bullet$, $K F= \{k f: k\in K, f\in
F\}$ and $F\Delta K F= (F\setminus K F)\cup (K F\setminus F)$.
Let $K\in \mathcal{F}_G$ and $\delta> 0$. Set $K^{- 1}= \{k^{- 1}: k\in K\}$.
$A\in \mathcal{F}_G$ is called \emph{$(K, \delta)$-invariant}, if
$$|K^{- 1} A\cap K^{- 1} (G\setminus
A)|< \delta |A|.$$
A sequence $\{F_n: n\in \mathbb{N}\}$ in $\mathcal{F}_G$ is called
a
\emph{F\o lner sequence}, if
\begin{equation}
\lim_{n\rightarrow \infty} \frac{|g F_n\Delta F_n|}{|F_n|}= 0
\end{equation}
for each $g\in G$.
It is
not too hard to obtain the usual asymptotic invariance property from this, viz.:
$G$ is amenable if and only if  $G$ has a F\o lner sequence
$\{ F_n \}_{n\in \mathbb{N}}$. In the class of countable discrete groups, amenable groups include all solvable groups.

In the group  $G=\Z$, it is well known that $F_n=\{ 0,1,\cdots,n-1\}$ defines a F\o lner sequence, as, indeed, does  $\{ a_n,a_n+1,\cdots,a_n+n- 1 \}$ for any sequence $\{a_n\}_{n\in
\mathbb{N}}\subseteq \Z$.

\medskip

\medskip

{\noindent \bf Standard Assumption 1. \it Throughout the current paper, we will assume that $G$ is always an infinite countable discrete amenable group.}

\medskip

\medskip

 The following terminology and results are due to Ornstein and Weiss
\cite{OW} (see also \cite{RW, WZ}).

Let $A_1, \cdots,
A_k, A\in \mathcal{F}_G$ and $\epsilon\in (0, 1)$, $\alpha\in (0, 1]$.
\begin{enumerate}

\item
Subsets $A_1,
\cdots, A_k$ are \emph{$\epsilon$-disjoint} if there are $B_1,
\cdots, B_k\in \mathcal{F}_G$ such that
$$B_i\subseteq A_i, \frac{|B_i|}{|A_i|}> 1- \epsilon\ \text{and}\ B_i\cap B_j= \emptyset\ \text{whenever}\ 1\le i\neq j\le k.$$

\item $\{A_1, \cdots, A_k\}$ \emph{$\alpha$-covers} $A$ if
\begin{equation*}
\frac{|A\cap \bigcup\limits_{i= 1}^k A_i|}{|A|}\ge \alpha.
\end{equation*}

\item
$A_1, \cdots, A_k$ \emph{$\epsilon$-quasi-tile} $A$ if there exist $C_1, \cdots, C_k\in \mathcal{F}_G$ such that
\begin{enumerate}

\item for each $i= 1, \cdots, k$, $A_i C_i\subseteq A$ and $\{A_i c: c\in C_i\}$
forms an $\epsilon$-disjoint family,

\item $A_i C_i\cap A_j C_j= \emptyset$ if $1\le i\neq j\le k$ and

\item $\{A_i C_i: i= 1, \cdots, k\}$ forms a $(1- \epsilon)$-cover
of $A$.
\end{enumerate}
The subsets $C_1, \cdots, C_k$ are called the \emph{tiling centers}.
\end{enumerate}

We have (see for example \cite[Proposition 2.3]{HYZ}, \cite{OW} or
\cite[Theorem 2.6]{WZ}):

\begin{prop} \label{ow-prop}
Let $\{F_n: n\in \mathbb{N}\}$ and $\{F_n': n\in \mathbb{N}\}$ be two F\o
lner sequences of $G$. Assume that $e_G\in F_1\subseteq
F_2\subseteq \cdots$. Then for any $\epsilon\in (0, \frac{1}{4})$
and each $N\in \mathbb{N}$, there exist integers $n_1, \cdots, n_k$
with $N\le n_1< \cdots< n_k$ such that $F_{n_1}, \cdots, F_{n_k}$
$\epsilon$-quasi-tile $F_m'$ whenever $m$ is sufficiently large.
\end{prop}

It is a well-known fact in analysis that if $\{a_n: n\in \N\}\subseteq \R$ is a sequence satisfying that $a_{n+ m}\le a_n+ a_m$ for all $n, m\in \N$, then  the
sequence $\{\frac{a_n}{n}: n\in \N\}$ converges and
\begin{equation} \label{1207211730}
\lim_{n\rightarrow \infty} \frac{a_n}{n}= \inf_{n\in \N} \frac{a_n}{n}\ge - \infty.
\end{equation}
Similar facts can be proved in the setting of an amenable group as follows.

\medskip

 Let $f: \mathcal{F}_G\rightarrow \R$ be a function. Following \cite{HYZ}, we say that f is:
\begin{enumerate}

\item
\emph{monotone}, if $f (E)\le f (F)$ for any $E, F\in \mathcal{F}_G$ satisfying
$E\subseteq F$;

\item \emph{non-negative}, if $f (F)\ge 0$ for any $F\in \mathcal{F}_G$;

\item
\emph{$G$-invariant}, if $f (F g)= f (F)$ for any $F\in \mathcal{F}_G$ and
$g\in G$;

\item
\emph{sub-additive}, if $f (E\cup F)\le f (E)+ f (F)$ for any $E, F\in
\mathcal{F}_G$.
\end{enumerate}

The following convergence property is well known (see for example \cite[Lemma 2.4]{HYZ} or \cite[Theorem 6.1]{Lin-Wei}).

\begin{prop} \label{1006122129}
Let $f: \mathcal{F}_G\rightarrow \R$ be a monotone non-negative
$G$-invariant sub-additive function. Then
for any F\o lner sequence $\{F_n: n\in \mathbb{N}\}$ of $G$, the
sequence $\{\frac{f(F_n)}{|F_n|}: n\in \N\}$ converges and the
value of the limit is independent of the choice of the F\o lner
sequence $\{F_n: n\in \mathbb{N}\}$.
\end{prop}

For a function $f$ as in Proposition \ref{1006122129}, in general we cannot conclude that the limit of the sequence $\{\frac{f(F_n)}{|F_n|}: n\in \N\}$ is its infimum. This is shown by Example \ref{1104071212} constructed at the end of this section (see also Remark \ref{1207221726} for more details).

In order to deduce properties analogous to those of \eqref{1207211730} for the sequence $\{\frac{f(F_n)}{|F_n|}: n\in \N\}$, some additional conditions must be added to the assumptions of Proposition \ref{1006122129}. This can be done in two different ways, both of which will be important for us.

The first extension is:

\begin{prop} \label{1102111944}
Let $f: \mathcal{F}_G\rightarrow \mathbb{R}$ be a function. Assume that $f (E g)= f (E)$ and $f (E\cap F)+ f (E\cup F)\le f (E)+ f (F)$ whenever $g\in G$ and $E, F\in \mathcal{F}_G$ (we set $f (\emptyset)= 0$ by convention). Then for any F\o lner sequence $\{F_n: n\in \mathbb{N}\}$ of $G$, the sequence $\{\frac{f (F_n)}{|F_n|}: n\in \mathbb{N}\}$ converges and the value of the limit is independent of the choice of the F\o lner sequence $\{F_n: n\in \mathbb{N}\}$.  More precisely,
\begin{equation*}
\lim_{n\rightarrow \infty} \frac{f (F_n)}{|F_n|}= \inf_{F\in \mathcal{F}_G} \frac{f (F)}{|F|}\ (\text{and so}\ = \inf_{n\in \mathbb{N}} \frac{f (F_n)}{|F_n|}).
\end{equation*}
\end{prop}

\begin{rem} \label{1106292205}
The above proposition was proved by Moulin Ollagnier (cf \cite[Lemma 2.2.16, Definition 3.1.5, Remark 3.1.7 and Proposition 3.1.9]{MO}).
We are grateful to Hanfeng Li, Benjy Weiss and the referee for pointing this out to us.
\end{rem}

Now we introduce our second extension of Proposition \ref{1006122129}.

Let $\emptyset\neq T\subseteq G$. We say that \emph{$T$ tiles $G$} if there exists $\emptyset\neq G_T\subseteq G$ such that $\{T c: c\in G_T\}$ forms a partition of $G$, that is, $T c_1\cap T c_2= \emptyset$ if $c_1$ and $c_2$ are different elements from $G_T$ and
 $\bigcup\limits_{c\in G_T} T c= G$.
Denote by $\mathcal{T}_G$ the set of all non-empty finite subsets of $G$ which tile $G$. Observe that $\mathcal{T}_G\neq \emptyset$, as $\mathcal{T}_G\supseteq \{\{g\}: g\in G\}$.

As shown by \cite[Theorem 3.3 and Proposition 3.6]{We}, tiling sets play a key role in establishing a counterpart of Rokhlin's Lemma for infinite countable discrete amenable group actions.

The class of countable amenable groups admitting a \emph{tiling F\o lner sequence} (i.e. a F\o lner sequence consisting of tiling subsets of the group) is large, and includes all countable amenable linear groups and all countable residually finite amenable groups \cite{W0}. Recall that a \emph{linear group} is a group isomorphic to a matrix group over a field $K$ (i.e. a group consisting of invertible matrices over $K$); a group is \emph{residually finite} if the intersection of all its normal subgroups of finite index is trivial. Note that any finitely generated nilpotent group is residually finite.

If the group $G$ admits a tiling F\o lner sequence, we may state our second generalization of Proposition \ref{1006122129} as follows:

\begin{prop} \label{p1006172118}
Let $f: \mathcal{F}_G\rightarrow \mathbb{R}$ be a function. Assume that $f (E g)= f (E)$ and $f (E\cup F)\le f (E)+ f (F)$ whenever $g\in G$ and $E, F\in \mathcal{F}_G$ satisfy $E\cap F= \emptyset$. Then for any tiling F\o lner sequence $\{F_n: n\in \mathbb{N}\}$ of $G$, the sequence $\{\frac{f (F_n)}{|F_n|}: n\in \mathbb{N}\}$ converges and the limit is independent of the choice of the tiling F\o lner sequence $\{F_n: n\in \mathbb{N}\}$.  Furthermore,
\begin{equation*}
\lim_{n\rightarrow \infty} \frac{f (F_n)}{|F_n|}= \inf_{F\in \mathcal{T}_G} \frac{f (F)}{|F|}\ (\text{and so}\ = \inf_{n\in \mathbb{N}} \frac{f (F_n)}{|F_n|}).
\end{equation*}
\end{prop}
\begin{proof}
Let $\{F_n: n\in \mathbb{N}\}$ be a tiling F\o lner sequence for $G$. Then there exists $M\in \mathbb{R}$ such that $f (\{g\})= M$ for each $g\in G$. Set
$$h: \mathcal{F}_G\rightarrow \mathbb{R}, E\mapsto |E| M- f (E)$$
 for each $E\in \mathcal{F}_G$. The function $h: \mathcal{F}_G\rightarrow \mathbb{R}_+$ satisfies $h (E g)= h (E)$ and $h (E\cup F)\ge h (E)+ h (F)$ whenever $g\in G$ and $E, F\in \mathcal{F}_G$ satisfy $E\cap F= \emptyset$. Thus, we only need show that the sequence $\{\frac{h (F_n)}{|F_n|}: n\in \mathbb{N}\}$ converges and
\begin{equation} \label{1006172118}
\lim_{n\rightarrow \infty} \frac{h (F_n)}{|F_n|}= \sup_{F\in \mathcal{T}_G} \frac{h (F)}{|F|}.
\end{equation}
It is clear that
\begin{equation} \label{1006172119}
\limsup_{n\rightarrow \infty} \frac{h (F_n)}{|F_n|}\le \sup_{F\in \mathcal{T}_G} \frac{h (F)}{|F|}.
\end{equation}

For the other direction, first let $\epsilon> 0$ and $F\in \mathcal{T}_G$ be fixed: then  $G_F$ is a subset of $G$ such that $\{F g: g\in G_F\}$ forms a partition of $G$. As $\{F_n: n\in \mathbb{N}\}$ is a tiling F\o lner sequence of $G$,
$F_n$ is $(F, \epsilon)$-invariant whenever $n\in \mathbb{N}$ is large enough.
Now for each $n\in \mathbb{N}$ set $E_n'= \{g\in G_F: F g\subseteq F_n\}$ and $E_n= \{g\in G_F: F g\cap F_n\neq \emptyset\}$, one has $$E_n\setminus E_n'\subseteq F^{- 1} F_n\cap F^{- 1} (G\setminus F_n).$$
Thus if $n\in \mathbb{N}$ is sufficiently large,
\begin{equation*}
\frac{|F_n|}{|F|}\le |E_n|\le |E_n'|+ \epsilon |F_n|,\ \text{i.e.}\ |E_n'|\ge (\frac{1}{|F|}- \epsilon) |F_n|,
\end{equation*}
and thus
\begin{equation*}
\frac{h (F_n)}{|F_n|}\ge \frac{h (F E_n')}{|F_n|}\ge \frac{h (F) |E_n'|}{|F_n|}\ge (\frac{1}{|F|}- \epsilon) h (F).
\end{equation*}
This implies
\begin{equation*}
\liminf_{n\rightarrow \infty} \frac{h (F_n)}{|F_n|}\ge (\frac{1}{|F|}- \epsilon) h (F).
\end{equation*}
Since both $\epsilon> 0$ and $F\in \mathcal{T}_G$ are arbitrary, one may conclude
\begin{equation} \label{1006172120}
\liminf_{n\rightarrow \infty} \frac{h (F_n)}{|F_n|}\ge \sup_{F\in \mathcal{T}_G} \frac{h (F)}{|F|}.
\end{equation}
Now \eqref{1006172118} follows directly from \eqref{1006172119} and \eqref{1006172120}. This completes the proof.
\end{proof}

\begin{rem} \label{1207221549}
In \cite[Theorem 5.9]{We}, Weiss proved the same conclusion under the additional assumptions that $0\le f (E)\le f (F)$ for all $E, F\in \mathcal{F}_G$ with $E\subseteq F$. The trivial example satisfying the assumptions of Proposition \ref{p1006172118} is the function $f$ given by $f (E)= - |E|^2$ for all $E\in \mathcal{F}_G$, to which \cite[Theorem 5.9]{We} does not apply.
\end{rem}

The following example highlights the difference between Proposition \ref{1102111944} and Proposition \ref{p1006172118} in the setting of $G=\Z$.

\begin{exam} \label{1104071212}
There exists a monotone non-negative
$\Z$-invariant sub-additive function $f: \mathcal{F}_\Z\rightarrow \R$ (in particular, $f$ satisfies the assumption of Proposition \ref{p1006172118} and so the sequence $\{\frac{f (\{1, \cdots, n\})}{n}: n\in \N\}$ converges) such that
\begin{equation} \label{1104071303}
\lim_{n\rightarrow \infty} \frac{f (\{1, \cdots, n\})}{n}> \inf_{E\in \mathcal{F}_\Z} \frac{f (E)}{|E|}.
\end{equation}
Thus, $f$ does not satisfy the assumption of Proposition \ref{1102111944}.
\end{exam}
\begin{proof}[Construction of Example \ref{1104071212}]
The function $f$ is constructed as follows: let $E\in \mathcal{F}_\Z$,
$$f (E)= \min\{|E|- |F|: \{p+ S: p\in F\}\ \text{is a disjoint family of subsets of}\ E\},$$
here $S= \{1, 2, 4\}$ and $F$ may be empty. For example, $f (S)= 2, f (\{1, 2, 3, 4\})= 3$.

Now we claim that $f$ has the required property.

First, we aim to prove that $f$ is a monotone non-negative $\Z$-invariant sub-additive function by claiming $f (E)\le f (E\cup \{a\})$ with $E\in \mathcal{F}_\Z, a\in \Z\setminus E$ and $f (E_1\cup E_2)\le f (E_1)+ f (E_2)$ with $E_1, E_2\in \mathcal{F}_\Z, E_1\cap E_2= \emptyset$.

We can select $F$ such that $f (E\cup \{a\})= |E|+ 1- |F|$ and $\{p+ S: p\in F\}$ is a disjoint family of subsets of $E\cup \{a\}$. If $a\notin \cup \{p+ S: p\in F\}$ then $\{p+ S: p\in F\}$ is also a disjoint family of subsets of $E$ and so $f (E)\le |E|- |F|$. If
$a\in p_0+ S$ for some $p_0\in F$ then $\{p+ S: p\in F\setminus \{p_0\}\}$ is a disjoint family of subsets of $E$ and so $f (E)\le |E|- |F\setminus \{p_0\}|$. Summing up, $f (E)\le f (E\cup \{a\})$.

Now let $F_i$ be such that $f (E_i)= |E_i|- |F_i|$ and $\{p+ S: p\in F_i\}$ is a disjoint family of subsets of $E_i, i= 1, 2$. As $E_1\cap E_2= \emptyset$ It is easy to see that $F_1\cap F_2= \emptyset$ and $\{p+ S: p\in F_1\cup F_2\}$ is a disjoint family of subsets of $E_1\cup E_2$, and so $f (E_1\cup E_2)\le |E_1\cup E_2|- |F_1\cup F_2|= f (E_1)+ f (E_2)$.

Secondly, let $n\in \N$. We prove that $f (\{1, \cdots, 4 n\})= 3 n$. It is easy to check that $f (\{1, \cdots, 4 n\})\le 3 n$. Assume that $f (\{1, \cdots, 4 n\})< 3 n$: in particular, there exists $F\in \mathcal{F}_\Z$ such that $\{p+ S: p\in F\}$ is a disjoint family of subsets of $\{1, \cdots, 4 n\}$ and $|F|> n$. Observe that there exists at least one $k$ such that $\{4k -3, 4k- 2, 4k- 1, 4k\}\cap F$ contains at least two different elements. In particular, there exists $i', j'\in \{4k- 3, 4k- 2, 4k- 1, 4k\}$ such that $i'+ S$ and $j'+ S$ are disjoint, a contradiction to the fact that $(i+ S)\cap (j+ S)\neq \emptyset$ whenever $i, j\in \{1, 2, 3, 4\}$ (this can be verified directly). Thus, $f (\{1, \cdots, 4 n\})= 3 n$.

Finally, we finish the proof of the strict inequality \eqref{1104071303} by observing that $\inf\limits_{E\in \mathcal{F}_\Z} \frac{f (E)}{|E|}= \frac{2}{3}$. This finishes the construction.
\end{proof}

Obviously, by standard modifications, we could obtain such an example with
$$\lim_{n\rightarrow \infty} \frac{f (\{1, \cdots, n\})}{n}> 0= \inf_{E\in \mathcal{F}_\Z} \frac{f (E)}{|E|}.$$

\begin{rem} \label{1207221726}
It is direct to check that $\{F_n: n\in \mathbb{N}\}$ is a F\o lner sequence of $\Z$, where $F_n= \{1, \cdots, 4 n- 3, 4 n- 2, 4 n\}$ for each $n\in \N$. From the construction, it is easy to see $3 (n- 1)\le f (F_n)\le 3 n- 1$ and then
$$\inf_{n\ge m} \frac{f (F_n)}{|F_n|}\le \frac{3 m- 1}{4 m- 1} < \frac{3}{4} = \lim_{n\rightarrow \infty} \frac{f (F_n)}{|F_n|}$$
for each $m\in \N$. This shows that: in general we can not conclude that the limit of the sequence $\{\frac{f(F_n)}{|F_n|}: n\in \N\}$
in Proposition \ref{1006122129} will be the infimum of the sequence (even neglecting finitely many elements of the sequence).
\end{rem}

Based on the convergence results Proposition \ref{1006122129} and Proposition \ref{1102111944}, we end this section with the following assumption throughout the remainder of the paper.

\medskip

\medskip

{\noindent \bf Standard Assumption 2. \it
From now on, fix $\{F_n: n\in \mathbb{N}\}$, a F\o lner sequence of $G$ with the property that $e_G\subseteq F_1\subsetneq F_2\subsetneq \cdots$, and hence $|F_n|\ge n$ for each $n\in \N$ (it is easy to see that such a F\o lner sequence of $G$ must exist).}

\section{Measurable dynamical systems} \label{1006291551}

By a \emph{measurable dynamical $G$-system} (MDS) $(Y, \mathcal{D}, \nu, G)$ we mean a probability space $(Y, \mathcal{D}, \nu)$  and a group $G$ of invertible measure-preserving transformations of $(Y, \mathcal{D}, \nu)$ with $e_G$ acting as the identity transformation.

In this section we give some background on measurable dynamical systems used in our discussion of a continuous bundle random dynamical system. We also obtain the relative Pinsker formula of an MDS for an infinite countable discrete amenable group action. This was obtained in \cite{GTW} in the case of a Lebesgue space.

We believe that Theorem \ref{1006301434} is an interesting new result. The related Question \ref{1102041455} is a step towards understanding the entropy theory of an MDS. Theorem \ref{1006301434} leads to our discussions of entropy tuples for a continuous bundle random dynamical system in \S \ref{entropy tuple}.

\medskip

Let $(Y, \mathcal{D}, \nu)$ be a probability space.
A \emph{cover of $(Y, \mathcal{D}, \nu)$} is a family $\mathcal{W}\subseteq \mathcal{D}$ satisfying $\bigcup\limits_{W\in \mathcal{W}} W= Y$; if all elements of a cover $\mathcal{W}$ are disjoint, then $\mathcal{W}$ is called a \emph{partition of $(Y, \mathcal{D}, \nu)$}. Denote by $\mathbf{C}_Y$ and $\mathbf{P}_Y$ the set of all finite covers and finite partitions of $(Y, \mathcal{D}, \nu)$, respectively. Let $\alpha$ be a partition of $(Y, \mathcal{D}, \nu)$ and $y\in Y$. Denote by $\alpha (y)$ the atom of $\alpha$ containing $y$.
Let $\mathcal{W}_1, \mathcal{W}_2\in \mathbf{C}_Y$.
If each element of $\mathcal{W}_1$ is contained in some element of $\mathcal{W}_2$ then we say that $\mathcal{W}_1$ is \emph{finer than $\mathcal{W}_2$} (denote by $\mathcal{W}_1\succeq \mathcal{W}_2$ or $\mathcal{W}_2\preceq \mathcal{W}_1$).
The join $\mathcal{W}_1\vee \mathcal{W}_2$ of $\mathcal{W}_1$ and $\mathcal{W}_2$ is given by
$$\mathcal{W}_1\vee \mathcal{W}_2= \{W_1\cap W_2: W_1\in \mathcal{W}_1, W_2\in \mathcal{W}_2\}.$$
The definition extends naturally to a finite collection of covers.

 Fix $\mathcal{W}_1\in \mathbf{C}_Y$ and denote by $\mathcal{P} (\mathcal{W}_1)\in \mathbf{P}_Y$ the finite partition generated by $\mathcal{W}_1$: that is, if we say $\mathcal{W}_1= \{W_1^1, \cdots, W_1^m\}, m\in \mathbb{N}$ then
\begin{equation*}
\mathcal{P} (\mathcal{W}_1)= \{\bigcap\limits_{i= 1}^m A_i: A_i\in \{W_1^i, (W_1^i)^c\}, 1\le i\le m\}.
\end{equation*}
We introduce a finite collection of partitions which we will use in the sequel. Let
\begin{equation*}
\mathbf{P} (\mathcal{W}_1)= \{\alpha\in \mathbf{P}_Y: \mathcal{P} (\mathcal{W}_1)\succeq \alpha\succeq \mathcal{W}_1\}.
\end{equation*}

 Now let $\mathcal{C}$ be a sub-$\sigma$-algebra  of $\mathcal{D}$ and $\mathcal{W}_1\in \mathbf{P}_Y$. We set
\begin{equation*}
H_\nu (\mathcal{W}_1| \mathcal{C})= - \sum_{W_1\in \mathcal{W}_1} \int_Y \nu (W_1| \mathcal{C}) (y) \log \nu (W_1| \mathcal{C}) (y) d \nu (y),
\end{equation*}
(by convention, we set $0 \log 0= 0$).
Here, $\nu (W_1| \mathcal{C})$ denotes the conditional expectation with respect to $\nu$ of the function $1_{W_1}$ relative to $\mathcal{C}$. It is a standard fact that $H_\nu (\mathcal{W}_1| \mathcal{C})$ increases with $\mathcal{W}_1$ (ordered by $\succeq$ ) and decreases as $\mathcal{C}$ increases (ordered by $\subseteq$). In fact, if the sequence of sub-$\sigma$-algebras $\{\mathcal{C}_n: n\in \mathbb{N}\}$ increases or decreases to $\mathcal{C}$ then the sequence $\{H_\nu (\mathcal{W}_1| \mathcal{C}_n): n\in \N\}$ decreases or increases to $H_\nu (\mathcal{W}_1| \mathcal{C})$, respectively (see for example \cite[Theorem 14.28]{G1}).

If $\mathcal{N}_Y\doteq \{\emptyset, Y\}$ is the trivial $\sigma$-algebra, one has
\begin{equation*}
H_\nu (\mathcal{W}_1| \mathcal{N}_Y)= - \sum_{W_1\in \mathcal{W}_1} \nu (W_1)\log \nu (W_1)\ge H_\nu (\mathcal{W}_1| \mathcal{C}).
\end{equation*}
We will write for short $H_\nu (\mathcal{W}_1)= H_\nu (\mathcal{W}_1| \mathcal{N}_Y)$.

Let $\mathcal{W}_2\in \mathbf{P}_Y$. Then $\mathcal{W}_2$ naturally generates a sub-$\sigma$-algebra of $\mathcal{D}$ (also denoted by $\mathcal{W}_2$ if there is no ambiguity). It is easy to see that
\begin{equation*}
H_\nu (\mathcal{W}_1| \mathcal{W}_2)= H_\nu (\mathcal{W}_1\vee \mathcal{W}_2)- H_\nu (\mathcal{W}_2).
\end{equation*}
In fact, more generally,
\begin{equation} \label{1007021626}
H_\nu (\mathcal{W}_1| \mathcal{C}\vee \mathcal{W}_2)= H_\nu (\mathcal{W}_1\vee \mathcal{W}_2| \mathcal{C})- H_\nu (\mathcal{W}_2| \mathcal{C}),
\end{equation}
here, $\mathcal{C}\vee \mathcal{W}_2$ denotes the sub-$\sigma$-algebra of $\mathcal{D}$ generated by sub-$\sigma$-algebras $\mathcal{C}$ and $\mathcal{W}_2$ (the notation works similarly for any given family of sub-$\sigma$-algebras of $\mathcal{D}$).

Now let $\mathcal{W}_1\in \mathbf{C}_Y$. Following the ideas of Romagnoli \cite{R} we set
\begin{equation*}
H_\nu (\mathcal{W}_1| \mathcal{C})= \inf_{\alpha\in \mathbf{P}_Y, \alpha\succeq \mathcal{W}_1} H_\nu (\alpha| \mathcal{C}).
\end{equation*}
It is clear that there is no ambiguity in this notation. Moreover, it remains true that $H_\nu (\mathcal{W}_1| \mathcal{C})$ increases with $\mathcal{W}_1$ and decreases as $\mathcal{C}$ increases.

Similarly, we can introduce $H_\nu (\mathcal{W}_1)$. Note that (see for example \cite[Proposition 6]{R})
\begin{equation} \label{1006291640}
H_\nu (\mathcal{W}_1)= \min_{\alpha\in \mathbf{P} (\mathcal{W}_1)} H_\nu (\alpha).
\end{equation}

\medskip

Let $(Y, \mathcal{D}, \nu, G)$ be an MDS, $\mathcal{W}\in \mathbf{C}_X$ and $\mathcal{C}\subseteq \mathcal{D}$ a sub-$\sigma$-algebra. For each $F\in \mathcal{F}_G$, set $\mathcal{W}_F= \bigvee\limits_{g\in F} g^{- 1} \mathcal{W}$. If $\mathcal{C}$ is $G$-invariant, i.e. $g^{- 1} \mathcal{C}= \mathcal{C}$ (up to $\nu$ null sets) for each $g\in G$,
then it is easy to check that
$$H_\nu (\mathcal{W}_\bullet| \mathcal{C}): \mathcal{F}_G\rightarrow \R, F\mapsto H_\nu (\mathcal{W}_F| \mathcal{C})$$
 is a monotone non-negative
$G$-invariant sub-additive function. Now, following ideas in \cite{R} by Romagnoli,
we may define
the {\it measure-theoretic
$\nu$-entropy of $\mathcal{W}$ with respect to $\mathcal{C}$} and the {\it measure-theoretic
$\nu, +$-entropy of $\mathcal{W}$ with respect to $\mathcal{C}$} by
\begin{equation*}
h_\nu (G, \mathcal{W}| \mathcal{C})= \lim_{n\rightarrow \infty}
\frac{1}{|F_n|} H_\nu (\mathcal{W}_{F_n}| \mathcal{C})
\end{equation*}
and
\begin{equation*}
h_{\nu, +} (G, \mathcal{W}| \mathcal{C})= \inf_{\alpha\in \mathbf{P}_Y, \alpha\succeq \mathcal{W}} h_\nu (G, \alpha| \mathcal{C})\ge h_\nu (G, \mathcal{W}| \mathcal{C}),
\end{equation*}
respectively. By Proposition \ref{1006122129}, $h_\nu (G, \mathcal{W}| \mathcal{C})$ and thus $h_{\nu, +} (G, \mathcal{W}| \mathcal{C})$ are well defined. Observe that if $\alpha\in \mathbf{P}_Y$ then $h_\nu (G, \alpha| \mathcal{C})= h_{\nu, +} (G, \alpha| \mathcal{C})$ and
\begin{equation} \label{1006272232}
h_\nu (G, \alpha| \mathcal{C})= \inf_{F\in \mathcal{F}_G} \frac{1}{|F|} H_\nu (\alpha_F| \mathcal{C})\le H_\nu (\alpha| \mathcal{C}),
\end{equation}
which is a direct corollary of Proposition \ref{1102111944}, see also \cite[(2)]{DZ}. Then the \emph{measure-theoretic $\nu$-entropy of $(Y, \mathcal{D}, \nu, G)$ with respect to $\mathcal{C}$} is defined as
\begin{equation*}
h_\nu (G, Y| \mathcal{C})= \sup_{\alpha\in \mathbf{P}_Y} h_\nu (G, \alpha| \mathcal{C}).
\end{equation*}
By Proposition \ref{1006122129}, all values of these invariants are independent of the selection of the F\o lner sequence $\{F_n: n\in \mathbb{N}\}$.

To simplify notation, when $\mathcal{C}= \mathcal{N}_Y$ we will omit the qualification ``with respect to $\mathcal{C}$" or ``$| \mathcal{C}$". When $T$ is an invertible measure-preserving transformation of $(Y, \mathcal{D}, \nu)$ and we consider the group action of $\{T^n: n\in \mathbb{Z}\}$, we will replace  ``$\{T^n: n\in \mathbb{Z}\}$" by ``$T$".

\medskip

It is not hard to obtain:

\begin{prop} \label{0911192237}
Let $(Y, \mathcal{D}, \nu, G)$ be an MDS, $\mathcal{W}_1, \mathcal{W}_2\in \mathbf{C}_Y, \alpha_1, \alpha_2\in \mathbf{P}_Y, F\in \mathcal{F}_G$ and $\mathcal{C}\subseteq \mathcal{D}$ a $G$-invariant sub-$\sigma$-algebra. Then
\begin{enumerate}

\item \label{1102032251} $h_\nu (G, \mathcal{W}_1| \mathcal{C})\le h_\nu (G, \mathcal{W}_2| \mathcal{C})$ and $h_{\nu, +} (G, \mathcal{W}_1| \mathcal{C})\le h_{\nu, +} (G, \mathcal{W}_2| \mathcal{C})$ if $\mathcal{W}_1\preceq \mathcal{W}_2$.

\item \label{1102032252} $h_\nu (G, \mathcal{W}_1\vee \mathcal{W}_2| \mathcal{C})\le h_\nu (G, \mathcal{W}_1| \mathcal{C})+ h_\nu (G, \mathcal{W}_2| \mathcal{C})$ and $h_{\nu, +} (G, \mathcal{W}_1\vee \mathcal{W}_2| \mathcal{C})\le h_{\nu, +} (G, \mathcal{W}_1| \mathcal{C})+ h_{\nu, +} (G, \mathcal{W}_2| \mathcal{C})$.

\item \label{1006301514} $h_\nu (G, (\mathcal{W}_1)_F| \mathcal{C})= h_\nu (G, \mathcal{W}_1| \mathcal{C})\le h_{\nu, +} (G, \mathcal{W}_1| \mathcal{C})\le H_\nu (\mathcal{W}_1| \mathcal{C})\le \log |\mathcal{W}_1|$.

\item \label{1007041247}
$h_\nu (G, \alpha_1\vee \alpha_2| \mathcal{C})\le h_\nu (G, \alpha_2| \mathcal{C})+ H_\nu (\alpha_1| \mathcal{C}\vee \alpha_2)\le h_\nu (G, \alpha_2| \mathcal{C})+ H_\nu (\alpha_1| \alpha_2)$.


\item \label{1102032253} $h_\nu (G, Y| \mathcal{C})= \sup\limits_{\mathcal{W}\in \mathbf{C}_Y} h_\nu (G, \mathcal{W}| \mathcal{C})= \sup\limits_{\mathcal{W}\in \mathbf{C}_Y} h_{\nu, +} (G, \mathcal{W}| \mathcal{C})$.
\end{enumerate}
\end{prop}
\begin{proof}
Equations \eqref{1102032251} and \eqref{1102032253} are easy to verify.

Equations \eqref{1102032252} and \eqref{1007041247} follow directly from
\begin{equation*}
H_\nu ((\mathcal{W}_1\vee \mathcal{W}_2)_E| \mathcal{C})\le H_\nu ((\mathcal{W}_1)_E| \mathcal{C})+ H_\nu ((\mathcal{W}_2)_E| \mathcal{C})
\end{equation*}
and
\begin{equation*}
H_\nu ((\alpha_1\vee \alpha_2)_E| \mathcal{C})\le H_\nu ((\alpha_2)_E| \mathcal{C})+ |E| H_\nu (\alpha_1| \alpha_2\vee \mathcal{C})
\end{equation*}
for each $E\in \mathcal{F}_G$, respectively, neither of which is hard to obtain.

Thus, we only need prove \eqref{1006301514}.
Note that if $\alpha\in \mathbf{P}_Y$ satisfies $\alpha\succeq \mathcal{W}_1$ then
$$h_{\nu, +} (G, \mathcal{W}_1| \mathcal{C})\le h_\nu (G, \alpha| \mathcal{C})\le H_\nu (\alpha| \mathcal{C})$$
by \eqref{1006272232}, which implies that
$$h_{\nu, +} (G, \mathcal{W}_1| \mathcal{C})\le H_\nu (\mathcal{W}_1| \mathcal{C})\le H_\nu (\mathcal{W}_1)\le \log |\mathcal{W}_1|.$$
 It remains to prove that
 $$h_\nu (G, (\mathcal{W}_1)_F| \mathcal{C})= h_\nu (G, \mathcal{W}_1| \mathcal{C}).$$
 We should point out that if $\{F_n: n\in \mathbb{N}\}$ is a F\o lner sequence of $G$ then
  $\{F F_n: n\in \mathbb{N}\}$ is also a F\o lner sequence of $G$
 and
 $\lim\limits_{n\rightarrow \infty} \frac{|F F_n|}{|F_n|}= 1$, which implies that
\begin{eqnarray*}
& &\hskip -26pt h_\nu (G, (\mathcal{W}_1)_F| \mathcal{C}) \\
&= & \lim_{n\rightarrow \infty} \frac{1}{|F_n|} H_\nu (((\mathcal{W}_1)_F)_{F_n}| \mathcal{C}) \\
&= & \lim_{n\rightarrow \infty} \frac{1}{|F_n|} H_\nu ((\mathcal{W}_1)_{F F_n}| \mathcal{C}) \\
&= & \lim_{n\rightarrow \infty} \frac{1}{|F F_n|} H_\nu ((\mathcal{W}_1)_{F F_n}| \mathcal{C})\ (\text{as $\lim\limits_{n\rightarrow \infty} \frac{|F F_n|}{|F_n|}= 1$}) \\
&= & h_\nu (G, \mathcal{W}_1| \mathcal{C})\ (\text{as $\{F F_n: n\in \mathbb{N}\}$ is also a F\o lner sequence of $G$}).
\end{eqnarray*}
This proves \eqref{1006301514} and so completes our proof.
\end{proof}

The following result will be used subsequently.

\begin{thm} \label{1006272300}
Let $(Y, \mathcal{D}, \nu, G)$ be an MDS, $\mathcal{W}\in \mathbf{C}_Y$ and $\mathcal{C}\subseteq \mathcal{D}$ a $G$-invariant sub-$\sigma$-algebra. Assume that $(Y, \mathcal{D}, \nu)$ is a Lebesgue space. Then
$$h_\nu (G, \mathcal{W}| \mathcal{C})= h_{\nu, +} (G, \mathcal{W}| \mathcal{C}).$$
 Thus, using \eqref{1006272232} we have an alternative expression for $h_\nu (G, \mathcal{W}| \mathcal{C})$:
\begin{equation} \label{1006301447}
h_\nu (G, \mathcal{W}| \mathcal{C})= \inf_{F\in \mathcal{F}_G} \frac{1}{|F|} \inf_{\alpha\in \mathbf{P}_Y, \alpha\succeq \mathcal{W}} H_\nu (\alpha_F| \mathcal{C}).
\end{equation}
\end{thm}

\begin{rem} \label{1006282157} As shown by \cite{HYZ}, Theorem \ref{1006272300} plays an important role in the establishment of the local theory of entropy for a topological $G$-action.

In order to prove Theorem \ref{1006272300}, we shall use Danilenko's orbital approach to the entropy theory of an MDS as a crucial tool.
In fact, notice that using the arguments of Danilenko \cite{D} one can re-write the whole process carried out in \cite[\S 4]{HYZ}.  In other words, we can extend the whole \cite[\S 4]{HYZ} to the relative case, where we are given a $G$-invariant sub-$\sigma$-algebra $\mathcal{C}\subseteq \mathcal{D}$. As this is a straightforward re-writing of the arguments of \cite[\S 4]{HYZ}, we omit the details and leave their verification to the interested reader. We only remark that, based on the results from \cite{GW4, HMRY, R}, the equivalence of these two kinds of entropy for finite measurable covers was first pointed out by Huang, Ye and Zhang in \cite{HYZ1} for $\Z$-actions.
\end{rem}

As in the case of a measurable dynamical $\Z$-system, one can define a relative Pinsker formula in our setting.

\begin{thm} \label{1007031233}
Let $(Y, \mathcal{D}, \nu, G)$ be an MDS, $\mathcal{C}\subseteq \mathcal{D}$ a $G$-invariant sub-$\sigma$-algebra and $\alpha, \beta\in \mathbf{P}_Y$. Then, for $\beta_G$, the sub-$\sigma$-algebra of $\mathcal{D}$ generated by $g^{- 1} \beta, g\in G$,
\begin{equation} \label{1007021621}
\lim_{n\rightarrow \infty} \frac{1}{|F_n|} H_\nu (\alpha_{F_n}| \beta_{F_n}\vee \mathcal{C})= h_\nu (G, \alpha| \beta_G\vee \mathcal{C})
\end{equation}
and so
$$h_\nu (G, \alpha\vee \beta| \mathcal{C})= h_\nu (G, \beta| \mathcal{C})+ h_\nu (G, \alpha| \beta_G\vee \mathcal{C}).$$
\end{thm}

Before establishing \eqref{1007021621}, we first make a remark.

\begin{rem} \label{1007-21545}
Under the assumptions of Theorem \ref{1007031233}, we cannot deduce the convergence of the sequence
$\{\frac{1}{|F_n|} H_\nu (\alpha_{F_n}| \beta_{F_n}\vee \mathcal{C}): n\in \N\}$ from Proposition \ref{1006122129}.

In fact, it is not hard to check that
 $$H_\nu (\alpha_\bullet| \beta_\bullet\vee \mathcal{C}): \mathcal{F}_G\rightarrow \mathbb{R}, F\mapsto H_\nu (\alpha_F| \beta_F\vee \mathcal{C})$$
  is a non-negative $G$-invariant function. By  (\eqref{1007021626}), it is also sub-additive:
 \begin{eqnarray*}
H_\nu (\alpha_{E\cup F}| \beta_{E\cup F}\vee \mathcal{C})&\le & H_\nu (\alpha_E| \beta_{E\cup F}\vee \mathcal{C})+ H_\nu (\alpha_F| \beta_{E\cup F}\vee \mathcal{C}) \\
&\le & H_\nu (\alpha_E| \beta_E\vee \mathcal{C})+ H_\nu (\alpha_F| \beta_F\vee \mathcal{C})
 \end{eqnarray*}
  whenever $E, F\in \mathcal{F}_G$.
In general, this function is not monotone: we give an example of this.

Let $G= \Z_2\times \Z$. Then $G$ is an infinite countable discrete amenable group, and with unit $(0, 0)$. Consider the MDS $$(\{a, b\}^G, \mathcal{B}_{\{a, b\}^G}, \bigotimes\limits_{g\in G} \{\frac{1}{2}, \frac{1}{2}\}, G),$$
  where $\mathcal{B}_{\{a, b\}^G}$ denotes the Borel $\sigma$-algebra of the compact metric space $\{a, b\}^G$. Now $G$ acts naturally on $(\{a, b\}^G, \mathcal{B}_{\{a, b\}^G}, \bigotimes\limits_{g\in G} \{\frac{1}{2}, \frac{1}{2}\})$ and preserves the measure. Set
  $$\alpha= \{[a]_{(0, 0)}, [b]_{(0, 0)}\}\ \text{and}\ \beta= (1, 0)^{- 1} \alpha$$
   with
  $[i]_{(0, 0)}= \{(x_g)_{g\in G}: x_{(0, 0)}= i\}$, $i\in \{a, b\}$.
 Let $S\in \mathcal{F}_\Z$ and set
  $$E= \{(0, s): s\in S\}\in \mathcal{F}_G\ \text{and}\ F= \{(1, s): s\in S\}= (1, 0)\cdot E\in \mathcal{F}_G.$$
  Using \eqref{1007021626}, it is straightforward to check
  $$H_\nu (\alpha_F| \beta_F\vee \mathcal{N}_{\{a, b\}^G})= H_\nu (\alpha_F\vee \beta_F| \mathcal{N}_{\{a, b\}^G})- H_\nu (\beta_F| \mathcal{N}_{\{a, b\}^G})= |S| \log 2;$$
whereas,
 $$\alpha_F= \alpha_{(1, 0)\cdot E}= ((1, 0)^{- 1} \alpha)_E= \beta_E\ \text{and similarly}\ \alpha_E= \beta_F.$$
It follows that
$$H_\nu (\alpha_{E\cup F}| \beta_{E\cup F}\vee \mathcal{N}_{\{a, b\}^G})= 0< |S| \log 2= H_\nu (\alpha_F| \beta_F\vee \mathcal{N}_{\{a, b\}^G}).$$
\end{rem}

Now we prove Theorem \ref{1007031233}.

\begin{proof}[Proof of Theorem \ref{1007031233}]
For each $n\in \N$ one has by \eqref{1007021626},
\begin{equation} \label{star}
H_\nu ((\alpha\vee \beta)_{F_n}| \mathcal{C})= H_\nu (\beta_{F_n}| \mathcal{C})+ H_\nu (\alpha_{F_n}| \beta_{F_n}\vee \mathcal{C}).
\end{equation}
To finish the proof it is sufficient to prove \eqref{1007021621}.

As a sub-$\sigma$-algebra of $\mathcal{D}$, $\beta_G$ (and likewise $\beta_G\vee \mathcal{C}$) is $G$-invariant, and hence
\begin{equation*}
h_\nu (G, \alpha| \beta_G\vee \mathcal{C})= \lim_{n\rightarrow \infty} \frac{1}{|F_n|} H_\nu (\alpha_{F_n}| \beta_G\vee \mathcal{C}).
\end{equation*}
Set $M= H_\nu (\alpha| \beta\vee \mathcal{C})$ (and so $M= H_\nu (\alpha_{\{g\}}| \beta_{\{g\}}\vee \mathcal{C})$ for each $g\in G$) and
\begin{equation*}
c= \lim_{n\rightarrow \infty} \frac{1}{|F_n|} H_\nu (\alpha_{F_n}| \beta_{F_n}\vee \mathcal{C}).
\end{equation*}
Observe that by Proposition \ref{1006122129} the limit $c$ must exist (using \eqref{star}).

Obviously, $c\ge h_\nu (G, \alpha| \beta_G\vee \mathcal{C})$.
To complete the proof, we only need show that $c\le h_\nu (G, \alpha| \beta_G\vee \mathcal{C})$.
The proof follows from the methods of \cite[Proposition 4.3]{WZ}.

Let $\epsilon\in (0, \frac{1}{4})$. Clearly, there exists $N\in \mathbb{N}$ such that if $n> N$ then
\begin{equation} \label{1007021902}
|\frac{1}{|F_n|} H_\nu (\alpha_{F_n}| \beta_{F_n}\vee \mathcal{C})- c|< \epsilon\ \text{and}\ |\frac{1}{|F_n|} H_\nu (\alpha_{F_n}| \beta_G\vee \mathcal{C})- h_\nu (\alpha| \beta_G\vee \mathcal{C})|< \epsilon.
\end{equation}
By Proposition \ref{ow-prop}, there exist integers $n_1, \cdots, n_k$ such that $N\le n_1< \cdots< n_k$ and $F_{n_1}, \cdots, F_{n_k}$ $\epsilon$-quasi-tile $F_m$ whenever $m$ is sufficiently large. Note that there must exist $B\in \mathcal{F}_G$ such that
\begin{equation} \label{1007021907}
H_\nu (\alpha_{F_{n_i}}| \beta_B\vee \mathcal{C})\le H_\nu (\alpha_{F_{n_i}}| \beta_G\vee \mathcal{C})+ \epsilon
\end{equation}
 for each $i= 1, \cdots, k$. Now let $m\in \mathbb{N}, m> N$ be large enough such that $F_m$ is $(B\cup \{e_G\}, \frac{\epsilon}{\sum\limits_{i= 1}^k |F_{n_i}|})$-invariant and $F_{n_1}, \cdots, F_{n_k}$ $\epsilon$-quasi-tile $F_m$ with tiling centers $C_1^m, \cdots,$ $C_k^m$. Then, by the selection of $C_1^m, \cdots, C_k^m$, one has
 \begin{enumerate}

\item \label{1007022102} for $A_m\doteq \{g\in F_m: B g\subseteq F_m\}= F_m\setminus B^{- 1} (G\setminus F_m)$, as $F_m\setminus A_m\subseteq F_m\cap B^{- 1} (G\setminus F_m)$ and $F_m$ is $(B\cup \{e_G\}, \frac{\epsilon}{\sum\limits_{i= 1}^k |F_{n_i}|})$-invariant, then
 \begin{equation*}
  |F_m\setminus A_m|< \frac{\epsilon |F_m|}{\sum\limits_{i= 1}^k |F_{n_i}|};
 \end{equation*}

\item \label{1007022105}
 $C_i^m\subseteq F_m, i= 1, \cdots, k$ (as $e_G\subseteq F_1\subseteq F_2\subseteq \cdots$) and

 \item
 \label{1007021916}
 $F_m\supseteq \bigcup\limits_{i= 1}^k F_{n_i} C_i^m\ \text{and}\ |\bigcup\limits_{i= 1}^k F_{n_i} C_i^m|\ge \max \{(1- \epsilon) |F_m|, (1- \epsilon) \sum\limits_{i= 1}^k |C_i^m| |F_{n_i}|\}$.
 \end{enumerate}
 Moreover, we have
\begin{eqnarray} \label{1007022115}
& &\hskip -36pt  \frac{1}{|F_m|} H_\nu (\alpha_{F_m}| \beta_{F_m}\vee \mathcal{C})\nonumber \\
&\le & \frac{1}{|F_m|} \{H_\nu (\alpha_{\bigcup\limits_{i= 1}^k F_{n_i} C_i^m}| \beta_{F_m}\vee \mathcal{C})+ H_\nu (\alpha_{F_m\setminus \bigcup\limits_{i= 1}^k F_{n_i} C_i^m}| \mathcal{C})\}\nonumber \\
&\le & \frac{1}{(1- \epsilon) \sum\limits_{i= 1}^k |C_i^m| |F_{n_i}|} \sum_{i= 1}^k H_\nu (\alpha_{F_{n_i} C_i^m}| \beta_{F_m}\vee \mathcal{C})+ \epsilon \log |\alpha|,
\end{eqnarray}
where the last inequality follows from the above \eqref{1007021916},
moreover, for each $i= 1, \cdots, k$,
\begin{eqnarray} \label{1007022116}
& & \hskip -36 pt \frac{1}{|C_i^m| |F_{n_i}|} H_\nu (\alpha_{F_{n_i} C_i^m}| \beta_{F_m}\vee \mathcal{C})\nonumber \\
&\le & \frac{1}{|C_i^m|} \sum_{c\in C_i^m} \frac{1}{|F_{n_i}|} H_\nu (\alpha_{F_{n_i} c}| \beta_{F_m}\vee \mathcal{C})\nonumber \\
&= & \frac{1}{|C_i^m|} \sum_{c\in C_i^m} \frac{1}{|F_{n_i}|} H_\nu (\alpha_{F_{n_i}}| \beta_{F_m c^{- 1}}\vee \mathcal{C})\nonumber \\
&\le & \frac{1}{|C_i^m|} \{\sum_{c\in C_i^m\cap A_m} \frac{1}{|F_{n_i}|} H_\nu (\alpha_{F_{n_i}}| \beta_{F_m c^{- 1}}\vee \mathcal{C})+\nonumber \\
& & \hskip 66pt \sum_{c\in C_i^m\setminus A_m} \frac{1}{|F_{n_i}|} H_\nu (\alpha_{F_{n_i}}| \beta_{F_m c^{- 1}}\vee \mathcal{C})\}\nonumber \\
&\le & \frac{1}{|F_{n_i}|} H_\nu (\alpha_{F_{n_i}}| \beta_B\vee \mathcal{C})+ \frac{1}{|C_i^m|} \sum_{c\in F_m\setminus A_m} \frac{1}{|F_{n_i}|} H_\nu (\alpha_{F_{n_i}}| \beta_{F_m c^{- 1}}\vee \mathcal{C})\nonumber \\
& & \hskip 66pt (\text{by the selection of $A_m$ and the above \eqref{1007022105}})\nonumber \\
&\le & \frac{1}{|F_{n_i}|} H_\nu (\alpha_{F_{n_i}}| \beta_G\vee \mathcal{C})+ \epsilon+ \frac{|F_m\setminus A_m|}{|C_i^m|} \log |\alpha|\ (\text{using \eqref{1007021907}}).
\end{eqnarray}
Combining \eqref{1007022115} and \eqref{1007022116}, we obtain
\begin{eqnarray*}
& &\hskip -50pt \frac{1}{|F_m|} H_\nu (\alpha_{F_m}| \beta_{F_m}\vee \mathcal{C}) \\
&\le & \frac{1}{1- \epsilon} \sum_{i= 1}^k \frac{|C_i^m| |F_{n_i}|}{\sum\limits_{j= 1}^k |C_j^m| |F_{n_j}|} \{\frac{1}{|F_{n_i}|} H_\nu (\alpha_{F_{n_i}}| \beta_G\vee \mathcal{C})+ \nonumber \\
& & \hskip 66pt \epsilon+ \frac{|F_m\setminus A_m|}{|C_i^m|} \log |\alpha|\}+ \epsilon \log |\alpha| \\
&\le & \frac{1}{1- \epsilon} \{\max_{1\le i\le k} \frac{1}{|F_{n_i}|} H_\nu (\alpha_{F_{n_i}}| \beta_G\vee \mathcal{C})+ \nonumber \\
& & \hskip 66pt \epsilon+ \frac{\epsilon |F_m|}{\sum\limits_{i= 1}^k |C_i^m| |F_{n_i}|} \log |\alpha|\}+\epsilon \log |\alpha|\ (\text{using \eqref{1007022102}}) \\
&\le & \frac{1}{1- \epsilon} \max_{1\le i\le k} \frac{1}{|F_{n_i}|} H_\nu (\alpha_{F_{n_i}}| \beta_G\vee \mathcal{C})+ \nonumber \\
& & \hskip 66pt\frac{1}{1- \epsilon} (\epsilon+ \frac{\epsilon}{1- \epsilon} \log |\alpha|)+ \epsilon \log |\alpha|\ (\text{using \eqref{1007021916}}).
\end{eqnarray*}
Combining this with \eqref{1007021902}, one has
\begin{equation*}
c< \frac{1}{1- \epsilon} h_\nu (G, \alpha| \beta_G\vee \mathcal{C})+ \frac{1}{1- \epsilon} (2 \epsilon+ \frac{\epsilon}{1- \epsilon} \log |\alpha|)+ \epsilon (1+ \log |\alpha|).
\end{equation*}
Finally, $c\le h_\nu (G, \alpha| \beta_G\vee \mathcal{C})$ follows by letting $\epsilon\rightarrow 0$. This finishes our proof.
\end{proof}

\begin{rem} \label{1010241458}
Remark that the case where $(Y, \mathcal{D}, \nu)$ is a Lebesgue space was proved by Glasner, Thouvenot and Weiss \cite[Lemma 1.1]{GTW}. The relative Pinsker formula for a measurable dynamical $\Z$-system is proved as \cite[Theorem 3.3]{ZJLMS}.
\end{rem}

As a direct corollary of Theorem \ref{1007031233}, we can obtain the well-known Abramov-Rokhlin entropy addition formula (see for example \cite[Theorem 0.2]{D} or \cite{WZ}), which will be used in our discussions of continuous bundle random dynamical systems.

\begin{prop} \label{1006291603}
Let $(Y, \mathcal{D}, \nu, G)$ be an MDS and $\mathcal{C}_1\subseteq \mathcal{C}_2\subseteq \mathcal{D}$ two $G$-invariant sub-$\sigma$-algebras. Assume that $(Y, \mathcal{D}, \nu)$ is a Lebesgue space. Then
$$h_\nu (G, Y| \mathcal{C}_1)= h_\nu (G, Y| \mathcal{C}_2)+ h_\nu (G, Y, \mathcal{C}_2| \mathcal{C}_1).$$
Here, $h_\nu (G, Y, \mathcal{C}_2| \mathcal{C}_1)$ denotes the measure-theoretic $\nu$-entropy of the MDS $(Y, \mathcal{C}_2, \nu,$ $G)$ with respect to $\mathcal{C}_1$.
\end{prop}

Let $(Y, \mathcal{D}, \nu, G)$ be an MDS and $\mathcal{C}\subseteq \mathcal{D}$ a $G$-invariant sub-$\sigma$-algebra. Define the \emph{Pinsker $\sigma$-algebra of $(Y, \mathcal{D}, \nu, G)$ with respect to $\mathcal{C}$}, $\mathcal{P}^\mathcal{C} (Y, \mathcal{D}, \nu, G)$, to be the sub-$\sigma$-algebra of $\mathcal{D}$ generated by $\{\alpha\in \mathbf{P}_Y: h_\nu (G, \alpha| \mathcal{C})= 0\}$. In the case of $\mathcal{C}= \mathcal{N}_Y$ we will write $\mathcal{P} (Y, \mathcal{D}, \nu, G)= \mathcal{P}^{\mathcal{N}_Y} (Y, \mathcal{D}, \nu, G)$ and call it the \emph{Pinsker $\sigma$-algebra of $(Y, \mathcal{D}, \nu, G)$}. Obviously $\mathcal{P}^\mathcal{C} (Y, \mathcal{D}, \nu, G)\subseteq \mathcal{D}$ is a $G$-invariant sub-$\sigma$-algebra and $\mathcal{C}\cup \mathcal{P} (Y, \mathcal{D}, \nu, G)\subseteq \mathcal{P}^\mathcal{C} (Y, \mathcal{D}, \nu, G)$.

We say that $(Y, \mathcal{D}, \nu, G)$ has \emph{$\mathcal{C}$-relative c.p.e.} if $\mathcal{P}^\mathcal{C} (Y, \mathcal{D}, \nu, G)= \mathcal{C}$ (in the sense of mod $\nu$), and has \emph{c.p.e.} if it has $\mathcal{N}_Y$-relative c.p.e.

The following result was proved in \cite{D, RW}.

\begin{prop} \label{1007012030}
Let $(Y, \mathcal{D}, \nu, G)$ be an MDS and $\mathcal{C}\subseteq \mathcal{D}$ a $G$-invariant sub-$\sigma$-algebra. Assume that $(Y, \mathcal{D}, \nu)$ is a Lebesgue space. Then $(Y, \mathcal{D}, \nu, G)$ has $\mathcal{C}$-relative c.p.e. if and only if for each $\alpha\in \mathbf{P}_Y$ and any $\epsilon> 0$ there exists $K\in \mathcal{F}_G$ such that if $F\in \mathcal{F}_G$ satisfies $F F^{- 1}\cap (K\setminus \{e_G\})= \emptyset$ then
\begin{equation*}
|\frac{1}{|F|} H_\nu (\alpha_F| \mathcal{C})- H_\nu (\alpha| \mathcal{C})|< \epsilon.
\end{equation*}
\end{prop}

In the proof of Theorem \ref{1006301434}, we will use the following well-known result.

\begin{prop} \label{1007012034}
Let $(Y, \mathcal{D}, \nu, G)$ be an MDS, $\mathcal{C}\subseteq \mathcal{D}$ a $G$-invariant sub-$\sigma$-algebra and $\alpha\in \mathbf{P}_Y$. Assume that $(Y, \mathcal{D}, \nu)$ is a Lebesgue space. Then
\begin{equation} \label{1007012036}
h_\nu (G, \alpha| \mathcal{C})= h_\nu (G, \alpha| \mathcal{P}^\mathcal{C} (Y, \mathcal{D}, \nu, G)).
\end{equation}
In particular,
$(Y, \mathcal{D}, \nu, G)$ has $\mathcal{P}^\mathcal{C} (Y, \mathcal{D}, \nu, G)$-relative c.p.e.
\end{prop}

Let $(Y, \mathcal{D}, \nu, G)$ be an MDS and $\mathcal{C}\subseteq \mathcal{D}$ a $G$-invariant sub-$\sigma$-algebra. For each $n\in \mathbb{N}\setminus \{1\}$, following ideas from \cite{G, HY, HYZ2, HYZ}, we introduce a probability measure $\lambda_n^\mathcal{C} (\nu)$ over $(Y^n, \mathcal{D}^n)$ as follows:
\begin{equation} \label{1208032307}
\lambda_n^\mathcal{C} (\nu) (\prod_{i= 1}^n A_i)= \int_Y \prod_{i= 1}^n \nu (A_i| \mathcal{P}^\mathcal{C} (Y, \mathcal{D}, \nu, G)) d \nu,
\end{equation}
whenever $A_1, \cdots, A_n\in \mathcal{D}$.
Here, $Y^n= Y\times \cdots\times Y$ ($n$-times) and $\mathcal{D}^n= \mathcal{D}\times \cdots\times \mathcal{D}$ ($n$-times).
 As $G$ acts naturally on $(Y^n, \mathcal{D}^n)$, it is not hard to check that the measure $\lambda_n^\mathcal{C} (\nu)$ is $G$-invariant (recall that the sub-$\sigma$-algebra $\mathcal{P}^\mathcal{C} (Y, \mathcal{D}, \nu, G)\subseteq \mathcal{D}$ is $G$-invariant) and so $(Y^n, \mathcal{D}^n, \lambda_n^\mathcal{C} (\nu), G)$ forms an MDS.

Following the method of proof of \cite[Lemma 6.8 and Theorem 6.11]{HYZ}, it is not hard to obtain:

\begin{lem} \label{1007032027}
Let $(Y, \mathcal{D}, \nu, G)$ be an MDS, $\mathcal{C}\subseteq \mathcal{D}$ a $G$-invariant sub-$\sigma$-algebra and $\mathcal{W}= \{W_1, \cdots,$ $W_n\}\in \mathbf{C}_Y$ with $n>1$. Then
\begin{enumerate}

\item \label{1007032038n} $\lambda_n^\mathcal{C} (\nu) (\prod\limits_{i= 1}^n W_i^c)> 0$ if and only if $h_\nu (G, \beta| \mathcal{C})> 0$ whenever $\beta\in \mathbf{P}_Y$ satisfies $\beta\succeq \mathcal{W}$.

\item \label{1007032038} if $\lambda_n^\mathcal{C} (\nu) (\prod\limits_{i= 1}^n W_i^c)> 0$ then there exist $\epsilon> 0$ and $\alpha\in \mathbf{P}_Y$ such that $\alpha\succeq \mathcal{W}$ and, whenever $\beta\in \mathbf{P}_Y$ satisfies $\beta\succeq \mathcal{W}$,
    $$H_\nu (\alpha| \beta\vee \mathcal{P}^\mathcal{C} (Y, \mathcal{D}, \nu, G))\le H_\nu (\alpha| \mathcal{P}^\mathcal{C} (Y, \mathcal{D}, \nu, G))- \epsilon.$$
\end{enumerate}
\end{lem}

Lemma \ref{1007032027} \eqref{1007032038n} can be strengthened as follows. We will use this version in our discussions of entropy tuples for a continuous bundle random dynamical system in \S \ref{entropy tuple}.

\begin{thm} \label{1006301434}
Let $(Y, \mathcal{D}, \nu, G)$ be an MDS, $\mathcal{C}\subseteq \mathcal{D}$ a $G$-invariant sub-$\sigma$-algebra and $\mathcal{W}= \{W_1, \cdots, W_n\}\in \mathbf{C}_Y$ with $n\in \mathbb{N}\setminus \{1\}$. Assume that $(Y, \mathcal{D}, \nu)$ is a Lebesgue space. Then the following statements are equivalent:
\begin{enumerate}

\item $h_\nu (G, \beta| \mathcal{C})> 0$ whenever $\beta\in \mathbf{P}_Y$ satisfies $\beta\succeq \mathcal{W}$.

\item \label{1007032101} $\lambda_n^\mathcal{C} (\nu) (\prod\limits_{i= 1}^n W_i^c)> 0$.

\item \label{1007032102} $\inf\limits_{F\in \mathcal{F}_G} \frac{1}{|F|} H_\nu (\mathcal{W}_F| \mathcal{C})> 0$.

\item $h_\nu (G, \mathcal{W}| \mathcal{C})> 0$.
\end{enumerate}
\end{thm}
\begin{proof}
The equivalence $(1)\Longleftrightarrow$\eqref{1007032101} is established by Lemma \ref{1007032027} and the implications \eqref{1007032102}$\Longrightarrow (4)\Longrightarrow (1)$ follow directly from the definitions.

Thus, it suffices to prove \eqref{1007032101}$\Longrightarrow$\eqref{1007032102}.

Now assume that $\lambda_n^\mathcal{C} (\nu) (\prod\limits_{i= 1}^n W_i^c)> 0$. Using Lemma \ref{1007032027} again, there exist $\alpha\in \mathbf{P}_Y$ and $\epsilon> 0$ such that
\begin{equation} \label{1007032137}
H_\nu (\alpha| \beta\vee \mathcal{P}^\mathcal{C} (Y, \mathcal{D}, \nu, G))\le H_\nu (\alpha| \mathcal{P}^\mathcal{C} (Y, \mathcal{D}, \nu, G))- \epsilon
\end{equation}
whenever $\beta\in \mathbf{P}_Y$ satisfies $\beta\succeq \mathcal{W}$.
By Proposition \ref{1007012030} and Proposition \ref{1007012034}, we can choose $K\in
\mathcal{F}_G$ such that if $F\in \mathcal{F}_G$ satisfies $F F^{-1}\cap (K\setminus \{e_G\})= \emptyset$ then
\begin{equation} \label{0911151253}
|\frac{1}{|F|} H_\nu (\alpha_F| \mathcal{P}^\mathcal{C} (Y, \mathcal{D}, \nu, G))-
H_\nu (\alpha| \mathcal{P}^\mathcal{C} (Y, \mathcal{D}, \nu, G))|< \frac{\epsilon}{2}.
\end{equation}
For $E\in \mathcal{F}_G$ and $g\in E$, there exists $S\in \mathcal{F}_G$ such that $S S^{- 1}\cap (K\setminus \{e_G\})= \emptyset, g\in S\subseteq E$ and
$(S\cup \{g'\}) (S\cup \{g'\})^{- 1}\cap (K\setminus \{e_G\})\neq \emptyset$ for any $g'\in E\setminus S$.
Thus,
\begin{equation} \label{0911151315}
|\frac{1}{|S|} H_\nu (\alpha_S| \mathcal{P}^\mathcal{C} (Y, \mathcal{D}, \nu, G))-
H_\nu (\alpha| \mathcal{P}^\mathcal{C} (Y, \mathcal{D}, \nu, G))|< \frac{\epsilon}{2}\ (\text{using \eqref{0911151253}}).
\end{equation}
It is now not hard to check that
$$E\setminus S\subseteq (K\setminus \{e_G\}) S\cup (K\setminus \{e_G\})^{- 1} S= (K\cup K^{- 1}\setminus \{e_G\}) S,$$
 hence $S\subseteq E\subseteq (K\cup K^{- 1}\cup \{e_G\}) S$, one has $(2 |K|+ 1) |S|\ge |E|$.
So, if $\beta\in \mathbf{P}_Y$ satisfies $\beta\succeq \mathcal{W}_S$ then $g \beta\succeq \mathcal{W}$ for each $g\in S$, hence
\begin{eqnarray*}
& & H_\nu (\beta| \mathcal{C})\nonumber \\
& &\ge H_\nu (\beta| \mathcal{P}^\mathcal{C} (Y, \mathcal{D}, \nu, G))\nonumber \\
& &= H_\nu (\beta\vee \alpha_S| \mathcal{P}^\mathcal{C} (Y, \mathcal{D}, \nu, G))-
H_\nu (\alpha_S| \beta\vee \mathcal{P}^\mathcal{C} (Y, \mathcal{D}, \nu, G))\nonumber \\
& &\ge H_\nu (\alpha_S| \mathcal{P}^\mathcal{C} (Y, \mathcal{D}, \nu, G))- \sum_{g\in S}
H_\nu (\alpha| g \beta\vee \mathcal{P}^\mathcal{C} (Y, \mathcal{D}, \nu, G)) \nonumber \\
& &\ge H_\nu (\alpha_S| \mathcal{P}^\mathcal{C} (Y, \mathcal{D}, \nu, G))-
|S| (H_\nu (\alpha| \mathcal{P}^\mathcal{C} (Y, \mathcal{D}, \nu, G))- \epsilon)\ (\text{using \eqref{1007032137}})\nonumber \\
& &\ge
\frac{|S| \epsilon}{2}\
(\text{using \eqref{0911151315}}).
\end{eqnarray*}
Since $\beta$ is arbitrary,
\begin{equation} \label{1207311055}
H_\nu (\mathcal{W}_E)\ge H_\nu (\mathcal{W}_S)\ge \frac{|S| \epsilon}{2}.
\end{equation}
Recall that $(2 |K|+ 1) |S|\ge |E|$. In \eqref{1207311055} letting $E$ vary over all elements from $\mathcal{F}_G$ we obtain \eqref{1007032102}. This completes the proof.
\end{proof}

\begin{ques} \label{1102041455}
Let $(Y, \mathcal{D}, \nu, G)$ be an MDS, $\mathcal{C}\subseteq \mathcal{D}$ a $G$-invariant sub-$\sigma$-algebra and $\mathcal{W}\in \mathbf{C}_Y$. We conjecture that the following equation holds:
\begin{equation*}
h_\nu (G, \mathcal{W}| \mathcal{C})= \inf\limits_{F\in \mathcal{F}_G} \frac{1}{|F|} H_\nu (\mathcal{W}_F| \mathcal{C}).
\end{equation*}
\begin{enumerate}

\item Unfortunately, the proof of \eqref{1006272232} does not work in this case, since if $\alpha\in \mathbf{P}_Y$,  \eqref{1007021626} easily implies strong sub-additivity of
  \begin{equation} \label{1102241455}
  H_\nu (\alpha_{E\cap F}| \mathcal{C})+ H_\nu (\alpha_{E\cup F}| \mathcal{C})\le H_\nu (\alpha_E| \mathcal{C})+ H_\nu (\alpha_F| \mathcal{C})
  \end{equation}
    whenever $E, F\in \mathcal{F}_G$ (setting $\alpha_\emptyset= \mathcal{N}_Y$). However, we do not know whether \eqref{1102241455} holds for a general cover $\mathcal{W}\in \mathbf{C}_Y$.

\item
From the definitions, the inequality $\ge$ holds directly. Moreover, by Theorem \ref{1006301434}, if $(Y, \mathcal{D}, \nu)$ is a Lebesgue space then
$$\inf\limits_{F\in \mathcal{F}_G} \frac{1}{|F|} H_\nu (\mathcal{W}_F| \mathcal{C})> 0\ \text{if and only if}\ h_\nu (G, \mathcal{W}| \mathcal{C})> 0.$$

\item The conjecture should be compared with Proposition \ref{1102111944}, Proposition \ref{p1006172118} and Example \ref{1104071212}.

\end{enumerate}
Observe that in  the topological setting, we have a similar result \cite[Lemma 6.1]{DZ}, and a similar conjecture can be made.
\end{ques}

Let $(Y, \mathcal{D}, \nu)$ be a Lebesgue space. If $\{\alpha_i: i\in I\}$ is a countable family in $\mathbf{P}_Y$, the
partition $\alpha= \bigvee\limits_{i\in I} \alpha_i\doteq \{\bigcap\limits_{i\in
I} A_i: A_i\in \alpha_i, i\in I\}$ is called a \emph{measurable
partition}. Note that the sets $C\in \mathcal{D}$, which are
unions of atoms of $\alpha$, form a sub-$\sigma$-algebra of
$\mathcal{D}$. In fact, every
sub-$\sigma$-algebra of $\mathcal{D}$ coincides with a
$\sigma$-algebra constructed in this way modulo $\nu$-null sets
(cf \cite{Rokhlin}).

Now let $\mathcal{C}\subseteq \mathcal{D}$ be a sub-$\sigma$-algebra. Then we may disintegrate $\nu$ over $\mathcal{C}$, i.e. we write $\nu= \int_Y \nu_y d \nu (y)$, where $\nu_y$ is a probability measure over $(Y, \mathcal{D})$ for $\nu$-a.e. $y\in Y$. In fact, if $\alpha$ is a measurable partition of $(Y, \mathcal{D}, \nu)$ which generates $\mathcal{C}$, then, for $\nu$-a.e. $y\in Y$, $\nu_y$ is supported on $\alpha (y)$ (i.e. $\nu_y (\alpha (y))= 1$) and $\nu_{y_1}= \nu_{y_2}$ for $\nu_y$-a.e. $y_1, y_2\in \alpha (y)$. The disintegration can be characterized as follows. For each $f\in L^1 (Y, \mathcal{D}, \nu)$, if we denote by $\nu (f| \mathcal{C})$ the conditional expectation with respect to $\nu$ of the function $f$ relative to $\mathcal{C}$, then: $f\in L^1 (Y, \mathcal{D}, \nu_y)$ for $\nu$-a.e. $y\in Y$,
 the function $y\mapsto \int_Y f d \nu_y$ is in $L^1 (Y, \mathcal{C}, \nu)$, and
 $\nu (f| \mathcal{C}) (y)= \int_Y f d \nu_y$ for $\nu$-a.e. $y\in Y$.
From this, it follows that if $f\in L^1 (Y, \mathcal{D}, \nu)$ then
\begin{equation} \label{1006281756}
\int_Y (\int_Y f d \nu_y) d \nu (y)= \int_Y f d \nu,
\end{equation}
and so it is simple to check that if $\beta\in \mathbf{P}_Y$ then
\begin{equation} \label{1006281842}
H_\nu (\beta| \mathcal{C})= \int_Y H_{\nu_y} (\beta) d \nu (y).
\end{equation}
Note that the disintegration is unique in the sense that if
$\nu= \int_Y \nu_y d \nu (y)$ and $\nu= \int_Y \nu'_y d \nu (y)$
are both the disintegrations of $\nu$ over $\mathcal{C}$, then $\nu_y=
\nu'_y$ for $\nu$-a.e. $y\in Y$. For details see for example \cite{Fu, Rokhlin}.

Now let $(Y, \mathcal{D}, \nu, G)$ be an MDS and $\mathcal{C}\subseteq \mathcal{D}$ a $G$-invariant sub-$\sigma$-algebra.
Assume that $(Y, \mathcal{D}, \nu)$ is a Lebesgue space and $\nu= \int_Y \nu_y d \nu (y)$ is the disintegration of $\nu$ over $\mathcal{P}^\mathcal{C} (Y, \mathcal{D}, \nu, G)$. Then, for each $n\in \mathbb{N}\setminus \{1\}$, by the construction of $\lambda_n^\mathcal{C} (\nu)$ one has:
\begin{equation*}
\lambda_n^\mathcal{C} (\nu)= \int_Y \nu_y\times \cdots\times \nu_y\ (\text{$n$-times})\ d \nu (y).
\end{equation*}

We need the following result in next section which is similar to \cite[Lemma 3.8]{HYZ2}.

\begin{lem} \label{1006291630}
Let $(Y, \mathcal{D}, \nu)$ be a Lebesgue space and $\mathcal{W}\in \mathbf{C}_Y$. Let  $\mathcal{C}\subseteq \mathcal{D}$ be a sub-$\sigma$-algebra and $\nu= \int_Y \nu_y d \nu (y)$ the disintegration of $\nu$ over $\mathcal{C}$. Then
$$H_\nu (\mathcal{W}| \mathcal{C})= \int_Y H_{\nu_y} (\mathcal{W}) d \nu (y).$$
\end{lem}

A probability space $(Y, \mathcal{D}, \nu)$ is \emph{purely atomic} if there exists a countably family $\{D_i: i\in I\}\subseteq \mathcal{D}$ such that $\nu (\bigcup\limits_{i\in I} D_i)= 1$ and for each
$i\in I$, $\nu (D_i)> 0$ and if $D_i'\subseteq D_i$ is measurable then $\nu (D_i')$ is either $0$ or
$\nu (D_i)$.

We will also need the following result in Part \ref{skdj}.

\begin{prop} \label{0811192316}
Let $(Y, \mathcal{D}, \nu, G)$ be an MDS and $\mathcal{C}\subseteq \mathcal{D}$ a $G$-invariant sub-$\sigma$-algebra. Assume that $(Y, \mathcal{D}, \nu)$ is a Lebesgue space and
$\nu= \int_Y \nu_y d \nu (y)$ is the disintegration of $\nu$ over
$\mathcal{C}$. If $\nu_y$ is purely atomic for $\nu$-a.e. $y\in Y$ then
$h_\nu (G, Y| \mathcal{C})= 0$.
 Conversely, if $h_\nu (G, Y| \mathcal{C})> 0$ then there is $A\in \mathcal{D}$ such that $\nu (A)> 0$ and $\nu_y$
is not purely atomic for each $y\in A$.
\end{prop}

The case where $\nu$ is ergodic in Proposition \ref{0811192316} is well known (see \cite[Theorem 4.1.15]{Dbook} for a stronger version), and then it is not hard to obtain Proposition \ref{0811192316} in the general case by applying ergodic decomposition.

\section{Continuous bundle random dynamical systems} \label{third}

In this section we define and establish basic properties of a continuous bundle random dynamical system associated to an infinite countable discrete amenable group action, given some known results for the special case of $\Z$ from \cite{Bog1, K1, KL1, Liu}.

\medskip

For convenience we restrict our setting as follows.

\medskip

\medskip

{\noindent \bf Standard Assumption 3. \it
Let $(\Omega, \mathcal{F}, \mathbb{P}, G)$ denote an MDS, where $(\Omega, \mathcal{F}, \mathbb{P})$ is a Lebesgue space. In particular, $(\Omega, \mathcal{F}, \mathbb{P})$ is complete and countably separated.}

\medskip

\medskip

Now let $(X, \mathcal{B})$ be a measurable space and $\mathcal{E}\in \mathcal{F}\times
\mathcal{B}$. Then $(\mathcal{E}, (\mathcal{F}\times
\mathcal{B})_\mathcal{E})$ forms naturally a measurable space with $(\mathcal{F}\times
\mathcal{B})_\mathcal{E}$ the $\sigma$-algebra of $\mathcal{E}$ given by restricting $\mathcal{F}\times \mathcal{B}$ over $\mathcal{E}$. Set $\mathcal{E}_\omega= \{x\in X: (\omega, x)\in \mathcal{E}\}$ for each $\omega\in \Omega$. A \emph{bundle random dynamical system} or \emph{random dynamical system} (RDS) \emph{associated to $(\Omega,
\mathcal{F}, \mathbb{P}, G)$} is a family $\mathbf{F}= \{F_{g, \omega}: \mathcal{E}_\omega\rightarrow \mathcal{E}_{g \omega}| g\in G, \omega\in \Omega\}$ satisfying:
\begin{enumerate}

\item for each $\omega\in \Omega$, the transformation $F_{e_G, \omega}$ is  the identity over $\mathcal{E}_\omega$,

\item for each $g\in G$, the map $(\mathcal{E}, (\mathcal{F}\times
\mathcal{B})_\mathcal{E})\rightarrow (X, \mathcal{B})$, given by $(\omega, x)\mapsto F_{g, \omega} (x)$, is measurable and

\item for each $\omega\in \Omega$ and all $g_1, g_2\in G$, $F_{g_2, g_1 \omega}\circ F_{g_1, \omega}= F_{g_2 g_1, \omega}$ (and so $F_{g^{- 1}, \omega}= (F_{g, g^{- 1} \omega})^{- 1}$ for each $g\in G$).
\end{enumerate}
In this case, $G$ has a natural measurable action on $\mathcal{E}$ with $(\omega, x)\rightarrow (g \omega, F_{g, \omega} x)$ for each $g\in G$, called the corresponding \emph{skew product transformation}.

Let the family $\mathbf{F}= \{F_{g, \omega}: \mathcal{E}_\omega\rightarrow \mathcal{E}_{g \omega}| g\in G, \omega\in \Omega\}$ be an RDS over $(\Omega,
\mathcal{F}, \mathbb{P}, G)$, where $X$ is a compact metric space with metric $d$ and equipped with the Borel $\sigma$-algebra $\mathcal{B}_X$. If for $\mathbb{P}$-a.e. $\omega\in \Omega$, $\emptyset\neq
\mathcal{E}_\omega\subseteq X$ is a compact subset and $F_{g, \omega}$ is
a continuous map for each $g\in G$ (and so $F_{g, \omega}: \mathcal{E}_\omega\rightarrow \mathcal{E}_{g \omega}$ is a homeomorphism for $\mathbb{P}$-a.e. $\omega\in \Omega$ and each $g\in G$), then it is called a \emph{continuous bundle RDS}.

By \cite[Chapter III]{CV}, the mapping
$\omega\mapsto \mathcal{E}_\omega$ is measurable with respect to
the Borel $\sigma$-algebra induced by the Hausdorff
topology on the hyperspace $2^X$ of all non-empty compact subsets of $X$, and
the distance function $d (x, \mathcal{E}_\omega)$ is measurable in
$\omega\in \Omega$ for each $x\in X$.

\medskip

These concepts generalize the classical concepts of dynamical systems as follows. We say that $(\Omega, \mathcal{F}, \mathbb{P}, G)$ is a \emph{trivial MDS}, if $\Omega$ is a singleton.

For a measurable space $(X, \mathcal{B})$, a bundle RDS associated to a trivial MDS is one where the group $G$ acts on a measurable subset $E\in \mathcal{B}$: that is, there exists a family of invertible measurable transformations $\{F_g: E\rightarrow E| g\in G\}$ such that $F_{g_2}\circ F_{g_1}= F_{g_2 g_1}$ for all $g_1, g_2\in G$ and $F_{e_G}$ acts as the identity over $E$.

For a compact metric space $X$, a continuous bundle RDS associated to a trivial MDS means that there is a \emph{topological $G$-action} $(K, G)$ for some non-empty compact subset $K\subseteq X$, that is, the group $G$ acts on $K$ in the sense that there exists a family of homeomorphisms $\{F_g: g\in G\}$ of $K$ such that   $F_{g_2}\circ F_{g_1}= F_{g_2 g_1}$ for all $g_1, g_2\in G$ and $F_{e_G}$ acts as the identity transformation over $K$. The pair $(K, G)$ is also called a \emph{topological dynamical $G$-system} (TDS).

\medskip

Among interesting examples of continuous bundle RDS's are random
sub-shifts.

In the case where $G= \Z$, these are treated in detail in \cite{BG, KK, K1}. We present a brief recall of some of their properties.

Let $(\Omega, \mathcal{F}, \mathbb{P})$ be a Lebesgue space and $\vartheta: (\Omega, \mathcal{F}, \mathbb{P})\rightarrow (\Omega, \mathcal{F}, \mathbb{P})$ an invertible measure-preserving transformation. Set $X=\{(x_i: i\in \Z): x_i\in \N\cup \{+ \infty\}, i\in \Z\}$, a compact metric space equipped with the metric
\begin{equation*}
d ((x_i: i\in \Z), (y_i: i\in \Z))= \sum_{i\in \Z} \frac{1}{2^{|i|}} |x_{i}^{-
1}- y_{i}^{- 1}|,
\end{equation*}
and let $F: X\rightarrow X$ be the translation $(x_i: i\in \Z)\mapsto (x_{i+ 1}: i\in \Z)$. Then the integer group $\Z$ acts on $(\Omega\times X, \mathcal{F}\times \mathcal{B}_X)$ measurably with $(\omega, x)\mapsto (\vartheta^i \omega, F^i x)$ for each $i\in \Z$, where $\mathcal{B}_X$ denotes the Borel $\sigma$-algebra of the space $X$.
Now let $\mathcal{E}\in \mathcal{F}\times \mathcal{B}_X$ be an invariant subset of $\Omega\times X$ (under the $\Z$-action) such that $\emptyset\neq \mathcal{E}_\omega\subseteq X$ is compact for
$\mathbb{P}$-a.e. $\omega\in \Omega$.
This defines a continuous bundle RDS where, for $\mathbb{P}$-a.e. $\omega\in \Omega$, $F_{i, \omega}$ is just the restriction of $F^i$ over $\mathcal{E}_\omega$ for each $i\in \Z$.

A very special case is when the subset $\mathcal{E}$ is given as follows. Let $k$ be a random
$\N$-valued random variable satisfying
\begin{equation*}
0< \int_\Omega \log k (\omega) d \mathbb{P} (\omega)< + \infty,
\end{equation*}
and, for $\mathbb{P}$-a.e. $\omega\in \Omega$, and
 let $M (\omega)$ be a random matrix $(m_{i, j} (\omega): i= 1, \cdots, k
(\omega), j$ $= 1, \cdots, k (\vartheta \omega))$ with entries $0$ and
$1$. Then the random variable $k$ and the random matrix $M$ generate a random sub-shift of finite type, where
\begin{equation*}
\mathcal{E}= \{(\omega, (x_i: i\in \mathbb{Z})): \omega\in \Omega, 1\le x_i\le k (\vartheta^i
\omega), m_{x_i, x_{i+ 1}} (\vartheta^i \omega)= 1, i\in \mathbb{Z}\}.
\end{equation*}
It is not hard to see that this is a continuous bundle RDS.

\medskip

There are many other interesting examples of continuous bundle RDS's coming from smooth ergodic theory, see for example \cite{LL, Liu2}, where one considers not only the action of $\Z$ on a compact metric state space but also $\Z_+$ on a Polish state space. (Recall that a \emph{Polish space} is a complete separable metric space.)

 Let $M$ be a $C^\infty$ compact connected Riemannian manifold without boundary and $C^r (M, M), r\in \Z_+\cup \{+ \infty\}$
the space of all $C^r$ maps from $M$ into itself endowed with the usual $C^r$ topology and the Borel $\sigma$-algebra. As above, $(\Omega, \mathcal{F}, \mathbb{P})$ is a Lebesgue space and $\vartheta: (\Omega, \mathcal{F}, \mathbb{P})\rightarrow (\Omega, \mathcal{F}, \mathbb{P})$ is an invertible measure-preserving transformation.
Now let $F: \Omega\rightarrow C^r (M, M)$ be a measurable map and define the family of the randomly composed maps $F_{n, \omega}, n\in \Z$ or $\Z_+$, $\omega\in \Omega$ as follows:
$$
F_{n, \omega}=
\begin{cases}
F (\vartheta^{n- 1} \omega)\circ \cdots\circ F (\vartheta \omega)\circ F (\omega),& \text{if}\ n> 0\\
id,& \text{if}\ n= 0\\
F (\vartheta^n \omega)^{- 1}\circ \cdots\circ F (\vartheta^{- 2} \omega)^{- 1}\circ F (\vartheta^{- 1} \omega)^{- 1}, & \text{if}\ n< 0
\end{cases}.
$$
Here $F_{n, \omega}, n< 0$ is defined when $F (\omega)\in \text{Diff}^r (M)$ for $\mathbb{P}$-a.e. $\omega\in \Omega$. In the case of $r= 0$ we may replace $M$ with a compact metric space.

\medskip

\medskip

{\noindent \bf Standard Assumption 4. \it
Henceforth, we will fix the family $\mathbf{F}= \{F_{g, \omega}: \mathcal{E}_\omega\rightarrow \mathcal{E}_{g \omega}| g\in G, \omega\in \Omega\}$ to be a continuous bundle RDS over $(\Omega,
\mathcal{F}, \mathbb{P}, G)$ with a compact metric space $(X, d)$ as its state space.}

\medskip

\medskip

As discussed in \S \ref{1006291551}, one can introduce $\mathbf{C}_\mathcal{E}, \mathbf{P}_\mathcal{E}$
and other related notations. Moreover, for $S\subseteq \mathcal{E}$, if for $\mathbb{P}$-a.e. $\omega\in \Omega$ all fibers $S_\omega\subseteq \mathcal{E}_\omega$ are open or closed,
then $S$ is called an \emph{open} or a \emph{closed random set}.
Denote by $\mathbf{C}^o_\mathcal{E}$ the set of all elements from $\mathbf{C}_\mathcal{E}$ consisting of subsets of open random sets.
Similarly, we can introduce $\mathbf{C}_X$, $\mathbf{P}_X$, $\mathbf{C}^o_X$ and other related notations. Moreover, for $\xi\in \mathbf{C}_\Omega$ and $\mathcal{W}\in \mathbf{C}_X$, we introduce the notation
$$(\xi\times \mathcal{W})_\mathcal{E}= \{(C\times W)\cap \mathcal{E}: C\in \xi, W\in \mathcal{W}\}\in \mathbf{C}_\mathcal{E}.$$
In special cases, we will write $(\Omega\times \mathcal{W})_\mathcal{E}= (\{\Omega\}\times \mathcal{W})_\mathcal{E}$ and $(\xi\times X)_\mathcal{E}= (\xi\times \{X\})_\mathcal{E}$.

In the sequel we will need the following result, which is a re-statement of \cite[Theorem III.23]{CV} and \cite[Theorem 1]{EV}.

Recall that by Standard Assumption 3,  $(\Omega, \mathcal{F}, \mathbb{P})$ is a Lebesgue space; in particular, it is a complete probability space, and by Standard Assumption 4, $X$ is a compact metric space. Thus, we can apply \cite[Theorem III.23]{CV} and \cite[Theorem 1]{EV} to $(\Omega, \mathcal{F}, \mathbb{P})$ and $X$, and obtain:

\begin{lem} \label{1007152201}
Let $\pi: \Omega\times X\rightarrow \Omega$ be the natural projection and $A\in \mathcal{F}\times \mathcal{B}_X$. Then
 $\pi (A)\in \mathcal{F}$, and there exists a measurable map $p: (\Omega, \mathcal{F})\rightarrow (X, \mathcal{B}_X)$ such that $(\omega, p (\omega))\in A$ for each $\omega\in \pi (A)$.
\end{lem}

 Denote by $\mathcal{P}_\mathbb{P} (\Omega\times X)$
the space of all probability measures on $\Omega\times X$ having marginal
$\mathbb{P}$ on $\Omega$. Every such a probability measure $\mu$ has the property that $\mu (A\times X)= \mathbb{P} (A)$ for each $A\in \mathcal{F}$. Put
$\mathcal{P}_\mathbb{P} (\mathcal{E})=
\{\mu\in \mathcal{P}_\mathbb{P} (\Omega\times X): \mu (\mathcal{E})= 1\}$. It is well known that $\mathcal{P}_\mathbb{P} (\mathcal{E})\neq \emptyset$ under our standard assumptions.

Let $\mathcal{F}_\mathcal{E}$ be the $\sigma$-algebra of all sets of the form $(A\times X)\cap \mathcal{E}, A\in \mathcal{F}$.
Note that each $\mu\in \mathcal{P}_\mathbb{P} (\mathcal{E})$ can be disintegrated
as
$$d \mu (\omega, x)= d \mu_\omega (x) d \mathbb{P} (\omega),$$
where $\mu_\omega, \omega\in \Omega$ are
regular conditional probability measures with respect to the $\sigma$-algebra
$\mathcal{F}_\mathcal{E}$, that is, for $\mathbb{P}$-a.e. $\omega\in \Omega$, $\mu_\omega$ is a Borel probability measure on $\mathcal{E}_\omega$
 and, for any measurable subset $R\subseteq \mathcal{E}$,
\begin{equation} \label{1006271735}
\mu_\omega (R_\omega)= \mu (R| \mathcal{F}_\mathcal{E}) (\omega)
\end{equation}
 where $R_\omega= \{x\in X: (\omega, x)\in R\}$. It follows that
\begin{equation*} \label{1006271736}
\mu (R)=
\int_\Omega \mu_\omega (R_\omega) d \mathbb{P} (\omega).
\end{equation*}
Since $X$ is a compact metric space, such a disintegration of $\mu$ exists (\cite[Proposition 10.2.8]{Du}).

Now let $\mu\in \mathcal{P}_\mathbb{P} (\mathcal{E})$ with disintegration $d \mu (\omega, x)= d \mu_\omega (x) d \mathbb{P} (\omega)$ as above. Assume that $\alpha\in \mathbf{P}_\mathcal{E}$ and $\mathcal{U}\in \mathbf{C}_\mathcal{E}$.
Then
\begin{eqnarray}
\hskip 16pt H_\mu (\alpha| \mathcal{F}_\mathcal{E})&= & -\int_\Omega \sum_{A\in \alpha}
\mu (A| \mathcal{F}_\mathcal{E}) (\omega)\log \mu (A|
\mathcal{F}_\mathcal{E}) (\omega) d \mathbb{P} (\omega)\nonumber \label{1006291626} \\
&= & \int_\Omega H_{\mu_\omega}
 (\alpha_\omega) d \mathbb{P} (\omega)\ (\text{using \eqref{1006271735}}), \label{0906282007}
\end{eqnarray}
here, $\alpha_\omega= \{A_\omega: A\in \alpha\}$ is a partition of
$\mathcal{E}_\omega$. In fact, by Lemma \ref{1006291630} we have
\begin{equation} \label{1007041151}
H_\mu (\mathcal{U}| \mathcal{F}_\mathcal{E})= \int_\Omega H_{\mu_\omega} (\mathcal{U}_\omega) d \mathbb{P} (\omega).
\end{equation}
Note that for each $F\in \mathcal{F}_G$ and for any $\omega\in \Omega$, one has
\begin{equation} \label{1006131722}
(\mathcal{U}_F)_\omega= \bigvee_{g\in F} (g^{- 1} \mathcal{U})_\omega= \bigvee_{g\in F} (F_{g, \omega})^{- 1} \mathcal{U}_{g \omega}= \bigvee_{g\in F} F_{g^{- 1}, g \omega} \mathcal{U}_{g \omega},
\end{equation}
and so, in view of \eqref{1007041151},
\begin{equation} \label{0906282008}
H_\mu (\mathcal{U}_F|
\mathcal{F}_\mathcal{E})= \int_\Omega H_{\mu_\omega} \left(\bigvee_{g\in F} F_{g^{- 1}, g \omega} \mathcal{U}_{g \omega}\right) d \mathbb{P} (\omega).
\end{equation}
Moreover, for any $\omega\in \Omega$, denote by $N (\mathcal{U}, \omega)$ the minimal cardinality of a sub-family of $\mathcal{U}_\omega$ covering $\mathcal{E}_\omega$ (equivalently, the minimal cardinality of a sub-family of $\mathcal{U}$ covering $\mathcal{E}_\omega$), it is easy to check $H_{\mu_\omega} (\mathcal{U}_\omega)\le \log N (\mathcal{U}, \omega)$.

Then we have:

\begin{prop} \label{1101311055}
Let $\mathcal{U}\in \mathbf{C}_\mathcal{E}$. Then $N (\mathcal{U}, \omega)$ is measurable in $\omega\in \Omega$, and
\begin{equation} \label{1007121506}
H_\mu (\mathcal{U}| \mathcal{F}_\mathcal{E})\le \int_\Omega \log N (\mathcal{U}, \omega) d \mathbb{P} (\omega).
\end{equation}
\end{prop}
\begin{proof}
We will call $\pi: \mathcal{E}\rightarrow \Omega$ the \emph{natural projection}.

Let $n\in \N$. Then $N (\mathcal{U}, \omega)\le n$ if and only if there exists $U_1, \cdots, U_n$ from $\mathcal{U}$ such that $\mathcal{E}_\omega\subseteq \bigcup\limits_{i= 1}^n U_i$. Equivalently, $\omega\notin \pi (\mathcal{E}\setminus \bigcup\limits_{i= 1}^n U_i)$. Observe that for given $U_1, \cdots, U_n$ the subset $\pi (\mathcal{E}\setminus \bigcup\limits_{i= 1}^n U_i)$ is measurable from Lemma \ref{1007152201}. From this it is easy to see that $N (\mathcal{U}, \omega)$ is measurable in $\omega\in \Omega$, and hence we obtain the inequality \eqref{1007121506}.
\end{proof}

Recall that $\mathcal{F}_\mathcal{E}= \{(A\times X)\cap \mathcal{E}: A\in \mathcal{F}\}$, in particular, $\mathcal{F}_\mathcal{E}\subseteq (\mathcal{F}\times \mathcal{B}_X)\cap \mathcal{E}$ is a $G$-invariant sub-$\sigma$-algebra with respect to the skew product transformation $(\mathcal{E}, (\mathcal{F}\times \mathcal{B}_X)\cap \mathcal{E}, G)$.
 It is not hard to check that, for
 $\mu\in \mathcal{P}_\mathbb{P} (\mathcal{E})$ with $d \mu (\omega, x)= d \mu_\omega (x) d \mathbb{P} (\omega)$ as its disintegration, $\mu$ is $G$-invariant
 if and only if $F_{g, \omega} \mu_\omega= \mu_{g
\omega}$ for $\mathbb{P}$-a.e. $\omega\in \Omega$ and each $g\in G$, here $F_{g, \omega} \mu_\omega (\bullet)$ is given by $\mu_\omega (F_{g, \omega}^{- 1} \bullet)$.

 Denote by $\mathcal{P}_\mathbb{P}
(\mathcal{E}, G)$ the set of all $G$-invariant elements from $\mathcal{P}_\mathbb{P}
(\mathcal{E})$. Observe that if $\mu\in \mathcal{P}_\mathbb{P} (\mathcal{E}, G)$ is ergodic then $(\Omega, \mathcal{F}, \mathbb{P}, G)$ is also ergodic. Conversely, if $(\Omega, \mathcal{F}, \mathbb{P}, G)$ is not ergodic then each element from $\mathcal{P}_\mathbb{P} (\mathcal{E}, G)$ is also not ergodic.

Note that, as $G$ is amenable, any topological $G$-action admits a $G$-invariant Borel probability measure on the compact metric state space. Hence by the observation made at the beginning of this section, one has $\mathcal{P}_\mathbb{P} (\mathcal{E}, G)\neq \emptyset$ when $(\Omega, \mathcal{F}, \mathbb{P}, G)$ is a trivial MDS. In fact, $\mathcal{P}_\mathbb{P} (\mathcal{E}, G)\neq \emptyset$ still holds even if the MDS $(\Omega, \mathcal{F}, \mathbb{P}, G)$ is non-trivial. The argument may be made as follows.

Remark that by Standard Assumptions 3 and 4, $(\Omega, \mathcal{F}, \mathbb{P}, G)$ is an MDS with $(\Omega, \mathcal{F}, \mathbb{P})$ a Lebesgue space, $X$ is a compact metric space, and $\mathcal{E}\in \mathcal{F}\times \mathcal{B}_X$ satisfies that $\emptyset\neq
\mathcal{E}_\omega\subseteq X$ is a compact subset for $\mathbb{P}$-a.e. $\omega\in \Omega$.
For each real-valued function
 $f$ on $\mathcal{E}$ which is measurable in $(\omega, x)\in \mathcal{E}$
 and continuous
in $x\in \mathcal{E}_\omega$ (for each fixed $\omega\in \Omega$), we set
 \begin{equation*}
 ||f||_1= \int_\Omega ||f (\omega)||_\infty d \mathbb{P} (\omega),\ \text{where}\
 ||f (\omega)||_\infty= \sup_{x\in \mathcal{E}_\omega} | f (\omega, x)|.
 \end{equation*}
Denote by $\mathbf{L}_\mathcal{E}^1 (\Omega, C (X))$ the space of
all such functions with $||f||_1< + \infty$, where we will identify two such functions $f$ and $g$ provided $||f- g||_1= 0$.
It is easy to check that $(\mathbf{L}_\mathcal{E}^1 (\Omega, C (X)), ||\bullet||_1)$
 becomes a Banach space.

 As we will see, the role of $\mathbf{L}_\mathcal{E}^1 (\Omega, C (X))$ in the set-up of a continuous bundle RDS is just that of $C (Y)$, is played by the set of all real-valued continuous functions over $Y$, when we consider a topological $G$-action $(Y, G)$.

We will introduce the weak* topology in $\mathcal{P}_\mathbb{P} (\mathcal{E})$ as follows. Let $\mu, \mu_n\in \mathcal{P}_\mathbb{P} (\mathcal{E}), n\in \mathbb{N}$. We will say that the sequence $\{\mu_n: n\in \mathbb{N}\}$ converges to $\mu$ in $\mathcal{P}_\mathbb{P} (\mathcal{E})$ if the sequence $\{\int_\mathcal{E} f d \mu_n: n\in \mathbb{N}\}$ converges to $\int_\mathcal{E} f d \mu$ for each $f\in \mathbf{L}_\mathcal{E}^1 (\Omega, C (X))$ (obviously, $\int_\mathcal{E} f d \mu_n$ and $\int_\mathcal{E} f d \mu$ are well defined from the above definitions).

It is well known that $\mathcal{P}_\mathbb{P} (\mathcal{E})$ is a non-empty compact space in the weak* topology, see for example \cite[Lemma 2.1 (i)]{K1}. Remark again that $(\Omega, \mathcal{F}, \mathbb{P})$ is a Lebesgue space by Standard Assumption 3, and so additionally, by \cite[Theorem 5.6]{CR}, one sees that $\mathcal{P}_\mathbb{P} (\mathcal{E})$ is also a metric space.

Recall that a non-empty subset of a topological space is \emph{clopen} if it is simultaneously open and closed.

The following result (cf \cite[Lemma 2.1]{K1}) will be useful in the proof of Theorem \ref{1007141414}. Observe that Proposition \ref{1007192023}
\eqref{1208011203} follows directly from the aforementioned compactness of the space $\mathcal{P}_\mathbb{P} (\mathcal{E})$.

\begin{prop} \label{1007192023}
Let $\mathcal{P}_\mathbb{P} (\mathcal{E})$ be equipped with the  weak* topology introduced above.
\begin{enumerate}

\item \label{1208011203} Let $\{\nu_n: n\in \mathbb{N}\}\subseteq \mathcal{P}_\mathbb{P} (\mathcal{E})$. The set of limit points of the sequence $$\{\mu_n\doteq \frac{1}{|F_n|} \sum\limits_{g\in F_n} g \nu_n: n\in \mathbb{N}\}$$ is non-empty and is contained in $\mathcal{P}_\mathbb{P} (\mathcal{E}, G)$.

\item \label{1207301347} Let $\{\mu_n: n\in \mathbb{N}\}$ be a sequence in $\mathcal{P}_\mathbb{P} (\mathcal{E})$ converging to $\mu\in \mathcal{P}_\mathbb{P} (\mathcal{E})$ with $d \mu (\omega, x)= d \mu_\omega (x) d \mathbb{P} (\omega)$ the disintegration of $\mu$ over $\mathcal{F}_\mathcal{E}$. If $\alpha\in \mathbf{P}_\mathcal{E}$ satisfies that $\alpha_\omega$ is a clopen partition of $\mathcal{E}_\omega$ (i.e. each element of $\alpha_\omega$ is clopen) for $\mathbb{P}$-a.e. $\omega\in \Omega$, then
$$\limsup\limits_{n\rightarrow \infty} H_{\mu_n} (\alpha| \mathcal{F}_\mathcal{E})\le H_\mu (\alpha| \mathcal{F}_\mathcal{E}).$$
\end{enumerate}
\end{prop}

 From now on, the topological space $\mathcal{P}_\mathbb{P} (\mathcal{E})$ (as well as its closed non-empty subspace $\mathcal{P}_\mathbb{P} (\mathcal{E}, G)$) is assumed to be equipped with the weak* topology.

\medskip

Now let $\mu\in \mathcal{P}_\mathbb{P} (\mathcal{E}, G)$ and $\mathcal{U}\in \mathbf{C}_\mathcal{E}$. As $\mathcal{F}_\mathcal{E}\subseteq (\mathcal{F}\times \mathcal{B}_X)\cap \mathcal{E}$ is a $G$-invariant sub-$\sigma$-algebra, we can introduce
 the
 \emph{$\mu$-fiber entropy of $\mathbf{F}$ with respect to $\mathcal{U}$}
 by
 \begin{equation*}
h_\mu^{(r)} (\mathbf{F}, \mathcal{U})= h_\mu (G, \mathcal{U}| \mathcal{F}_\mathcal{E}).
 \end{equation*}
And then we define the \emph{$\mu$-fiber entropy of $\mathbf{F}$} as
\begin{equation*}
h_\mu^{(r)} (\mathbf{F})= \sup_{\alpha\in \mathbf{P}_\mathcal{E}}
h_\mu^{(r)} (\mathbf{F}, \alpha).
\end{equation*}
From the definitions we have immediately $h_\mu^{(r)} (\mathbf{F})= h_\mu (G, \mathcal{E}| \mathcal{F}_\mathcal{E})$.

Recall that $(\Omega, \mathcal{F}, \mathbb{P})$ is a Lebesgue space by Standard Assumption 3, one has that $(\mathcal{E}, (\mathcal{F}\times \mathcal{B}_X)\cap \mathcal{E}, \mu)$ is also a Lebesgue space. Then, using Theorem \ref{1006272300} and Proposition \ref{1006291603}, respectively, we have
\begin{equation} \label{1208011804}
h_\mu^{(r)} (\mathbf{F}, \mathcal{U})= h_{\mu, +} (G, \mathcal{U}| \mathcal{F}_\mathcal{E})
\end{equation}
and
\begin{equation*} \label{1208011805}
h_\mu (G, \mathcal{E})= h_\mu^{(r)} (\mathbf{F})+ h_\mathbb{P} (G, \Omega).
\end{equation*}

The following observation will be used below. Its proof is standard.

\begin{lem} \label{1007261204}
Let $\mu\in \mathcal{P}_\mathbb{P} (\mathcal{E}, G)$ and $\alpha_1, \alpha_2\in \mathbf{P}_\mathcal{E}, \mathcal{U}_1, \mathcal{U}_2\in \mathbf{C}_\mathcal{E}$.
\begin{enumerate}

\item \label{1007261236} If $(\alpha_1)_\omega\succeq (\alpha_2)_\omega$ for $\mathbb{P}$-a.e. $\omega\in \Omega$ then $H_\mu (\alpha_1| \mathcal{F}_\mathcal{E})\ge H_\mu (\alpha_2| \mathcal{F}_\mathcal{E})$ and $h_\mu^{(r)} (\mathbf{F}, \alpha_1)\ge h_\mu^{(r)} (\mathbf{F}, \alpha_2)$.

\item \label{1007261237} If $\alpha\in \mathbf{P}_\mathcal{E}$ and $\mathcal{U}\in \mathbf{C}_\mathcal{E}$ satisfy $\alpha_\omega\succeq \mathcal{U}_\omega$ for $\mathbb{P}$-a.e. $\omega\in \Omega$ then there exists $\alpha'\in \mathbf{P}_\mathcal{E}$ such that $\alpha'\succeq \mathcal{U}$ and $\alpha'_\omega= \alpha_\omega$ for $\mathbb{P}$-a.e. $\omega\in \Omega$, and so $H_\mu (\alpha| \mathcal{F}_\mathcal{E})= H_\mu (\alpha'| \mathcal{F}_\mathcal{E})\ge H_\mu (\mathcal{U}| \mathcal{F}_\mathcal{E})$.

\item \label{1007261239} If $(\mathcal{U}_1)_\omega\succeq (\mathcal{U}_2)_\omega$ for $\mathbb{P}$-a.e. $\omega\in \Omega$ then $h_\mu^{(r)} (\mathbf{F}, \mathcal{U}_1)\ge h_\mu^{(r)} (\mathbf{F}, \mathcal{U}_2)$.
 And so, if $(\mathcal{U}_1)_\omega= (\mathcal{U}_2)_\omega$ for $\mathbb{P}$-a.e. $\omega\in \Omega$ then $h_\mu^{(r)} (\mathbf{F}, \mathcal{U}_1)= h_\mu^{(r)} (\mathbf{F}, \mathcal{U}_2)$.
\end{enumerate}
\end{lem}

As a direct corollary, we have:

\begin{prop} \label{1007241716}
Let $\mu\in \mathcal{P}_\mathbb{P} (\mathcal{E}, G)$.
\begin{enumerate}

\item \label{1007241734}
If $\mathcal{W}\in \mathbf{C}_\Omega$ then $h_{\mu}^{(r)} (\mathbf{F}, (\mathcal{W}\times X)_\mathcal{E})= h_{\mu}^{(r)} (\mathbf{F}, (\{\Omega\}\times X)_\mathcal{E})= 0$.

\item \label{1207261253} If $\xi\in \mathbf{P}_\Omega$ and $\mathcal{V}\in \mathbf{C}_X$ then
\begin{equation*}
\inf_{\beta\in \mathbf{P}_X, \beta\succeq \mathcal{V}} h_\mu^{(r)} (\mathbf{F}, (\Omega\times \beta)_\mathcal{E})\ge h_{\mu}^{(r)} (\mathbf{F}, (\Omega\times \mathcal{V})_\mathcal{E})= h_{\mu}^{(r)} (\mathbf{F}, (\xi\times \mathcal{V})_\mathcal{E}).
\end{equation*}

\item \label{1007241735}
Assume that $\mathcal{U}\in \mathbf{C}_\mathcal{E}$ is in the form of $\mathcal{U}= \{(\Omega_i\times B_i)^c: i= 1, \cdots, n\}, n\in \mathbb{N}\setminus \{1\}$, where $\Omega_i\in \mathcal{F}$ and $B_i\in \mathcal{B}_X$ for each $i= 1, \cdots, n$. If $\mathbb{P} (\bigcap\limits_{i= 1}^n \Omega_i)= 0$ then $h_{\mu}^{(r)} (\mathbf{F}, \mathcal{U})= 0$.
\end{enumerate}
\end{prop}
\begin{proof}
\eqref{1007241734} and \eqref{1207261253} follow directly from Lemma \ref{1007261204}.

Now we check \eqref{1007241735}.
By the assumption, $\mathcal{U}= \{(\Omega_i\times B_i)^c: i= 1, \cdots, n\}\in \mathbf{C}_\mathcal{E}$, where $\Omega_i\in \mathcal{F}$ and $B_i\in \mathcal{B}_X$ for each $i= 1, \cdots, n$, and $\mathbb{P} (\bigcap\limits_{i= 1}^n \Omega_i)= 0$. Obviously, $\mathcal{W}^*\doteq (\{\Omega_1^c, \cdots, \Omega_n^c\}\times X)_\mathcal{E}\in \mathbf{C}_\mathcal{E}$ satisfies $\mathcal{W}^*\succeq \mathcal{U}$ (in the sense of $\mu$). Thus, by \eqref{1007241734},
$0\le h_{\mu}^{(r)} (\mathbf{F}, \mathcal{U})\le h_{\mu}^{(r)} (\mathbf{F}, \mathcal{W}^*)= 0$.
This completes the proof of \eqref{1007241735}.
\end{proof}

We now come to the main result of this section. It will be very important in Part \ref{skdj}.

\begin{thm} \label{1006122212}
Let $\mu\in \mathcal{P}_\mathbb{P} (\mathcal{E}, G)$. Then
\begin{eqnarray*}
h_\mu^{(r)} (\mathbf{F})&= & \sup_{\mathcal{U}\in \mathbf{C}_\mathcal{E}}
h_\mu^{(r)} (\mathbf{F}, \mathcal{U})= \sup_{\mathcal{U}\in \mathbf{C}^o_\mathcal{E}}
h_\mu^{(r)} (\mathbf{F}, \mathcal{U}) \\
&= & \sup_{\mathcal{V}\in \mathbf{C}_X}
h_\mu^{(r)} (\mathbf{F}, (\Omega\times \mathcal{V})_\mathcal{E})= \sup_{\mathcal{V}\in \mathbf{C}^o_X}
h_\mu^{(r)} (\mathbf{F}, (\Omega\times \mathcal{V})_\mathcal{E}).
\end{eqnarray*}
\end{thm}
\begin{proof}
By the definitions, we only need prove
\begin{equation} \label{1006271113}
h_\mu^{(r)} (\mathbf{F})= \sup_{\alpha\in \mathbf{P}_X}
h_\mu^{(r)} (\mathbf{F}, (\Omega\times \alpha)_\mathcal{E})
\end{equation}
and, for each $\beta\in \mathbf{P}_X$,
\begin{equation} \label{1006271227}
h_\mu^{(r)} (\mathbf{F}, (\Omega\times \beta)_\mathcal{E})\le \sup_{\mathcal{V}\in \mathbf{C}^o_X}
h_\mu^{(r)} (\mathbf{F}, (\Omega\times \mathcal{V})_\mathcal{E}).
\end{equation}

 Observe that, for convenience, $\mu$ may be viewed as a probability measure over  $(\Omega\times X, \mathcal{F}\times \mathcal{B}_X)$ and so $(\Omega\times X, \mathcal{F}\times \mathcal{B}_X, \mu, G)$ may be viewed as an MDS defined up to $\mu$-null sets. Let $\mathbf{P}_{\Omega\times X}$ be the set of all finite partitions of $(\Omega\times X, \mathcal{F}\times \mathcal{B}_X, \mu)$.

 \medskip

Let us first prove \eqref{1006271113}.

Let $\gamma\in \mathbf{P}_{\Omega\times X}$.
 Recall that $\mathcal{F}\times \mathcal{B}_X$ is the sub-$\sigma$-algebra generated by $A\times B, A\in \mathcal{F}$ and $B\in \mathcal{B}_X$, then there exist $\xi\in \mathbf{P}_\Omega$ and $\alpha\in \mathbf{P}_X$ such that $H_\mu (\gamma| \xi\times \alpha)$ is sufficiently small, and so by Proposition \ref{0911192237} \eqref{1007041247} we estimate $h_\mu (G, \gamma| \mathcal{F}\times \{X\})$  arbitrarily closely from above by $h_\mu (G, \xi\times \alpha| \mathcal{F}\times \{X\})$,
as $\mathcal{F}\times \{X\}\subseteq \mathcal{F}\times \mathcal{B}_X$ is a $G$-invariant sub-$\sigma$-algebra. We conclude that
\begin{equation} \label{1007041429}
h_\mu (G, \Omega\times X| \mathcal{F}\times \{X\})= \sup_{\xi\in \mathbf{P}_\Omega} \sup_{\alpha\in \mathbf{P}_X} h_\mu (G, \xi\times \alpha| \mathcal{F}\times \{X\}).
\end{equation}

Now $d \mu (\omega, x)= d \mu_\omega (x) d \mathbb{P} (\omega)$, the disintegration of $\mu\in \mathcal{P}_\mathbb{P} (\mathcal{E}, G)$ introduced following Lemma \ref{1007152201}, may also be viewed as the disintegration of $\mu$ over $\mathcal{F}\times \{X\}$. Hence, whenever $\xi\in \mathbf{P}_\Omega, \alpha\in \mathbf{P}_X$, using the argument of \eqref{0906282008}, one has
\begin{eqnarray} \label{1007041440}
h_\mu (G, \xi\times \alpha| \mathcal{F}\times \{X\})&= & \lim_{n\rightarrow \infty} \frac{1}{|F_n|} H_\mu ((\xi\times \alpha)_{F_n}| \mathcal{F}\times \{X\})\nonumber \\
&= & \lim_{n\rightarrow \infty} \frac{1}{|F_n|} \int_\Omega H_{\mu_\omega} \left(\bigvee_{g\in F_n} F_{g^{- 1}, g \omega} \alpha\right) d \mathbb{P} (\omega)\nonumber \\
&= & \lim_{n\rightarrow \infty} \frac{1}{|F_n|} H_\mu (((\Omega\times \alpha)_\mathcal{E})_{F_n}| \mathcal{F}_\mathcal{E})= h_\mu^{(r)} (\mathbf{F}, (\Omega\times \alpha)_\mathcal{E}).
\end{eqnarray}

Furthermore, it is easy to check that
\begin{equation} \label{1007041439}
h_\mu (G, \Omega\times X| \mathcal{F}\times \{X\})= h_\mu^{(r)} (\mathbf{F}).
\end{equation}
Then \eqref{1006271113} follows from \eqref{1007041429}, \eqref{1007041440} and \eqref{1007041439}.

\medskip

Now we turn to the proof of \eqref{1006271227}.

 Recall that as $(\Omega, \mathcal{F}, \mathbb{P})$ is a Lebesgue space by Standard Assumption 3, we can view $(\Omega, \mathcal{F}, \mathbb{P})$ as a Borel subset of the unit interval $[0, 1]$ equipped with a Borel probability measure. Furthermore, by Standard Assumption 4, $X$ is a compact metric space. Thus $\mu$ can be viewed as a Borel probability measure on the compact metric space $[0, 1]\times X$. In particular, it is regular.

Let $\beta\in \mathbf{P}_X, \epsilon> 0$ and say $\beta= \{B_1, \cdots, B_n\}, n\in \mathbb{N}$.
Observe that there exists $\delta> 0$ such that if $\xi= \{C_1, \cdots, C_n\}\in \mathbf{P}_\mathcal{E}$
 satisfies $\sum\limits_{i= 1}^n \mu ((\Omega\times B_i)\cap \mathcal{E}\Delta C_i)< \delta$ then
 $$H_\mu ((\Omega\times \beta)_\mathcal{E}| \xi)+ H_\mu (\xi| (\Omega\times \beta)_\mathcal{E})< \epsilon.$$
 For each $i= 1, \cdots, n$, by the regularity of $\mu$, it is not hard to choose a compact subset $K_i\subseteq B_i$ with $\mu (\Omega\times (B_i\setminus K_i))< \frac{\delta}{n ^2}$. Set
 $$\mathcal{U}= \{K_i\cup U: i= 1, \cdots, n\},\ \text{where}\ U= X\setminus (K_1\cup \cdots\cup K_n).$$
  Then $\mathcal{U}\in \mathbf{C}^o_X$ and $\mu (\Omega\times U)< \frac{\delta}{n}$.

Hence, once $\gamma\in \mathbf{P}_\mathcal{E}$ satisfies $\gamma\succeq (\Omega\times \mathcal{U})_\mathcal{E}$, there exists $\eta= \{A_1, \cdots, A_n\}\in \mathbf{P}_\mathcal{E}$ such that $\gamma\succeq \eta$ and $A_i\subseteq \Omega\times (K_i\cup U)$ for each $i= 1, \cdots, n$. By the choice of $\eta$ one has: $\Omega\times K_i\subseteq A_i$ (up to $\mu$-null sets) and $K_i\subseteq B_i\subseteq K_i\cup U$ for each $i= 1, \cdots, n$, which implies
\begin{equation*}
\sum_{i= 1}^n \mu (A_i\Delta (\Omega\times B_i))< n \mu (\Omega\times U)< \delta,
\end{equation*}
thus, by the choice of $\delta$,
\begin{equation} \label{1207311214}
H_\mu ((\Omega\times \beta)_\mathcal{E}| \gamma)\le H_\mu ((\Omega\times \beta)_\mathcal{E}| \eta)< \epsilon.
\end{equation}
Now, for each $F\in \mathcal{F}_G$, if $\zeta\in \mathbf{P}_\mathcal{E}$ satisfies $\zeta\succeq ((\Omega\times \mathcal{U})_\mathcal{E})_F$.  Thus, for each $g\in F$, $g \zeta\succeq (\Omega\times \mathcal{U})_\mathcal{E}$ and, using \eqref{1207311214},
\begin{equation} \label{1207311215}
H_\mu ((\Omega\times \beta)_\mathcal{E}| g \zeta)< \epsilon.
\end{equation}
It follows that
\begin{eqnarray*}
H_\mu (((\Omega\times \beta)_\mathcal{E})_F| \mathcal{F}_\mathcal{E})&\le & H_\mu (\zeta| \mathcal{F}_\mathcal{E})+ H_\mu (((\Omega\times \beta)_\mathcal{E})_F| \zeta) \\
&\le & H_\mu (\zeta| \mathcal{F}_\mathcal{E})+ \sum_{g\in F} H_\mu ((\Omega\times \beta)_\mathcal{E}| g \zeta) \\
&< & H_\mu (\zeta| \mathcal{F}_\mathcal{E})+ |F| \epsilon\ (\text{using \eqref{1207311215}}),
\end{eqnarray*}
which implies
$$H_\mu (((\Omega\times \beta)_\mathcal{E})_F| \mathcal{F}_\mathcal{E})\le H_\mu (((\Omega\times \mathcal{U})_\mathcal{E})_F| \mathcal{F}_\mathcal{E})+ |F| \epsilon.$$

Lastly, for each $m\in \mathbb{N}$, substituting $F$ by $F_m$, dividing both sides by $|F_m|$ and letting $m$ tend to infinity, we obtain
$$h_\mu^{(r)} (\mathbf{F}, (\Omega\times \beta)_\mathcal{E})\le h_\mu^{(r)} (\mathbf{F}, (\Omega\times \mathcal{U})_\mathcal{E})+ \epsilon,$$ from which \eqref{1006271227} follows easily. This completes the proof.
\end{proof}

\newpage

\part{ A Local Variational Principle for Fiber Topological Pressure} \label{skdj}

In this part we present and prove our main results. More precisely,
given the continuous bundle random dynamical system associated to an infinite countable discrete amenable group action and a monotone sub-additive invariant family of random continuous functions,
we introduce and discuss local fiber topological pressure for a finite measurable cover and establish an associated variational principle.   This relates the local fiber topological pressure to measure-theoretic entropy, under some assumptions.

Before proceeding, we reminder the reader that, by Standard Assumptions 1, 2, 3 and 4 made in Part \ref{preli}, the family $\mathbf{F}= \{F_{g, \omega}: \mathcal{E}_\omega\rightarrow \mathcal{E}_{g \omega}| g\in G, \omega\in \Omega\}$ will always be a continuous bundle RDS over the MDS $(\Omega,
\mathcal{F}, \mathbb{P}, G)$, where:
\begin{enumerate}

\item
$G$ is an infinite countable discrete amenable group with a F\o lner sequence $\{F_n: n\in \mathbb{N}\}$ satisfying $e_G\subseteq F_1\subsetneq F_2\subsetneq \cdots$,

\item
 $(\Omega,
\mathcal{F}, \mathbb{P})$ is a Lebesgue space and

\item
the state space of $\mathbf{F}$ is a compact metric space $(X, d)$.
\end{enumerate}

\section{Local fiber topological pressure} \label{fourth}

In this section, given a continuous bundle random dynamical system associated to an infinite countable discrete amenable group action and a monotone sub-additive invariant family of random continuous functions, we introduce the concept of the local fiber topological pressure for a finite measurable cover and discuss some basic properties. Our discussion follows the ideas of \cite{CFH, HYi, R, Z-DCDS}.

\medskip

Let us first introduce the concept of a monotone sub-additive invariant family of random continuous functions.

 We say that $f\in \mathbf{L}_\mathcal{E}^1 (\Omega, C (X))$ is \emph{non-negative} if for
 $\mathbb{P}$-a.e. $\omega\in \Omega$, $f (\omega, x)$ is a non-negative function on $\mathcal{E}_\omega$.

 Let $\mathbf{D}= \{d_F: F\in \mathcal{F}_G\}$ be a family in $\mathbf{L}_\mathcal{E}^1 (\Omega, C (X))$.
We say that $\mathbf{D}$ is
 \begin{enumerate}

   \item
 \emph{non-negative} if each element of $\mathbf{D}$ is non-negative;

   \item
\emph{sub-additive} if for
 $\mathbb{P}$-a.e. $\omega\in \Omega$, $d_{E\cup F g} (\omega, x)\le d_E (\omega, x)+ d_F
 (g (\omega, x))$ whenever $E, F\in \mathcal{F}_G$ and $g\in G$ satisfy $E\cap F g= \emptyset$ and $x\in \mathcal{E}_\omega$;

   \item \emph{$G$-invariant} if for $\mathbb{P}$-a.e. $\omega\in \Omega$, $d_{F g} (\omega, x)= d_F (g (\omega, x))$ whenever $F\in \mathcal{F}_G, g\in G$ and $x\in \mathcal{E}_\omega$;

   \item \emph{monotone} if for $\mathbb{P}$-a.e. $\omega\in \Omega$, $d_E (\omega, x)\le d_F (\omega, x)$ whenever $E, F\in \mathcal{F}_G$ satisfy $E\subseteq F$ and $x\in \mathcal{E}_\omega$.
 \end{enumerate}
 For example, for each $f\in
 \mathbf{L}_\mathcal{E}^1 (\Omega, C (X))$, it is easy to check that
 $$\mathbf{D}^f\doteq \{d_F^f (\omega, x)\doteq \sum\limits_{g\in F} f (g (\omega, x)): F\in \mathcal{F}_G\}$$
  is a sub-additive $G$-invariant family in $\mathbf{L}_\mathcal{E}^1 (\Omega, C (X))$. If $f$ is non-negative it is also monotone. Observe that in $\mathbf{L}_\mathcal{E}^1 (\Omega, C (X))$ not every sub-additive $G$-invariant family is of this form. In fact, if $f\in
 \mathbf{L}_\mathcal{E}^1 (\Omega, C (X))$ then the family
  $$\{d_F (\omega, x)\doteq \sum\limits_{g\in F} f (g (\omega, x))+ \sqrt{|F|}: F\in \mathcal{F}_G\}\subseteq \mathbf{L}_\mathcal{E}^1 (\Omega, C (X))$$
is also sub-additive and $G$-invariant.

  We can introduce analogues of these families in $L^1 (\Omega, \mathcal{F}, \mathbb{P})$.

 It is easy to check that:

 \begin{prop} \label{1006151230}
 Let $\mathbf{D}= \{d_F: F\in \mathcal{F}_G\}\subseteq \mathbf{L}_\mathcal{E}^1 (\Omega, C (X))$ be a sub-additive $G$-invariant family and $\mu\in \mathcal{P}_\mathbb{P} (\mathcal{E}, G)$. Then, for the function
 $$f: \mathcal{F}_G\rightarrow \mathbb{R}, F\mapsto \int_{\mathcal{E}} d_F (\omega, x) d \mu (\omega, x),$$
 $f (E g)= f (E)$ and $f (E\cup F)\le f (E)+ f (F)$ whenever $g\in G$ and $E, F\in \mathcal{F}_G$ satisfy $E\cap F= \emptyset$. Moreover, if $\mathbf{D}$ is monotone then $\mathbf{D}$ is non-negative, and so $f$ is a monotone non-negative sub-additive $G$-invariant function.

A similar conclusion also holds if the family belongs to $L^1 (\Omega, \mathcal{F}, \mathbb{P})$.
 \end{prop}
 \begin{proof}
 We only need check that if $\mathbf{D}$ is monotone, then $\mathbf{D}$ is non-negative. In fact, this follows directly from the assumptions of sub-additivity and monotonicity.

 Let $F\in \mathcal{F}_G$. Then for each $E\in \mathcal{F}_G$ satisfying $E\cap F= \emptyset$, by the assumptions of sub-additivity and monotonicity over $\mathbf{D}$ we have: for $\mathbb{P}$-a.e. $\omega\in \Omega$,
 \begin{equation*}
 d_E (\omega, x)\le d_{E\cup F} (\omega, x)\le d_E (\omega, x)+ d_F (\omega, x),
 \end{equation*}
and so $d_F (\omega, x)\ge 0$ for each $x\in \mathcal{E}_\omega$. This finishes our proof.
 \end{proof}

Let $\mathbf{D}= \{d_F: F\in \mathcal{F}_G\}\subseteq \mathbf{L}_\mathcal{E}^1 (\Omega, C (X))$ be a sub-additive $G$-invariant family and $\mathcal{U}\in \mathbf{C}_\mathcal{E}$.
For each $F\in \mathcal{F}_G$ and any $\omega\in \Omega$ we set
\begin{eqnarray}
& &\hskip -36pt P_\mathcal{E} (\omega, \mathbf{D}, F, \mathcal{U}, \mathbf{F})\nonumber \\
&= & \inf \left\{\sum_{A
(\omega)\in \alpha
(\omega)} \sup_{x\in A (\omega)} e^{d_F (\omega, x)}: \alpha (\omega)\in
\mathbf{P}_{\mathcal{E}_\omega}, \alpha (\omega)\succeq
(\mathcal{U}_F)_\omega\right\}\nonumber \\
&= & \inf \left\{\sum_{A
(\omega)\in \alpha
(\omega)} \sup_{x\in A (\omega)} e^{d_F (\omega, x)}: \alpha (\omega)\in
\mathbf{P}_{\mathcal{E}_\omega}, \alpha (\omega)\succeq
 \bigvee_{g\in F} F_{g^{- 1}, g \omega} \mathcal{U}_{g \omega}\right\}, \label{1007051535}
\end{eqnarray}
where $\mathbf{P}_{\mathcal{E}_\omega}$ is introduced as in previous sections and \eqref{1007051535} follows from \eqref{1006131722}.

In fact, it is easy to obtain an alternative expression for $P_\mathcal{E} (\omega, \mathbf{D}, F, \mathcal{U}, \mathbf{F})$ viz.:
\begin{equation} \label{1007111222}
P_\mathcal{E} (\omega, \mathbf{D}, F, \mathcal{U}, \mathbf{F})= \inf \left\{\sum_{A\in \alpha} \sup_{x\in A_\omega} e^{d_F (\omega, x)}: \alpha\succeq
\mathcal{U}_F\right\}.
\end{equation}
To see this, for $\alpha (\omega)\in
\mathbf{P}_{\mathcal{E}_\omega}$ with $\alpha (\omega)\succeq
(\mathcal{U}_F)_\omega$, define
$$\beta= \{\{\omega\}\times A: A\in \alpha (\omega)\}\cup \{U\setminus (\{\omega\}\times \mathcal{E}_\omega): U\in \mathcal{P} (\mathcal{U}_F)\}.$$
As $(\Omega, \mathcal{F}, \mathbb{P})$ is a Lebesgue space by Standard Assumption 3, then $\{\omega\}\in \mathcal{F}$ and so by the construction it is clear to see that $\beta\in \mathbf{P}_\mathcal{E}$. Further, $\beta_\omega= \alpha (\omega), \beta\succeq \mathcal{U}_F$ (as $\alpha (\omega)\succeq (\mathcal{U}_F)_\omega$ and $\mathcal{P} (\mathcal{U}_F)\succeq \mathcal{U}_F$), and hence
$$\sum_{A
(\omega)\in \alpha
(\omega)} \sup_{x\in A (\omega)} e^{d_F (\omega, x)}= \sum_{B\in \beta} \sup_{x\in B_\omega} e^{d_F (\omega, x)}.$$

\medskip

Before proceeding, we need the following observation, whose proof is obvious.

\begin{lem} \label{1007051758}
Let $\mathcal{U}\in \mathbf{C}_\mathcal{E}$ and $\omega\in \Omega$. Then $\mathbf{P} (\mathcal{U}_\omega)= \{\alpha_\omega: \alpha\in \mathbf{P} (\mathcal{U})\}$.
\end{lem}

We now have the following alternative formula for $P_\mathcal{E} (\omega, \mathbf{D}, F, \mathcal{U}, \mathbf{F})$.

\begin{prop} \label{1006141641}
Let $\mathbf{D}= \{d_F: F\in \mathcal{F}_G\}\subseteq \mathbf{L}_\mathcal{E}^1 (\Omega, C (X))$ be a sub-additive $G$-invariant family and $\mathcal{U}\in \mathbf{C}_\mathcal{E}$, $F\in \mathcal{F}_G$, $\omega\in \Omega$. Then
\begin{eqnarray}
\hskip 16pt
P_\mathcal{E} (\omega, \mathbf{D}, F, \mathcal{U}, \mathbf{F})&= & \min \left\{\sum_{A
(\omega)\in \alpha (\omega)} \sup_{x\in A (\omega)} e^{d_F (\omega, x)}: \alpha (\omega)\in
\mathbf{P} ((\mathcal{U}_F)_\omega)\right\} \label{1007052308} \\
&= & \min \left\{\sum_{A\in \alpha} \sup_{x\in A_\omega} e^{d_F (\omega, x)}: \alpha\in
\mathbf{P} (\mathcal{U}_F)\right\}. \label{1007052309}
\end{eqnarray}
\end{prop}
\begin{proof}
Note that \eqref{1007052309} follows directly from Lemma \ref{1007051758} and \eqref{1007052308}. Thus we only need prove \eqref{1007052308}.
We should point out that $d_F (\omega, x)$ is continuous in $x\in \mathcal{E}_\omega$ and $(\mathcal{U}_F)_\omega\in \mathbf{C}_{\mathcal{E}_\omega}$ (where $\mathbf{C}_{\mathcal{E}_\omega}$ is introduced as in previous sections). The proof will therefore be complete if we can prove that, for each continuous function $f$ over $\mathcal{E}_\omega$ and any $\mathcal{W}\in \mathbf{C}_{\mathcal{E}_\omega}$,
\begin{equation*} \label{1007051115}
\inf_{\gamma\in \mathbf{P}_{\mathcal{E}_\omega}, \gamma\succeq \mathcal{W}} \sum_{B\in \gamma} \sup_{x\in B} e^{f (x)}= \min_{\zeta\in \mathbf{P} (\mathcal{W})} \sum_{C\in \zeta} \sup_{x\in C} e^{f (x)},
\end{equation*}
where $\mathbf{P} (\mathcal{W})$ is introduced above. However, this is easy to see. (For example see the proof of \cite[Lemma 2.1]{HYi}.)
This establishes \eqref{1007052308} and ends our proof.
\end{proof}

Thus:

\begin{prop} \label{1006141621}
Let $\mathbf{D}= \{d_F: F\in \mathcal{F}_G\}\subseteq \mathbf{L}_\mathcal{E}^1 (\Omega, C (X))$ be a sub-additive $G$-invariant family and $\mathcal{U}\in \mathbf{C}_\mathcal{E}$. Then
\begin{enumerate}

\item \label{1006141704} For each $F\in \mathcal{F}_G$, the function $P_\mathcal{E} (\omega, \mathbf{D}, F, \mathcal{U}, \mathbf{F})$ is measurable in $\omega\in \Omega$.

\item \label{1006141705}
$\{\log P_\mathcal{E} (\omega, \mathbf{D}, F, \mathcal{U}, \mathbf{F}): F\in \mathcal{F}_G\}$ is a sub-additive $G$-invariant family in $L^1 (\Omega, \mathcal{F}, \mathbb{P})$.

\item \label{1007111801} For the function $p: \mathcal{F}_G\rightarrow \R, F\mapsto \int_\Omega \log P_\mathcal{E} (\omega, \mathbf{D}, F, \mathcal{U}, \mathbf{F}) d \mathbb{P} (\omega)$, one has $p (E g)= p (E)$ and $p (E\cup F)\le p (E)+ p (F)$ whenever $E, F\in \mathcal{F}_G$ and $g\in G$ satisfy $E\cap F= \emptyset$; moreover, if $\mathbf{D}$ is monotone then $p$ is a monotone non-negative $G$-invariant sub-additive function.
\end{enumerate}
\end{prop}
\begin{proof}
\eqref{1006141704} Let $F\in \mathcal{F}_G$. By \eqref{1007052309}, we only need to prove that
$\sup\limits_{x\in A_\omega} e^{d_F (\omega, x)}$
is measurable in $\omega\in \Omega$ for each $A\in (\mathcal{F}\times \mathcal{B}_X)\cap \mathcal{E}$. In fact, let $A\in (\mathcal{F}\times \mathcal{B}_X)\cap \mathcal{E}$ and letting $\pi: \mathcal{E}\rightarrow \Omega$ be the natural projection, we have
\begin{equation*}
\{\omega\in \Omega: \sup_{x\in A_\omega} e^{d_F (\omega, x)}> r\}= \pi (\{(\omega, x)\in A: e^{d_F (\omega, x)}> r\})
\end{equation*}
for each $r\in \mathbb{R}$.
By Lemma \ref{1007152201},
$$\{\omega\in \Omega: \sup\limits_{x\in A_\omega} e^{d_F (\omega, x)}> r\}$$
 is measurable, which implies that $\sup\limits_{x\in A_\omega} e^{d_F (\omega, x)}$ is measurable in $\omega\in \Omega$.

\eqref{1006141705} Let $E, F\in \mathcal{F}_G, g\in G$ satisfy $E\cap F g= \emptyset$ and $\omega\in \Omega$. Then by \eqref{1007111222} one has
\begin{eqnarray*}
e^{- ||d_E (\omega)||_\infty}&\le & P_\mathcal{E} (\omega, \mathbf{D}, E, \mathcal{U}, \mathbf{F}) \\
&= & \inf \left\{\sum_{A\in \alpha} \sup_{x\in A_\omega} e^{d_E (\omega, x)}: \alpha\succeq
\mathcal{U}_E\right\}\le |\mathcal{U}_E| e^{||d_E (\omega)||_\infty},
\end{eqnarray*}
which implies $\log P_\mathcal{E} (\omega, \mathbf{D}, E, \mathcal{U}, \mathbf{F})\in L^1 (\Omega, \mathcal{F}, \mathbb{P})$ (by the definition of $\mathbf{L}_\mathcal{E}^1 (\Omega, C (X))$).

Moreover, by the $G$-invariance of the family $\mathbf{D}$ one has, for $\mathbb{P}$-a.e. $\omega\in \Omega$,
\begin{eqnarray}
P_\mathcal{E} (\omega, \mathbf{D}, F g, \mathcal{U}, \mathbf{F})&= & \inf \left\{\sum_{A\in \alpha} \sup_{x\in A_\omega} e^{d_{F g} (\omega, x)}: \alpha\succeq
\mathcal{U}_{F g}\right\}\ (\text{using \eqref{1007111222}})\nonumber \\
&= & \inf \left\{\sum_{A\in \alpha} \sup_{x\in A_\omega} e^{d_F (g (\omega, x))}: g \alpha\succeq
\mathcal{U}_F\right\}\nonumber \\
&= & \inf \left\{\sum_{A\in \alpha} \sup_{x\in A_{g \omega}} e^{d_F (g \omega, x)}: \alpha\succeq
\mathcal{U}_F\right\}= P_\mathcal{E} (g \omega, \mathbf{D}, F, \mathcal{U}, \mathbf{F}), \label{1007112119}
\end{eqnarray}
which implies the $G$-invariance of $\log P_\mathcal{E} (\omega, \mathbf{D}, F, \mathcal{U}, \mathbf{F})$.

Finally, by the $G$-invariance of $\log P_\mathcal{E} (\omega, \mathbf{D}, F, \mathcal{U}, \mathbf{F})$ and the sub-additivity of the family $\mathbf{D}$, one has, for $\mathbb{P}$-a.e. $\omega\in \Omega$,
\begin{eqnarray*}
& &
P_\mathcal{E} (\omega, \mathbf{D}, E\cup F g, \mathcal{U}, \mathbf{F}) \\
&= & \inf \left\{\sum_{A\in \alpha} \sup_{x\in A_\omega} e^{d_{E\cup F g} (\omega, x)}: \alpha\succeq
\mathcal{U}_{E\cup F g}\right\}\ (\text{using \eqref{1007111222}}) \\
&\le & \inf \left\{\sum_{A\in \alpha, B\in \beta} \sup_{x\in A_\omega\cap B_\omega} e^{d_E (\omega, x)+ d_F (g (\omega, x))}: \alpha\succeq \mathcal{U}_E, \beta\succeq
\mathcal{U}_{F g}\right\} \\
&\le & \inf \left\{\sum_{A\in \alpha, B\in \beta} \sup_{x\in A_\omega} e^{d_E (\omega, x)} \sup_{x\in B_\omega} e^{d_F (g (\omega, x))}: \alpha\succeq \mathcal{U}_E, \beta\succeq
\mathcal{U}_{F g}\right\} \\
&= & \inf \left\{\sum_{A\in \alpha} \sup_{x\in A_\omega} e^{d_E (\omega, x)}: \alpha\succeq \mathcal{U}_E\right\} \inf \left\{\sum_{B\in \beta} \sup_{x\in B_\omega} e^{d_F (g (\omega, x))}: \beta\succeq
\mathcal{U}_{F g}\right\} \\
&= & P_\mathcal{E} (\omega, \mathbf{D}, E, \mathcal{U}, \mathbf{F}) P_\mathcal{E} (g \omega, \mathbf{D}, F, \mathcal{U}, \mathbf{F})\ (\text{using \eqref{1007111222} and \eqref{1007112119}}),
\end{eqnarray*}
which implies the sub-additivity of $\log P_\mathcal{E} (\omega, \mathbf{D}, F, \mathcal{U}, \mathbf{F})$.

\eqref{1007111801} follows directly from Proposition \ref{1006151230} and \eqref{1006141705}.
\end{proof}

Let $\mathbf{D}= \{d_F: F\in \mathcal{F}_G\}\subseteq \mathbf{L}_\mathcal{E}^1 (\Omega, C (X))$ be a monotone sub-additive $G$-invariant family and $\mathcal{U}\in \mathbf{C}_\mathcal{E}$. Then by Proposition \ref{1006122129} and Proposition \ref{1006141621} we may define
 the \emph{fiber topological $\mathbf{D}$-pressure of
$\mathbf{F}$ with respect to $\mathcal{U}$} and the \emph{fiber topological $\mathbf{D}$-pressure of
$\mathbf{F}$}, respectively, by
\begin{equation} \label{1208011239}
P_\mathcal{E} (\mathbf{D}, \mathcal{U}, \mathbf{F})=
\lim_{n\rightarrow \infty} \frac{1}{|F_n|} \int_\Omega
\log P_\mathcal{E} (\omega, \mathbf{D}, F_n, \mathcal{U}, \mathbf{F}) d \mathbb{P} (\omega)
\end{equation}
and
\begin{equation} \label{1208030124}
P_\mathcal{E} (\mathbf{D}, \mathbf{F})= \sup_{\mathcal{V}\in \mathbf{C}_X^o}
P_\mathcal{E} (\mathbf{D}, (\Omega\times \mathcal{V})_\mathcal{E}, \mathbf{F}).
\end{equation}
Recall that $\mathbf{D}^0$ is the family whose only element is the constant zero function, and so is a monotone sub-additive $G$-invariant family. It is simple to check that
$$P_\mathcal{E} (\omega, \mathbf{D}^0, F, \mathcal{U}, \mathbf{F})= N (\mathcal{U}_F, \omega)$$
 whenever $\omega\in \Omega$ and $F\in \mathcal{F}_G$, and so one has
\begin{equation} \label{1007120926}
P_\mathcal{E} (\mathbf{D}^0, \mathcal{U}, \mathbf{F})=
\lim_{n\rightarrow \infty} \frac{1}{|F_n|} \int_\Omega
\log N (\mathcal{U}_{F_n}, \omega) d \mathbb{P} (\omega).
\end{equation}
We call this the \emph{fiber topological entropy of $\mathbf{F}$ with respect to $\mathcal{U}$} (also denoted by $h_{\text{top}}^{(r)} (\mathbf{F}, \mathcal{U})$). Moreover, $P_\mathcal{E} (\mathbf{D}^0, \mathbf{F})$ will be  called the \emph{fiber topological entropy of $\mathbf{F}$} (also denoted by $h_{\text{top}}^{(r)} (\mathbf{F})$). Remark that by Proposition \ref{1006122129} the values of these invariants are all independent of the selection of the F\o lner sequence $\{F_n: n\in \mathbb{N}\}$.

\medskip

By the strict convexity of $x \log x$ over $[0, \infty)$ it is easy to obtain:

\begin{lem} \label{1007131745}
Let $a_1, p_1, \cdots, a_k, p_k\in \mathbb{R}$ with $p_1, \cdots, p_k\ge 0$ and $\sum\limits_{i= 1}^k p_i= p$. Then
\begin{equation*}
\sum_{i= 1}^k p_i (a_i- \log p_i)\le p \log (\sum_{i= 1}^k e^{a_i})- p \log p.
\end{equation*}
Equality holds if and only if $p_i= \frac{p e^{a_i}}{\sum\limits_{j= 1}^k e^{a_j}}$ for each $i= 1, \cdots, k$. In particular,
$$\sum\limits_{i= 1}^k - p_i\log p_i\le p\log k- p \log p.$$
\end{lem}

Let $\mathbf{D}= \{d_F: F\in \mathcal{F}_G\}\subseteq \mathbf{L}_\mathcal{E}^1 (\Omega, C (X))$ be a monotone sub-additive $G$-invariant family and $\mu\in \mathcal{P}_\mathbb{P} (\mathcal{E}, G)$. We can define
\begin{equation*}
\mu (\mathbf{D})= \lim_{n\rightarrow \infty} \frac{1}{|F_n|} \int_{\mathcal{E}} d_{F_n} (\omega, x) d \mu (\omega, x)\ge 0.
\end{equation*}
The above limit is well defined by Proposition \ref{1006122129} and Proposition \ref{1006151230}, and its value is independent of the selection of the F\o lner sequence $\{F_n: n\in \mathbb{N}\}$ by Proposition \ref{1006122129}. Observe that the sequence $\{\frac{1}{|F_n|} \int_{\mathcal{E}} d_{F_n} (\omega, x) d \mu (\omega, x): n\in \N\}$ need not be convergent, if we no longer assume $G$-invariance of the measure (that is, we only assume $\mu\in \mathcal{P}_\mathbb{P} (\mathcal{E})$).

It is not hard to see:

\begin{prop} \label{1007120919}
Let $\mathbf{D}= \{d_F: F\in \mathcal{F}_G\}\subseteq \mathbf{L}_\mathcal{E}^1 (\Omega, C (X))$ be a sub-additive $G$-invariant family, $\mathcal{U}\in \mathbf{C}_\mathcal{E}$ and $\mu\in \mathcal{P}_\mathbb{P} (\mathcal{E}, G)$ with $d \mu (\omega, x)= d \mu_\omega (x) d \mathbb{P} (\omega)$ the disintegration of $\mu$ over $\mathcal{F}_\mathcal{E}$.
\begin{enumerate}

\item \label{1007121531} Let $\omega\in \Omega$. If $\nu_\omega$ is a Borel probability measure over $\mathcal{E}_\omega$, then, for each $F\in \mathcal{F}_G$,
     \begin{equation*}
     H_{\nu_\omega} ((\mathcal{U}_F)_\omega)+ \int_{\mathcal{E}_\omega} d_F (\omega, x) d \nu_\omega (x) \le \log P_\mathcal{E} (\omega, \mathbf{D}, F, \mathcal{U}, \mathbf{F}).
     \end{equation*}

\item \label{1007121532} If $\mathbf{D}$ is monotone then $P_\mathcal{E} (\mathbf{D}, \mathcal{U}, \mathbf{F})\ge h_\mu^{(r)} (\mathbf{F}, \mathcal{U})+ \mu (\mathbf{D})$,
 and so $P_\mathcal{E} (\mathbf{D}, \mathbf{F})$ $\ge h_\mu^{(r)} (\mathbf{F})+ \mu (\mathbf{D})$. In particular,
\begin{equation} \label{1207272032}
h_{\text{top}}^{(r)} (\mathbf{F}, \mathcal{U})\ge h_\mu^{(r)} (\mathbf{F}, \mathcal{U})\ \text{and}\ h_{\text{top}}^{(r)} (\mathbf{F})\ge h_\mu^{(r)} (\mathbf{F}).
\end{equation}
\end{enumerate}
\end{prop}
\begin{proof}
\eqref{1007121531} In fact, using Lemma \ref{1007131745} one has
     \begin{eqnarray*}
    & & \log P_\mathcal{E} (\omega, \mathbf{D}, F, \mathcal{U}, \mathbf{F}) \\
    &= & \inf \log \left\{\sum_{A
(\omega)\in \alpha
(\omega)} \sup_{x\in A (\omega)} e^{d_F (\omega, x)}: \alpha (\omega)\in
\mathbf{P}_{\mathcal{E}_\omega}, \alpha (\omega)\succeq
(\mathcal{U}_F)_\omega\right\} \\
&\ge & \inf_{\alpha (\omega)\in
\mathbf{P}_{\mathcal{E}_\omega}, \alpha (\omega)\succeq
(\mathcal{U}_F)_\omega} \sum_{A
(\omega)\in \alpha
(\omega)} \nu_\omega (A (\omega)) \left(\sup_{x\in A (\omega)} d_F (\omega, x)- \log \nu_\omega (A (\omega))\right) \\
&\ge & \inf_{\alpha (\omega)\in
\mathbf{P}_{\mathcal{E}_\omega}, \alpha (\omega)\succeq
(\mathcal{U}_F)_\omega} \left\{\int_{\mathcal{E}_\omega} d_F (\omega, x) d \nu_\omega (x)+ H_{\nu_\omega} (\alpha
(\omega))\right\} \\
&= & H_{\nu_\omega} ((\mathcal{U}_F)_\omega)+ \int_{\mathcal{E}_\omega} d_F (\omega, x) d \nu_\omega (x).
     \end{eqnarray*}

\eqref{1007121532} Now we assume in addition that the family $\mathbf{D}$ is monotone. It follows directly from \eqref{1007121531} and the definitions that $P_\mathcal{E} (\mathbf{D}, \mathcal{U}, \mathbf{F})\ge h_\mu^{(r)} (\mathbf{F}, \mathcal{U})+ \mu (\mathbf{D})$, and then $P_\mathcal{E} (\mathbf{D}, \mathbf{F})\ge h_\mu^{(r)} (\mathbf{F})+ \mu (\mathbf{D})$ by
 Theorem \ref{1006122212}. Finally, applying the conclusion to the constant zero family $\mathbf{D}^0$ we obtain \eqref{1207272032}.
\end{proof}

Observe that if $\mathbf{D}= \{d_F: F\in \mathcal{F}_G\}\subseteq \mathbf{L}_\mathcal{E}^1 (\Omega, C (X))$ is a monotone sub-additive $G$-invariant family then it is not hard to check that the family
$$\{\sup_{x\in \mathcal{E}_\omega} d_F (\omega, x)= ||d_F (\omega)||_\infty: F\in \mathcal{F}\}\subseteq L^1 (\Omega, \mathcal{F}, \mathbb{P})$$
is also monotone sub-additive and $G$-invariant. Hence we may define
\begin{equation} \label{1208011225}
\text{sup}_\mathbb{P} (\mathbf{D})= \lim_{n\rightarrow \infty} \frac{1}{|F_n|} \int_\Omega \sup_{x\in \mathcal{E}_\omega} d_F (\omega, x) d \mathbb{P} (\omega).
\end{equation}
Remark that by Proposition \ref{1006122129} and Proposition \ref{1006141621}, the limit is well defined and its value is independent of the selection of the F\o lner sequence $\{F_n: n\in \mathbb{N}\}$.

From the definition, it is easy to see:

\begin{lem} \label{1102071655}
Let $\mathbf{D}= \{d_F: F\in \mathcal{F}_G\}\subseteq \mathbf{L}_\mathcal{E}^1 (\Omega, C (X))$ be a monotone sub-additive $G$-invariant family and $\mu\in \mathcal{P}_\mathbb{P} (\mathcal{E})$. Then
\begin{equation*}
\text{sup}_\mathbb{P} (\mathbf{D})\ge \limsup_{n\rightarrow \infty} \frac{1}{|F_n|} \int_{\mathcal{E}} d_{F_n} (\omega, x) d \mu (\omega, x).
\end{equation*}
In particular, $\text{sup}_\mathbb{P} (\mathbf{D})\ge \mu (\mathbf{D})$ for each $\mu\in \mathcal{P}_\mathbb{P} (\mathcal{E}, G)$.
\end{lem}

As in Lemma \ref{1007261204} and Proposition \ref{1007241716}, one has:

\begin{prop} \label{1102041733}
Let $\mathbf{D}= \{d_F: F\in \mathcal{F}_G\}\subseteq \mathbf{L}_\mathcal{E}^1 (\Omega, C (X))$ be a sub-additive $G$-invariant family and $\mathcal{U}, \mathcal{U}_1, \mathcal{U}_2\in \mathbf{C}_\mathcal{E}$.
\begin{enumerate}

\item Let $\omega\in \Omega$ and $F\in \mathcal{F}_G$. Then
\begin{equation*}
\sup_{x\in \mathcal{E}_\omega} e^{d_F (\omega, x)}\le P_\mathcal{E} (\omega, \mathbf{D}, F, \mathcal{U}, \mathbf{F})\le N (\mathcal{U}_F, \omega) \sup_{x\in \mathcal{E}_\omega} e^{d_F (\omega, x)}.
\end{equation*}

\item If $(\mathcal{U}_1)_\omega\succeq (\mathcal{U}_2)_\omega$ for $\mathbb{P}$-a.e. $\omega\in \Omega$, then
    \begin{equation*}
  \log P_\mathcal{E} (\omega, \mathbf{D}, F, \mathcal{U}_1, \mathbf{F})\ge \log P_\mathcal{E} (\omega, \mathbf{D}, F, \mathcal{U}_2, \mathbf{F})
    \end{equation*}
    for $\mathbb{P}$-a.e. $\omega\in \Omega$ and each $F\in \mathcal{F}_G$.

\item If $(\mathcal{U}_1)_\omega= (\mathcal{U}_2)_\omega$ for $\mathbb{P}$-a.e. $\omega\in \Omega$, then
    \begin{equation*}
  \log P_\mathcal{E} (\omega, \mathbf{D}, F, \mathcal{U}_1, \mathbf{F})= \log P_\mathcal{E} (\omega, \mathbf{D}, F, \mathcal{U}_2, \mathbf{F})
    \end{equation*}
    for $\mathbb{P}$-a.e. $\omega\in \Omega$ and each $F\in \mathcal{F}_G$.

\item If $\mathbf{D}$ is monotone then,
for $\mathbb{P}$-a.e. $\omega\in \Omega$ and each $F\in \mathcal{F}_G$,
\begin{equation*}
e^{||d_F (\omega)||_\infty}\le P_\mathcal{E} (\omega, \mathbf{D}, F, \mathcal{U}, \mathbf{F})\le N (\mathcal{U}_F, \omega) e^{||d_F (\omega)||_\infty},
\end{equation*}
and hence
$$\text{sup}_\mathbb{P} (\mathbf{D})\le P_\mathcal{E} (\mathbf{D}, \mathcal{U}, \mathbf{F})\le h_{\text{top}}^{(r)} (\mathbf{F}, \mathcal{U})+ \text{sup}_\mathbb{P} (\mathbf{D}).$$

\item
Assume that $\mathcal{U}$ has the form  $\mathcal{U}= \{(\Omega_i\times B_i)^c: i= 1, \cdots, n\}, n\in \mathbb{N}\setminus \{1\}$ with $\Omega_i\in \mathcal{F}$ and $B_i\in \mathcal{B}_X$ for each $i= 1, \cdots, n$. If $\mathbb{P} (\bigcap\limits_{i= 1}^n \Omega_i)= 0$ then $h_{\text{top}}^{(r)} (\mathbf{F}, \mathcal{U})= 0$. so If, in addition, $\mathbf{D}$ is monotone, then
$$P_\mathcal{E} (\mathbf{D}, \mathcal{U}, \mathbf{F})= \text{sup}_\mathbb{P} (\mathbf{D}).$$
\end{enumerate}
\end{prop}

As a direct corollary, we have:

\begin{cor} \label{1102052218}
Let $\mathbf{D}= \{d_F: F\in \mathcal{F}_G\}\subseteq \mathbf{L}_\mathcal{E}^1 (\Omega, C (X))$ be a monotone sub-additive $G$-invariant family. Then
\begin{equation*}
P_\mathcal{E} (\mathbf{D}, \mathbf{F})= \sup_{\xi\in \mathbf{P}_\Omega, \mathcal{V}\in \mathbf{C}_X^o}
P_\mathcal{E} (\mathbf{D}, (\xi\times \mathcal{V})_\mathcal{E}, \mathbf{F}).
\end{equation*}
\end{cor}

We end this section with a question.

\begin{ques} \label{1102061758}
Let $\mathbf{D}= \{d_F: F\in \mathcal{F}_G\}\subseteq \mathbf{L}_\mathcal{E}^1 (\Omega, C (X))$ be a monotone sub-additive $G$-invariant family. Do we have
\begin{equation*}
P_\mathcal{E} (\mathbf{D}, \mathbf{F})= \sup_{\mathcal{U}\in \mathbf{C}_\mathcal{E}^o}
P_\mathcal{E} (\mathbf{D}, \mathcal{U}, \mathbf{F})?
\end{equation*}
\end{ques}

Observe that if $\Omega$ is a compact metric space with $\mathcal{F}= \mathcal{B}_\Omega$ and $\mathcal{U}\in \mathbf{C}_{\Omega\times X}^o$, it is not hard to find $\mathcal{W}\in \mathbf{C}_\Omega^o$ and $\mathcal{V}\in \mathbf{C}_X^o$ with $\mathcal{W}\times \mathcal{V}\succeq \mathcal{U}$, and hence $\xi\times \mathcal{V}\succeq \mathcal{U}$ for some $\xi\in \mathbf{P}_\Omega$,
 thus, using Corollary \ref{1102052218} one has
 \begin{equation*}
P_\mathcal{E} (\mathbf{D}, \mathbf{F})= \sup_{\mathcal{U}\in \mathbf{C}_\mathcal{E}^{t, o}}
P_\mathcal{E} (\mathbf{D}, \mathcal{U}, \mathbf{F}).
\end{equation*}
 Here, we
 denote by $\mathbf{C}_\mathcal{E}^{t, o}$ the set of all $\mathcal{U}\cap \mathcal{E}, \mathcal{U}\in \mathbf{C}_{\Omega\times X}^o$, and clearly $\mathbf{C}_\mathcal{E}^{t, o}\subseteq \mathbf{C}_\mathcal{E}^o$.

\section{Factor excellent and good covers} \label{factor good}

 In this section we introduce and discuss the concept of factor excellent and factor good covers which are necessary assumptions in our main result Theorem \ref{1007141414}. As shown by Theorem \ref{1007212202} and Theorem \ref{cover}, many interesting covers belong to this special class of finite measurable covers.

 \medskip

Recall that a topological space is \emph{zero-dimensional} if it has a topological base consisting of clopen subsets. For a zero-dimensional compact metric space, the set of all clopen subsets is countable.

Let $\mathcal{U}\in \mathbf{C}_\mathcal{E}$. Say $\mathcal{U}= \{U_1, \cdots, U_N\}, N\in \mathbb{N}$. Set
$$\mathbf{P}_\mathcal{U}= \{\{A_1, \cdots, A_N\}\in \mathbf{P}_\mathcal{E}: A_i\subseteq U_i, i= 1, \cdots, N\}.$$

Before proceeding, we state a well-known fact.

\begin{lem} \label{1007242349}
Let $Z$ be a zero-dimensional compact metric space and $\mathcal{W}\in \mathbf{C}_Z^o$. Set
$$\mathbf{P}_c (\mathcal{W})= \{\beta\in \mathbf{P}_\mathcal{W}: \beta\ \text{is clopen}\},$$
where $\mathbf{P}_\mathcal{W}$ is introduced similarly. Then $\mathbf{P}_c (\mathcal{W})$ is a countable family and, for each $\gamma\in \mathbf{P}_\mathcal{W}$, if $(Z, \mathcal{B}_Z, \eta)$ is a probability space then
$$\inf\limits_{\beta\in \mathbf{P}_c (\mathcal{W})} [H_\eta (\gamma| \beta)+ H_\eta (\beta| \gamma)]= 0.$$
\end{lem}

In the development of the local entropy theory of $\Z$-actions (or more generally the local entropy theory of a countable discrete amenable group action), a key point is a local version of the classical variational principle for entropy of finite open covers.
Lemma \ref{1007242349} plays an important role in the process of building the local variational principle. That is, people first prove the local variational principle in the case that the state space is zero-dimensional with the help of Lemma \ref{1007242349}, and then,
starting from this, people can obtain the local variational principle for a general dynamical system by standard arguments.
See \cite{BGH, HYZ1, HYZ} for more details.

\medskip

Inspired by Lemma \ref{1007242349}, we introduce the following concepts which, together with their variants, factor excellent and factor good, serve as essential assumptions in our main results.

Let $\mathcal{U}\in \mathbf{C}_\mathcal{E}$.
$\mathcal{U}$ is called \emph{excellent} (\emph{good}, respectively) if there exists a sequence $\{\alpha_n: n\in \mathbb{N}\}\subseteq \mathbf{P}_\mathcal{U}$ satisfying properties \eqref{1007202201} and \eqref{1007202202} (properties \eqref{1007202201} and \eqref{1007202203}, respectively), where
\begin{enumerate}

\item \label{1007202201}
For each $n\in \mathbb{N}$, $(\alpha_n)_\omega$ is a clopen partition of $\mathcal{E}_\omega$ for $\mathbb{P}$-a.e. $\omega\in \Omega$;

\item \label{1007202202} For each $\beta\in \mathbf{P}_\mathcal{U}$, if $\mu\in \mathcal{P}_\mathbb{P} (\mathcal{E})$ then
    \begin{equation} \label{1207201652}
    \inf_{n\in \mathbb{N}} [H_\mu (\beta| \alpha_n\vee \mathcal{F}_\mathcal{E})+ H_\mu (\alpha_n| \beta\vee \mathcal{F}_\mathcal{E})]= 0,
    \end{equation}
in fact, if $d \mu (\omega, x)= d \mu_\omega (x) d \mathbb{P} (\omega)$ is the disintegration of $\mu$ over $\mathcal{F}_\mathcal{E}$, then using \eqref{1007021626} and \eqref{0906282007}, condition \eqref{1207201652} is equivalent to:
\begin{equation*}
\inf_{n\in \mathbb{N}} \int_\Omega [H_{\mu_\omega} (\beta_\omega| (\alpha_n)_\omega)+ H_{\mu_\omega} ((\alpha_n)_\omega| \beta_\omega)] d \mathbb{P} (\omega)= 0.
\end{equation*}

\item \label{1007202203}
 For each $\mu\in \mathcal{P}_\mathbb{P} (\mathcal{E}, G)$, $h_{\mu}^{(r)} (\mathbf{F}, \mathcal{U})= \inf\limits_{n\in \mathbb{N}} h_\mu^{(r)} (\mathbf{F}, \alpha_n)$, by \eqref{1208011804}, this is equivalent to $h_{\mu}^{(r)} (\mathbf{F}, \beta)\ge \inf\limits_{n\in \mathbb{N}} h_\mu^{(r)} (\mathbf{F}, \alpha_n)$ for each $\beta\in \mathbf{P}_\mathcal{U}$.
\end{enumerate}
Property \eqref{1007202202} implies property \eqref{1007202203} by Proposition \ref{0911192237} \eqref{1007041247}, and excellent implies good.

\begin{rem} \label{1207262017}
In the above definitions (as distinct from the hypothesis of Lemma \ref{1007242349}), $X$ need not be zero-dimensional. For example, for $\xi\in \mathbf{P}_\Omega$,  $(\xi\times \mathcal{V})_\mathcal{E}$ will always be excellent in $\mathbf{C}_\mathcal{E}$ if we put $\alpha_n= (\xi\times \mathcal{V})_\mathcal{E}$ for each $n\in \N$, whenever $\mathcal{V}\in \mathbf{P}_X$ consists of clopen subsets of $X$. (Obviously such $\mathcal{V}$ may exist even if $X$ is not zero-dimensional.)
\end{rem}

It is easy to check:

\begin{lem} \label{1102061738}
Let $\mathcal{U}\in \mathbf{C}_\mathcal{E}^o$ and $\mathcal{U}'\in \mathbf{C}_\mathcal{E}^o$ with $\mathcal{U}'\succeq \mathcal{U}$, such that $\mathcal{U}'$ is good and $h_{\mu}^{(r)} (\mathbf{F}, \mathcal{U}')= h_{\mu}^{(r)} (\mathbf{F}, \mathcal{U})$ for each $\mu\in \mathcal{P}_\mathbb{P} (\mathcal{E}, G)$. Then $\mathcal{U}$ is good.
\end{lem}

We also have:

\begin{lem} \label{1102061920}
Let $\mathcal{U}_1, \mathcal{U}_2\in \mathbf{C}_\mathcal{E}^o$ and $W\in \mathcal{F}$. If both $\mathcal{U}_1$ and $\mathcal{U}_2$ are excellent then $\mathcal{U}_1\vee \mathcal{U}_2, [\mathcal{U}_1\cap (W\times X)]\cup [\mathcal{U}_2\cap (W^c\times X)]\in \mathbf{C}_\mathcal{E}^o$ and both of them are excellent.
\end{lem}
\begin{proof}
It is obvious that $\mathcal{U}_1\vee \mathcal{U}_2, [\mathcal{U}_1\cap (W\times X)]\cup [\mathcal{U}_2\cap (W^c\times X)]\in \mathbf{C}_\mathcal{E}^o$.

By assumption, for each $i= 1, 2$, there exists $\{\alpha_n^i: n\in \mathbb{N}\}\subseteq \mathbf{P}_{\mathcal{U}_i}$ satisfying
\begin{enumerate}

\item for each $n\in \mathbb{N}$, $(\alpha_n^i)_\omega$ is a clopen partition of $\mathcal{E}_\omega$ for $\mathbb{P}$-a.e. $\omega\in \Omega$ and

\item for each $\beta^i\in \mathbf{P}_{\mathcal{U}_i}$ and any $\mu\in \mathcal{P}_\mathbb{P} (\mathcal{E})$,
$$\inf_{n\in \mathbb{N}} [H_\mu (\beta^i| \alpha_n^i\vee \mathcal{F}_\mathcal{E})+ H_\mu (\alpha_n^i| \beta^i\vee \mathcal{F}_\mathcal{E})]= 0.$$
\end{enumerate}

First we consider $\mathcal{U}_1\vee \mathcal{U}_2$. For each $n_1, n_2\in \mathbb{N}$ set $\alpha_{n_1, n_2}= \alpha_{n_1}^1\vee \alpha_{n_2}^2$, it is clear that $\alpha_{n_1, n_2}\in \mathbf{C}_\mathcal{E}^o$ and $(\alpha_{n_1, n_2})_\omega$ is a clopen partition of $\mathcal{E}_\omega$ for $\mathbb{P}$-a.e. $\omega\in \Omega$. Now let $\beta\in \mathbf{P}_{\mathcal{U}_1\vee \mathcal{U}_2}$. Suppose that $\beta= \{B_{U_1, U_2}\subseteq U_1\cap U_2: U_1\in \mathcal{U}_1, U_2\in \mathcal{U}_2\}$. Set
$$\beta^1= \{\bigcup_{U_2\in \mathcal{U}_2} B_{U_1, U_2}: U_1\in \mathcal{U}_1\}\ \text{and}\ \beta^2= \{\bigcup_{U_1\in \mathcal{U}_1} B_{U_1, U_2}: U_2\in \mathcal{U}_2\}.$$
Then $\beta^i\in \mathbf{P}_{\mathcal{U}_i}, i= 1, 2$ and $\beta= \beta^1\vee \beta^2$. Let $\mu\in \mathcal{P}_\mathbb{P} (\mathcal{E})$. So
\begin{eqnarray*}
& & \inf_{n_1, n_2\in \mathbb{N}} [H_\mu (\beta| \alpha_{n_1, n_2}\vee \mathcal{F}_\mathcal{E})+ H_\mu (\alpha_{n_1, n_2}| \beta\vee \mathcal{F}_\mathcal{E})] \\
&= & \inf_{n_1, n_2\in \mathbb{N}} [H_\mu (\beta^1\vee \beta^2| \alpha_{n_1}^1\vee \alpha_{n_2}^2\vee \mathcal{F}_\mathcal{E})+ H_\mu (\alpha_{n_1}^1\vee \alpha_{n_2}^2| \beta^1\vee \beta^2\vee \mathcal{F}_\mathcal{E})] \\
&\le & \inf_{n_1, n_2\in \mathbb{N}} [H_\mu (\beta^1| \alpha_{n_1}^1\vee \mathcal{F}_\mathcal{E})+ H_\mu (\beta^2| \alpha_{n_2}^2\vee \mathcal{F}_\mathcal{E}) \\
& & \hskip 26pt + H_\mu (\alpha_{n_1}^1| \beta^1\vee \mathcal{F}_\mathcal{E})+ H_\mu (\alpha_{n_2}^2| \beta^2\vee \mathcal{F}_\mathcal{E})],
\end{eqnarray*}
by assumption, the sequence $\{\alpha_n^i: n\in \mathbb{N}\}\subseteq \mathbf{P}_{\mathcal{U}_i}, i= 1, 2$ satisfies:
\begin{equation*}
\inf_{n_1, n_2\in \mathbb{N}} [H_\mu (\beta| \alpha_{n_1, n_2}\vee \mathcal{F}_\mathcal{E})+ H_\mu (\alpha_{n_1, n_2}| \beta\vee \mathcal{F}_\mathcal{E})]= 0.
\end{equation*}
That is, $\mathcal{U}_1\vee \mathcal{U}_2$ is excellent.

Now let us consider $\mathcal{U}\doteq [\mathcal{U}_1\cap (W\times X)]\cup [\mathcal{U}_2\cap (W^c\times X)]$.

For each $n_1, n_2\in \mathbb{N}$ set $\alpha_{n_1, n_2}= [\alpha_{n_1}^1\cap (W\times X)]\cup [\alpha_{n_2}^2\cap (W^c\times X)]$. Obviously $\alpha_{n_1, n_2}\in \mathbf{C}_\mathcal{E}^o$ and $(\alpha_{n_1, n_2})_\omega$ is a clopen partition of $\mathcal{E}_\omega$ for $\mathbb{P}$-a.e. $\omega\in \Omega$.
Let $\beta\in \mathbf{P}_\mathcal{U}$. It is easy to choose
$\beta^i\in \mathbf{P}_{\mathcal{U}_i}, i= 1, 2$ such that $\beta= [\beta^1\cap (W\times X)]\cup [\beta^2\cap (W^c\times X)]$. Hence if $\mu\in \mathcal{P}_\mathbb{P} (\mathcal{E})$, and $d \mu (\omega, x)= d \mu_\omega (x) d \mathbb{P} (\omega)$ is the disintegration of $\mu$ over $\mathcal{F}_\mathcal{E}$, then, by the assumptions on the sequence $\{\alpha_n^i: n\in \mathbb{N}\}\subseteq \mathbf{P}_{\mathcal{U}_i}, i= 1, 2$ and using \eqref{1007021626} and \eqref{0906282007}, one has
\begin{equation} \label{1102062215}
\inf_{n\in \mathbb{N}} \int_\Omega [H_{\mu_\omega} ((\beta^i)_\omega| (\alpha_n^i)_\omega)+ H_{\mu_\omega} ((\alpha_n^i)_\omega| (\beta^i)_\omega)] d \mathbb{P} (\omega)= 0, i= 1, 2,
\end{equation}
and then, by the construction of $\beta^1, \beta^2, \alpha_{n_1, n_2}, n_1, n_2\in \N$,
\begin{eqnarray*}
& & \inf_{n_1, n_2\in \mathbb{N}} [H_\mu (\beta| \alpha_{n_1, n_2}\vee \mathcal{F}_\mathcal{E})+ H_\mu (\alpha_{n_1, n_2}| \beta\vee \mathcal{F}_\mathcal{E})] \\
&= & \inf_{n_1, n_2\in \mathbb{N}} \int_\Omega [H_{\mu_\omega} (\beta_\omega| (\alpha_{n_1, n_2})_\omega)+ H_{\mu_\omega} ((\alpha_{n_1, n_2})_\omega| \beta_\omega)] d \mathbb{P} (\omega) \\
&= & \inf_{n_1, n_2\in \mathbb{N}} \left\{\int_W [H_{\mu_\omega} (\beta^1_\omega| (\alpha_{n_1}^1)_\omega)+ H_{\mu_\omega} ((\alpha_{n_1}^1)_\omega| \beta^1_\omega)] d \mathbb{P} (\omega)\right. \\
& & \hskip 26pt \left.+ \int_{W^c} [H_{\mu_\omega} (\beta^2_\omega| (\alpha_{n_2}^2)_\omega)+ H_{\mu_\omega} ((\alpha_{n_2}^2)_\omega| \beta^2_\omega)] d \mathbb{P} (\omega)\right\} \\
&= & 0\ (\text{using \eqref{1102062215}}).
\end{eqnarray*}
This means that $\mathcal{U}$ is excellent.
\end{proof}

We now have the following important observation.

\begin{prop} \label{1007202212}
Assume that
$X$ is a zero-dimensional space.
\begin{enumerate}

\item \label{1007241819} If $\xi\in \mathbf{P}_\Omega$ and $\mathcal{V}\in \mathbf{C}_X^o$ then $(\xi\times \mathcal{V})_\mathcal{E}$ is excellent.

\item \label{1007241820} If $\mathcal{U}\in \mathbf{C}_\mathcal{E}^o$ is in the form of $\mathcal{U}= \{(\Omega_i\times U_i)^c: i= 1, \cdots, m\}, m\in \mathbb{N}\setminus \{1\}$ with $\Omega_i\in \mathcal{F}$ for each $i= 1, \cdots, m$ and $\{U_1^c, \cdots, U_m^c\}\in \mathbf{C}_X^o$, then $\mathcal{U}$ is good, in fact, there exists $\mathcal{U}'\in \mathbf{C}_\mathcal{E}^o$ such that $\mathcal{U}'\succeq \mathcal{U}$, $\mathcal{U}'$ is excellent and $h_{\mu}^{(r)} (\mathbf{F}, \mathcal{U}')= h_{\mu}^{(r)} (\mathbf{F}, \mathcal{U})$ for each $\mu\in \mathcal{P}_\mathbb{P} (\mathcal{E}, G)$.
\end{enumerate}
\end{prop}
\begin{proof}
\eqref{1007241819}
First, we shall prove the Proposition in the case where $\xi= \{\Omega\}$.

As $(\Omega, \mathcal{F}, \mathbb{P})$ is a Lebesgue space by Standard Assumption 3, there exists an isomorphism $\phi: (\Omega, \mathcal{F}, \mathbb{P})\rightarrow (Z, \mathcal{Z}, p)$ between probability spaces, where $Z$ is a zero-dimensional compact metric space and $\mathcal{Z}$ is the $p$-completion of $\mathcal{B}_Z$. Hence there exist $\Omega^*\in \mathcal{F}, Z^*\in \mathcal{Z}$ and an invertible measure-preserving transformation $\psi: \Omega^*\rightarrow Z^*$ such that $\mathbb{P} (\Omega^*)= 1= p (Z^*)$.
In fact, without loss of generality, we may assume that $\Omega= \Omega^*$.

Now define $\psi_*: (\Omega\times X, \mathcal{F}\times \mathcal{B}_X)\rightarrow (Z^*\times X, \mathcal{Z}^*\times \mathcal{B}_X), (\omega, x)\rightarrow (\psi \omega, x)$, where $\mathcal{Z}^*$ is the restriction of $\mathcal{Z}$ to $Z^*$. Then $\psi^*$ is an invertible bi-measurable map, by a standard proof.

For each $B\in \mathcal{Z}\times \mathcal{B}_X$, we set
$$B_\psi= \{(\psi^{- 1} z, x): (z, x)\in B\ \text{and}\ z\in Z^*\}.$$
In fact, $B_\psi= \psi^{- 1}_* (B\cap (Z^*\times X))$, in particular, $B_\psi\in \mathcal{F}\times \mathcal{B}_X$. Moreover, if $(\Omega\times X, \mathcal{F}\times \mathcal{B}_X, \mu)$ is a probability space, set $\mu_\psi (B)= \mu (B_\psi)$ for each $B\in \mathcal{B}_Z\times \mathcal{B}_X$. This defines a probability measure on $(Z\times X, \mathcal{B}_Z\times \mathcal{B}_X)$.

Now suppose that $\mathcal{V}= \{V_1, \cdots, V_N\}, N\in \mathbb{N}$ and set
\begin{equation*}
\mathbf{P}_c^* (\Omega\times \mathcal{V})= \{\{(A_1)_\psi\cap \mathcal{E}, \cdots, (A_N)_\psi\cap \mathcal{E}\}: \{A_1, \cdots, A_N\}\in \mathbf{P}_c (Z\times \mathcal{V})\}.
\end{equation*}
Observe that $Z\times X$ is a zero-dimensional compact metric space, and by Lemma \ref{1007242349}, $\mathbf{P}_c (Z\times \mathcal{V})$ is a countable family. Hence  $\mathbf{P}_c^* (\Omega\times \mathcal{V})$ is also a countable family.

We shall show that $\mathbf{P}_c^* (\Omega\times \mathcal{V})$ satisfies the required properties.

First, by construction, it is easy to see that, for each $\alpha\in \mathbf{P}_c^* (\Omega\times \mathcal{V})$, $\alpha\in \mathbf{P}_{(\Omega\times \mathcal{V})_\mathcal{E}}$ and $\alpha_\omega$ is a clopen partition of $\mathcal{E}_\omega$ for $\mathbb{P}$-a.e. $\omega\in \Omega$.
Now if $\beta= \{B_1, \cdots, B_N\}\in \mathbf{P}_\mathcal{E}$ satisfies $B_i\subseteq \Omega\times V_i$ for each $i= 1, \cdots, N$, it is not hard to obtain some $\beta'= \{B_1', \cdots, B_N'\}\in \mathbf{P}_{Z\times X}$ with $\psi_* (B_i)\subseteq B_i'\subseteq Z\times V_i$ for each $i= 1, \cdots, N$.
For each $\mu\in \mathcal{P}_\mathbb{P} (\mathcal{E})$, $\mu$ may be viewed as a probability measure over $(\Omega\times X, \mathcal{F}\times \mathcal{B}_X)$, and so by Lemma \ref{1007242349}
for each $\epsilon> 0$ there exists $\alpha'= \{A_1, \cdots, A_N\}\in \mathbf{P}_c (Z\times \mathcal{V})$ with
$$H_{\mu_\psi} (\alpha'| \beta')+ H_{\mu_\psi} (\beta'| \alpha')< \epsilon.$$
Set $\alpha= \{A_\psi\cap \mathcal{E}: A\in \alpha'\}\in \mathbf{P}^*_c (\Omega\times \mathcal{V})$. As $\mu (\mathcal{E})= 1$, it is easy from the constructions above, to check that
$$\mu (B_i)= \mu_\psi (B_i'), \mu ((A_i)_\psi\cap \mathcal{E})= \mu_\psi (A_i)\ \text{and}\ \mu ((A_i)_\psi\cap \mathcal{E}\cap B_j)= \mu_\psi (A_i\cap B_j')$$
 for all $i, j= 1, \cdots, N$, and so
\begin{eqnarray*}
& & \hskip -26pt H_\mu (\alpha| \beta\vee \mathcal{F}_\mathcal{E})+ H_\mu (\beta| \alpha\vee \mathcal{F}_\mathcal{E}) \\
& & \le H_\mu (\alpha| \beta)+ H_\mu (\beta| \alpha)= H_{\mu_\psi} (\alpha'| \beta')+ H_{\mu_\psi} (\beta'| \alpha')< \epsilon.
\end{eqnarray*}
This finishes the proof in the case of $\xi= \{\Omega\}$.

\medskip

Now we shall prove the Proposition for a general $\xi\in\mathbf{P}_\Omega$.
In fact,
$$(\xi\times \mathcal{V})_\mathcal{E}= (\xi\times X)_\mathcal{E}\vee (\Omega\times \mathcal{V})_\mathcal{E}.$$
Now from the definition it follows that $(\xi\times X)_\mathcal{E}\in \mathbf{C}_X^o$ is excellent (as $\mathbf{P}_{(\xi\times X)_\mathcal{E}}= \{(\xi\times X)_\mathcal{E}\}$), and by the above arguments $(\Omega\times \mathcal{V})_\mathcal{E}\in \mathbf{C}_X^o$ is excellent, thus using Lemma \ref{1102061920} one sees that $(\xi\times \mathcal{V})_\mathcal{E}$ is also excellent.

 \eqref{1007241820} Obviously, in $\mathcal{F}$ there exist disjoint $\Omega_i'\subseteq \Omega_i^c, i= 1, \cdots, m$ with $\bigcup\limits_{i= 1}^m \Omega_i'= \bigcup\limits_{i= 1}^m \Omega_i^c$. Now set $\Omega_0= \Omega\setminus \bigcup\limits_{i= 1}^m \Omega_i'= \bigcap\limits_{i= 1}^m \Omega_i$ and
 $$\mathcal{U}'= \{(\Omega_i'\times X)\cap \mathcal{E}: i= 1, \cdots, m\}\cup \{(\Omega_0\times U_i^c)\cap \mathcal{E}: i= 1, \cdots, m\}.$$
 It is easy to see that $\mathcal{U}'\in \mathbf{C}_\mathcal{E}^o$ and $\mathcal{U}'\succeq \mathcal{U}$. In fact, $\mathcal{U}'_\omega= \mathcal{U}_\omega$ for $\mathbb{P}$-a.e. $\omega\in \Omega$ and so by Lemma \ref{1007261204} one has $h_{\mu}^{(r)} (\mathbf{F}, \mathcal{U}')= h_{\mu}^{(r)} (\mathbf{F}, \mathcal{U})$ for each $\mu\in \mathcal{P} (\mathcal{E}, G)$.

 Now with the help of Lemma \ref{1102061738} we shall finish our proof by showing that $\mathcal{U}'$ is excellent.
 In fact, suppose that $\xi= \{\Omega_i': i= 1, \cdots, m\}\cup \{\Omega_0\}\in \mathbf{P}_\Omega$. Then
 $$\mathcal{U}'= [(\xi\times X)_\mathcal{E}\cap (\Omega_0^c\times X)]\cup [(\Omega\times \mathcal{V})_\mathcal{E}\cap (\Omega_0\times X)],$$
 where $\mathcal{V}= \{U_1^c, \cdots, U_m^c\}$. Observe that by \eqref{1007241819} one has that $(\xi\times X)_\mathcal{E}, (\Omega\times \mathcal{V})_\mathcal{E}\in \mathbf{C}_\mathcal{E}^o$ are both excellent, and so using Lemma \ref{1102061920}
we conclude that $\mathcal{U}'$ is excellent.
\end{proof}

Before proceeding, we need to introduce the concept of a factor map in the setting of continuous bundle RDS's.

For each $i= 1, 2$, let $X_i$ be a compact metric space with $\mathcal{E}_i\in \mathcal{F}\times \mathcal{B}_{X_i}$ and the family $\mathbf{F}_i= \{(F_i)_{g, \omega}: (\mathcal{E}_i)_\omega\rightarrow (\mathcal{E}_i)_{g \omega}| g\in G, \omega\in \Omega\}$ the corresponding continuous bundle RDS. By a \emph{factor map from $\mathbf{F}_1$ to $\mathbf{F}_2$} we mean a measurable map $\pi: \mathcal{E}_1\rightarrow \mathcal{E}_2$ satisfying
\begin{enumerate}

\item $\pi_\omega$, the restriction of $\pi$ over $(\mathcal{E}_1)_\omega$, is a continuous surjection from $(\mathcal{E}_1)_\omega$ to $(\mathcal{E}_2)_\omega$ for $\mathbb{P}$-a.e. $\omega\in \Omega$ and

\item $\pi_{g \omega}\circ (F_1)_{g, \omega}= (F_2)_{g, \omega}\circ \pi_\omega$ for each $g\in G$ and $\mathbb{P}$-a.e. $\omega\in \Omega$.
\end{enumerate}
In this case, it is obvious that $\pi^{- 1} (\mathcal{U}_2)\in \mathbf{P}_{\mathcal{E}_1}$ ($\mathbf{C}_{\mathcal{E}_1}$, $\mathbf{C}_{\mathcal{E}_1}^o$, respectively) if $\mathcal{U}_2\in \mathbf{P}_{\mathcal{E}_2}$ ($\mathbf{C}_{\mathcal{E}_2}$, $\mathbf{C}_{\mathcal{E}_2}^o$, respectively). Then $\mathcal{U}_2\in \mathbf{C}_{\mathcal{E}_2}^o$ is called \emph{factor excellent} (\emph{factor good}, respectively) if there exists a factor map $\pi$ such that $\pi^{- 1} (\mathcal{U}_2)$ is excellent (good, respectively).

Let $\mathcal{U}\in \mathbf{C}_\mathcal{E}^o$. In general we don't know whether $\mathcal{U}$ is (factor) good, even if $X$ is a zero-dimensional space. However, we have:

\begin{lem} \label{1007192218}
Let $\mathcal{U}= \{U_1, \cdots, U_N\}\in \mathbf{C}_\mathcal{E}^o, N\in \mathbb{N}$. Assume that $X$ is a zero-dimensional space.
Then
 there exists $\alpha= \{A_1, \cdots, A_N\}\in \mathbf{P}_\mathcal{E}$ such that $\alpha\succeq \mathcal{U}$ and $\alpha_\omega$ is a clopen partition of $\mathcal{E}_\omega$ for $\mathbb{P}$-a.e. $\omega\in \Omega$.
\end{lem}
\begin{proof}
Let $\pi: \Omega\times X\rightarrow X$ be the natural projection. We may assume without any loss of generality that $\mathcal{E}_\omega$ is a non-empty compact subset of $X$ and $\mathcal{U}_\omega\in \mathbf{C}_{\mathcal{E}_\omega}^o$ for each $\omega\in \Omega$.

As $X$ is zero-dimensional, there exists a countable topological basis $\{V_n: n\in \mathbb{N}\}$ of $X$ consisting of clopen subsets (here, we take $V_1= \emptyset$).

Note that, if $I_1, \cdots, I_N$ are $N$ finite disjoint non-empty subsets of $\mathbb{N}$, and we set
\begin{equation*}
\Omega (I_1, \cdots, I_N)= \pi ((\Omega\times X\setminus \bigcup_{j\in \bigcup\limits_{i= 1}^N I_i} V_j)\cap \mathcal{E})\cup \bigcup_{i= 1}^N \pi ((\Omega\times \bigcup_{j\in I_i} V_j\setminus U_i)\cap \mathcal{E}),
\end{equation*}
then by Lemma \ref{1007152201} we have $\Omega (I_1, \cdots, I_N)\in \mathcal{F}$. Moreover, $\omega\notin \Omega (I_1, \cdots, I_N)$ if and only if
$\mathcal{E}_\omega\subseteq \bigcup\limits_{j\in \bigcup\limits_{i= 1}^N I_i} V_j$ and
$\bigcup\limits_{j\in I_i} V_j\cap \mathcal{E}_\omega\subseteq (U_i)_\omega$ for each $i= 1, \cdots, N$.

Now for any given $\omega\in \Omega$, as $\mathcal{U}_\omega\in \mathbf{C}_{\mathcal{E}_\omega}^o$ and as $X$ is a zero-dimensional space, there exists $\alpha (\omega)\in \mathbf{P}_{\mathcal{E}_\omega}$ consisting of clopen subsets $A_1 (\omega), \cdots, A_N (\omega)$ with the property that $A_i (\omega)\subseteq (U_i)_\omega, i= 1, \cdots, N$. Furthermore, there exist $N$ finite disjoint non-empty subsets $I_1 (\omega), \cdots, I_N (\omega)\subseteq \mathbb{N}$ such that $A_i (\omega)= \bigcup\limits_{j\in I_i (\omega)} V_j\cap \mathcal{E}_\omega$ for each $i= 1, \cdots, N$. In particular, $\omega\in \Omega (I_1 (\omega), \cdots, I_N (\omega))^c$.

 Thus, there exists a countably family $\{\{I_1^n, \cdots, I_N^n\}: n\in \mathbb{N}\}$ of $N$ finite disjoint non-empty subsets of $\mathbb{N}$ and a sequence $\{\Omega_n: n\in \mathbb{N}\}\subseteq \mathcal{F}$ such that $\bigcup\limits_{n\in \mathbb{N}} \Omega_n= \Omega$, $\Omega_n\cap \Omega_m= \emptyset$ whenever $1\le n< m$ and $\mathbb{P} (\Omega_n)> 0, \Omega_n\subseteq \Omega (I_1^n, \cdots, I_N^n)^c$ for each $n\in \mathbb{N}$. Now set
 \begin{equation*}
 \alpha= \{\bigcup_{n\in \mathbb{N}} (\Omega_n\times \bigcup_{j\in I_i^n} V_j)\cap \mathcal{E}: i= 1, \cdots, N\}.
 \end{equation*}
From the above construction it is not hard to check that $\alpha$ has the claimed properties. This completes the proof.
\end{proof}

We also have:

\begin{prop} \label{1007211043}
Let $\mathbf{F}= \{F_{g, \omega}: \mathcal{E}_\omega\rightarrow \mathcal{E}_{g \omega}| g\in G, \omega\in \Omega\}$ be a continuous bundle RDS over $(\Omega,
\mathcal{F}, \mathbb{P}, G)$. Then there exists a family $\mathbf{F}'= \{F_{g, \omega}': \mathcal{E}_\omega'\rightarrow \mathcal{E}_{g \omega}'| g\in G, \omega\in \Omega\}$ (with $\mathcal{E}'\in \mathcal{F}\times \mathcal{B}_{X'}$ and $X'$ a compact metric state space), which is a continuous bundle RDS over $(\Omega, \mathcal{F}, \mathbb{P}, G)$, and a factor map $\pi: \mathcal{E}'\rightarrow \mathcal{E}$ from $\mathbf{F}'$ to $\mathbf{F}$, such that $X'$ is a zero-dimensional space. In fact, $\pi$ is induced by a continuous surjection from $X'$ to $X$.
\end{prop}
\begin{proof}
It is well known that there exists a continuous surjection $\phi: C\rightarrow X$, where $C$ is a Cantor space. Then $G$ acts naturally on the space $C^G$ with $g': (c_g)_{g\in G}\mapsto (c_{g' g})_{g\in G}$ whenever $g'\in G$. There is a natural projection
$$\psi: \Omega\times C^G\rightarrow \Omega\times X, (\omega, (c_g)_{g\in G})\mapsto (\omega, \phi (c_{e_G})).$$
Now we consider $X'= C^G$, which is a zero-dimensional compact metric space, and
\begin{equation*}
\mathcal{E}'= \{(\omega, (c_g)_{g\in G})\in \psi^{- 1} (\mathcal{E}): \phi (c_g)= F_{g, \omega} \phi (c_{e_G})\ \text{for each}\ g\in G\ \text{and any}\ \omega\in \Omega\}
\end{equation*}
with the family $\mathbf{F}'= \{F_{g, \omega}': \mathcal{E}_\omega'\rightarrow \mathcal{E}_{g \omega}'| g\in G, \omega\in \Omega\}$ given by
$$F_{g', \omega}: \mathcal{E}_\omega'\ni (c_g)_{g\in G}\mapsto (c_{g' g})_{g\in G}, g'\in G, \omega\in \Omega.$$
The map $\pi: \mathcal{E}'\rightarrow \mathcal{E}$ given by $(\omega, (c_g)_{g\in G})\mapsto (\omega, \phi (c_{e_G}))$ is clearly well defined. In the following we shall check step-by-step that  $X', \mathcal{E}', \mathbf{F}'$ and $\pi$ as constructed satisfy the required properties.

\medskip

\begin{claim} \label{claim 1}
The family $\mathbf{F}'= \{F_{g, \omega}': \mathcal{E}_\omega'\rightarrow \mathcal{E}_{g \omega}'| g\in G, \omega\in \Omega\}$, which is well defined naturally, is a continuous bundle RDS over $(\Omega, \mathcal{F}, \mathbb{P}, G)$.
\end{claim}

\begin{proof}[Proof of Claim \ref{claim 1}]
First, the map
$$\psi_G: \Omega\times C^G\rightarrow \Omega\times X^G, (\omega, (c_g)_{g\in G})\mapsto (\omega, (\phi c_g)_{g\in G})$$
 is obviously measurable. We let $\mathcal{E}'= \psi_G^{- 1} (\mathcal{E}_G)$, where
    \begin{equation*}
    \mathcal{E}_G= \{(\omega, (x_g)_{g\in G}): (\omega, x_{e_G})\in \mathcal{E}, x_g= F_{g, \omega} x_{e_G}\ \text{for each}\ g\in G\ \text{and any}\ \omega\in \Omega\}.
    \end{equation*}
    Since $\mathcal{E}_G\in \mathcal{F}\times \mathcal{B}_{X^G}$, it follows that $\mathcal{E}'\in \mathcal{F}\times \mathcal{B}_{X'}$.

    The measurability of
    $$(\omega, (c_g)_{g\in G})\in \mathcal{E}'\mapsto F'_{g', \omega} ((c_g)_{g\in G})= (c_{g' g})_{g\in G}$$
     for fixed $g'\in G$ and the equality $F'_{g_2, g_1 \omega}\circ F'_{g_1, \omega}= F'_{g_2 g_1, \omega}$ for each $\omega\in \Omega$ and all $g_1, g_2\in G$ are easy to see. Finally, it is not hard to check that $\emptyset\neq \mathcal{E}_\omega'\subseteq X'$ is a compact subset and $F'_{g, \omega}$ is continuous for each $g\in G$. We have shown that the family $\mathbf{F}'$ is a continuous bundle RDS over $(\Omega, \mathcal{F}, \mathbb{P}, G)$.
     \end{proof}

\begin{claim} \label{claim 2}
$\pi$ is a factor map from $\mathcal{E}'$ to $\mathcal{E}$.
\end{claim}

\begin{proof}[Proof of Claim \ref{claim 2}]
In fact, let $\omega\in \Omega$, obviously $\pi_\omega: \mathcal{E}_\omega'\rightarrow \mathcal{E}_\omega$ is a continuous surjection; now let $g'\in G$, if $(\omega, (c_g)_{g\in G})\in \mathcal{E}'$ then
    \begin{eqnarray*}
    \pi_{g' \omega}\circ F'_{g', \omega} ((c_g)_{g\in G})= \pi_{g' \omega} ((c_{g' g})_{g\in G})&= & \phi (c_{g'}) \\
    &= & F_{g', \omega}\circ \phi (c_{e_G})= F_{g', \omega}\circ \pi_\omega ((c_g)_{g\in G}),
    \end{eqnarray*}
    which establishes the identity $\pi_{g' \omega}\circ F'_{g', \omega}= F_{g', \omega}\circ \pi_\omega$.
\end{proof}

It is clear that $\pi$ is induced by the continuous surjection $X'\rightarrow X, (c_g)_{g\in G}\mapsto \phi (c_{e_G})$. This completes the proof.
\end{proof}

 By Proposition \ref{1007202212} and Proposition \ref{1007211043}, one has:

 \begin{thm} \label{1007212202}
The following statements hold:
\begin{enumerate}

\item If $\xi\in \mathbf{P}_\Omega$ and $\mathcal{V}\in \mathbf{C}_X^o$ then $(\xi\times \mathcal{V})_\mathcal{E}$ is factor excellent.

\item If $\mathcal{U}\in \mathbf{C}_\mathcal{E}^o$ has the form
$\mathcal{U}= \{(\Omega_i\times U_i)^c: i= 1, \cdots, n\}$,
$n\in \mathbb{N}\setminus \{1\}$ with $\Omega_i\in \mathcal{F}, i= 1, \cdots, n$ and $\{U_1^c, \cdots, U_n^c\}\in \mathbf{C}_X^o$, then $\mathcal{U}$ is factor good.
\end{enumerate}
 \end{thm}

 By Lemma \ref{1007242349} and Proposition \ref{1007211043}, one has:

 \begin{thm} \label{cover}
 Assume that $\Omega$ is a zero-dimensional compact metric space with $\mathcal{F}= \mathcal{B}_\Omega$. Then each member of $\mathbf{C}_\mathcal{E}^{t, o}$ is factor excellent.

 (Recall from the end of \S \ref{fourth} that $\mathbf{C}_\mathcal{E}^{t, o}$ is the set of all $\mathcal{U}\cap \mathcal{E}, \mathcal{U}\in \mathbf{C}_{\Omega\times X}^o$.)
 \end{thm}

Suppose that the family $\mathbf{F}_i= \{(F_i)_{g, \omega}: (\mathcal{E}_i)_\omega\rightarrow (\mathcal{E}_i)_{g \omega}| g\in G, \omega\in \Omega\}$ is a continuous bundle RDS over $(\Omega,
\mathcal{F}, \mathbb{P}, G)$, $i= 1, 2$ and $\pi: \mathcal{E}_1\rightarrow \mathcal{E}_2$ a factor map from $\mathbf{F}_1$ to $\mathbf{F}_2$. $\pi$ naturally induces a map from $\mathcal{P}_\mathbb{P} (\mathcal{E}_1)$ to $\mathcal{P}_\mathbb{P} (\mathcal{E}_2)$, which we may denote by $\pi$ without any ambiguity.

It is simple to see:

\begin{lem} \label{1007212122}
Suppose that for $i= 1, 2$ the family $\mathbf{F}_i= \{(F_i)_{g, \omega}: (\mathcal{E}_i)_\omega\rightarrow (\mathcal{E}_i)_{g \omega}| g\in G, \omega\in \Omega\}$ is a continuous bundle RDS over $(\Omega,
\mathcal{F}, \mathbb{P}, G)$ with corresponding compact metric state space $X_i$.
Assume that $\pi: \mathcal{E}_1\rightarrow \mathcal{E}_2$ is a factor map from $\mathbf{F}_1$ to $\mathbf{F}_2, \mu\in \mathcal{P}_\mathbb{P} (\mathcal{E}_1, G), \mathcal{U}\in \mathbf{C}_{\mathcal{E}_2}$ and $\mathbf{D}= \{d_F: F\in \mathcal{F}_G\}\subseteq \mathbf{L}^1_{\mathcal{E}_2} (\Omega, C (X_2))$ is a sub-additive $G$-invariant family. Then
 \begin{enumerate}

 \item If the sequence $\{\eta_n: n\in \N\}$ converges to $\eta$ in $\mathcal{P}_\mathbb{P} (\mathcal{E}_1)$  then the sequence $\{\pi \eta_n: n\in \N\}$ converges to $\pi \eta$ in $\mathcal{P}_\mathbb{P} (\mathcal{E}_2)$.  In other words, the map $\pi: \mathcal{P}_\mathbb{P} (\mathcal{E}_1)\rightarrow \mathcal{P}_\mathbb{P} (\mathcal{E}_2)$ is continuous.

 \item $\pi \mu\in \mathcal{P}_\mathbb{P} (\mathcal{E}_2, G)$.

 \item $\mathbf{D}\circ \pi\doteq \{d_F\circ \pi: F\in \mathcal{F}_G\}$ is a sub-additive $G$-invariant family in $\mathbf{L}^1_{\mathcal{E}_1} (\Omega, C (X_1))$. Moreover, if $\mathbf{D}$ is monotone then $\mathbf{D}\circ \pi$ is also monotone.

         \item \label{1007271634} $h_\mu^{(r)} (\mathbf{F}_1, \pi^{- 1} \mathcal{U})= h_{\pi \mu}^{(r)} (\mathbf{F}_2, \mathcal{U})$ and so $h_\mu^{(r)} (\mathbf{F}_1)\ge h_{\pi \mu}^{(r)} (\mathbf{F}_2)$.

         \item \label{1007271635} For each $F\in \mathcal{F}_G$ and  for any $\omega\in \Omega$,
          $$P_{\mathcal{E}_1} (\omega, \mathbf{D}\circ \pi, F, \pi^{- 1} \mathcal{U}, \mathbf{F}_1)= P_{\mathcal{E}_2} (\omega, \mathbf{D}, F, \mathcal{U}, \mathbf{F}_2).$$
         Hence if $\mathbf{D}$ is monotone then $P_{\mathcal{E}_1} (\mathbf{D}\circ \pi, \pi^{- 1} \mathcal{U}, \mathbf{F}_1)= P_{\mathcal{E}_2} (\mathbf{D}, \mathcal{U}, \mathbf{F}_2)$. In particular, $h_{\text{top}}^{(r)} (\mathbf{F}_1, \pi^{- 1} \mathcal{U})= h_{\text{top}}^{(r)} (\mathbf{F}_2, \mathcal{U})$.  As a consequence,
             $$P_{\mathcal{E}_1} (\mathbf{D}\circ \pi, \mathbf{F}_1)\ge P_{\mathcal{E}_2} (\mathbf{D}, \mathbf{F}_2)\ \text{and}\  h_{\text{top}}^{(r)} (\mathbf{F}_1)\ge h_{\text{top}}^{(r)} (\mathbf{F}_2).$$
 \end{enumerate}
 \end{lem}
 \begin{proof}
The first three statements are easy to check; we prove the last two.

In fact, the last item follows from \eqref{1007052309} and the fact of $\mathbf{P} ((\pi^{- 1} \mathcal{U})_F)= \pi^{- 1} \mathbf{P} (\mathcal{U}_F)\doteq \{\{\pi^{- 1} B: B\in \beta\}: \beta\in \mathbf{P} (\mathcal{U}_F)\}$ for each $F\in \mathcal{F}_G$.

As for \eqref{1007271634}, suppose that $d \mu (\omega, x)= d \mu_\omega (x) d \mathbb{P} (\omega)$ is the disintegration of $\mu$ over $\mathcal{F}_{\mathcal{E}_1}$. Then it is not hard to check that $d (\pi \mu) (\omega, y)= d (\pi_\omega \mu_\omega) (y) d \mathbb{P} (\omega)$ is the disintegration of $\pi \mu$ over $\mathcal{F}_{\mathcal{E}_2}$. Hence for each $F\in \mathcal{F}_G$,
\begin{eqnarray}
& & \hskip -36pt
H_\mu ((\pi^{- 1} \mathcal{U})_F| \mathcal{F}_{\mathcal{E}_1})\nonumber \\
& &= \int_\Omega H_{\mu_\omega} (((\pi^{- 1} \mathcal{U})_F)_\omega) d \mathbb{P} (\omega)\ (\text{using \eqref{1007041151}})\nonumber \\
& &= \int_\Omega \inf_{\beta (\omega)\in \mathbf{P} (((\pi^{- 1} \mathcal{U})_F)_\omega)} H_{\mu_\omega} (\beta (\omega)) d \mathbb{P} (\omega)\ (\text{using \eqref{1006291640}})\nonumber \\
& &= \int_\Omega \inf_{\alpha\in \mathbf{P} ((\pi^{- 1} \mathcal{U})_F)} H_{\mu_\omega} (\alpha_\omega) d \mathbb{P} (\omega)\ (\text{using Lemma \ref{1007051758}}) \label{1007271651} \\
& &= \int_\Omega \inf_{\beta\in \mathbf{P} (\mathcal{U}_F)} H_{\mu_\omega} ((\pi^{- 1} \beta)_\omega) d \mathbb{P} (\omega)\ (\text{as $\mathbf{P} ((\pi^{- 1} \mathcal{U})_F)= \pi^{- 1} \mathbf{P} (\mathcal{U}_F)$})\nonumber \\
& &= \int_\Omega \inf_{\beta\in \mathbf{P} (\mathcal{U}_F)} H_{\pi_\omega \mu_\omega} (\beta_\omega) d \mathbb{P} (\omega)\nonumber \\
& &= H_{\pi \mu} (\mathcal{U}_F| \mathcal{F}_{\mathcal{E}_2})\ (\text{by a reasoning similar to \eqref{1007271651}}), \nonumber
\end{eqnarray}
and so $h_\mu^{(r)} (\mathbf{F}_1, \pi^{- 1} \mathcal{U})= h_{\pi \mu}^{(r)} (\mathbf{F}_2, \mathcal{U})$. This finishes our proof.
 \end{proof}

 We end this section with the following nice property of a factor good cover.

 A generalized real-valued function $f$ defined on a compact space $Z$ is called \emph{upper semi-continuous} (u.s.c.) if one of the following equivalent conditions holds:
 \begin{enumerate}

\item $\limsup\limits_{z'\rightarrow z} f (z')\le f (z)$ for each $z\in Z$.

\item for each $r\in \R$, the set $\{z\in Z: f (z)\ge r\}\subseteq Z$ is closed.
 \end{enumerate}
Notice that the infimum of any family of u.s.c. functions is again u.s.c., and similarly both the sum and the supremum of finitely many u.s.c. functions are u.s.c.

It follows that:

 \begin{prop} \label{1102102326}
Assume that $\mathcal{U}\in \mathbf{C}_\mathcal{E}^o$ is factor good. Then $h_\bullet^{(r)} (\mathbf{F}, \mathcal{U}): \mathcal{P}_\mathbb{P} (\mathcal{E}, G)$ $\rightarrow \R, \mu\mapsto h_\mu^{(r)} (\mathbf{F}, \mathcal{U})$ is a u.s.c. function.
 \end{prop}
 \begin{proof}
 First, we prove the Proposition in the case that $\mathcal{U}$ is good. By assumption, there exists a sequence $\{\alpha_n: n\in \mathbb{N}\}\subseteq \mathbf{P}_\mathcal{U}$ satisfying:
\begin{enumerate}

\item
For each $n\in \mathbb{N}$, $(\alpha_n)_\omega$ is a clopen partition of $\mathcal{E}_\omega$ for $\mathbb{P}$-a.e. $\omega\in \Omega$ and

\item \label{1208011936}
 For each $\mu\in \mathcal{P}_\mathbb{P} (\mathcal{E}, G)$, $h_{\mu}^{(r)} (\mathbf{F}, \mathcal{U})= \inf\limits_{n\in \mathbb{N}} h_\mu^{(r)} (\mathbf{F}, \alpha_n)$.
\end{enumerate}
Observe that $X$ is not necessarily zero-dimensional by Remark \ref{1207262017}. The existence of the sequence $\{\alpha_n: n\in \mathbb{N}\}$
follows from the assumption that $\mathcal{U}$ is good.
By the assumptions on the sequence $\{\alpha_n: n\in \mathbb{N}\}$, one sees that for each $F\in \mathcal{F}_G$ and for any $n\in \N$, $(\alpha_n)_F\in \mathbf{P}_\mathcal{E}$ satisfies that $((\alpha_n)_F)_\omega$ is a clopen partition of $\mathcal{E}_\omega$ for $\mathbb{P}$-a.e. $\omega\in \Omega$, and so applying Proposition \ref{1007192023} \eqref{1207301347} to $(\alpha_n)_F$ we obtain that
the function
$$H_\bullet ((\alpha_n)_F| \mathcal{F}_\mathcal{E}): \mathcal{P}_\mathbb{P} (\mathcal{E})\rightarrow \R, \mu\mapsto H_\mu ((\alpha_n)_F| \mathcal{F}_\mathcal{E})$$
 is u.s.c. It follows that the function
 $$h_{\bullet}^{(r)} (\mathbf{F}, \alpha_n): \mathcal{P}_\mathbb{P} (\mathcal{E}, G)\rightarrow \R, \mu\mapsto h_{\mu}^{(r)} (\mathbf{F}, \alpha_n)$$
  is also u.s.c. for each $n\in \N$ (using \eqref{1006272232}), which implies that the function
$$h_{\bullet}^{(r)} (\mathbf{F}, \mathcal{U}): \mathcal{P}_\mathbb{P} (\mathcal{E}, G)\rightarrow \R, \mu\mapsto h_{\mu}^{(r)} (\mathbf{F}, \mathcal{U})= \inf\limits_{n\in \mathbb{N}} h_\mu^{(r)} (\mathbf{F}, \alpha_n)\ (\text{using \eqref{1208011936}})$$
is u.s.c., as it is the infimum of a family of u.s.c. functions.

 For the general case, our assumptions imply that there exists a continuous bundle RDS $\mathbf{F}'= \{F_{g, \omega}': \mathcal{E}_\omega'\rightarrow \mathcal{E}_{g \omega}'| g\in G, \omega\in \Omega\}$ (with $\mathcal{E}'\in \mathcal{F}\times \mathcal{B}_{X'}$ and $X'$ a compact metric state space) and a factor map $\pi: \mathcal{E}'\rightarrow \mathcal{E}$ from $\mathbf{F}'$ to $\mathbf{F}$ such that $\pi^{- 1} \mathcal{U}$ is good. By the above arguments, the function
 $$h_{\bullet}^{(r)} (\mathbf{F}', \pi^{- 1} \mathcal{U}): \mathcal{P}_\mathbb{P} (\mathcal{E}', G)\rightarrow \R, \mu'\mapsto h_{\mu'}^{(r)} (\mathbf{F}', \pi^{- 1} \mathcal{U})$$
  is u.s.c. Now applying Lemma \ref{1007212122} we deduce
 \begin{equation*}
 h_{\pi \mu'}^{(r)} (\mathbf{F}, \mathcal{U})= h_{\mu'}^{(r)} (\mathbf{F}', \pi^{- 1} \mathcal{U})\ \text{for each $\mu'\in \mathcal{P}_\mathbb{P} (\mathcal{E}', G)$}.
 \end{equation*}
Recall that $\mathcal{P}_\mathbb{P} (\mathcal{E}', G)$ and $\mathcal{P}_\mathbb{P} (\mathcal{E}, G)$ are both compact metric spaces, the map $\pi: \mathcal{P}_\mathbb{P} (\mathcal{E}', G)\rightarrow \mathcal{P}_\mathbb{P} (\mathcal{E}, G)$ is continuous by Lemma \ref{1007212122} and $\pi \mathcal{P}_\mathbb{P} (\mathcal{E}', G)= \mathcal{P}_\mathbb{P} (\mathcal{E}, G)$ (cf \cite[Proposition 2.5]{Liu} for the special case of $G= \Z$). Thus, for any $r\in \R$,
 \begin{equation*}
 \{\mu\in \mathcal{P}_\mathbb{P} (\mathcal{E}, G): h_{\mu}^{(r)} (\mathbf{F}, \mathcal{U})\ge r\}= \pi (\{\mu'\in \mathcal{P}_\mathbb{P} (\mathcal{E}', G): h_{\mu'}^{(r)} (\mathbf{F}', \pi^{- 1} \mathcal{U})\ge r\})
 \end{equation*}
 is a closed subset, which completes our proof.
 \end{proof}

\section{A variational principle for local fiber topological pressure}\label{variational principle concerning pressure}

In this section we present our main result, Theorem \ref{1007141414}. We will  postpone its proof to next section: here we give the statement, and some remarks and direct applications of it.

Here is our main result.

\begin{thm} \label{1007141414}
Let $\mathbf{D}= \{d_F: F\in \mathcal{F}_G\}\subseteq \mathbf{L}_\mathcal{E}^1 (\Omega, C (X))$ be a monotone sub-additive $G$-invariant family satisfying:

\medskip

\begin{center}
{\bf $(\spadesuit)$}
\begin{minipage}[]{110mm}
\emph{for any given sequence $\{\nu_n: n\in \mathbb{N}\}\subseteq \mathcal{P}_\mathbb{P} (\mathcal{E})$, set
$\mu_n= \frac{1}{|F_n|} \sum\limits_{g\in F_n} g \nu_n$
 for each $n\in \mathbb{N}$, then there exists a subsequence $\{n_j: j\in \mathbb{N}\}\subseteq \mathbb{N}$ such that $\{\mu_{n_j}: j\in \mathbb{N}\}$ converges to some $\mu\in \mathcal{P}_\mathbb{P} (\mathcal{E})$ (and hence $\mu\in \mathcal{P}_\mathbb{P} (\mathcal{E}, G)$) with
\begin{equation*}
\limsup_{j\rightarrow \infty} \frac{1}{|F_{n_j}|} \int_{\mathcal{E}} d_{F_{n_j}} (\omega, x) d \nu_{n_j} (\omega, x)\le \mu (\mathbf{D}).
\end{equation*}}
\end{minipage}
\end{center}

\medskip

\noindent Assume that $\mathcal{U}\in \mathbf{C}_\mathcal{E}^o$ is factor good. Then
\begin{equation} \label{1207281618}
P_\mathcal{E} (\mathbf{D}, \mathcal{U}, \mathbf{F})= \max_{\mu\in \mathcal{P}_\mathbb{P} (\mathcal{E}, G)} [h_\mu^{(r)} (\mathbf{F}, \mathcal{U})+ \mu (\mathbf{D})]
\end{equation}
and
\begin{equation} \label{1207281619}
P_\mathcal{E} (\mathbf{D}, \mathbf{F})= \sup_{\mu\in \mathcal{P}_\mathbb{P} (\mathcal{E}, G)} [h_\mu^{(r)} (\mathbf{F})+ \mu (\mathbf{D})].
\end{equation}
In particular,
\begin{equation} \label{1207281620}
h_{\text{top}}^{(r)} (\mathbf{F}, \mathcal{U})= \max_{\mu\in \mathcal{P}_\mathbb{P} (\mathcal{E}, G)} h_\mu^{(r)} (\mathbf{F}, \mathcal{U})\ \text{and}\
h_{\text{top}}^{(r)} (\mathbf{F})= \sup_{\mu\in \mathcal{P}_\mathbb{P} (\mathcal{E}, G)} h_\mu^{(r)} (\mathbf{F}).
\end{equation}
\end{thm}

In view of Theorem \ref{1007141414}, and in particular \eqref{1207281619},
$h_\mu^{(r)} (\mathbf{F})+ \mu (\mathbf{D})$ may be viewed as the general definition of measure-theoretic pressure in the setting of a continuous bundle RDS and a monotone sub-additive $G$-invariant family satisfying $(\spadesuit)$. Note that \cite{Z-DCDS} provides another way of defining measure-theoretic pressure in the setting of topological dynamical systems.

We believe that Theorem \ref{1007141414} holds for all $\mathcal{U}\in \mathbf{C}_\mathcal{E}^o$, but we have not so far been able to prove it in full generality. In fact, Proposition \ref{1007211043} tells us that, each continuous bundle RDS can be lifted to another continuous bundle RDS with a zero-dimensional compact metric space as its state space, and the associated map between these two continuous bundle RDS's is induced by a continuous surjection between their compact metric state spaces; and Theorem \ref{1007212202} and Theorem \ref{cover} show that many elements of $\mathbf{C}_\mathcal{E}^o$ are indeed factor good.
 Observing these facts, it seems possible to prove that each $\mathcal{U}\in \mathbf{C}_\mathcal{E}^o$ is factor good, and if this were the case then Theorem \ref{1007141414} would hold for all $\mathcal{U}\in \mathbf{C}_\mathcal{E}^o$.

 We also believe that a monotone sub-additive $G$-invariant family $\mathbf{D}= \{d_F: F\in \mathcal{F}_G\}\subseteq \mathbf{L}_\mathcal{E}^1 (\Omega, C (X))$ always satisfies assumption $(\spadesuit)$, and if this were the case, Theorem \ref{1007141414} would hold for any monotone sub-additive $G$-invariant family. In fact, let $\nu_n\in \mathcal{P}_\mathbb{P} (\mathcal{E})$ and $\mu_n= \frac{1}{|F_n|} \sum\limits_{g\in F_n} g \nu_n$
 for each $n\in \mathbb{N}$.
 By compactness of the space $\mathcal{P}_\mathbb{P} (\mathcal{E})$, there always exists a subsequence $\{n_j: j\in \mathbb{N}\}\subseteq \mathbb{N}$ such that the sequence $\{\mu_{n_j}: j\in \mathbb{N}\}$ converges to some $\mu\in \mathcal{P}_\mathbb{P} (\mathcal{E}, G)$ (cf Proposition \ref{1007192023}). Observe that, by Proposition \ref{1008292115}, the family $\mathbf{D}$ does indeed satisfy $(\spadesuit)$ if  $G$ is abelian.
  Further investigation of  $(\spadesuit)$, is made in \S \ref{assumption} and \S \ref{special}.

   \medskip

Remark that, if we remove the assumption of monotonicity from the family $\mathbf{D}= \{d_F: F\in \mathcal{F}_G\}\subseteq \mathbf{L}_\mathcal{E}^1 (\Omega, C (X))$ in Theorem \ref{1007141414} and assume that $\mathbf{D}$ is just a sub-additive $G$-invariant family satisfying $(\spadesuit)$ and if, in addition, there exists a finite constant $C\in \R_+$ such that $\mathbf{D}'= \{d_F': F\in \mathcal{F}_G\}\subseteq \mathbf{L}_\mathcal{E}^1 (\Omega, C (X))$ is a monotone sub-additive $G$-invariant family, where $d_F'= d_F+ |F| C$ for each $F\in \mathcal{F}_G$, then we can introduce $P_\mathcal{E} (\mathbf{D}, \mathcal{U}, \mathbf{F}), P_\mathcal{E} (\mathbf{D}, \mathbf{F})$ and $\mu (\mathbf{D})$ similarly for each $\mathcal{U}\in \mathbf{C}_\mathcal{E}$ and any $\mu\in \mathcal{P}_\mathbb{P} (\mathcal{E}, G)$. In fact,
\begin{equation}
\label{1208091655} P_\mathcal{E} (\mathbf{D}, \mathcal{U}, \mathbf{F})= P_\mathcal{E} (\mathbf{D}', \mathcal{U}, \mathbf{F})- C
\ \text{and}\ P_\mathcal{E} (\mathbf{D}, \mathbf{F})= P_\mathcal{E} (\mathbf{D}', \mathbf{F})- C.
\end{equation}
 It is easy to check that the family $\mathbf{D}'$ also satisfies $(\spadesuit)$. Hence
 in the case that $\mathcal{U}\in \mathbf{C}_\mathcal{E}^o$ is factor good,
we may apply Theorem \ref{1007141414} to $\mathbf{D}'$ and $\mathcal{U}$, and then using \eqref{1208091655} we obtain
\begin{equation} \label{1207291438}
P_\mathcal{E} (\mathbf{D}, \mathcal{U}, \mathbf{F})= \max_{\mu\in \mathcal{P}_\mathbb{P} (\mathcal{E}, G)} [h_\mu^{(r)} (\mathbf{F}, \mathcal{U})+ \mu (\mathbf{D})]
\end{equation}
and
\begin{equation} \label{1207291439}
P_\mathcal{E} (\mathbf{D}, \mathbf{F})= \sup_{\mu\in \mathcal{P}_\mathbb{P} (\mathcal{E}, G)} [h_\mu^{(r)} (\mathbf{F})+ \mu (\mathbf{D})].
\end{equation}

\begin{rem} \label{1010210035}
As we will see in \S\S \ref{1208042011}, the equations \eqref{1207291438} and \eqref{1207291439} can be used to obtain the main results of  Ledrappier and Walters \cite{LW} .
 \end{rem}

Now let $\mathbf{D}= \{d_F: F\in \mathcal{F}_G\}\subseteq \mathbf{L}_\mathcal{E}^1 (\Omega, C (X))$ be a family satisfying:

\medskip

\begin{center}
{\bf $(\clubsuit)$}
\begin{minipage}[]{112mm}
\emph{for any given sequence $\{\nu_n: n\in \mathbb{N}\}\subseteq \mathcal{P}_\mathbb{P} (\mathcal{E})$, set
$\mu_n= \frac{1}{|F_n|} \sum\limits_{g\in F_n} g \nu_n$
 for each $n\in \mathbb{N}$, there always exists a subsequence $\{n_j: j\in \mathbb{N}\}\subseteq \mathbb{N}$ such that the sequence $\{\mu_{n_j}: j\in \mathbb{N}\}$ converges to some $\mu\in \mathcal{P}_\mathbb{P} (\mathcal{E}, G)$ with
\begin{equation*}
\limsup_{j\rightarrow \infty} \frac{1}{|F_{n_j}|} \int_{\mathcal{E}} d_{F_{n_j}} (\omega, x) d \nu_{n_j} (\omega, x)\le \limsup_{n\rightarrow \infty} \frac{1}{|F_n|} \int_{\mathcal{E}} d_{F_n} (\omega, x) d \mu (\omega, x).
\end{equation*}}
\end{minipage}
\end{center}

\medskip

\noindent
Remark that, for each $f\in
 \mathbf{L}_\mathcal{E}^1 (\Omega, C (X))$, $\mathbf{D}^f$ is a family in $\mathbf{L}_\mathcal{E}^1 (\Omega, C (X))$ satisfying the above assumption, since
  $$\mathbf{D}^f= \{d_F^f (\omega, x)\doteq \sum\limits_{g\in F} f (g (\omega, x)): F\in \mathcal{F}_G\}.$$
As in \S \ref{fourth}, $P_\mathcal{E} (\omega, \mathbf{D}, F_n, \mathcal{U}, \mathbf{F})$ can be introduced similarly.
By a reasoning similar to the proof of Theorem \ref{1007141414}, which we present in \S \ref{seventh}, it is not hard to see that, if $\mathcal{U}\in \mathbf{C}_\mathcal{E}^o$ is factor good,
\begin{eqnarray} \label{1103031445}
& & \limsup_{n\rightarrow \infty} \frac{1}{|F_n|} \int_\Omega \log P_\mathcal{E} (\omega, \mathbf{D}, F_n, \mathcal{U}, \mathbf{F}) d \mathbb{P} (\omega)\nonumber \\
& & \hskip 26pt = \max_{\mu\in \mathcal{P}_\mathbb{P} (\mathcal{E}, G)} [h_\mu^{(r)} (\mathbf{F}, \mathcal{U})+ \limsup_{n\rightarrow \infty} \frac{1}{|F_n|} \int_{\mathcal{E}} d_{F_n} (\omega, x) d \mu (\omega, x)].
\end{eqnarray}
Then by Theorem \ref{1006122212} and Theorem \ref{1007212202} we have
\begin{eqnarray} \label{1208011114}
& & \sup_{\mathcal{V}\in \mathbf{C}_X^o}
\limsup_{n\rightarrow \infty} \frac{1}{|F_n|} \int_\Omega \log P_\mathcal{E} (\omega, \mathbf{D}, F_n, (\Omega\times \mathcal{V})_\mathcal{E}, \mathbf{F}) d \mathbb{P} (\omega)\nonumber \\
& & \hskip 26pt = \sup_{\mu\in \mathcal{P}_\mathbb{P} (\mathcal{E}, G)} [h_{\mu}^{(r)} (\mathbf{F})+ \limsup_{n\rightarrow \infty} \frac{1}{|F_n|} \int_{\mathcal{E}} d_{F_n} (\omega, x) d \mu (\omega, x)].
\end{eqnarray}
In particular,
\begin{eqnarray} \label{1208011732}
& & \limsup_{n\rightarrow \infty} \frac{1}{|F_n|} \int_\Omega \log P_\mathcal{E} (\omega, \mathbf{D}^f, F_n, \mathcal{U}, \mathbf{F}) d \mathbb{P} (\omega)\nonumber \\
& & \hskip 26pt = \max_{\mu\in \mathcal{P}_\mathbb{P} (\mathcal{E}, G)} [h_\mu^{(r)} (\mathbf{F}, \mathcal{U})+ \int_\mathcal{E} f (\omega, x) d \mu (\omega, x)]
\end{eqnarray}
for factor good $\mathcal{U}\in \mathbf{C}_\mathcal{E}^o$, and
\begin{eqnarray} \label{1207292302}
& & \sup_{\mathcal{V}\in \mathbf{C}_X^o}
\limsup_{n\rightarrow \infty} \frac{1}{|F_n|} \int_\Omega \log P_\mathcal{E} (\omega, \mathbf{D}^f, F_n, (\Omega\times \mathcal{V})_\mathcal{E}, \mathbf{F}) d \mathbb{P} (\omega)\nonumber \\
& & \hskip 26pt = \sup_{\mu\in \mathcal{P}_\mathbb{P} (\mathcal{E}, G)} [h_{\mu}^{(r)} (\mathbf{F})+ \int_\mathcal{E} f (\omega, x) d \mu (\omega, x)].
\end{eqnarray}

\begin{rem} \label{1103031512}
Recall from \S \ref{third} that by a TDS we mean: $G$ acts over a compact metric space as a group of homeomorphisms of the space.

Let $(Y, G)$ be a TDS and $f\in C (Y)$. Denote by $P (f, Y)$, the topological pressure of $f$ over $Y$, \cite[Definition 5.2.1]{MO}. As shown at the beginning of \S \ref{third}, the setting of $(Y, G)$ and $f$
can be viewed as a continuous bundle RDS with:
 \begin{enumerate}

\item \label{8901} $(\Omega, \mathcal{F}, \mathbb{P}, G)$ is a trivial MDS in the sense that $\Omega$ is a singleton $\{\omega_0\}$,

 \item \label{8902} $\mathcal{E}= \{\omega_0\}\times Y$,

\item \label{8903} $\mathbf{F}= \{F_{g, \omega_0}: \{\omega_0\}\times Y\rightarrow \{\omega_0\}\times Y| g\in G\}$, where
 $F_{g, \omega_0}: (\omega_0, y)\mapsto (\omega_0, g y)$ for each $g\in G$ and any $y\in Y$, and

 \item \label{8904} $\mathbf{D}= \{d_F: F\in \mathcal{F}_G\}$, where $d_F (\omega_0, y)= \sum\limits_{g\in F} f (g y)$ for each $y\in Y$.
 \end{enumerate}

 Applying \eqref{1207292302} to \eqref{8901}, \eqref{8902}, \eqref{8903} and \eqref{8904}, we obtain
 \begin{equation} \label{1207292336}
 P (f, Y)= \sup [h_\mu (G, Y)+ \int_Y f (y) d \mu (y)],
 \end{equation}
 where the supremum is taken over all $G$-invariant Borel probability measures $\mu$ over $Y$. Observe that \eqref{1207292336} is indeed  \cite[Variational Principle 5.2.7]{MO}.

 Now let $\mathcal{V}\in \mathbf{C}_Y^o$.
 Denote by $P (f, \mathcal{V})$, the topological $\mathcal{V}$-pressure of $f$, introduced in \cite[\S 2]{LY}. Note that
 $\mathcal{V}$ can be viewed naturally as $(\{\omega_0\}\times \mathcal{V})_\mathcal{E}\in \mathbf{C}_\mathcal{E}^o$ in the above setting of \eqref{8901}, \eqref{8902}, \eqref{8903} and \eqref{8904}. As $(\{\omega_0\}\times \mathcal{V})_\mathcal{E}$ is factor good by Theorem \ref{1007212202}, applying \eqref{1208011732} to this setting we obtain
 \begin{equation} \label{1208012336}
 P (f, \mathcal{V})= \max [h_\mu (G, \mathcal{V})+ \int_Y f (y) d \mu (y)],
 \end{equation}
 where the maximum is taken over all $G$-invariant Borel probability measures $\mu$ over $Y$.
 Observe that \eqref{1208012336} is indeed \cite[Corollary 1.2]{LY}, recovering \cite[Theorem 5.1]{HYZ}.

 In fact, using similar arguments as above and observing \cite[Lemma 3.6]{LY}, it is not hard to see that \eqref{1103031445} generalizes \cite[Theorem 1.1]{LY}, the main result of \cite{LY} by Liang and Yan. Remark that, just before the submission of the paper, we found a preprint version of \cite{LY} from the internet.
 \end{rem}

As a direct corollary of Theorem \ref{1007141414}, we can extend Lemma \ref{1102071655} as follows with the additional assumption that the family $\mathbf{D}$ satisfies $(\spadesuit)$.

\begin{prop} \label{1102071746}
Let $\mathbf{D}= \{d_F: F\in \mathcal{F}_G\}\subseteq \mathbf{L}_\mathcal{E}^1 (\Omega, C (X))$ be a monotone sub-additive $G$-invariant family satisfying $(\spadesuit)$. Then
$$\text{sup}_\mathbb{P} (\mathbf{D})= \max_{\mu\in \mathcal{P}_\mathbb{P} (\mathcal{E}, G)} \mu (\mathbf{D}).$$
\end{prop}
\begin{proof}
It is easy to see that $\{\mathcal{E}\}= (\Omega\times \{X\})_\mathcal{E}\in \mathbf{C}_\mathcal{E}^o$ is excellent, and so by Theorem \ref{1007141414} one has
$$P_\mathcal{E} (\mathbf{D}, \{\mathcal{E}\}, \mathbf{F})= \max_{\mu\in \mathcal{P}_\mathbb{P} (\mathcal{E}, G)} [h_\mu^{(r)} (\mathbf{F}, \{\mathcal{E}\})+ \mu (\mathbf{D})].$$
Observe that $h_{\text{top}}^{(r)} (\mathbf{F}, \{\mathcal{E}\})= 0$ and $h_\mu^{(r)} (\mathbf{F}, \{\mathcal{E}\})= 0$ for each $\mu\in \mathcal{P}_\mathbb{P} (\mathcal{E}, G)$, and so by Proposition \ref{1102041733} we have the conclusion.
\end{proof}

As in the discussions preceding Remark \ref{1103031512}, we assume that
$\mathbf{D}= \{d_F: F\in \mathcal{F}_G\}\subseteq \mathbf{L}_\mathcal{E}^1 (\Omega, C (X))$ is a family satisfying $(\clubsuit)$.
Thus we can apply \eqref{1103031445} to $\{\mathcal{E}\}= (\Omega\times \{X\})_\mathcal{E}\in \mathbf{C}_\mathcal{E}^o$ and obtain
\begin{eqnarray*}
& & \limsup_{n\rightarrow \infty} \frac{1}{|F_n|} \int_\Omega \log P_\mathcal{E} (\omega, \mathbf{D}, F_n, \{\mathcal{E}\}, \mathbf{F}) d \mathbb{P} (\omega)\nonumber \\
& & \hskip 26pt = \max_{\mu\in \mathcal{P}_\mathbb{P} (\mathcal{E}, G)} [h_\mu^{(r)} (\mathbf{F}, \{\mathcal{E}\})+ \limsup_{n\rightarrow \infty} \frac{1}{|F_n|} \int_{\mathcal{E}} d_{F_n} (\omega, x) d \mu (\omega, x)].
\end{eqnarray*}
In other words,
\begin{eqnarray} \label{1208011057}
& & \limsup_{n\rightarrow \infty} \frac{1}{|F_n|} \int_\Omega \sup_{x\in \mathcal{E}_\omega} d_{F_n} (\omega, x) d \mathbb{P} (\omega)\nonumber \\
& & \hskip 26pt = \max_{\mu\in \mathcal{P}_\mathbb{P} (\mathcal{E}, G)} \limsup_{n\rightarrow \infty} \frac{1}{|F_n|} \int_{\mathcal{E}} d_{F_n} (\omega, x) d \mu (\omega, x).
\end{eqnarray}

\medskip

The concept of a principal extension was first introduced and studied by Ledrappier in \cite{Le}. This concept is that a topological dynamical system and its factor have the same measure-theoretic entropy for all invariant Borel probability measures of the system. This plays an important role in the study of relative entropy theory and entropy theory of symbolic extensions \cite{BD}.

Based on the same ideas, we can introduce a similar concept in our setting.

Let the family $\mathbf{F}_i= \{(F_i)_{g, \omega}: (\mathcal{E}_i)_\omega\rightarrow (\mathcal{E}_i)_{g \omega}| g\in G, \omega\in \Omega\}$ be a continuous bundle RDS over $(\Omega,
\mathcal{F}, \mathbb{P}, G)$ with $X_i$ the corresponding compact metric state space, $i= 1, 2$ and $\pi: \mathcal{E}_1\rightarrow \mathcal{E}_2$ a factor map from $\mathbf{F}_1$ to $\mathbf{F}_2$. $\pi$ is called \emph{principal} if $h_{\mu_1}^{(r)} (\mathbf{F}_1)= h_{\pi \mu_1}^{(r)} (\mathbf{F}_2)$ for each $\mu_1\in \mathcal{P}_\mathbb{P} (\mathcal{E}_1, G)$.

Before proceeding, we also need the following result.

\begin{lem} \label{1008171752}
Let the family $\mathbf{F}_i= \{(F_i)_{g, \omega}: (\mathcal{E}_i)_\omega\rightarrow (\mathcal{E}_i)_{g \omega}| g\in G, \omega\in \Omega\}$ be a continuous bundle RDS over $(\Omega,
\mathcal{F}, \mathbb{P}, G)$ with $X_i$ the corresponding compact metric state space, $i= 1, 2$ and $\pi: \mathcal{E}_1\rightarrow \mathcal{E}_2$ a factor map from $\mathbf{F}_1$ to $\mathbf{F}_2$. Assume that $\mathbf{D}= \{d_F: F\in \mathcal{F}_G\}\subseteq \mathbf{L}^1_{\mathcal{E}_2} (\Omega, C (X_2))$ satisfies the assumption of $(\spadesuit)$ with respect to $\mathbf{F}_2$. Then $\mathbf{D}\circ \pi$ satisfies the assumption of $(\spadesuit)$ with respect to $\mathbf{F}_1$.
\end{lem}
\begin{proof}
 Let $\{\nu_n: n\in \mathbb{N}\}\subseteq \mathcal{P}_\mathbb{P} (\mathcal{E}_1)$ be a given sequence and set $\mu_n= \frac{1}{|F_n|} \sum\limits_{g\in F_n} g \nu_n$ for each $n\in \mathbb{N}$. As $\mathbf{D}$ satisfies the assumption of $(\spadesuit)$ with respect to $\mathbf{F}_2$, then
there always exists some subsequence $\{n_j: j\in \mathbb{N}\}\subseteq \mathbb{N}$ such that the sequence $\{\pi \mu_{n_j}: j\in \mathbb{N}\}$ converges to some $\mu'\in \mathcal{P}_\mathbb{P} (\mathcal{E}_2, G)$ and
\begin{equation} \label{1008171814}
\limsup_{j\rightarrow \infty} \frac{1}{|F_{n_j}|} \int_{\mathcal{E}_2} d_{F_{n_j}} (\omega, x) d (\pi \nu_{n_j}) (\omega, x)\le \mu' (\mathbf{D}).
\end{equation}
Note that by Proposition \ref{1007192023} we may assume that $\{\mu_{n_j}: j\in \mathbb{N}\}$ converges to some $\mu\in \mathcal{P}_\mathbb{P} (\mathcal{E}_1, G)$ (by selecting a subsequence of $\{n_j: j\in \mathbb{N}\}$ if necessary). Obviously, $\pi \mu= \mu'$ and then \eqref{1008171814} can be restated as:
\begin{equation*}
\limsup_{j\rightarrow \infty} \frac{1}{|F_{n_j}|} \int_{\mathcal{E}_1} d_{F_{n_j}}\circ \pi (\omega, x) d \nu_{n_j} (\omega, x)\le \mu (\mathbf{D}\circ \pi).
\end{equation*}
That is, $\mathbf{D}\circ \pi$ satisfies the assumption of $(\spadesuit)$ with respect to $\mathbf{F}_1$.
\end{proof}

Now, given continuous bundle RDS's $\mathbf{F}_1$ and $\mathbf{F}_2$ over $(\Omega,
\mathcal{F}, \mathbb{P}, G)$, and given a factor map $\pi: \mathcal{E}_1\rightarrow \mathcal{E}_2$ from $\mathbf{F}_1$ to $\mathbf{F}_2$, observe $\pi \mathcal{P}_\mathbb{P} (\mathcal{E}_1, G)= \mathcal{P}_\mathbb{P} (\mathcal{E}_2, G)$ (cf \cite[Proposition 2.5]{Liu} for the special case of $G= \Z$). Thus, by the definition, Theorem \ref{1007141414} and Lemma \ref{1008171752} one has:

\begin{prop} \label{1010241620}
Let the family $\mathbf{F}_i= \{(F_i)_{g, \omega}: (\mathcal{E}_i)_\omega\rightarrow (\mathcal{E}_i)_{g \omega}| g\in G, \omega\in \Omega\}$ be a continuous bundle RDS over $(\Omega,
\mathcal{F}, \mathbb{P}, G)$ with $X_i$ the corresponding compact metric state space, $i= 1, 2$ and $\pi: \mathcal{E}_1\rightarrow \mathcal{E}_2$ a factor map from $\mathbf{F}_1$ to $\mathbf{F}_2$. Assume that $\mathbf{D}= \{d_F: F\in \mathcal{F}_G\}\subseteq \mathbf{L}^1_{\mathcal{E}_2} (\Omega, C (X_2))$ is a monotone sub-additive $G$-invariant family satisfying the assumption of $(\spadesuit)$ with respect to $\mathbf{F}_2$. If $\pi$ is principal then
\begin{equation*}
P_{\mathcal{E}_2} (\mathbf{D}, \mathbf{F}_2)= P_{\mathcal{E}_1} (\mathbf{D}\circ \pi, \mathbf{F}_1),\ \text{particularly,}\ h_{\text{top}}^{(r)} (\mathbf{F}_2)= h_{\text{top}}^{(r)} (\mathbf{F}_1).
\end{equation*}
\end{prop}

Let the family $\mathbf{F}_i= \{(F_i)_{g, \omega}: (\mathcal{E}_i)_\omega\rightarrow (\mathcal{E}_i)_{g \omega}| g\in G, \omega\in \Omega\}$ be a continuous bundle RDS over $(\Omega,
\mathcal{F}, \mathbb{P}, G)$ with $X_i$ the corresponding compact metric state space, $i= 1, 2$, and $\pi: \mathcal{E}_1\rightarrow \mathcal{E}_2$ a factor map from $\mathbf{F}_1$ to $\mathbf{F}_2$.

Denote by $|\pi^{- 1} (\omega, x_2)|$ the cardinality of $\pi^{- 1} (\omega, x_2)$. In the case of $G= \Z$ Liu proved the following result \cite[Theorem 2.3]{Liu}: if in addition $|\pi^{- 1} (\omega, x_2)|$ is measurable in $(\omega, x_2)\in \mathcal{E}_2$ and, for $\mathbb{P}$-a.e. $\omega\in \Omega$,  $|\pi^{- 1} (\omega, x_2)|$ is finite for each $x_2\in (\mathcal{E}_2)_\omega$, then, in our notation of a continuous bundle RDS,
$$P_{\mathcal{E}_2} (\mathbf{D}^f, \mathbf{F}_2)= P_{\mathcal{E}_1} (\mathbf{D}^f\circ \pi, \mathbf{F}_1)$$
for each $f\in \mathbf{L}^1_{\mathcal{E}_2} (\Omega, C (X_2))$.

  For each $\mu_1\in \mathcal{P}_\mathbb{P} (\mathcal{E}_1, G)$ we see that $\pi$ may be viewed as a given $G$-invariant sub-$\sigma$-algebra $\mathcal{C}$ of an MDS $(\mathcal{E}_1, (\mathcal{F}\times \mathcal{B}_{X_1})\cap \mathcal{E}_1, \mu_1, G)$.
  As the state space $(\Omega, \mathcal{F}, \mathbb{P})$ is a Lebesgue space, the special case where $\pi$ is a principal extension is $h_{\mu_1} (G, \mathcal{E}_1| \mathcal{C})= 0$ for each $\mu_1\in \mathcal{P}_\mathbb{P} (\mathcal{E}_1, G)$. Then the well-known  Abramov-Rokhlin entropy addition formula (cf Proposition \ref{1006291603}) states:
$$h_{\mu_1}^{(r)} (\mathbf{F}_1)\le h_{\pi \mu_1}^{(r)} (\mathbf{F}_2)+ h_{\mu_1} (G, \mathcal{E}_1| \mathcal{C}),$$
in our notation. Thus, by Proposition \ref{0811192316} one sees that the assumption  $\pi$ in \cite[Theorem 2.3]{Liu} is just a very special case of a principal extension, and so \cite[Theorem 2.3]{Liu} can be deduced from \eqref{1207292302}, or from Corollary \ref{1102082226} and Proposition \ref{1008292115}, variants of Theorem \ref{1007141414} and Proposition \ref{1010241620}.

\section{Proof of main result Theorem \ref{1007141414}} \label{seventh}

In this section, we present our somewhat complicated proof of Theorem \ref{1007141414} following the ideas of \cite{HYZ1, HYZ, Mis, Z-Thesis, Z-DCDS} and the references therein.

In fact, Theorem \ref{1007141414} follows from the following result.

\begin{prop} \label{1006151236}
Let $\mathbf{D}= \{d_F: F\in \mathcal{F}_G\}\subseteq \mathbf{L}_\mathcal{E}^1 (\Omega, C (X))$ be a monotone sub-additive $G$-invariant family satisfying $(\spadesuit)$ and $\mathcal{U}\in \mathbf{C}_\mathcal{E}^o$. If $\mathcal{U}$ is factor good then, for some $\mu\in \mathcal{P}_\mathbb{P} (\mathcal{E}, G)$,
\begin{equation*}
 h_{\mu}^{(r)} (\mathbf{F}, \mathcal{U})+ \mu (\mathbf{D})\ge P_\mathcal{E} (\mathbf{D}, \mathcal{U}, \mathbf{F}).
\end{equation*}
\end{prop}

Let us first present the proof of Theorem \ref{1007141414} assuming Proposition \ref{1006151236}.

\begin{proof}[Proof of Theorem \ref{1007141414}]
\eqref{1207281618} follows from Proposition \ref{1007120919} and Proposition \ref{1006151236}.

Observe that \eqref{1207281618} holds for each $(\Omega\times \mathcal{V})_\mathcal{E}, \mathcal{V}\in \mathbf{C}_X^o$, as $(\Omega\times \mathcal{V})_\mathcal{E}\in \mathbf{C}_\mathcal{E}^o$ is factor good by Theorem \ref{1007212202}.
Combining this with
Theorem \ref{1006122212} we obtain \eqref{1207281619}.

It is clear that $\mathbf{D}^0$ is a monotone sub-additive $G$-invariant family satisfying $(\spadesuit)$, thus \eqref{1207281620} follows directly from \eqref{1207281618} and \eqref{1207281619}. This finishes our proof.
\end{proof}

Now let us prove Proposition \ref{1006151236}. Our proof is preceded by five Lemmas and a Proposition.

\medskip

First we need the following result. Recall that by Standard Assumptions 3 and 4, $(\Omega, \mathcal{F}, \mathbb{P})$ is a Lebesgue space and $X$ is a compact metric space.

\begin{lem} \label{1007152310}
Let $p: \Omega\rightarrow X$ be a measurable map and $\alpha\in \mathbf{P}_\mathcal{E}$. Assume that $B\doteq \{(\omega, p (\omega)): \omega\in \Omega\}\subseteq \mathcal{E}$. Then
$$\bigcup\limits_{\omega\in \Omega} \{\omega\}\times \alpha_{\omega} (p (\omega))\in \mathcal{F}\times \mathcal{B}_X.$$
\end{lem}
\begin{proof}
Let $\pi: \Omega\times X\rightarrow \Omega$ be the natural projection.
Since $X$ is a compact metric space, it is well known that $B\in \mathcal{F}\times \mathcal{B}_X$ (see for example \cite[Proposition III.13]{CV}).
  Note that $B\subseteq \mathcal{E}$ and $\alpha\in \mathbf{P}_\mathcal{E}$,
so clearly there exist distinct atoms $A_1, \cdots, A_n, n\in \mathbb{N}$ from $\alpha$ such that $B\subseteq \bigcup\limits_{i= 1}^n A_i$ and $B\cap A_i\neq \emptyset$, for each $i= 1, \cdots, n$. In fact, for each $i= 1, \cdots, n$, set $C_i= \pi (B\cap A_i)$. Then $\bigcup\limits_{k= 1}^n C_k= \Omega$ and $C_i\cap C_j= \emptyset$ once $1\le i\neq j\le n$. Moreover, for each $i= 1, \cdots, n$, $C_i\in \mathcal{F}$ by Lemma \ref{1007152201}. Thus $\{C_1, \cdots, C_n\}\in \mathbf{P}_\Omega$, and then
\begin{equation*}
\bigcup_{\omega\in \Omega} \{\omega\}\times \alpha_{\omega} (p (\omega))= \bigcup_{i= 1}^n \bigcup_{\omega\in C_i} \{\omega\}\times \alpha_{\omega} (p (\omega))= \bigcup_{i= 1}^n [(C_i\times X)\cap A_i]\in \mathcal{F}\times \mathcal{B}_X.
\end{equation*}
This completes the proof.
\end{proof}

The following selection lemma is a random variation of \cite[Lemma 3.1]{Z-DCDS}, and plays an important role in the proof of Proposition \ref{1006151236}.

\begin{lem} \label{1007141419}
Let $\mathbf{D}= \{d_F: F\in \mathcal{F}_G\}\subseteq \mathbf{L}_\mathcal{E}^1 (\Omega, C (X))$ and $\mathcal{U}\in \mathbf{C}_\mathcal{E}$. Assume that $\alpha_k\in \mathbf{P}_\mathcal{E}$ satisfies $\alpha_k\succeq \mathcal{U}$ for each $1\le k\le K$, where $K\in \mathbb{N}$. Then for each $F\in \mathcal{F}_G$ there exists a family of finite subsets $B_{F, \omega}\subseteq \mathcal{E}_\omega, \omega\in \Omega$ such that
\begin{enumerate}

\item For $B_F\doteq \{(\omega, x): \omega\in \Omega, x\in B_{F, \omega}\}$,
$$\sum\limits_{x\in B_{F, \omega}} e^{d_F (\omega, x)}> \frac{1}{K} \left[\inf\limits_{\beta (\omega)\in \mathbf{P}_{\mathcal{E}_\omega}, \beta (\omega)\succeq (\mathcal{U}_F)_\omega} \sum\limits_{B\in \beta (\omega)} \sup\limits_{x\in B} e^{d_F (\omega, x)}- \frac{1}{2} e^{- ||d_F (\omega)||_\infty}\right],$$

\item The family depends measurably on $\omega\in \Omega$ in the sense that $B_F\in \mathcal{F}\times \mathcal{B}_X$ and

\item Each atom of $((\alpha_k)_F)_\omega$ contains at most one point from $B_{F, \omega}, 1\le k\le K$.
\end{enumerate}
\end{lem}
\begin{proof}
We may assume that $\alpha_1, \cdots, \alpha_K\in \mathbf{P}_\mathcal{U}$. Recall that, if $\mathcal{U}= \{U_1, \cdots, U_n\}, n\in \mathbb{N}$, then
$$\mathbf{P}_\mathcal{U}= \{\{A_1, \cdots, A_n\}\in \mathbf{P}_\mathcal{E}: A_i\subseteq U_i, i= 1, \cdots, n\}.$$
In particular, the cardinality of each $\alpha_k, k= 1, \cdots, K$ is at most $n$.


Let $\pi: \Omega\times X\rightarrow \Omega$ be the natural projection.

Set $\mathcal{E}_0= \mathcal{E}$. We may assume without loss of generality that $\pi (\mathcal{E}_0)= \Omega$. Observe that, by Lemma \ref{1007152201}, there exists a measurable map $p_1: \Omega\rightarrow X$ such that
$(\omega, p_1 (\omega))\in \mathcal{E}_0$ for each $\omega\in \pi (\mathcal{E}_0)$ and
\begin{equation*}
e^{d_F (\omega, p_1 (\omega))}\ge \sup_{x\in (\mathcal{E}_0)_\omega} e^{d_F (\omega, x)}- \frac{1}{2^{1+ 1} K} e^{- ||d_F (\omega)||_\infty}.
\end{equation*}
Note, by Lemma \ref{1007152310}, for each $k= 1, \cdots, K$, $$\bigcup\limits_{\omega\in \Omega} \{\omega\}\times ((\alpha_k)_F)_\omega (p_1 (\omega))\in \mathcal{F}\times \mathcal{B}_X,$$
and so
\begin{equation*}
\mathcal{E}_1\doteq \mathcal{E}_0\setminus \bigcup_{k= 1}^K \bigcup_{\omega\in \pi (\mathcal{E}_0)} \{\omega\}\times ((\alpha_k)_F)_\omega (p_1 (\omega))\in \mathcal{F}\times \mathcal{B}_X.
\end{equation*}
If $\mathcal{E}_1= \emptyset$ the proof is finished. If not, by using Lemma \ref{1007152201} $\pi (\mathcal{E}_1)\in \mathcal{F}$, and there exists a measurable map $p_2: \pi (\mathcal{E}_1)\rightarrow X$ such that
\begin{equation*}
e^{d_F (\omega, p_2 (\omega))}\ge \sup_{x\in (\mathcal{E}_1)_\omega} e^{d_F (\omega, x)}- \frac{1}{2^{2+ 1} K} e^{- ||d_F (\omega)||_\infty}
\end{equation*}
and $(\omega, p_2 (\omega))\in \mathcal{E}_1$
for each $\omega\in \pi (\mathcal{E}_1)$. Set
\begin{equation*}
\mathcal{E}_2= \mathcal{E}_1\setminus \bigcup_{k= 1}^K \bigcup_{\omega\in \pi (\mathcal{E}_1)} \{\omega\}\times ((\alpha_k)_F)_\omega (p_2 (\omega))\in \mathcal{F}\times \mathcal{B}_X.
\end{equation*}

It is not hard to see that, after finitely many steps we obtain
\begin{equation*}
\mathcal{E}_0\in \mathcal{F}\times \mathcal{B}_X\ \text{and}\ \mathcal{E}_j= \mathcal{E}_{j- 1}\setminus \bigcup_{k= 1}^K \bigcup_{\omega\in \pi (\mathcal{E}_{j- 1})} \{\omega\}\times ((\alpha_k)_F)_\omega (p_j (\omega))\in \mathcal{F}\times \mathcal{B}_X,
\end{equation*}
where $p_j: \pi (\mathcal{E}_{j- 1})\rightarrow X$ is a measurable map satisfying
\begin{equation} \label{1207312011}
e^{d_F (\omega, p_j (\omega))}\ge \sup_{x\in (\mathcal{E}_{j- 1})_\omega} e^{d_F (\omega, x)}- \frac{1}{2^{j+ 1} K} e^{- ||d_F (\omega)||_\infty}
\end{equation}
and $(\omega, p_j (\omega))\in \mathcal{E}_{j- 1}$ for each $\omega\in \pi (\mathcal{E}_{j- 1})$,
 $j= 1, \cdots, m$ and $\mathcal{E}_{m- 1}\neq \emptyset$ while $\mathcal{E}_m= \emptyset$.

  Observe that, for all $j= 1, \cdots, m$ and $j_1, j_2\in \{0, 1, \cdots, j- 1\}$, for $j_1\neq j_2$,  $((\alpha_k)_F)_\omega (p_{j_1+ 1} (\omega))$ and $((\alpha_k)_F)_\omega (p_{j_2+ 1} (\omega))$ are different non-empty atoms of the partition $((\alpha_k)_F)_\omega$ for each $k= 1, \cdots, K$ and any $\omega\in \pi (\mathcal{E}_{j- 1})$.
  By assumption, the cardinality of each $\alpha_k, k= 1, \cdots, K$ is at most $n$, the cardinality of $\mathcal{U}$.
Thus, we can deduce that $\mathcal{E}_{m- 1}\neq \emptyset$ while $\mathcal{E}_m= \emptyset$ for some $m\in \mathbb{N}$.

Now for each $\omega\in \Omega$, set
$$B_{F, \omega}= \{p_j (\omega): j\in \{1, \cdots, m\}, \omega\in \mathcal{E}_{j- 1}\}.$$
From the construction, it is easy to see that, for $\omega\in \Omega$, each atom of $((\alpha_k)_F)_\omega$ contains at most one point from $B_{F, \omega}, 1\le k\le K$, and $B_F\in \mathcal{F}\times \mathcal{B}_X$ (using \cite[Proposition III.13]{CV}). Here we use the assumption that $X$ is a compact metric space.

To finish the proof, let $\omega\in \Omega$. We only need check
\begin{equation*}
\sum_{x\in B_{F, \omega}} e^{d_F (\omega, x)}> \frac{1}{K} \left[\inf_{\beta (\omega)\in \mathbf{P}_{\mathcal{E}_\omega}, \beta (\omega)\succeq (\mathcal{U}_F)_\omega} \sum_{B\in \beta (\omega)} \sup_{x\in B} e^{d_F (\omega, x)}- \frac{1}{2} e^{- ||d_F (\omega)||_\infty}\right].
\end{equation*}

In fact, suppose that $m (\omega)\in \{1, \cdots, m\}$ is the first number $J$ in $\{1, \cdots, m\}$ such that $\omega\notin \pi (\mathcal{E}_J)$, and set
\begin{equation*}
\gamma (\omega)= \{(\mathcal{E}_{j- 1})_\omega\cap ((\alpha_k)_F)_\omega (p_j (\omega)): j= 1, \cdots, m (\omega), k= 1, \cdots, K\}.
\end{equation*}
It is easy to check that $\gamma (\omega)\in \mathbf{C}_{\mathcal{E}_\omega}, \gamma (\omega)\succeq (\mathcal{U}_F)_\omega$. Moreover,
\begin{eqnarray*}
& & \sum_{x\in B_{F, \omega}} e^{d_F (\omega, x)}= \sum_{j= 1}^{m (\omega)} e^{d_F (\omega, p_j (\omega))} \\
& &\hskip 26pt \ge \sum_{j= 1}^{m (\omega)} \frac{1}{K}\sum_{k= 1}^K \left[ \sup_{x\in (\mathcal{E}_{j- 1})_\omega\cap ((\alpha_k)_F)_\omega (p_j (\omega))} e^{d_F (\omega, x)}\right. \\
& & \hskip 126pt \left.- \frac{1}{2^{j+ 1} K} e^{- ||d_F (\omega)||_\infty}\right]\ (\text{using \eqref{1207312011}}) \\
& &\hskip 26pt > \frac{1}{K} \left[\sum_{B (\omega)\in \gamma (\omega)} \sup_{x\in B (\omega)} e^{d_F (\omega, x)}- \frac{1}{2} e^{- ||d_F (\omega)||_\infty}\right] \\
& &\hskip 26pt \ge \frac{1}{K} \left[\inf_{\beta (\omega)\in \mathbf{P}_{\mathcal{E}_\omega}, \beta (\omega)\succeq (\mathcal{U}_F)_\omega} \sum_{B \in \beta (\omega)} \sup_{x\in B} e^{d_F (\omega, x)}- \frac{1}{2} e^{- ||d_F (\omega)||_\infty}\right].
\end{eqnarray*}
This finishes our proof.
\end{proof}

In the process of proving Proposition \ref{1006151236}, we will also need the following result, which was essentially proved in the proof of \cite[Lemma 3.1]{HYZ} (see also \cite[Lemma 2.2]{DZ}).

\begin{lem} \label{1007221736}
Let $(Y, \mathcal{D}, \nu)$ be a probability space, $\mathcal{C}\subseteq \mathcal{D}$ a sub-$\sigma$-algebra and $\alpha\in \mathbf{P}_Y$. Assume that $G$ acts as a group of invertible measurable transformations (which may be not measure-preserving) over $(Y, \mathcal{D}, \nu)$. If $E, F\in \mathcal{F}_G$ then
\begin{equation*}
H_\nu (\alpha_F| \mathcal{C})\le \sum_{g\in F} \frac{1}{|E|} H_\nu (\alpha_{E g}| \mathcal{C})+ |F\setminus \{g\in G: E^{- 1} g\subseteq F\}| \log |\alpha|.
\end{equation*}
\end{lem}

The following result is well known.


\begin{lem} \label{1007222138}
Let $(Y, \mathcal{D}, \nu_i)$ be a Lebesgue space, for $i= 1, \cdots, n,$ and some $ n\in \mathbb{N}$, $\mathcal{C}\subseteq \mathcal{D}$ a sub-$\sigma$-algebra of $\mathcal D$, and suppose that  $0< \lambda_1, \cdots, \lambda_n< 1$ satisfy $\lambda_1+ \cdots+ \lambda_n= 1$. Then there exists $\lambda> 0$ (depending on $\lambda_1, \cdots, \lambda_n$) such that, for each $\alpha\in \mathbf{P}_Y$,
\begin{equation*}
\lambda+ \sum_{i= 1}^n \lambda_i H_{\nu_i} (\alpha| \mathcal{C})\ge H_{\lambda_1 \nu_1+ \cdots+ \lambda_n \nu_n} (\alpha| \mathcal{C})\ge \sum_{i= 1}^n \lambda_i H_{\nu_i} (\alpha| \mathcal{C}).
\end{equation*}
\end{lem}

With the above preparation we can prove:

\begin{prop} \label{1007142016}
Let $\mathbf{D}= \{d_F: F\in \mathcal{F}_G\}\subseteq \mathbf{L}_\mathcal{E}^1 (\Omega, C (X))$ be a monotone sub-additive $G$-invariant family satisfying $(\spadesuit)$ and $\mathcal{U}\in \mathbf{C}_\mathcal{E}^o$.
If $\mathcal{U}$ is good then, for some $\mu\in \mathcal{P}_\mathbb{P} (\mathcal{E}, G)$,
$$h_{\mu}^{(r)} (\mathbf{F}, \mathcal{U})+ \mu (\mathbf{D})\ge P_\mathcal{E} (\mathbf{D}, \mathcal{U}, \mathbf{F}).$$
\end{prop}
\begin{proof}
As $\mathcal{U}\in \mathbf{C}_\mathcal{E}^o$ is good, there exists a sequence $\{\alpha_n: n\in \mathbb{N}\}\subseteq \mathbf{P}_\mathcal{U}$ such that
\begin{enumerate}

\item[(a)]
for each $n\in \mathbb{N}$, $(\alpha_n)_\omega$ is a clopen partition of $\mathcal{E}_\omega$ for $\mathbb{P}$-a.e. $\omega\in \Omega$ and

\item[(b)] $h_{\nu}^{(r)} (\mathbf{F}, \mathcal{U})= \inf\limits_{n\in \mathbb{N}} h_\nu^{(r)} (\mathbf{F}, \alpha_n)$ for each $\nu\in \mathcal{P}_\mathbb{P} (\mathcal{E}, G)$.
\end{enumerate}

Let $n\in \mathbb{N}$ be fixed.
By Lemma
\ref{1007141419},
there exists a family of finite subsets $B_{n, \omega}\subseteq \mathcal{E}_\omega, \omega\in \Omega$ such that
\begin{enumerate}

\item \label{1007212243}  For $B_n\doteq \{(\omega, x): \omega\in \Omega, x\in B_{n, \omega}\}$,
$$\sum\limits_{x\in B_{n, \omega}} e^{d_{F_n} (\omega, x)}> \frac{1}{n} \left[\inf\limits_{\beta (\omega)\in \mathbf{P}_{\mathcal{E}_\omega}, \beta (\omega)\succeq (\mathcal{U}_{F_n})_\omega} \sum\limits_{B\in \beta (\omega)} \sup\limits_{x\in B} e^{d_{F_n} (\omega, x)}- \frac{1}{2} e^{- ||d_{F_n} (\omega)||_\infty}\right],$$

\item \label{1007212245} The family depends measurably on $\omega\in \Omega$ in the sense that $B_n\in \mathcal{F}\times \mathcal{B}_X$ and

\item \label{1007212244} Each atom of $((\alpha_k)_{F_n})_\omega$ contains at most one point from $B_{n, \omega}, 1\le k\le n$.
\end{enumerate}
Now we introduce a probability measure $\nu^{(n)}$ over $\mathcal{E}$ by a measurable disintegration
$d \nu^{(n)} (\omega, x)= d \nu^{(n)}_\omega (x) d \mathbb{P} (\omega)$, where
\begin{equation*}
\nu^{(n)}_\omega= \sum_{x\in B_{n, \omega}} \frac{e^{d_{F_n} (\omega, x)} \delta_x}{\sum\limits_{y\in B_{n, \omega}} e^{d_{F_n} (\omega, y)}}.
\end{equation*}
Hence we may define another probability measure $\mu^{(n)}$ on $\mathcal{E}$ by
\begin{equation*}
\mu^{(n)}= \frac{1}{|F_n|} \sum_{g\in F_n} g \nu^{(n)}.
\end{equation*}
Observe that by assumption \eqref{1007212245} the measure $\nu^{(n)}$ (and hence $\mu^{(n)}$) is well defined.

As the family $\mathbf{D}$ satisfies $(\spadesuit)$, we can choose a subsequence $\{n_j: j\in \mathbb{N}\}\subseteq \mathbb{N}$ such that the sequence $\{\mu^{(n_j)}: j\in \mathbb{N}\}$ converges to some $\mu\in \mathcal{P}_\mathbb{P} (\mathcal{E})$ (which then necessarily belongs to $ \mathcal{P}_\mathbb{P} (\mathcal{E}, G)$) and
\begin{equation} \label{1008161700}
\limsup_{j\rightarrow \infty} \frac{1}{|F_{n_j}|} \int_{\mathcal{E}} d_{F_{n_j}} (\omega, x) d \nu^{n_j} (\omega, x)\le \mu (\mathbf{D}).
\end{equation}

 By the assumptions on the sequence $\{\alpha_n: n\in \mathbb{N}\}$, to finish the proof, it suffices to prove $h_\mu^{(r)} (\mathbf{F}, \alpha_l)+ \mu (\mathbf{D})\ge P_\mathcal{E} (\mathbf{D}, \mathcal{U}, \mathbf{F})$ for each $l\in \mathbb{N}$.

Let $l\in \mathbb{N}$ be fixed.

For each $n> l$,
 from the construction of $\nu^{(n)}_\omega$, one has
 \begin{eqnarray} \label{1007221805}
 H_{\nu_\omega^{(n)}} (((\alpha_l)_{F_n})_\omega)&= & \sum_{x\in B_{n, \omega}} - \frac{e^{d_{F_n} (\omega, x)}}{\sum\limits_{y\in B_{n, \omega}} e^{d_{F_n} (\omega, y)}} \log \frac{e^{d_{F_n} (\omega, x)}}{\sum\limits_{y\in B_{n, \omega}} e^{d_{F_n} (\omega, y)}}\nonumber \\
 &= & \sum_{x\in B_{n, \omega}} - \frac{e^{d_{F_n} (\omega, x)} d_{F_n} (\omega, x)}{\sum\limits_{y\in B_{n, \omega}} e^{d_{F_n} (\omega, y)}}+ \log \sum_{y\in B_{n, \omega}} e^{d_{F_n} (\omega, y)}\nonumber \\
 &= & - \int_X d_{F_n} (\omega, x) d \nu^{(n)}_\omega (x)+ \log \sum_{y\in B_{n, \omega}} e^{d_{F_n} (\omega, y)},
 \end{eqnarray}
 as each atom of $((\alpha_l)_{F_n})_\omega$ contains at most one point from $B_{n, \omega}$.
This implies
\begin{eqnarray} \label{1007221806}
& &\hskip -16pt \log P_\mathcal{E} (\omega, \mathbf{D}, F_n, \mathcal{U}, \mathbf{F})- \log 2- \log n\nonumber \\
& &\le \log \left[P_\mathcal{E} (\omega, \mathbf{D}, F_n, \mathcal{U}, \mathbf{F})- \frac{1}{2} e^{- ||d_{F_n} (\omega)||_\infty}\right]- \log n\ (\text{from the definitions})\nonumber \\
& &= \log \left[\inf\limits_{\beta (\omega)\in \mathbf{P}_{\mathcal{E}_\omega}, \beta (\omega)\succeq (\mathcal{U}_{F_n})_\omega} \sum\limits_{B\in \beta (\omega)} \sup\limits_{x\in B} e^{d_{F_n} (\omega, x)}- \frac{1}{2} e^{- ||d_{F_n} (\omega)||_\infty}\right]- \log n\nonumber \\
& &< \log \sum_{x\in B_{n, \omega}} e^{d_{F_n} (\omega, x)}\ (\text{by assumption \eqref{1007212243}})\nonumber \\
& &= H_{\nu_\omega^{(n)}} (((\alpha_l)_{F_n})_\omega)+ \int_X d_{F_n} (\omega, x) d \nu^{(n)}_\omega (x)\ (\text{using \eqref{1007221805}}),
\end{eqnarray}
and so by Proposition \ref{1006141621} (1) and the construction of $\nu^{(n)}$ (using \eqref{0906282007}), for each $B\in \mathcal{F}_G$ we have
\begin{eqnarray} \label{1007222025}
& & \int_\Omega \log P_\mathcal{E} (\omega, \mathbf{D}, F_n, \mathcal{U}, \mathbf{F}) d \mathbb{P} (\omega)- \log 2- \log n\nonumber \\
&< & H_{\nu^{(n)}} ((\alpha_l)_{F_n}| \mathcal{F}_\mathcal{E})+ \int_\mathcal{E} d_{F_n} (\omega, x) d \nu^{(n)} (\omega, x)\nonumber \\
&\le & \sum_{g\in F_n} \frac{1}{|B|} H_{\nu^{(n)}} ((\alpha_l)_{B g}| \mathcal{F}_\mathcal{E})+ |F_n\setminus \{g\in G: B^{- 1} g\subseteq F_n\}|\log |\alpha_l|\nonumber \\
& &\hskip 26pt + \int_\mathcal{E} d_{F_n} (\omega, x) d \nu^{(n)} (\omega, x)\ (\text{using Lemma \ref{1007221736}})\nonumber \\
&= & \frac{|F_n|}{|B|} \sum_{g\in F_n} \frac{1}{|F_n|} H_{g \nu^{(n)}} ((\alpha_l)_B| \mathcal{F}_\mathcal{E})+ |F_n\setminus \{g\in G: B^{- 1} g\subseteq F_n\}|\log |\alpha_l|\nonumber \\
& &\hskip 26pt + \int_\mathcal{E} d_{F_n} (\omega, x) d \nu^{(n)} (\omega, x)\ (\text{using the $G$-invariance of $\mathcal{F}_\mathcal{E}$})\nonumber \\
&\le & \frac{|F_n|}{|B|} H_{\mu^{(n)}} ((\alpha_l)_B| \mathcal{F}_\mathcal{E})+ |F_n\setminus \{g\in G: B^{- 1} g\subseteq F_n\}|\log |\alpha_l|\nonumber \\
& &\hskip 26pt + \int_\mathcal{E} d_{F_n} (\omega, x) d \nu^{(n)} (\omega, x)\ (\text{using Lemma \ref{1007222138}}).
\end{eqnarray}

Let $B\in \mathcal{F}_G$ be fixed. Observe that, as $\{F_n: n\in \mathbb{N}\}$ is a F\o lner sequence,
\begin{equation*}
\lim_{n\rightarrow \infty} \frac{1}{|F_n|} |F_n\setminus \{g\in G: B^{- 1} g\subseteq F_n\}|= 0;
\end{equation*}
moreover, by the choice of $\alpha_l$,  $((\alpha_l)_B)_\omega$ is a clopen partition of $\mathcal{E}_\omega$ for $\mathbb{P}$-a.e. $\omega\in \Omega$, and so we have (using Proposition \ref{1007192023} (2))
\begin{equation} \label{last1}
\limsup_{n\rightarrow \infty} H_{\mu^{(n)}} ((\alpha_l)_B| \mathcal{F}_\mathcal{E})\le H_\mu ((\alpha_l)_B| \mathcal{F}_\mathcal{E}).
\end{equation}
Recall that $|F_n|\ge n$ for each $n\in \mathbb{N}$. Combining \eqref{last1} and \eqref{1007222025} (divided by $|F_n|$) we obtain
\begin{equation*}
P_\mathcal{E} (\mathbf{D}, \mathcal{U}, \mathbf{F})\le \frac{1}{|B|} H_\mu ((\alpha_l)_B| \mathcal{F}_\mathcal{E})+ \mu (\mathbf{D})\ \text{(using \eqref{1008161700})}.
\end{equation*}
Lastly, taking the infimum over all $B\in \mathcal{F}_G$ we obtain
 $$P_\mathcal{E} (\mathbf{D}, \mathcal{U}, \mathbf{F})\le h_\mu (G, \alpha_l| \mathcal{F}_\mathcal{E})+ \mu (\mathbf{D})\ (\text{using \eqref{1006272232}}),$$
or equivalently, $P_\mathcal{E} (\mathbf{D}, \mathcal{U}, \mathbf{F})\le h_\mu^{(r)} (\mathbf{F}, \alpha_l)+ \mu (\mathbf{D})$. This completes the proof.
\end{proof}

Now we can present the proof of Proposition \ref{1006151236}.

\begin{proof}[Proof of Proposition \ref{1006151236}]
As $\mathcal{U}$ is factor good, there exists a family $\mathbf{F}'= \{F_{g, \omega}': \mathcal{E}_\omega'\rightarrow \mathcal{E}_{g \omega}'| g\in G, \omega\in \Omega\}$ (with compact metric state space $X'$ and $\mathcal{E}'\in \mathcal{F}\times \mathcal{B}_{X'}$), which is a continuous bundle RDS over $(\Omega, \mathcal{F}, \mathbb{P}, G)$, and a factor map $\pi: \mathcal{E}'\rightarrow \mathcal{E}$ such that $\pi^{- 1} \mathcal{U}$ is good.
By
Lemma \ref{1007212122} and Lemma \ref{1008171752}, $\mathbf{D}\circ \pi$ is a monotone sub-additive $G$-invariant family satisfying $(\spadesuit)$,
 and so, by Proposition \ref{1007142016}, there exists $\mu'\in \mathcal{P}_\mathbb{P} (\mathcal{E}', G)$ such that
 $$h_{\mu'}^{(r)} (\mathbf{F}', \pi^{- 1} \mathcal{U})+ \mu' (\mathbf{D}\circ \pi)\ge P_{\mathcal{E}'} (\mathbf{D}\circ \pi, \pi^{- 1} \mathcal{U}, \mathbf{F}').$$
 Set $\mu= \pi \mu'$. By Lemma \ref{1007212122}, we have $\mu\in \mathcal{P}_\mathbb{P} (\mathcal{E}, G)$,
 $$h^{(r)}_{\mu} (\mathbf{F}, \mathcal{U})= h^{(r)}_{\mu'} (\mathbf{F}', \pi^{- 1} \mathcal{U})$$ and
  $$P_{\mathcal{E}'} (\mathbf{D}\circ \pi, \pi^{- 1} \mathcal{U}, \mathbf{F}')= P_\mathcal{E} (\mathbf{D}, \mathcal{U}, \mathbf{F}).$$
  It follows from the definition that $\mu' (\mathbf{D}\circ \pi)= \mu (\mathbf{D})$ and hence
   $$h_{\mu}^{(r)} (\mathbf{F}, \mathcal{U})+ \mu (\mathbf{D})\ge P_\mathcal{E} (\mathbf{D}, \mathcal{U}, \mathbf{F}).$$
This completes the proof.
\end{proof}

\section{Assumption $(\spadesuit)$ on the family $\mathbf{D}$} \label{assumption}

There are two essential assumptions appearing in Theorem \ref{1007141414}: that $\mathcal{U}\in \mathbf{C}_\mathcal{E}^o$ is factor good, and that $\mathbf{D}= \{d_F: F\in \mathcal{F}_G\}\subseteq \mathbf{L}_\mathcal{E}^1 (\Omega, C (X))$ satisfies $(\spadesuit)$.
In \S \ref{factor good} we have discussed the first assumption and in this section we discuss the second.

\medskip

Before proceeding, we introduce the property of strong sub-additivity. In his treatment of entropy theory for amenable group actions Moullin Ollagnier \cite{MO} used this property rather heavily.

 Let $(Y, \mathcal{D}, \nu, G)$ be an MDS and $\mathbf{D}= \{d_F: F\in \mathcal{F}_G\}\subseteq L^1 (Y, \mathcal{D}, \nu)$.
 $\mathbf{D}$ is called \emph{strongly sub-additive} if for
 $\nu$-a.e. $y\in Y$,
 $$d_{E\cup F} (y)+ d_{E\cap F} (y)\le d_E (y)+ d_F
 (y)$$
  whenever $E, F\in \mathcal{F}_G$ (here we set $d_\emptyset (y)= 0$ for $\nu$-a.e. $y\in Y$). For an invariant family, the property of strong sub-additivity is stronger than the property of sub-additivity, and, for each $f\in
 L^1 (Y, \mathcal{D}, \nu)$, $\mathbf{D}^f$ is always a strongly sub-additive $G$-invariant family in $L^1 (Y, \mathcal{D}, \nu)$. Similarly, we can introduce the property of strong sub-additivity for any given continuous bundle RDS.

 Let $\mathbf{D}= \{d_F: F\in \mathcal{F}_G\}\subseteq \mathbf{L}_\mathcal{E}^1 (\Omega, C (X))$ be a strongly sub-additive $G$-invariant family. For each $\mu\in \mathcal{P}_\mathbb{P} (\mathcal{E}, G)$, we use Proposition \ref{1102111944} to define $\mu (\mathbf{D})$ by
\begin{eqnarray*}
\mu (\mathbf{D})&= & \lim_{n\rightarrow \infty} \frac{1}{|F_n|} \int_{\mathcal{E}} d_{F_n} (\omega, x) d \mu (\omega, x)\nonumber \\
&= & \inf_{n\in \mathbb{N}} \frac{1}{|F_n|} \int_{\mathcal{E}} d_{F_n} (\omega, x) d \mu (\omega, x). \label{1202111533}
\end{eqnarray*}
Note that the value of $\mu (\mathbf{D})$ is independent of the choice of the F\o lner sequence $\{F_n: n\in \mathbb{N}\}$, in fact,
\begin{equation} \label{1208011217}
\mu (\mathbf{D})= \inf_{F\in \mathcal{F}_G} \frac{1}{|F|} \int_{\mathcal{E}} d_{F} (\omega, x) d \mu (\omega, x).
\end{equation}
Remark that $\mu (\mathbf{D})$ need not be non-negative, as here $\mathbf{D}$ need not be non-negative.
As show by the following trivial example, $\mu (\mathbf{D})$ may even take the value $- \infty$:
$$\mathbf{D}= \{d_F: F\in \mathcal{F}_G\}\ \text{with}\ d_{F} (\omega, x)= - |F|^2\ \text{for each}\ F\in \mathcal{F}_G.$$
We also remark that by \eqref{1208011217}, the function
 $$\bullet (\mathbf{D}): \mathcal{P}_\mathbb{P} (\mathcal{E}, G)\rightarrow \R\cup \{- \infty\}, \mu\mapsto \mu (\mathbf{D})$$
  is the infimum of a family of continuous functions over the compact metric space $\mathcal{P}_\mathbb{P} (\mathcal{E}, G)$, and hence itself is u.s.c.

\medskip

Here is the main result of this section.

 \begin{prop} \label{1008300004}
 Let $\mathbf{D}= \{d_F: F\in \mathcal{F}_G\}\subseteq \mathbf{L}_\mathcal{E}^1 (\Omega, C (X))$ be a strongly sub-additive $G$-invariant family. Then $\mathbf{D}$ satisfies $(\spadesuit)$.
 \end{prop}

In order to prove Proposition \ref{1008300004}, we will use \cite[Lemma 2.2.16]{MO}.

\begin{lem} \label{1008132304}
Let $(Y, \mathcal{D}, \nu, G)$ be an MDS and
$\mathbf{D}= \{d_F: F\in \mathcal{F}_G\}$ a strongly sub-additive family in $L^1 (Y, \mathcal{D}, \nu)$.
Assume that $1_{E}= \sum\limits_{i= 1}^n a_i 1_{E_i}$, where
$E, E_1, \cdots, E_n\in \mathcal{F}_G$
and $a_1, \cdots, a_n> 0, n\in \N$. Then
$d_E (y)\le \sum\limits_{i= 1}^n a_i d_{E_i} (y)$
 for $\nu$-a.e. $y\in Y$.
 A similar result holds for a continuous bundle RDS.
\end{lem}

We also need the following:

\begin{lem} \label{1008300008}
 Let $T, E\in \mathcal{F}_G$. Then $\sum\limits_{t\in T} 1_{t E}= \sum\limits_{g\in E} 1_{T g}$.
 \end{lem}
 \begin{proof}
Set $L= \sum\limits_{t\in T} 1_{t E}$ and $R= \sum\limits_{g\in E} 1_{T g}$. Let $g'\in G$. Then $L (g')> 0$ if and only if there exists $t\in T$ such that $g'\in t E$, if and only if there exists $g\in E$ such that $g'\in T g$, if and only if $R (g')> 0$. Moreover, for any given $n\in \mathbb{N}$, $L (g')= n$ if and only if there exist exactly $n$ distinct elements $t_1, \cdots, t_n$ of $T$ such that $g'\in t_i E$ (say $g'= t_i g_i$ for some $g_i\in E$) for each $i= 1, \cdots, n$, if and only if there exist exactly $n$ distinct elements $g_1, \cdots, g_n$ of $E$ such that $g'\in T g_i$ for each $i= 1, \cdots, n$, if and only if $R (g')= n$. This finishes the proof.
\end{proof}

Now we can prove Proposition \ref{1008300004} as follows.

\begin{proof}[Proof of Proposition \ref{1008300004}]
Let $\{\nu_n: n\in \mathbb{N}\}\subseteq \mathcal{P}_\mathbb{P} (\mathcal{E})$ be a given sequence. Set
\begin{equation*}
\mu_n= \frac{1}{|F_n|} \sum\limits_{g\in F_n} g \nu_n\ \text{for each}\ n\in \mathbb{N}.
\end{equation*}
By Proposition \ref{1007192023} \eqref{1208011203} there exists a subsequence $\{n_j: j\in \mathbb{N}\}\subseteq \mathbb{N}$ such that the sequence $\{\mu_{n_j}: j\in \mathbb{N}\}$ converges to some $\mu\in \mathcal{P}_\mathbb{P} (\mathcal{E}, G)$. Now we only need check
\begin{equation} \label{1008301422}
\limsup_{j\rightarrow \infty} \frac{1}{|F_{n_j}|} \int_{\mathcal{E}} d_{F_{n_j}} (\omega, x) d \nu_{n_j} (\omega, x)\le \mu (\mathbf{D}).
\end{equation}

For each $F\in \mathcal{F}_G$ set
$$d_F' (\omega, x)= d_F (\omega, x)- \sum\limits_{g\in F} d_{\{e_G\}} (g (\omega, x))$$ and put
$$\mathbf{D}'= \{d_F': F\in \mathcal{F}_G\}\subseteq \mathbf{L}_\mathcal{E}^1 (\Omega, C (X)).$$
As $\mathbf{D}$ is a strongly sub-additive $G$-invariant family, then the family $\mathbf{D}'$ is also strongly sub-additive $G$-invariant and $- \mathbf{D}'$ is non-negative. Observe that
\begin{eqnarray} \label{1008292315}
& & \limsup_{j\rightarrow \infty} \frac{1}{|F_{n_j}|} \int_{\mathcal{E}} d_{F_{n_j}} (\omega, x) d \nu_{n_j} (\omega, x)\nonumber \\
& &= \limsup_{j\rightarrow \infty} \frac{1}{|F_{n_j}|} \int_{\mathcal{E}} d_{F_{n_j}}' (\omega, x) d \nu_{n_j} (\omega, x)+ \limsup_{j\rightarrow \infty} \int_{\mathcal{E}} d_{\{e_G\}} (\omega, x) d \mu_{n_j} (\omega, x)\nonumber \\
& &= \limsup_{j\rightarrow \infty} \frac{1}{|F_{n_j}|} \int_{\mathcal{E}} d_{F_{n_j}}' (\omega, x) d \nu_{n_j} (\omega, x)+ \int_{\mathcal{E}} d_{\{e_G\}} (\omega, x) d \mu (\omega, x) \\
 & & \hskip 68pt (\text{as the sequence $\{\mu_{n_j}: j\in \mathbb{N}\}$ converges to $\mu$})\nonumber
\end{eqnarray}
and
\begin{eqnarray}\label{1008292319}
\mu (\mathbf{D})&= & \lim_{n\rightarrow \infty} \frac{1}{|F_n|} \int_{\mathcal{E}} d_{F_n} (\omega, x) d \mu (\omega, x)\nonumber \\
&= & \lim_{n\rightarrow \infty} \frac{1}{|F_n|} \int_{\mathcal{E}} d_{F_n}' (\omega, x) d \mu (\omega, x)\nonumber \\
& & \hskip 26pt + \lim_{n\rightarrow \infty} \frac{1}{|F_n|} \int_{\mathcal{E}} \sum_{g\in F_n} d_{\{e_G\}} (g (\omega, x)) d \mu (\omega, x)\nonumber \\
&= & \mu (\mathbf{D}')+ \int_{\mathcal{E}} d_{\{e_G\}} (\omega, x) d \mu (\omega, x)\ (\text{as $\mu\in \mathcal{P}_\mathbb{P} (\mathcal{E}, G)$}).
\end{eqnarray}
To prove \eqref{1008301422}, by \eqref{1008292315} and \eqref{1008292319}, we only need prove
\begin{equation} \label{1008292322}
\limsup_{j\rightarrow \infty} \frac{1}{|F_{n_j}|} \int_{\mathcal{E}} d_{F_{n_j}}' (\omega, x) d \nu_{n_j} (\omega, x)\le \mu (\mathbf{D}').
\end{equation}

Let $T\in \mathcal{F}_G$ be fixed. As $\{F_n: n\in \mathbb{N}\}$ is a F\o lner sequence of $G$, for each $n\in \mathbb{N}$ we set $E_n= F_n\cap \bigcap\limits_{g\in T} g^{- 1} F_n\subseteq F_n$, then $\lim\limits_{n\rightarrow \infty} \frac{|E_n|}{|F_n|}= 1$.
Set
$$w_n= \frac{1}{|E_n|} \sum\limits_{g\in E_n} g \nu_n\ \text{for each}\ n\in \mathbb{N}.$$
Observe that the sequence $\{\mu_{n_j}: j\in \mathbb{N}\}$ converges to $\mu$. By the choice of $E_n, n\in \N$, it is easy to see that the sequence $\{w_{n_j}: j\in \mathbb{N}\}$ also converges to $\mu$.

Now for each $n\in \mathbb{N}$, using Lemma \ref{1008300008}, one has
\begin{equation} \label{080901}
\sum\limits_{t\in T} 1_{t E_n}= \sum\limits_{g\in E_n} 1_{T g}.
\end{equation}
By the construction of $E_n$, $t E_n\subseteq F_n$ for any $t\in T$, there exist $E_1', \cdots, E_m'\in \mathcal{F}_G, m\in \{0\}\cup \N$ and rational numbers $a_1, \cdots, a_m> 0$ such that
\begin{equation*}
1_{F_n}= \frac{1}{|T|} \sum_{t\in T} 1_{t E_n}+ \sum_{j= 1}^m a_j 1_{E_j'},
\end{equation*}
and hence
\begin{equation} \label{1102111848}
1_{F_n}= \frac{1}{|T|} \sum\limits_{g\in E_n} 1_{T g}+ \sum_{j= 1}^m a_j 1_{E_j'}\ (\text{using \eqref{080901}}),
\end{equation}
which implies that, for $\mathbb{P}$-a.e. $\omega\in \Omega$,
\begin{eqnarray} \label{1008300040}
d'_{F_n} (\omega, x)&\le & \frac{1}{|T|} \sum_{g\in E_n} d'_{T g} (\omega, x)+ \sum_{j= 1}^m a_j d'_{E_j'} (\omega, x)\nonumber \\
& & (\text{using Lemma \ref{1008132304}, as the family $\mathbf{D}'$ is strongly sub-additive})\nonumber \\
&\le & \frac{1}{|T|} \sum_{g\in E_n} d'_{T g} (\omega, x)\ (\text{as the family $-\mathbf{D}'$ is non-negative})\nonumber \\
&= & \frac{1}{|T|} \sum_{g\in E_n} d'_T (g (\omega, x))\ (\text{as the family $\mathbf{D}'$ is $G$-invariant})
\end{eqnarray}
for each $x\in \mathcal{E}_\omega$.
It follows that
\begin{eqnarray} \label{1008301428}
& & \limsup_{j\rightarrow \infty} \frac{1}{|F_{n_j}|} \int_{\mathcal{E}} d'_{F_{n_j}} (\omega, x) d \nu_{n_j} (\omega, x)\nonumber \\
&= & \limsup_{j\rightarrow \infty} \frac{1}{|E_{n_j}|} \int_{\mathcal{E}} d'_{F_{n_j}} (\omega, x) d \nu_{n_j} (\omega, x)\ (\text{by the selection of $E_{n_j}$})\nonumber \\
&\le & \limsup_{j\rightarrow \infty} \frac{1}{|T|} \int_{\mathcal{E}} d'_T (\omega, x) d w_{n_j} (\omega, x)\ (\text{using \eqref{1008300040}})\nonumber \\
&= & \frac{1}{|T|} \int_{\mathcal{E}} d'_T (\omega, x) d \mu (\omega, x)\ (\text{as the sequence $\{w_{n_j}: j\in \mathbb{N}\}$ converges to $\mu$}).
\end{eqnarray}
Now \eqref{1008292322} follows from \eqref{1208011217} and \eqref{1008301428}. This completes the proof.
\end{proof}

Let $\mathbf{D}= \{d_F: F\in \mathcal{F}_G\}\subseteq \mathbf{L}_\mathcal{E}^1 (\Omega, C (X))$ be a strongly sub-additive $G$-invariant family.
 Observe that the family
 $$\{\sup\limits_{x\in \mathcal{E}_\omega} d_F (\omega, x): F\in \mathcal{F}_G\}\subseteq L^1 (\Omega, \mathcal{F}, \mathbb{P})$$
  may be not strongly sub-additive, as for $E, F\in \mathcal{F}_G$ it may happen that
     $$\hskip 26pt \sup_{x\in \mathcal{E}_\omega} d_{E\cap F} (\omega, x)+ \sup\limits_{x\in \mathcal{E}_\omega} d_{E\cup F} (\omega, x)> \sup\limits_{x\in \mathcal{E}_\omega} d_E (\omega, x)+ \sup\limits_{x\in \mathcal{E}_\omega} d_F (\omega, x),$$
     even though, by strong sub-additivity of $\mathbf{D}$,
      $$d_{E\cap F} (\omega, x)+ d_{E\cup F} (\omega, x)\le d_E (\omega, x)+ d_F (\omega, x).$$
      Thus we cannot apply Proposition \ref{1102111944} to define $\text{sup}_\mathbb{P} (\mathbf{D})$
      as we did in \eqref{1208011225}.

In a more general setting, for $\mathcal{U}\in \mathbf{C}_\mathcal{E}$, it may happen that the family $\{\log P_\mathcal{E} (\omega,$ $\mathbf{D}, F, \mathcal{U}, \mathbf{F}): F\in \mathcal{F}_G\}$ is not strongly sub-additive, and so, as above, we are similarly unable to apply Proposition \ref{1102111944} to define
  $P_\mathcal{E} (\mathbf{D}, \mathcal{U}, \mathbf{F})$.

However, by Proposition \ref{1008300004},
we can apply the discussions surrounding \eqref{1103031445}, \eqref{1208011114}, \eqref{1208011732}, \eqref{1207292302} and \eqref{1208011057} in \S \ref{variational principle concerning pressure} to $\mathbf{D}$. Hence we can show that
\begin{equation*}
\limsup_{n\rightarrow \infty} \frac{1}{|F_n|} \int_\Omega \sup_{x\in \mathcal{E}_\omega} d_{F_n} (\omega, x) d \mathbb{P} (\omega)= \max_{\mu\in \mathcal{P}_\mathbb{P} (\mathcal{E}, G)}  \mu (\mathbf{D}).
\end{equation*}

\section{The local variational principle for amenable groups admitting a tiling F\o lner sequence}\label{special}

In this section we  discuss the local variational principle for fiber topological pressure. In this section, we shall remove the assumption of monotonicity from the family $\mathbf{D}$ and instead, impose the assumption that the group $G$ admits a tiling F\o lner sequence.

\medskip

\emph{Throughout this section, we assume that each for each  $n\in \mathbb{N}$, $F_n$ tiles $G$.}

\medskip

Let $\mathbf{D}= \{d_F: F\in \mathcal{F}_G\}\subseteq \mathbf{L}_\mathcal{E}^1 (\Omega, C (X))$ be a sub-additive $G$-invariant family and $\mathcal{U}\in \mathbf{C}_\mathcal{E}$. Recall that we have not assumed that $\mathbf{D}$ is monotone as in \S \ref{fourth}. As each $F_n$ tiles $G$, by Proposition \ref{p1006172118} and Proposition \ref{1006141621} we may define
\begin{eqnarray*}
P_\mathcal{E} (\mathbf{D}, \mathcal{U}, \mathbf{F})&= &
\lim_{n\rightarrow \infty} \frac{1}{|F_n|} \int_\Omega
\log P_\mathcal{E} (\omega, \mathbf{D}, F_n, \mathcal{U}, \mathbf{F}) d \mathbb{P} (\omega) \\
&= & \inf_{n\in \mathbb{N}} \frac{1}{|F_n|} \int_\Omega
\log P_\mathcal{E} (\omega, \mathbf{D}, F_n, \mathcal{U}, \mathbf{F}) d \mathbb{P} (\omega)
\end{eqnarray*}
and
\begin{equation*}
P_\mathcal{E} (\mathbf{D}, \mathbf{F})= \sup_{\mathcal{V}\in \mathbf{C}_X^o}
P_\mathcal{E} (\mathbf{D}, (\Omega\times \mathcal{V})_\mathcal{E}, \mathbf{F}).
\end{equation*}
We will continue to call these the \emph{fiber topological $\mathbf{D}$-pressure of
$\mathbf{F}$ with respect to $\mathcal{U}$} and the \emph{fiber topological $\mathbf{D}$-pressure of
$\mathbf{F}$}, respectively. By the same reasoning, for each $\mu\in \mathcal{P}_\mathbb{P} (\mathcal{E}, G)$ we can define
\begin{eqnarray}
\mu (\mathbf{D})&= & \lim_{n\rightarrow \infty} \frac{1}{|F_n|} \int_{\mathcal{E}} d_{F_n} (\omega, x) d \mu (\omega, x)\nonumber \\
&= & \inf_{n\in \mathbb{N}} \frac{1}{|F_n|} \int_{\mathcal{E}} d_{F_n} (\omega, x) d \mu (\omega, x) \label{1102111533}
\end{eqnarray}
and
\begin{equation*}
\text{sup}_\mathbb{P} (\mathbf{D})= \lim_{n\rightarrow \infty} \frac{1}{|F_n|} \int_\Omega \sup_{x\in \mathcal{E}_\omega} d_F (\omega, x) d \mathbb{P} (\omega)\ge \mu (\mathbf{D}).
\end{equation*}
 As above, all these invariants are independent of the choice of tiling F\o lner sequences.
Remark that neither $\mu (\mathbf{D})$ nor $\text{sup}_\mathbb{P} (\mathbf{D})$ need be non-negative, as our assumption  does not imply that $\mathbf{D}$ is non-negative. In fact, they may take the value of $- \infty$. Moreover, by \eqref{1102111533}, the function $\bullet (\mathbf{D}): \mathcal{P}_\mathbb{P} (\mathcal{E}, G)\rightarrow \R\cup \{- \infty\}, \mu\mapsto \mu (\mathbf{D})$ is u.s.c. over the compact metric space $\mathcal{P}_\mathbb{P} (\mathcal{E}, G)$.

\medskip

With the above definitions, almost all of the results in the previous sections hold. We give a brief sketch of the main results.

As in Proposition \ref{1007120919} and Proposition \ref{1102041733} one has:

\begin{prop} \label{1008282301}
Let $\mathbf{D}= \{d_F: F\in \mathcal{F}_G\}\subseteq \mathbf{L}_\mathcal{E}^1 (\Omega, C (X))$ be a sub-additive $G$-invariant family and $\mathcal{U}\in \mathbf{C}_\mathcal{E}, \mu\in \mathcal{P}_\mathbb{P} (\mathcal{E}, G)$. Then
\begin{enumerate}

\item
$P_\mathcal{E} (\mathbf{D}, \mathcal{U}, \mathbf{F})\ge h_\mu^{(r)} (\mathbf{F}, \mathcal{U})+ \mu (\mathbf{D})$.

 \item If $\mu (\mathbf{D})> -\infty$ then $P_\mathcal{E} (\mathbf{D}, \mathbf{F})\ge h_\mu^{(r)} (\mathbf{F})+ \mu (\mathbf{D})$.

\item \label{1208011406}
$\text{sup}_\mathbb{P} (\mathbf{D})\le P_\mathcal{E} (\mathbf{D}, \mathcal{U}, \mathbf{F})\le h_{\text{top}}^{(r)} (\mathbf{F}, \mathcal{U})+ \text{sup}_\mathbb{P} (\mathbf{D})$.

\item \label{1208011408} If $\text{sup}_\mathbb{P} (\mathbf{D})= - \infty$ then $P_\mathcal{E} (\mathbf{D}, \mathbf{F})= -\infty$.
\end{enumerate}
\end{prop}
\begin{proof}
With the above definitions, the proof is similar to that of Proposition \ref{1007120919} and Proposition \ref{1102041733}, except the last item \eqref{1208011408}.
Now we assume $\text{sup}_\mathbb{P} (\mathbf{D})= - \infty$. Applying \eqref{1208011406} to each $\mathcal{V}\in \mathbf{C}_\mathcal{E}$ we obtain
$P_\mathcal{E} (\mathbf{D}, \mathcal{V}, \mathbf{F})= -\infty$, which implies $P_\mathcal{E} (\mathbf{D}, \mathbf{F})= -\infty$. The result follows as before.
\end{proof}

Moreover, we have:

\begin{thm} \label{1008272328}
Assume that $\mathcal{U}\in \mathbf{C}_\mathcal{E}^o$ is factor good.
\begin{enumerate}

\item \label{1008272344} If $\mathbf{D}= \{d_F: F\in \mathcal{F}_G\}\subseteq \mathbf{L}_\mathcal{E}^1 (\Omega, C (X))$ is a sub-additive $G$-invariant family satisfying the assumption of $(\spadesuit)$ then
    \begin{equation*}
P_\mathcal{E} (\mathbf{D}, \mathcal{U}, \mathbf{F})= \max_{\mu\in \mathcal{P}_\mathbb{P} (\mathcal{E}, G)} [h_\mu^{(r)} (\mathbf{F}, \mathcal{U})+ \mu (\mathbf{D})],
\end{equation*}
$$\text{sup}_\mathbb{P} (\mathbf{D})= \max_{\mu\in \mathcal{P}_\mathbb{P} (\mathcal{E}, G)} \mu (\mathbf{D}).$$

\item \label{1008272345} If $f\in
 \mathbf{L}_\mathcal{E}^1 (\Omega, C (X))$ then
 \begin{equation*}
P_\mathcal{E} (\mathbf{D}^f, \mathcal{U}, \mathbf{F})= \max_{\mu\in \mathcal{P}_\mathbb{P} (\mathcal{E}, G)} [h_\mu^{(r)} (\mathbf{F}, \mathcal{U})+ \int_{\mathcal{E}} f (\omega, x) d \mu (\omega, x)].
\end{equation*}
\end{enumerate}
\end{thm}
\begin{proof}
\eqref{1008272344} The proof is essentially a re-writing of the proof of Theorem \ref{1007141414} and Proposition \ref{1102071746}.

\eqref{1008272345} Obviously, $\mathbf{D}^f\subseteq \mathbf{L}_\mathcal{E}^1 (\Omega, C (X))$ is a sub-additive $G$-invariant family satisfying $(\spadesuit)$ and $\mu (\mathbf{D}^f)= \int_{\mathcal{E}} f (\omega, x) d \mu (\omega, x)$ for each $\mu\in \mathcal{P}_\mathbb{P} (\mathcal{E}, G)$. Hence the conclusion follows directly from \eqref{1008272344}.
\end{proof}

Combining Theorem \ref{1008272328} with Theorem \ref{1006122212} and Proposition \ref{1008282301}, we have, as a direct corollary:

\begin{cor} \label{1102082226}
Let $\mathbf{D}= \{d_F: F\in \mathcal{F}_G\}\subseteq \mathbf{L}_\mathcal{E}^1 (\Omega, C (X))$ be a sub-additive $G$-invariant family satisfying $(\spadesuit)$. Then
$$
P_\mathcal{E} (\mathbf{D}, \mathbf{F})=
\begin{cases}
- \infty,& \text{if}\ \text{sup}_\mathbb{P} (\mathbf{D})= -\infty\\
\sup\limits_{\mu\in \mathcal{P}_\mathbb{P} (\mathcal{E}, G), \mu (\mathbf{D})> -\infty} [h_\mu^{(r)} (\mathbf{F})+ \mu (\mathbf{D})],& \text{otherwise}
\end{cases}.
$$
In particular, for each $f\in \mathbf{L}_\mathcal{E}^1 (\Omega, C (X))$,
\begin{equation*}
P_\mathcal{E} (\mathbf{D}^f, \mathbf{F})= \sup_{\mu\in \mathcal{P}_\mathbb{P} (\mathcal{E}, G)} [h_{\mu}^{(r)} (\mathbf{F})+ \int_\mathcal{E} f (\omega, x) d \mu (\omega, x)].
\end{equation*}
\end{cor}

\medskip

For a given sub-additive $G$-invariant family $\mathbf{D}= \{d_F: F\in \mathcal{F}_G\}\subseteq \mathbf{L}_\mathcal{E}^1 (\Omega, C (X))$, does $\mathbf{D}$ always satisfy $(\spadesuit)$? In general, this assumption is not easy to check, except when  the family $\mathbf{D}$ is strongly  sub-additive (see \S \ref{assumption}).

In the remainder of this section, we will discuss this question again in the case where $G$ is abelian. Remark that if $G$ is abelian then it always admits a tiling F\o lner sequence \cite{W0}.

\begin{prop} \label{1008292115}
Let $\mathbf{D}= \{d_F: F\in \mathcal{F}_G\}\subseteq \mathbf{L}_\mathcal{E}^1 (\Omega, C (X))$ be a sub-additive $G$-invariant family. If $G$ is abelian then $\mathbf{D}$ satisfies $(\spadesuit)$.
\end{prop}

Before proving Proposition \ref{1008292115}, we make the following observation.

\begin{lem} \label{1008282306}
Let $\mathbf{D}= \{d_F: F\in \mathcal{F}_G\}\subseteq \mathbf{L}_\mathcal{E}^1 (\Omega, C (X))$ be a sub-additive $G$-invariant family and  $T\in \mathcal{T}_G, \epsilon> 0$. Assume that $G$ is abelian and the family $- \mathbf{D}$ is non-negative. Then,
whenever $n\in \mathbb{N}$ is sufficiently large, there exists $H_n\subseteq F_n$ such that $|F_n\setminus H_n|\le 2 \epsilon |F_n|$ and, for $\mathbb{P}$-a.e. $\omega\in \Omega$,
\begin{equation*}
d_{F_n} (\omega, x)\le \frac{1}{|T|} \sum_{g\in H_n} d_T (g (\omega, x))\ \text{for each}\ x\in \mathcal{E}_\omega.
\end{equation*}
\end{lem}
\begin{proof}
As $T\in \mathcal{T}_G$ and $\{F_n: n\in \mathbb{N}\}$ is a F\o lner sequence of $G$. Thus for $n\in \mathbb{N}$ large enough, there exists $E_n\in \mathcal{F}_G$ such that $T g, g\in E_n$ are pairwise disjoint,
$T E_n\subseteq T_n\doteq F_n\cap \bigcap\limits_{t\in T} t^{- 1} F_n$ and $|T E_n|\ge |T_n|- \epsilon |F_n|, |T_n|\ge (1- \epsilon) |F_n|$. Hence,
\begin{equation} \label{1208011519}
|T E_n|\ge (1- 2 \epsilon) |F_n|.
\end{equation}
By assumption, $\mathbf{D}$ is a sub-additive $G$-invariant family, $- \mathbf{D}$ is non-negative and the group $G$ is abelian. Thus, for $\mathbb{P}$-a.e. $\omega\in \Omega$,
\begin{eqnarray}
d_{F_n} (\omega, x)&\le &
d_{t T_n} (\omega, x)+ d_{F_n\setminus t T_n} (\omega, x)\ (\text{as $t T_n\subseteq F_n$})\nonumber \\
&\le & d_{t T E_n} (\omega, x)+ d_{t (T_n\setminus T E_n)} (\omega, x)\ (\text{as $T E_n\subseteq T_n$})\nonumber \\
&\le & \sum_{g\in E_n} d_{t T} (g (\omega, x))\ (\text{as $T g, g\in E_n$ are pairwise disjoint})\nonumber \\
&= & \sum_{g\in E_n} d_{T t} (g (\omega, x))= \sum_{g\in E_n} d_T (t g (\omega, x)) \label{1008292240}
\end{eqnarray}
for each $t\in T$ and any $x\in \mathcal{E}_\omega$. Summing \eqref{1008292240} over all $t\in T$ we obtain:
\begin{equation} \label{1008292249}
|T| d_{F_n} (\omega, x)\le \sum_{g\in T E_n} d_T (g (\omega, x))
\end{equation}
for $\mathbb{P}$-a.e. $\omega\in \Omega$ and each $x\in \mathcal{E}_\omega$ (observe that $T g, g\in E_n$ are pairwise disjoint). The theorem follows from \eqref{1208011519} and \eqref{1008292249} by setting $H_n= T E_n$.
\end{proof}

Now let us finish the proof of Proposition \ref{1008292115}.

\begin{proof}[Proof of Proposition \ref{1008292115}]
Let $\{\nu_n: n\in \mathbb{N}\}\subseteq \mathcal{P}_\mathbb{P} (\mathcal{E})$ be a given sequence. Set
$$\mu_n= \frac{1}{|F_n|} \sum\limits_{g\in F_n} g \nu_n\ \text{for each}\ n\in \mathbb{N}.$$
By Proposition \ref{1007192023} \eqref{1208011203} there exists a subsequence $\{n_j: j\in \mathbb{N}\}\subseteq \mathbb{N}$ such that the sequence $\{\mu_{n_j}: j\in \mathbb{N}\}$ converges to some $\mu\in \mathcal{P}_\mathbb{P} (\mathcal{E}, G)$. We show that
\begin{equation} \label{1008292311}
\limsup_{j\rightarrow \infty} \frac{1}{|F_{n_j}|} \int_{\mathcal{E}} d_{F_{n_j}} (\omega, x) d \nu_{n_j} (\omega, x)\le \mu (\mathbf{D}).
\end{equation}

As in the proof of Proposition \ref{1008300004}, we may assume that the family $- \mathbf{D}$ is non-negative.
Now applying Lemma \ref{1008282306} to $\mathbf{D}$ we see that, if we fix $T\in \mathcal{T}_G$  and $\epsilon> 0$, and if  $n\in \mathbb{N}$ is sufficiently large then there exists $T_n\subseteq F_n$ such that $|F_n\setminus T_n|\le 2 \epsilon |F_n|$ and, for $\mathbb{P}$-a.e. $\omega\in \Omega$,
\begin{equation} \label{1008292331}
d_{F_n} (\omega, x)\le \frac{1}{|T|} \sum_{g\in T_n} d_T (g (\omega, x))\ \text{for each}\ x\in \mathcal{E}_\omega.
\end{equation}
We see from this that, we may assume without loss of generality that: $T_n\subseteq F_n$ satisfies $\lim\limits_{n\rightarrow \infty} \frac{|T_n|}{|F_n|}= 1$ and, for $\mathbb{P}$-a.e. $\omega\in \Omega$,
\eqref{1008292331}
holds for all sufficiently large $n\in \N$. Now we set
$$w_n= \frac{1}{|T_n|} \sum\limits_{g\in T_n} g \nu_n\ \text{for each large enough}\ n\in \mathbb{N}.$$
 Observe that the sequence $\{\mu_{n_j}: j\in \mathbb{N}\}$ converges to $\mu$. By the choice of $T_n, n\in \mathbb{N}$, it is easy to see that the sequence $\{w_{n_j}: j\in \mathbb{N}\}$ also converges to $\mu$. Thus
\begin{eqnarray} \label{1008292345}
& & \limsup_{j\rightarrow \infty} \frac{1}{|F_{n_j}|} \int_{\mathcal{E}} d_{F_{n_j}} (\omega, x) d \nu_{n_j} (\omega, x)\nonumber \\
& &\le \limsup_{j\rightarrow \infty} \frac{1}{|F_{n_j}|} \int_{\mathcal{E}} \frac{1}{|T|} \sum_{g\in T_{n_j}} d_T (g (\omega, x)) d \nu_{n_j} (\omega, x)\ (\text{using \eqref{1008292331}})\nonumber \\
& &= \limsup_{j\rightarrow \infty} \frac{1}{|T_{n_j}|} \int_{\mathcal{E}} \frac{1}{|T|} \sum_{g\in T_{n_j}} d_T (g (\omega, x)) d \nu_{n_j} (\omega, x)\ (\text{by the selection of $T_{n_j}$})\nonumber \\
& &= \limsup_{j\rightarrow \infty} \frac{1}{|T|} \int_{\mathcal{E}} d_T (\omega, x) d w_{n_j} (\omega, x)\ (\text{by the definition of $w_{n_j}$})\nonumber \\
& &= \frac{1}{|T|} \int_\mathcal{E} d_T (\omega, x) d \mu (\omega, x)\ (\text{as the sequence $\{w_{n_j}: j\in \mathbb{N}\}$ converges to $\mu$}).
\end{eqnarray}
Now recall our assumption that $F_n\in \mathcal{T}_G$ for each $n\in \mathbb{N}$. By \eqref{1008292345} we have
\begin{equation*}
\limsup_{j\rightarrow \infty} \frac{1}{|F_{n_j}|} \int_{\mathcal{E}} d_{F_{n_j}} (\omega, x) d \nu_{n_j} (\omega, x)\le \frac{1}{|F_n|} \int_\mathcal{E} d_{F_n} (\omega, x) d \mu (\omega, x)
\end{equation*}
for each $n\in \mathbb{N}$, from which \eqref{1008292311} follows, once we take the infimum over all $n\in \mathbb{N}$ and observe \eqref{1102111533}. This finishes the proof.
\end{proof}

We conclude this section with further comments about our main results of this section for the special case $G= \Z$.

Assume $G= \Z$.
Then Proposition \ref{1008292115} is just \cite[Lemma 3.5]{ZC}, so we recover \cite[Lemma 2.3]{CFH} by Cao, Feng and Huang. Corollary \ref{1102082226} is exactly \cite[Theorem 4.1]{ZC}, the main result by Zhao and Cao  \cite{ZC}. Hence we recover not only the main result \cite[Theorem 1.1]{CFH} but also \cite[Proposition 2.2]{K1} by Kifer. Moreover, as we did in Remark \ref{1103031512}, we may use Proposition \ref{1008292115} to deduce \cite[Theorem 4.5]{Z-DCDS} and \cite[Theorem 6.4]{Z-DCDS} from Theorem \ref{1008272328} and Corollary \ref{1102082226}, respectively. As the argument is the same as that of Remark \ref{1103031512},  we omit the details.
We also remark that Feng and Huang \cite{FengHuang} considered the topological pressure of limit sub-additive potentials and obtained \cite[Theorem 3.1]{FengHuang} in a similar spirit.

\section{Another version of the local variational principle} \label{ninth}

In all of our previous discussions of the variational principle for local fiber topological pressure of a continuous bundle RDS, we have only considered finite random open covers. In this section, we relax this assumption and consider countable random open covers. This is inspired by Kifer's work \cite[\S 1]{K1}.

\medskip

Let us briefly discuss Kifer's ideas in \cite[\S 1]{K1}.

Let $(Z, s)$ be a metric space. For each $r> 0$ and for any compact subset $Y\subseteq Z$, denote by $N_Y (r)$ the minimal number $n\in \N$ such that there exists a family of closed balls with diameter $r$ and centers $z_1, \cdots, z_n\in Z$, which covers $Y$.

For the continuous bundle RDS $\mathbf{F}= \{F_{g, \omega}: \mathcal{E}_\omega\rightarrow \mathcal{E}_{g \omega}| g\in G, \omega\in \Omega\}$ by Standard Assumption 4. In \cite{K1}, Kifer considered fiber topological pressure, in the spirit of Bowen's separated subsets, for the special case of $G= \mathbb{Z}_+$ and $\mathbf{D}^f$ for given $f\in \mathbf{L}_\mathcal{E}^1 (\Omega, C (X))$.

Recall from \S \ref{fourth} that
$$\mathbf{D}^f= \{d_F^f (\omega, x)\doteq \sum\limits_{g\in F} f (g (\omega, x)): F\in \mathcal{F}_G\}\subseteq \mathbf{L}_\mathcal{E}^1 (\Omega, C (X)).$$

Now, for any positive random variable $\varepsilon: (\Omega, \mathcal{F}, \mathbb{P})\rightarrow \R_{> 0}$, the function $N_{\mathcal{E}_\omega} (\varepsilon (\omega))$ is measurable in $\omega\in \Omega$ \cite[Page 205]{K1}.
 Kifer  \cite[Definition 1.9]{K1} defined a class $\mathcal{N}$ as follows: $\varepsilon\in \mathcal{N}$ if and only if
\begin{equation} \label{1102121429}
\int_\Omega \log N_{\mathcal{E}_\omega} (\varepsilon (\omega)) d \mathbb{P} (\omega)< \infty.
\end{equation}
As the metric space $(X, d)$ is compact, any positive constant is contained in $\mathcal{N}$ if it is viewed as a constant function on $(\Omega, \mathcal{F}, \mathbb{P})$.

Kifer  \cite[Definition 1.3]{K1} introduced
 the fiber topological $\mathbf{D}^f$-pressure of $\mathbf{F}$ associated with any given positive random variable $\varepsilon$ in \cite[(1.2)]{K1}. He denoted this by  $P_\mathcal{E} (\mathbf{D}^f, \varepsilon, \mathbf{F})$, and defined the global fiber topological $\mathbf{D}^f$-pressure of $\mathbf{F}$ by:
\begin{equation} \label{1208022239}
P_\mathcal{E} (\mathbf{D}^f, \mathbf{F})= \sup P_\mathcal{E} (\mathbf{D}^f, \epsilon, \mathbf{F}),
\end{equation}
where $\epsilon$ varies over all positive constants.
Finally, he proved that this resulting pressure
 can be defined equivalently using separated subsets with a positive random variable $\varepsilon\in \mathcal{N}$ (\cite[Proposition 1.10]{K1}):
 \begin{equation} \label{1208030150}
 P_\mathcal{E} (\mathbf{D}^f, \mathbf{F})= \sup_{\varepsilon\in \mathcal{N}} P_\mathcal{E} (\mathbf{D}^f, \varepsilon, \mathbf{F}).
 \end{equation}

 \medskip

 Hence, a natural question is whether, by analogy with \cite[Proposition 1.10]{K1}, can there a similar result for random open covers not only for a finite family but also for a countable family? This section is devoted to proving a result of this type.

 \medskip
\begin{defn}
Denote by $\mathfrak{C}^o_\mathcal{E}$ the set of all countable families $\mathcal{U}\subseteq (\mathcal{F}\times \mathcal{B}_X)\cap \mathcal{E}$ satisfying:
\begin{enumerate}

\item
$\mathcal{U}$ covers the whole space $\mathcal{E}$,

\item $\mathcal{U}_\omega= \{U_\omega: U\in \mathcal{U}\}\in \mathbf{C}^o_{\mathcal{E}_\omega}$ for $\mathbb{P}$-a.e. $\omega\in \Omega$ and \label{1010222027}

\item There exists an increasing sequence $\{\Omega_1\subseteq \Omega_2\subseteq \cdots\}\subseteq \mathcal{F}$ such that $\lim\limits_{n\rightarrow \infty} \mathbb{P} (\Omega_n)= 1$ and $\mathcal{U}\cap (\Omega_n\times X)$ is a finite family for each $n\in \mathbb{N}$. \label{1010222038}
\end{enumerate}
\end{defn}

It is not hard to see that the function $N (\mathcal{U}, \omega)$ is measurable in $\omega\in \Omega$ for each $\mathcal{U}\in \mathfrak{C}^o_\mathcal{E}$. Equation \eqref{1010222038} might at first sight seem rather contrived. However, note that for any positive random variable $\varepsilon: (\Omega, \mathcal{F}, \mathbb{P})\rightarrow \R_{> 0}$, we have
\begin{equation} \label{1208022302}
\lim_{n\rightarrow \infty} \mathbb{P} (\{\omega\in \Omega: \epsilon (\omega)> \frac{1}{n}\})= 1.
\end{equation}
In fact,
  \eqref{1010222038} is just the counterpart of \eqref{1208022302} for random open covers.

Now let $\mathcal{U}\in \mathfrak{C}^o_\mathcal{E}$ and $\mathbf{D}= \{d_F: F\in \mathcal{F}_G\}\subseteq \mathbf{L}_\mathcal{E}^1 (\Omega, C (X))$ be a monotone sub-additive $G$-invariant family.
The definitions and notation related to $\mathbf{C}_\mathcal{E}^o$ can be extended naturally to $\mathfrak{C}^o_\mathcal{E}$, including $P_\mathcal{E} (\omega, \mathbf{D}, F, \mathcal{U}, \mathbf{F})$ for each $F\in \mathcal{F}_G$ and $\mathbb{P}$-a.e. $\omega\in \Omega$. In fact, let $F\in \mathcal{F}_G$, as in
Proposition \ref{1006141641} for $\mathbb{P}$-a.e. $\omega\in \Omega$ we also have
\begin{equation} \label{1010231225}
P_\mathcal{E} (\omega, \mathbf{D}, F, \mathcal{U}, \mathbf{F})= \min \left\{\sum_{A
(\omega)\in \alpha (\omega)} \sup_{x\in A (\omega)} e^{d_F (\omega, x)}: \alpha (\omega)\in
\mathbf{P} ((\mathcal{U}_F)_\omega)\right\}.
\end{equation}
Moreover, for each $n\in \mathbb{N}$ set
\begin{equation} \label{1208022325}
\mathcal{U}_n= [\mathcal{U}\cap (\Omega_n\times X)]\cup [\{(\Omega_n^c\times X)\}\cap \mathcal{E}],
\end{equation}
then $\mathcal{U}_n\in \mathbf{C}^o_\mathcal{E}$. It is now not hard to check that
the sequence $\{P_\mathcal{E} (\omega, \mathbf{D}, F, \mathcal{U}_n, \mathbf{F}): n\in \mathbb{N}\}$ increases to $P_\mathcal{E} (\omega, \mathbf{D}, F, \mathcal{U}, \mathbf{F})$ for $\mathbb{P}$-a.e. $\omega\in \Omega$. Observe that $\mathbf{D}$ is non-negative by Proposition \ref{1006151230}, applying Proposition \ref{1006141621}
we have:
\begin{enumerate}

\item[(4)] for each $F\in \mathcal{F}_G$, the function $P_\mathcal{E} (\omega, \mathbf{D}, F, \mathcal{U}, \mathbf{F})$ is measurable in $\omega\in \Omega$.
\end{enumerate}
If, in addition,
\begin{equation} \label{1208022316}
\int_\Omega \log N (\mathcal{U}, \omega) d \mathbb{P} (\omega)< \infty,
\end{equation}
then
\begin{enumerate}
\item[(5)]
$\{\log P_\mathcal{E} (\omega, \mathbf{D}, F, \mathcal{U}, \mathbf{F}): F\in \mathcal{F}_G\}$ is a non-negative sub-additive $G$-invariant family in $L^1 (\Omega, \mathcal{F}, \mathbb{P})$ and

\item[(6)] $p: \mathcal{F}_G\rightarrow \R, F\mapsto \int_\Omega \log P_\mathcal{E} (\omega, \mathbf{D}, F, \mathcal{U}, \mathbf{F}) d \mathbb{P} (\omega)$ is a monotone non-negative $G$-invariant sub-additive function.
\end{enumerate}
From this, we can introduce
\begin{equation*}
P_\mathcal{E} (\mathbf{D}, \mathcal{U}, \mathbf{F})=
\lim_{n\rightarrow \infty} \frac{1}{|F_n|} \int_\Omega
\log P_\mathcal{E} (\omega, \mathbf{D}, F_n, \mathcal{U}, \mathbf{F}) d \mathbb{P} (\omega).
\end{equation*}
We can also introduce $h_\mu^{(r)} (\mathbf{F}, \mathcal{U})$ for each $\mu\in \mathcal{P}_\mathbb{P} (\mathcal{E}, G)$, and, similarly to Proposition \ref{1007120919} it is easy to show that
 \begin{equation} \label{1102121609}
P_\mathcal{E} (\mathbf{D}, \mathcal{U}, \mathbf{F})\ge \sup_{\mu\in \mathcal{P}_\mathbb{P} (\mathcal{E}, G)} h_\mu^{(r)} (\mathbf{F}, \mathcal{U})+ \mu (\mathbf{D}).
 \end{equation}

\medskip

All the major results of the previous sections now hold in this extended setting. We single out the principal ones as follows.

First, in the above notation we have:

\begin{prop} \label{1010222039}
Let $\mathcal{U}\in \mathfrak{C}^o_\mathcal{E}$ with corresponding increasing sequence $\{\Omega_1\subseteq \Omega_2\subseteq \cdots\}\subseteq \mathcal{F}$ satisfying $\lim\limits_{n\rightarrow \infty} \mathbb{P} (\Omega_n)= 1$ and each $\mathcal{U}\cap (\Omega_n\times X), n\in \N$ is a finite family. We define $\mathcal{U}_n, n\in \mathbb{N}$ by \eqref{1208022325}. Assume that $\mathbf{D}= \{d_F: F\in \mathcal{F}_G\}\subseteq \mathbf{L}_\mathcal{E}^1 (\Omega, C (X))$ is a monotone sub-additive $G$-invariant family. Then
\begin{equation} \label{firstfirst}
\frac{P_\mathcal{E} (\omega, \mathbf{D}, F, \mathcal{U}, \mathbf{F})}{P_\mathcal{E} (\omega, \mathbf{D}, F, \mathcal{U}_n, \mathbf{F})}\le \exp \sum_{g\in F} 1_{\Omega\setminus \Omega_n} (g \omega) \log N (\mathcal{U}, g \omega)
\end{equation}
 for each $F\in \mathcal{F}_G$, $\mathbb{P}$-a.e. $\omega\in \Omega$ and any $n\in \mathbb{N}$. If, in addition,
\eqref{1208022316} holds, then
 \begin{equation} \label{second}
 \lim_{n\rightarrow \infty} P_\mathcal{E} (\mathbf{D}, \mathcal{U}_n, \mathbf{F})= P_\mathcal{E} (\mathbf{D}, \mathcal{U}, \mathbf{F}). \end{equation}
\end{prop}
\begin{proof}
First, we establish \eqref{firstfirst}.

Fix $n\in \N$ and $F\in \mathcal{F}_G, \omega\in \Omega$ such that $N (\mathcal{U}, g \omega)$ is finite for each $g\in F$. Set
$$F^1= \{g\in F: g \omega\in \Omega_n\}\ \text{and}\ F^2= \{g\in F: g \omega\in \Omega\setminus \Omega_n\}= F\setminus F^1.$$

By the construction of $\mathcal{U}_n$ \eqref{1208022325}  one has
\begin{eqnarray}
& & P_\mathcal{E} (\omega, \mathbf{D}, F, \mathcal{U}_n, \mathbf{F})\nonumber \\
&& \hskip 10pt = \inf \left\{\sum_{A
(\omega)\in \alpha (\omega)} \sup_{x\in A (\omega)} e^{d_F (\omega, x)}: \alpha (\omega)\in
\mathbf{P}_{\mathcal{E}_\omega}, \alpha (\omega)\succeq
((\mathcal{U}_n)_F)_\omega\right\}\nonumber \\
& & \hskip 10pt = \inf \left\{\sum_{A
(\omega)\in \alpha (\omega)} \sup_{x\in A (\omega)} e^{d_F (\omega, x)}: \alpha (\omega)\in
\mathbf{P}_{\mathcal{E}_\omega}, \right. \nonumber \\
& & \hskip 86pt \left.\alpha (\omega)\succeq
\bigvee_{g\in F} F_{g^{- 1}, g \omega} (\mathcal{U}_n)_{g \omega}\right\}\ (\text{using \eqref{1006131722}})\nonumber \\
& & \hskip 10pt = \inf \left\{\sum_{A
(\omega)\in \alpha (\omega)} \sup_{x\in A (\omega)} e^{d_F (\omega, x)}: \alpha (\omega)\in
\mathbf{P}_{\mathcal{E}_\omega}, \alpha (\omega)\succeq
\bigvee_{g\in F^1} F_{g^{- 1}, g \omega} (\mathcal{U}_n)_{g \omega}\right\}\nonumber \\
& & \hskip 10pt = \inf \left\{\sum_{A
(\omega)\in \alpha (\omega)} \sup_{x\in A (\omega)} e^{d_F (\omega, x)}: \alpha (\omega)\in
\mathbf{P}_{\mathcal{E}_\omega}, \alpha (\omega)\succeq
\bigvee_{g\in F^1} F_{g^{- 1}, g \omega} \mathcal{U}_{g \omega}\right\}. \label{1010231247}
\end{eqnarray}
Moreover,
\begin{eqnarray*}
& & P_\mathcal{E} (\omega, \mathbf{D}, F, \mathcal{U}, \mathbf{F})\nonumber \\
& &\hskip 2pt \le \inf \left\{\sum_{A (\omega)\in \alpha (\omega), B
(\omega)\in \beta (\omega)} \sup_{x\in A (\omega)\cap B (\omega)} e^{d_F (\omega, x)}:\right. \\
& & \hskip 86pt \left.\alpha (\omega)\in
\mathbf{P}_{\mathcal{E}_\omega}, \alpha (\omega)\succeq \mathcal{U}_{F^1}, \beta (\omega)\in
\mathbf{P}_{\mathcal{E}_\omega}, \beta (\omega)\succeq
\mathcal{U}_{F^2}\right\}\nonumber \\
& &\hskip 2pt \le \inf \left\{\sum_{A (\omega)\in \alpha (\omega)} \sup_{x\in A (\omega)} e^{d_F (\omega, x)}: \alpha (\omega)\in
\mathbf{P}_{\mathcal{E}_\omega}, \alpha (\omega)\succeq \mathcal{U}_{F^1}\right\}\nonumber \\
& &\hskip 86pt \inf \left\{\sum_{B
(\omega)\in \beta (\omega)} 1: \beta (\omega)\in
\mathbf{P}_{\mathcal{E}_\omega}, \beta (\omega)\succeq
\mathcal{U}_{F^2}\right\}\nonumber \\
& &\hskip 2pt \le P_\mathcal{E} (\omega, \mathbf{D}, F, \mathcal{U}_n, \mathbf{F})\cdot N (\mathcal{U}_{F^2}, \omega)\ \text{(using \eqref{1006131722} and \eqref{1010231247})}\\
& &\hskip 2pt \le P_\mathcal{E} (\omega, \mathbf{D}, F, \mathcal{U}_n, \mathbf{F})\cdot \prod_{g\in F^2} N (\mathcal{U}, g \omega),
\end{eqnarray*}
which implies the conclusion.

\medskip

Next we assume that \eqref{1208022316} holds and prove \eqref{second}.

It is not hard to check that the sequence $\{P_\mathcal{E} (\mathbf{D}, \mathcal{U}_n,$ $\mathbf{F}): n\in \mathbb{N}\}$ is increasing and each member is less than $P_\mathcal{E} (\mathbf{D}, \mathcal{U}, \mathbf{F})$, that is,
\begin{equation} \label{1102121546}
P_\mathcal{E} (\mathbf{D}, \mathcal{U}, \mathbf{F})\ge \lim_{n\rightarrow \infty} P_\mathcal{E} (\mathbf{D}, \mathcal{U}_n, \mathbf{F}).
\end{equation}

For each $n\in \mathbb{N}$ and may apply \eqref{firstfirst} to obtain
\begin{eqnarray} \label{1010231806}
& & P_\mathcal{E} (\mathbf{D}, \mathcal{U}, \mathbf{F})\nonumber \\
& &\hskip 16pt\le P_\mathcal{E} (\mathbf{D}, \mathcal{U}_n, \mathbf{F})+ \limsup_{m\rightarrow \infty} \frac{1}{|F_m|} \int_\Omega \sum_{g\in F_m} 1_{\Omega\setminus \Omega_n} (g \omega) \log N (\mathcal{U}, g \omega) d \mathbb{P} (\omega)\nonumber \\
& &\hskip 16pt= P_\mathcal{E} (\mathbf{D}, \mathcal{U}_n, \mathbf{F})+ \int_\Omega 1_{\Omega\setminus \Omega_n} (\omega) \log N (\mathcal{U}, \omega) d \mathbb{P} (\omega).
\end{eqnarray}
By the assumption that $\lim\limits_{n\rightarrow \infty} \mathbb{P} (\Omega_n)= 1$ one has
\begin{equation*}
\lim_{n\rightarrow \infty} \int_\Omega 1_{\Omega\setminus \Omega_n} (\omega) \log N (\mathcal{U}, \omega) d \mathbb{P} (\omega)= 0\ (\text{using \eqref{1208022316}}),
\end{equation*}
and hence, by \eqref{1010231806},
\begin{equation} \label{1208022340}
P_\mathcal{E} (\mathbf{D}, \mathcal{U}, \mathbf{F})\le \lim_{n\rightarrow \infty} P_\mathcal{E} (\mathbf{D}, \mathcal{U}_n, \mathbf{F}).
\end{equation}
Combining \eqref{1102121546} with \eqref{1208022340} we obtain \eqref{second}.
 \end{proof}

We can now extend the local variational principle Theorem \ref{1007141414} to countable random open covers.

\begin{thm} \label{1010222042}
Let $\mathcal{U}\in \mathfrak{C}^o_\mathcal{E}$ with $\Omega_n$ and $\mathcal{U}_n, n\in \mathbb{N}$ as in Proposition \ref{1010222039}. Assume that each $\mathcal{U}_n, n\in \N$ is factor good and \eqref{1208022316} holds. If $\mathbf{D}= \{d_F: F\in \mathcal{F}_G\}\subseteq \mathbf{L}_\mathcal{E}^1 (\Omega, C (X))$ is a monotone sub-additive $G$-invariant family satisfying $(\spadesuit)$ then
\begin{equation*}
P_\mathcal{E} (\mathbf{D}, \mathcal{U}, \mathbf{F})= \sup_{\mu\in \mathcal{P}_\mathbb{P} (\mathcal{E}, G)} [h_\mu^{(r)} (\mathbf{F}, \mathcal{U})+ \mu (\mathbf{D})].
\end{equation*}
\end{thm}
\begin{proof}
Obviously for each $n\in \mathbb{N}$ we have
\begin{equation*}
\sup_{\mu\in \mathcal{P}_\mathbb{P} (\mathcal{E}, G)} [h_\mu^{(r)} (\mathbf{F}, \mathcal{U})+ \mu (\mathbf{D})]\ge \sup_{\mu\in \mathcal{P}_\mathbb{P} (\mathcal{E}, G)} [h_\mu^{(r)} (\mathbf{F}, \mathcal{U}_n)+ \mu (\mathbf{D})]= P_\mathcal{E} (\mathbf{D}, \mathcal{U}_n, \mathbf{F}),
\end{equation*}
where the last identity follows from the assumptions and Theorem \ref{1007141414}. Thus
\begin{equation*}
\sup_{\mu\in \mathcal{P}_\mathbb{P} (\mathcal{E}, G)} [h_\mu^{(r)} (\mathbf{F}, \mathcal{U})+ \mu (\mathbf{D})]\ge P_\mathcal{E} (\mathbf{D}, \mathcal{U}, \mathbf{F})\ (\text{using Proposition \ref{1010222039}}).
\end{equation*}
Combining this inequality with \eqref{1102121609}, we obtain the conclusion.
\end{proof}

There is one simple case when $\mathcal{U}\in \mathfrak{C}^o_\mathcal{E}$ satisfies the assumptions of Theorem \ref{1010222042}: when
$\mathcal{U}\in \mathfrak{C}^o_\mathcal{E}$ has the form $\cup \{(A_i\times \mathcal{V}_i)\cap \mathcal{E}: i\in \mathbb{N}\}$, where $\{\mathcal{V}_i: i\in \mathbb{N}\}\subseteq \mathbf{C}_X^o$ and $\{A_i: i\in \mathbb{N}\}$ is a partition of $(\Omega, \mathcal{F}, \mathbb{P})$ satisfying $\sum\limits_{i\in \N} \mathbb{P} (A_i) |\mathcal{V}_i|< \infty$. In fact, assume that $\mathcal{U}$ is as above. It is easy to check $\mathcal{U}\in \mathfrak{C}^o_\mathcal{E}$ with $\Omega_n= \bigcup\limits_{i= 1}^n A_i\in \mathcal{F}$ for each $n\in \N$. From the construction \eqref{1208022325} one has, for each $n\in \N$,
$$\mathcal{U}_n= \bigcup_{i= 1}^n (A_i\times \mathcal{V}_i)\cap \mathcal{E}\cup \{(\Omega_n^c\times X)\}\cap \mathcal{E}\in \mathbf{C}^o_\mathcal{E}$$
is factor good
(by Lemma \ref{1102061920} and Theorem \ref{1007212202}). Moreover, by our assumptions
\begin{equation*}
\int_\Omega \log N (\mathcal{U}, \omega) d \mathbb{P} (\omega)\le \sum_{i\in \N} \mathbb{P} (A_i) |\mathcal{V}_i|< \infty,
\end{equation*}
that is, \eqref{1208022316} holds for $\mathcal{U}$.

\medskip

Comparing Theorem \ref{1007141414} with Theorem \ref{1010222042}, we have the following question.

\begin{ques} \label{1102121808}
Under the assumptions of Theorem \ref{1010222042}, is it true that
\begin{equation} \label{1102121809}
P_\mathcal{E} (\mathbf{D}, \mathcal{U}, \mathbf{F})= \max_{\mu\in \mathcal{P}_\mathbb{P} (\mathcal{E}, G)} [h_\mu^{(r)} (\mathbf{F}, \mathcal{U})+ \mu (\mathbf{D})]?
\end{equation}
\end{ques}

Observe that in Theorem \ref{1007141414},  the supremum over $\mathcal{P}_\mathbb{P} (\mathcal{E}, G)$ can be realized as a maximum, by the direct construction in the proof.

If $f\in \mathbf{L}_\mathcal{E}^1 (\Omega, C (X))$ and $\mathcal{U}\in \mathfrak{C}^o_\mathcal{E}$ with $\mathcal{U}_n, n\in \N$ fulfill the assumptions of Theorem \ref{1010222042}, we could obtain similarly to \eqref{1208011732},
 \begin{eqnarray} \label{1208030106}
& & \limsup_{m\rightarrow \infty} \frac{1}{|F_m|} \int_\Omega \log P_\mathcal{E} (\omega, \mathbf{D}^f, F_m, \mathcal{U}, \mathbf{F}) d \mathbb{P} (\omega)\nonumber \\
& & \hskip 26pt = \lim_{n\rightarrow \infty} \limsup_{m\rightarrow \infty} \frac{1}{|F_m|} \int_\Omega \log P_\mathcal{E} (\omega, \mathbf{D}^f, F_m, \mathcal{U}_n, \mathbf{F}) d \mathbb{P} (\omega) \\
& & \hskip 26pt = \lim_{n\rightarrow \infty} \max_{\mu\in \mathcal{P}_\mathbb{P} (\mathcal{E}, G)} [h_\mu^{(r)} (\mathbf{F}, \mathcal{U}_n)+ \int_\mathcal{E} f (\omega, x) d \mu (\omega, x)]\nonumber \\
& & \hskip 26pt = \sup_{\mu\in \mathcal{P}_\mathbb{P} (\mathcal{E}, G)} [h_\mu^{(r)} (\mathbf{F}, \mathcal{U})+ \int_\mathcal{E} f (\omega, x) d \mu (\omega, x)]\nonumber .
\end{eqnarray}

Further, if the group $G$ admits a tiling F\o lner sequence (cf \S \ref{special}), the previous discussions of this section can be carried out for any sub-additive $G$-invariant family $\mathbf{D}= \{d_F: F\in \mathcal{F}_G\}\subseteq \mathbf{L}_\mathcal{E}^1 (\Omega, C (X))$. In particular,
 Theorem \ref{1010222042} holds for any sub-additive $G$-invariant family satisfying $(\spadesuit)$.

\medskip

We end this section with further discussions showing how to deduce \eqref{1208030150}, i.e. \cite[Proposition 1.10]{K1}, from our main results of this section.

Here we only outline basic ideas using standard arguments (see \cite[\S 7.2]{W}).

As in \cite{K1}, we consider the setting of $\mathbf{D}^f\subseteq \mathbf{L}_\mathcal{E}^1 (\Omega, C (X))$ with $f\in \mathbf{L}_\mathcal{E}^1 (\Omega, C (X))$:

\medskip

\noindent {\bf Step One.} Let $\epsilon> 0$ be a positive constant and $\mathcal{V}_1, \mathcal{V}_2\in \mathbf{C}_X^o$ such that $2 \epsilon$ is a Lebesgue number of $\mathcal{V}_1$ and $\text{diam} (\mathcal{V}_2)< \epsilon$, where $\text{diam} (\mathcal{V}_2)$ denotes the maximal diameter of subsets $V_2\in \mathcal{V}_2$. It is straightforward to see:
    \begin{equation} \label{ki}
    P_\mathcal{E} (\mathbf{D}^f, (\Omega\times \mathcal{V}_1)_\mathcal{E}, \mathbf{F})\le P_\mathcal{E} (\mathbf{D}, \epsilon, \mathbf{F})\le P_\mathcal{E} (\mathbf{D}, (\Omega\times \mathcal{V}_2)_\mathcal{E}, \mathbf{F}).
    \end{equation}
From this one sees that our definition \eqref{1208030124} of $P_\mathcal{E} (\mathbf{D}^f, \mathbf{F})$ is equivalent to Kifer's definition \eqref{1208022239} for the global fiber topological $\mathbf{D}^f$-pressure of $\mathbf{F}$.

\medskip

\noindent {\bf Step Two.} Now suppose that $\varepsilon\in \mathcal{N}$.
It is not hard to construct $\varepsilon_1\in \mathcal{N}$ with $\varepsilon_1\le \varepsilon$ such that $\varepsilon_1$ has the form
    $$\varepsilon_1= \sum_{i\in I} a_i 1_{\Omega_i},$$
    where $I$ is a countable index set, $a_i> 0$ for each $i\in I$ and $\{\Omega_i: i\in I\}\subseteq \mathcal{F}$ forms a partition of $\Omega$. Then it is easy to construct $\mathcal{V}\in \mathfrak{C}_\mathcal{E}^o$ such that
    $$\mathcal{V}= \bigcup_{i\in I} \{\Omega_i\times \mathcal{V}_i\},$$
where $\text{diam} (\mathcal{V}_i)< a_i$ for each $i\in I$. As in \eqref{ki} one has
$$P_\mathcal{E} (\mathbf{D}^f, \varepsilon, \mathbf{F})\le P_\mathcal{E} (\mathbf{D}^f, \mathcal{V}, \mathbf{F}).$$
Thus by \eqref{1208030106}, a variation of Proposition \ref{1010222039}, we obtain \eqref{1208030150}.

The previous arguments show that $\mathfrak{C}^o_\mathcal{E}$ plays a role in our setting analogous to that of the positive random variables in Kifer's setting, where  condition \eqref{1208022316} plays the role of condition \eqref{1102121429}.

\newpage

\part{Applications of the Local Variational Principle}

In this part we give some applications of the local variational principle established in Part \ref{skdj}. Namely, following the line of local entropy theory (cf the book chapter \cite[Chapter 19]{G1}, the recent survey \cite{GY} and references therein), we introduce and discuss both topological and measure-theoretic entropy tuples for a continuous bundle RDS. We then establish a variational relationship between these two kinds of entropy tuples. Finally, in \S \ref{factor} we apply our results to obtain many known theorems, and some new ones, in local entropy theory.

\section{Entropy tuples for a continuous bundle random dynamical system}\label{entropy tuple}

Recall again that, by Standard Assumptions 3 and 4, the family $\mathbf{F}= \{F_{g, \omega}: \mathcal{E}_\omega\rightarrow \mathcal{E}_{g \omega}| g\in G, \omega\in \Omega\}$ is a continuous bundle RDS over MDS $(\Omega,
\mathcal{F}, \mathbb{P}, G)$, where $(\Omega,
\mathcal{F}, \mathbb{P})$ is a Lebesgue space and $X$ is a compact metric space.

In this section we introduce and discuss entropy tuples for $\mathbf{F}$ in both the topological and the measure-theoretic setting, and establish a variational relation between them.
Our ideas follow the development of local entropy theory (cf \cite{G1, GY}).

\medskip

Let $\mu\in \mathcal{P}_\mathbb{P} (\mathcal{E}, G)$ and $(x_1, \cdots, x_n)\in X^n\setminus \Delta_n (X)$, where $\Delta_n (X)$ is the diagonal $\{(x_1', \cdots, x_n'): x_1'= \cdots= x_n'\in X\}, n\in \mathbb{N}\setminus \{1\}$.

 We say that $(x_1, \cdots, x_n)$ is a:
\begin{enumerate}

\item \emph{fiber topological entropy $n$-tuple of $\mathbf{F}$} if;
 for any $m\in \mathbb{N}$, there exists a closed neighborhood $V_i$ of $x_i$ of diameter at most $\frac{1}{m}$ for each $i= 1, \cdots, n$, such that $\mathcal{V}\doteq \{V_1^c, \cdots, V_n^c\}\in \mathbf{C}_X^o$ and $h_{\text{top}}^{(r)} (\mathbf{F}, (\Omega\times \mathcal{V})_\mathcal{E})> 0$.

 Equivalently, whenever $V_i$ is a closed neighborhood of $x_i$ for each $i= 1, \cdots, n$ such that $\mathcal{V}\doteq \{V_1^c, \cdots, V_n^c\}\in \mathbf{C}_X^o$, then $h_{\text{top}}^{(r)} (\mathbf{F}, (\Omega\times \mathcal{V})_\mathcal{E})> 0$.

\item \emph{$\mu$-fiber entropy $n$-tuple of $\mathbf{F}$} if;
 for any $m\in \mathbb{N}$, there exists a closed neighborhood $V_i$ of $x_i$ with diameter at most $\frac{1}{m}$ for each $i= 1, \cdots, n$, such that $\mathcal{V}\doteq \{V_1^c, \cdots, V_n^c\}\in \mathbf{C}_X^o$ and $h_\mu^{(r)} (\mathbf{F}, (\Omega\times \mathcal{V})_\mathcal{E})> 0$.

 Equivalently, whenever $V_i$ is a closed neighborhood of $x_i$ for each $i= 1, \cdots, n$ such that $\mathcal{V}\doteq \{V_1^c, \cdots, V_n^c\}\in \mathbf{C}_X^o$, then $h_\mu^{(r)} (\mathbf{F}, (\Omega\times \mathcal{V})_\mathcal{E})> 0$.
\end{enumerate}
Denote by $_\mathbb{P}E^{(r)}_n (\mathcal{E}, G)$ and $E^{(r)}_{n, \mu} (\mathcal{E}, G)$ the set of all fiber topological entropy $n$-tuples of $\mathbf{F}$ and $\mu$-fiber entropy $n$-tuples of $\mathbf{F}$, respectively. Using the notation of $_\mathbb{P}E^{(r)}_n (\mathcal{E}, G)$, we denote by $\mathbb{P}$ the phase system $(\Omega, \mathcal{F}, \mathbb{P}, G)$.

From the definitions, it is not hard to show:

\begin{prop} \label{1008071740}
Let $\mu\in \mathcal{P}_\mathbb{P} (\mathcal{E}, G)$ and $n\in \mathbb{N}\setminus \{1\}$. Then both $_\mathbb{P}E_n^{(r)} (\mathcal{E}, G)\cup \Delta_n (X)$ and $E_{n, \mu}^{(r)} (\mathcal{E}, G)\cup \Delta_n (X)$ are closed subsets of $X^n$.
\end{prop}

We will use the following well-known result, which follows from Lemma \ref{1007222138}.

\begin{lem} \label{1008052343}
Let $(Y, \mathcal{D}, \nu_n, G)$ be an MDS, $\mathcal{C}\subseteq \mathcal{D}$ a $G$-invariant sub-$\sigma$-algebra and $\alpha\in \mathbf{P}_Y$, where $(Y, \mathcal{D}, \nu_n)$ is a Lebesgue space, $n\in \mathbb{N}$. Assume that $0\le \lambda_n\le 1, n\in \N$ satisfy $\sum\limits_{n\in \mathbb{N}} \lambda_n= 1$. Then
\begin{equation*}
h_{\sum\limits_{n\in \mathbb{N}}\lambda_n \nu_n} (G, \alpha| \mathcal{C})= \sum_{n\in \mathbb{N}} \lambda_n h_{\nu_n} (G, \alpha| \mathcal{C}).
\end{equation*}
\end{lem}

We have the following variational relation between these two kinds of entropy tuples.

\begin{thm} \label{1007232313}
Let $n\in \mathbb{N}\setminus \{1\}$ and $0< \lambda_1, \cdots, \lambda_p< 1$ satisfy $\sum\limits_{i= 1}^p \lambda_i= 1$, for some $p\in \mathbb{N}$.
\begin{enumerate}

\item \label{1008052351} If $\mu\in \mathcal{P}_\mathbb{P} (\mathcal{E}, G)$ then $E_{n, \mu}^{(r)} (\mathcal{E}, G)\subseteq _\mathbb{P}E_n^{(r)} (\mathcal{E}, G)$.

   \item \label{1008081646}
  If $\mu_1, \cdots, \mu_p\in \mathcal{P}_\mathbb{P} (\mathcal{E}, G)$ then
\begin{equation*}
E_{n, \sum\limits_{i= 1}^p \lambda_i \mu_i}^{(r)} (\mathcal{E}, G)= \bigcup_{i= 1}^p E_{n, \mu_i}^{(r)} (\mathcal{E}, G).
\end{equation*}

   \item \label{1008081647} $_\mathbb{P}E_n^{(r)} (\mathcal{E}, G)= \bigcup\limits_{\mu\in \mathcal{P}_\mathbb{P} (\mathcal{E}, G)} E_{n, \mu}^{(r)} (\mathcal{E}, G)$.
 \end{enumerate}
\end{thm}
\begin{proof}
\eqref{1008052351} follows directly from Proposition \ref{1007120919} and the definitions.

 \eqref{1008081646} The containment $\supseteq$ follows directly from Lemma \ref{1008052343}. In fact, it is also easy to obtain the containment $\subseteq$ from Lemma \ref{1008052343} as follows.

 Set $\nu= \sum\limits_{i= 1}^p \lambda_i \mu_i$ and let $(x_1, \cdots, x_n)\in E_{n, \nu}^{(r)} (\mathcal{E}, G)$. Then for any $m\in \mathbb{N}$, there exists a closed neighborhood $V_i^m$ of $x_i$ with diameter at most $\frac{1}{m}$ for each $i= 1, \cdots, n$, such that $\mathcal{V}^m\doteq \{(V_1^m)^c, \cdots, (V_n^m)^c\}\in \mathbf{C}_{X}^o$ and $h_\nu^{(r)} (\mathbf{F}, (\Omega\times \mathcal{V}^m)_\mathcal{E})> 0$, and so, by Lemma \ref{1008052343}, $h_{\mu_j}^{(r)} (\mathbf{F}, (\Omega\times \mathcal{V}^m)_\mathcal{E})> 0$ for some $j\in \{1, \cdots, p\}$. Clearly there exists $J\in \{1, \cdots, p\}$ such that, $h_{\mu_J}^{(r)} (\mathbf{F}, (\Omega\times \mathcal{V}^m)_\mathcal{E})> 0$ for infinitely many $m\in \mathbb{N}$, which implies $(x_1, \cdots, x_n)\in E_{n, \mu_J}^{(r)} (\mathcal{E}, G)$.

  \eqref{1008081647}
Let $(x_1, \cdots, x_n)\in _\mathbb{P}E_n^{(r)} (\mathcal{E}, G)$.
Observing \eqref{1008052351} we only need prove that $(x_1, \cdots, x_n)\in E_{n, \mu}^{(r)} (\mathcal{E}, G)$ for some $\mu\in \mathcal{P}_\mathbb{P} (\mathcal{E}, G)$.

In fact, from the assumption, for any $m\in \mathbb{N}$, there exists a closed neighborhood $V_i^m$ of $x_i$ with diameter at most $\frac{1}{m}$ for each $i= 1, \cdots, n$, such that $\mathcal{V}^m\doteq \{(V_1^m)^c, \cdots, (V_n^m)^c\}\in \mathbf{C}_X^o$ and $h_{\text{top}}^{(r)} (\mathbf{F}, (\Omega\times \mathcal{V}^m)_\mathcal{E})> 0$.
Using Proposition \ref{1007212202} one has that $(\Omega\times \mathcal{V}^m)_\mathcal{E}\in \mathbf{C}_\mathcal{E}^o$ is factor good, and so by \eqref{1207281620} there exists $\mu_m\in \mathcal{P}_\mathbb{P} (\mathcal{E}, G)$ such that $h_{\mu_m}^{(r)} (\mathbf{F}, (\Omega\times \mathcal{V}^m)_\mathcal{E})> 0$.
Now set $\mu= \sum\limits_{m\in \mathbb{N}} \frac{\mu_m}{2^m}$. Obviously, $\mu\in \mathcal{P}_\mathbb{P} (\mathcal{E}, G)$ and, for each $m\in \mathbb{N}$,
$$h_\mu^{(r)} (\mathbf{F}, (\Omega\times \mathcal{V}^m)_\mathcal{E})\ge \frac{1}{2^m} h_{\mu_m}^{(r)} (\mathbf{F}, (\Omega\times \mathcal{V}^m)_\mathcal{E})> 0\ \text{(using Lemma \ref{1008052343})}.$$
Thus $(x_1, \cdots, x_n)\in E_{n, \mu}^{(r)} (\mathcal{E}, G)$. This finishes the proof.
\end{proof}

In fact, we can strengthen Theorem \ref{1007232313} as follows.

\begin{thm} \label{1008072226}
There exists $\mu\in \mathcal{P}_\mathbb{P} (\mathcal{E}, G)$ such that $_\mathbb{P}E_n^{(r)} (\mathcal{E}, G)= E_{n, \mu}^{(r)} (\mathcal{E}, G)$ for each $n\in \mathbb{N}\setminus \{1\}$.
\end{thm}
\begin{proof}
Remark that $X$ is a compact metric space from Standard Assumption 4,
then for each $n\in \mathbb{N}\setminus \{1\}$, there exists a sequence $\{(x_1^m, \cdots, x_n^m): m\in \mathbb{N}\}$ which is dense in $_\mathbb{P}E_n^{(r)} (\mathcal{E}, G)$. For each $n\in \mathbb{N}\setminus \{1\}$ and any $m\in \N$, by Theorem \ref{1007232313} \eqref{1008081647}, there exists $\mu_n^m\in \mathcal{P}_\mathbb{P} (\mathcal{E}, G)$ with $(x_1^m, \cdots, x_n^m)\in E_{n, \mu_n^m}^{(r)} (\mathcal{E}, G)$. Now set
\begin{equation*}
\mu= \sum_{n\in \mathbb{N}\setminus \{1\}} \frac{1}{2^{n- 1}} \sum_{m\in \mathbb{N}} \frac{1}{2^m} \mu_n^m.
\end{equation*}
Obviously, $\mu\in \mathcal{P}_\mathbb{P} (\mathcal{E}, G)$. By the  arguments of Theorem \ref{1007232313} \eqref{1008081647}, it is easy to see that $(x_1^m, \cdots, x_n^m)\in E_{n, \mu}^{(r)} (\mathcal{E}, G)$ for each $n\in \mathbb{N}\setminus \{1\}$ and any $m\in \mathbb{N}$. Now, by the choice of $(x_1^m, \cdots, x_n^m), n\in \mathbb{N}\setminus \{1\}, m\in \mathbb{N}$ we may use Proposition \ref{1008071740} and Theorem \ref{1007232313}, to see that $\mu$ has the required property.
\end{proof}

The following result tells us that both kinds of entropy tuples have nice properties with respect to lifting and projection.

\begin{prop} \label{1007271514}
Let the family $\mathbf{F}_i= \{(F_i)_{g, \omega}: (\mathcal{E}_i)_\omega\rightarrow (\mathcal{E}_i)_{g \omega}| g\in G, \omega\in \Omega\}$ be a continuous bundle RDS over $(\Omega,
\mathcal{F}, \mathbb{P}, G)$ with $X_i$ the corresponding compact metric state space, $i= 1, 2$.
Assume that $\pi: \mathcal{E}_1\rightarrow \mathcal{E}_2$ is a factor map from $\mathbf{F}_1$ to $\mathbf{F}_2$ and $n\in \mathbb{N}\setminus \{1\}, \mu\in \mathcal{P}_\mathbb{P} (\mathcal{E}_1, G)$. If $\pi$ is induced by a continuous surjection $\phi: X_1\rightarrow X_2$ via $\pi (\omega, x)= (\omega, \phi x)$, then
\begin{enumerate}

\item \label{1007271546} $E^{(r)}_{n, \pi \mu} (\mathcal{E}_2, G)\subseteq (\phi\times \cdots\times \phi) E^{(r)}_{n, \mu} (\mathcal{E}_1, G)\subseteq E^{(r)}_{n, \pi \mu} (\mathcal{E}_2, G)\cup \Delta_n (X_2)$.

\item \label{1007271547} $_\mathbb{P}E^{(r)}_n (\mathcal{E}_2, G)\subseteq (\phi\times \cdots\times \phi) _\mathbb{P}E^{(r)}_n (\mathcal{E}_1, G)\subseteq _\mathbb{P}E^{(r)}_n (\mathcal{E}_2, G)\cup \Delta_n (X_2)$.
\end{enumerate}
\end{prop}
\begin{proof}
As the proofs are similar, we shall only prove \eqref{1007271546}.

The proof follows the ideas of the proof of \cite[Proposition 4]{B2}.

First, let $(x_1, \cdots, x_n)\in E^{(r)}_{n, \mu} (\mathcal{E}_1, G)$ with $(\phi (x_1), \cdots, \phi (x_n))\in X_2^n\setminus \Delta_n (X_2)$.
As $(x_1, \cdots, x_n)$ $\in E^{(r)}_{n, \mu} (\mathcal{E}_1, G)$, for any $M\in \mathbb{N}$ there exists a closed neighborhood $V_i^M$ of $x_i$ with diameter at most $\frac{1}{M}$ for each $i= 1, \cdots, n$, such that $\mathcal{V}^M\doteq \{(V_1^M)^c, \cdots, (V_n^M)^c\}\in \mathbf{C}_{X_1}^o$ and $h_\mu^{(r)} (\mathbf{F}_1, (\Omega\times \mathcal{V}^M)_{\mathcal{E}_1})> 0$.
Now let $m\in \mathbb{N}$ and suppose that $V_i\subseteq X_2$ is a closed neighborhood of $\phi (x_i)$ with diameter at most $\frac{1}{m}$, for each $i= 1, \cdots, n$, such that $\mathcal{V}\doteq \{V_1^c, \cdots, V_n^c\}\in \mathbf{C}_{X_2}^o$.

By the continuity of $\phi$, for $M$ sufficiently large, $\phi^{- 1} V_i\supseteq V_i^M$ for each $i= 1, \cdots, n$.
Since $\pi$ is induced by $\phi$, and one has $\pi^{- 1} (\Omega\times \mathcal{V})_{\mathcal{E}_2}\succeq (\Omega\times \mathcal{V}^M)_{\mathcal{E}_1}$, it follows that
$$h_\mu^{(r)} (\mathbf{F}_1, \pi^{- 1} (\Omega\times \mathcal{V})_{\mathcal{E}_2})> 0,$$
and hence, by Lemma \ref{1007212122},
 $$h_{\pi \mu}^{(r)} (\mathbf{F}_2, (\Omega\times \mathcal{V})_{\mathcal{E}_2})> 0.$$
 This means that $(\phi x_1, \cdots, \phi x_n)\in E^{(r)}_{n, \pi \mu} (\mathcal{E}_2, G)$.

 \medskip

Now let $(y_1, \cdots, y_n)\in E^{(r)}_{n, \pi \mu} (\mathcal{E}_2, G)$.
For any $m\in \mathbb{N}$ there exists a closed neighborhood $V_i$ of $y_i$ with diameter at most $\frac{1}{m}$ for each $i= 1, \cdots, n$, such that $\mathcal{V}\doteq \{(V_1)^c, \cdots, (V_n)^c\}\in \mathbf{C}_{X_2}^o$ and $h_{\pi \mu}^{(r)} (\mathbf{F}_2, (\Omega\times \mathcal{V})_{\mathcal{E}_2})> 0$.

For each $i= 1, \cdots, n$, we can cover $\phi^{- 1} (V_i)$ with finitely many compact non-empty subsets $V_i^1, \cdots, V_i^{k_i}$ $\subseteq \phi^{- 1} (V_i), k_i\in \mathbb{N}$ of diameter at most $\frac{1}{m}$. Set
$$\mathcal{W}_{j_1, \cdots, j_n}= \{(\Omega\times V_i^{j_i})^c: i= 1, \cdots, n\}\in \mathbf{C}_{\mathcal{E}_1}^o$$
 for any $j_i= 1, \cdots, k_i, i= 1, \cdots, n$.
 Observe that
\begin{equation*} \label{1008051623}
(\Omega\times \phi^{- 1} V_i)^c= \bigcap_{j= 1}^{k_i} (\Omega\times V_i^{j})^c
\end{equation*}
 for each $i= 1, \cdots, n$. One has
$$\pi^{- 1} (\Omega\times \mathcal{V})_{\mathcal{E}_2}\preceq \bigvee\limits_{j_1= 1}^{k_1} \cdots \bigvee\limits_{j_n= 1}^{k_n} \mathcal{W}_{j_1, \cdots, j_n},$$
 and so
\begin{eqnarray*}
0&< & h_\mu^{(r)} (\mathbf{F}_1, \pi^{- 1} (\Omega\times \mathcal{V})_{\mathcal{E}_2})\ \text{(using Lemma \ref{1007212122})} \\
&\le & h_\mu^{(r)} (\mathbf{F}_1, \bigvee\limits_{j_1= 1}^{k_1} \cdots \bigvee\limits_{j_n= 1}^{k_n} \mathcal{W}_{j_1, \cdots, j_n}) \\
&\le & \sum_{j_1= 1}^{k_1} \cdots \sum_{j_n= 1}^{k_n} h_\mu^{(r)} (\mathbf{F}_1, \mathcal{W}_{j_1, \cdots, j_n}),
\end{eqnarray*}
where the last inequality uses Proposition \ref{0911192237}.  Thus
$h_\mu^{(r)} (\mathbf{F}_1, \mathcal{W}_{s_1, \cdots, s_n})> 0$
  for some $s_j\in \{1, \cdots, k_j\}$ and each $j= 1, \cdots, n$.

  In other words, there exists
$\{(W_i^m)^c: i= 1, \cdots, n\}\in \mathbf{C}_{X_1}^o$ such that
\begin{enumerate}

\item[(a)] $h_\mu^{(r)} (\mathbf{F}_1, \mathcal{U}^m)> 0$, where $\mathcal{U}^m= \{(\Omega\times W_i^m)^c: i= 1, \cdots, n\}$ and

\item[(b)] for each $i= 1, \cdots, n$, both $W_i^m$ and $\phi (W_i^m)$ have diameters at most $\frac{1}{m}$ and the distance between $y_i$ and $\phi (W_i^m)$ is also at most $\frac{1}{m}$.
\end{enumerate}
From $(b)$, for each $i= 1, \cdots, n$, $\{W_i^m: m\in \mathbb{N}\}$ converges to some point $x_i\in X_1$. Moreover, it is obvious that $\phi (x_i)= y_i$ (using $(b)$ again, recall that $\phi: X_1\rightarrow X_2$ is continuous). Our proof will be complete if we show that $(x_1, \cdots, x_n)\in E_{n, \mu}^{(r)} (\mathcal{E}_1, G)$.

In fact, for any $p\in \mathbb{N}$ there exists a closed neighborhood $W_i$ of $x_i$ with diameter at most $\frac{1}{p}$ such that $\{W_1^c, \cdots, W_n^c\}\in \mathbf{C}_{X_1}^o$. For $m\in \mathbb{N}$ sufficiently large, $W_i^m\subseteq W_i$ for each $i= 1, \cdots, n$, and so, by $(a)$,
  $$h_\mu^{(r)} (\mathbf{F}_1, \mathcal{W})> 0,\ \text{where}\ \mathcal{W}\doteq \{(\Omega\times W_i)^c: i= 1, \cdots, n\}\succeq \mathcal{U}^m.$$
  This implies that $(x_1, \cdots, x_n)\in E_{n, \mu}^{(r)} (\mathcal{E}_1, G)$, completing the proof.
\end{proof}

Moreover, we can show:

\begin{prop} \label{1102131645}
Let $\mu\in \mathcal{P}_\mathbb{P} (\mathcal{E}, G)$ and $n\in \mathbb{N}\setminus \{1\}$. Then
\begin{enumerate}

\item \label{1102131649}
$E^{(r)}_{n, \mu} (\mathcal{E}, G)\neq \emptyset$ if and only if $h_\mu^{(r)} (\mathbf{F})> 0$.

\item \label{1102131650} $_\mathbb{P}E^{(r)}_n (\mathcal{E}, G)\neq \emptyset$ if and only if $h_\text{top}^{(r)} (\mathbf{F})> 0$.
\end{enumerate}
\end{prop}

Remark that by Theorem \ref{1006122212} and the definitions we can prove Proposition \ref{1102131645} following the ideas of Blanchard \cite{B2}, see also the proof of Proposition \ref{1007271514}. As this is standard, we omit the details.

\medskip

Recall again from \S \ref{third} that, by a TDS $(Z, G)$ we mean that $Z$ is a compact metric space and $G$ is a group of homeomorphisms of $Z$ with $e_G$ acting as the identity.

For the continuous bundle RDS $\mathbf{F}= \{F_{g, \omega}: \mathcal{E}_\omega\rightarrow \mathcal{E}_{g \omega}| g\in G, \omega\in \Omega\}$, if, in addition, $G$ acts over the state space $X$ as a TDS, and
$F_{g, \omega}$ is just the restriction of the action $g$ over $\mathcal{E}_\omega$ for each $g\in G$ and $\mathbb{P}$-a.e. $\omega\in \Omega$,
then we say that \emph{$\mathbf{F}$ is induced by TDS $(X, G)$}.

From the definitions it is easy to see:

\begin{prop} \label{1008060011}
Let $\mu\in \mathcal{P}_\mathbb{P} (\mathcal{E}, G)$ and $n\in \mathbb{N}\setminus \{1\}$. If $\mathbf{F}$ is induced by TDS $(X, G)$, then both $E^{(r)}_{n, \mu} (\mathcal{E}, G)$ and $_\mathbb{P}E^{(r)}_n (\mathcal{E}, G)$ are $G$-invariant subsets of $X^n$.
\end{prop}

Let $(x_1, \cdots, x_n)\in X^n\setminus \Delta_n (X), n\in \mathbb{N}\setminus \{1\}$. We call
$(x_1, \cdots, x_n)$ a
\emph{fiber $n$-tuple of $\mathbf{F}$} if for any $m\in \mathbb{N}$, there exist $\Omega^*\in \mathcal{F}$ and a closed neighborhood $V_i$ of $x_i$ with diameter at most $\frac{1}{m}$ for each $i= 1, \cdots, n$, such that $\mathcal{V}= \{V_1^c, \cdots, V_n^c\}\in \mathbf{C}_X^o$, $\mathbb{P} (\Omega^*)> 0$ and $\prod\limits_{i= 1}^n (\{\omega\}\times V_i)\cap \mathcal{E}^n\neq \emptyset$ for each $\omega\in \Omega^*$.

Denote by $_\mathbb{P}E_n^{(r)} (\mathcal{E})$ the set of all fiber $n$-tuples of $\mathbf{F}$.
It may happen $_\mathbb{P}E_n^{(r)} (\mathcal{E})= \emptyset$: the trivial example is where $\mathcal{E}_\omega$ is just a singleton for $\mathbb{P}$-a.e. $\omega\in \Omega$.

With the above definition, as in Proposition \ref{1008071740}, we have:

\begin{prop} \label{1008072218}
Let $n\in \mathbb{N}\setminus \{1\}$. Then $_\mathbb{P}E_n^{(r)} (\mathcal{E})\cup \Delta_n (X)\subseteq \overline{\bigcup\limits_{\omega\in \Omega} \mathcal{E}_\omega^n}\cup \Delta_n (X)$ is a closed subset. Moreover, if $\mathbf{F}$ is induced by TDS $(X, G)$ then the subset $_\mathbb{P}E_n^{(r)} (\mathcal{E})$ is $G$-invariant.
\end{prop}

As in the proof of Proposition \ref{1007271514}, we obtain:

\begin{prop} \label{lift-proj}
Let the family $\mathbf{F}_i= \{(F_i)_{g, \omega}: (\mathcal{E}_i)_\omega\rightarrow (\mathcal{E}_i)_{g \omega}| g\in G, \omega\in \Omega\}$ be a continuous bundle RDS over $(\Omega,
\mathcal{F}, \mathbb{P}, G)$ with $X_i$ the corresponding compact metric state space, $i= 1, 2$.
Assume that $\pi: \mathcal{E}_1\rightarrow \mathcal{E}_2$ is a factor map from $\mathbf{F}_1$ to $\mathbf{F}_2$ and $n\in \mathbb{N}\setminus \{1\}$. If $\pi$ is induced by a continuous surjection $\phi: X_1\rightarrow X_2$, then
$$_\mathbb{P}E^{(r)}_n (\mathcal{E}_2)\subseteq (\phi\times \cdots\times \phi) _\mathbb{P}E^{(r)}_n (\mathcal{E}_1)\subseteq _\mathbb{P}E^{(r)}_n (\mathcal{E}_2)\cup \Delta_n (X_2).$$
\end{prop}

Before proceeding, we observe the following two Lemmas.

\begin{lem} \label{1008072241}
Let $V_1, \cdots ,V_n\in \mathcal{B}_X, n\in \mathbb{N}\setminus \{1\}$. Then
$$\Omega (V_1, \cdots, V_n)\doteq \{\omega\in \Omega: \prod\limits_{i= 1}^n (\{\omega\}\times V_i)\cap \mathcal{E}^n= \emptyset\}\in \mathcal{F}.$$
\end{lem}
\begin{proof}
Let $\pi: \Omega\times X\rightarrow \Omega$ be the natural projection. Remark that $(\Omega, \mathcal{F}, \mathbb{P})$ is a Lebesgue space by Standard Assumption 3, and $X$ is a compact metric space by Standard Assumption 4, and so we could apply Lemma \ref{1007152201} to obtain:
\begin{equation*}
\Omega_0\doteq \{\omega\in \Omega: \prod_{i= 1}^n (\{\omega\}\times V_i)\cap \mathcal{E}^n\neq \emptyset\}= \bigcap_{i= 1}^n \pi ((\Omega\times V_i)\cap \mathcal{E})\in \mathcal{F}.
\end{equation*}
Observe $\Omega_0= \Omega\setminus \Omega (V_1, \cdots, V_n)$, one has $\Omega (V_1, \cdots, V_n)\in \mathcal{F}$.
\end{proof}

\begin{lem} \label{1008042307}
Let $\Omega^*\in \mathcal{F}$ and $\mathcal{V}= \{V_1^c, \cdots, V_n^c\}\in \mathbf{C}_X, n\in \mathbb{N}\setminus \{1\}$. Set $\mathcal{U}= \{(\Omega^*\times V_i)^c: i= 1, \cdots, n\}$ and $\mathcal{U}'= \{(\Omega'\times V_i)^c: i= 1, \cdots, n\}$, where $\Omega'= \Omega^*\setminus \Omega (V_1, \cdots, V_n)$. Then
\begin{enumerate}

\item \label{1008042216}
$\mathcal{U}\succeq \mathcal{U}'$ and $\mathcal{U}_\omega\supseteq \mathcal{U}'_\omega$, (and hence $\mathcal{U}'_\omega\succeq \mathcal{U}_\omega$) for each $\omega\in \Omega$.

\item \label{1008042217}
$h_{\text{top}}^{(r)} (\mathbf{F}, \mathcal{U})= h_{\text{top}}^{(r)} (\mathbf{F}, \mathcal{U}')$. In particular, if $h_{\text{top}}^{(r)} (\mathbf{F}, \mathcal{U})> 0$ then $\mathbb{P} (\Omega')> 0$.

\item \label{1008042218} if $\mu\in \mathcal{P}_\mathbb{P} (\mathcal{E}, G)$ then $h_\mu^{(r)} (\mathbf{F}, \mathcal{U})= h_\mu^{(r)} (\mathbf{F}, \mathcal{U}')$.
\end{enumerate}
\end{lem}
\begin{proof}
 \eqref{1008042216} From the assumptions on $\mathcal{U}'$, it is clear that $\{\mathcal{E}_\omega\}= \mathcal{U}'_\omega$ for each $\omega\in \Omega (V_1, \cdots, V_n)$. Thus we only need check that $\mathcal{E}_\omega\in \mathcal{U}_\omega$ for each $\omega\in \Omega (V_1, \cdots, V_n)$. In fact, if $\omega\in \Omega (V_1, \cdots, V_n)$ then $\{\omega\}\times V_i\cap \mathcal{E}= \emptyset$ for some $i\in \{1, \cdots, n\}$, which implies $\mathcal{E}_\omega\subseteq V_i^c$, particularly, $\mathcal{E}_\omega\in \mathcal{U}_\omega$.

Combining the above definition with Proposition \ref{0911192237}, Lemma \ref{1007261204} and Proposition \ref{1102041733}, both \eqref{1008042217} and
   \eqref{1008042218} follow directly from \eqref{1008042216}.
\end{proof}

Thus, we have:

\begin{prop} \label{1008041856}
Let $(x_1, \cdots, x_n)\in X^n\setminus \Delta_n (X), n\in \mathbb{N}\setminus \{1\}$. Then
 \begin{enumerate}

 \item \label{1008042152} $(x_1, \cdots, x_n)\in _\mathbb{P}E_n^{(r)} (\mathcal{E})$ if and only if, whenever $V_i$ is a closed neighborhood of $x_i$ for each $i= 1, \cdots, n$ such that $\{V_1^c, \cdots, V_n^c\}\in \mathbf{C}_X^o$, then $\mathbb{P} (\Omega (V_1, \cdots,$ $V_n))< 1$.

     \item \label{1008042154} $_\mathbb{P}E_n^{(r)} (\mathcal{E}, G)\subseteq _\mathbb{P}E_n^{(r)} (\mathcal{E})$.
 \end{enumerate}
\end{prop}
\begin{proof}
With the definitions introduced in this section,
 \eqref{1008042152} and \eqref{1008042154} follow from Lemma \ref{1008072241} and
  Lemma \ref{1008042307}, respectively.
\end{proof}

In the remainder of this section, we equip with $\Omega$ the structure of a topological space (and $\mathcal{F}$ is its Borel $\sigma$-algebra).

\medskip

Before proceeding, we need some preparations.

Let $\mu\in \mathcal{P}_\mathbb{P} (\mathcal{E}, G)$ and $n\in \mathbb{N}\setminus \{1\}$. We recall the definition of $\lambda_n^{\mathcal{F}_\mathcal{E}} (\mu)$ from \eqref{1208032307}:
\begin{equation*}
\lambda_n^{\mathcal{F}_\mathcal{E}} (\mu) (\prod_{i= 1}^n A_i)= \int_\mathcal{E} \prod_{i= 1}^n \mu (A_i| \mathcal{P}^{\mathcal{F}_\mathcal{E}} (\mathcal{E}, (\mathcal{F}\times \mathcal{B}_X)\cap \mathcal{E}, \mu, G)) d \mu
\end{equation*}
whenever $A_1, \cdots, A_n\in (\mathcal{F}\times \mathcal{B}_X)\cap \mathcal{E}$.

Let $Y$ be a topological space and $\nu$ a probability measure on $(Y, \mathcal{B}_Y)$. Denote by $\text{supp} (\nu)$ the set of all points $y\in Y$ such that $\nu (V)> 0$ whenever $V$ is an open neighborhood of $y$. Thus, $\text{supp} (\nu)\subseteq Y$ is a closed subset.

Observe that if $\Omega$ is a topological space with $\mathcal{F}= \mathcal{B}_\Omega$, then each $\mu\in \mathcal{P}_\mathbb{P} (\mathcal{E}, G)$ may be viewed as a Borel probability measure on the topological space $\Omega\times X$.

From the definition, it is easy to check:

\begin{lem} \label{small}
Let $\mu\in \mathcal{P}_\mathbb{P} (\mathcal{E}, G)$ and $n\in \mathbb{N}\setminus \{1\}$. Assume that $\Omega$ is a topological space with $\mathcal{F}= \mathcal{B}_\Omega$. Then $\text{supp} (\lambda_n^{\mathcal{F}_\mathcal{E}} (\mu))\subseteq \text{supp} (\mu)^n\subseteq (\text{supp} (\mathbb{P})\times X)^n$.
\end{lem}

We also have:

\begin{lem} \label{1008081528}
Let $\mu\in \mathcal{P}_\mathbb{P} (\mathcal{E}, G)$ and $((\omega_1, x_1), \cdots, (\omega_n, x_n))\in \text{supp} (\lambda_n^{\mathcal{F}_\mathcal{E}} (\mu)), n\in \mathbb{N}\setminus \{1\}$. Assume that $\Omega$ is a Hausdorff space with $\mathcal{F}= \mathcal{B}_\Omega$. Then $\omega_1= \cdots= \omega_n$.
\end{lem}
\begin{proof}
For each $i= 1, \cdots, n$, assume that $A_i\in (\mathcal{F}\times \mathcal{B}_X)\cap \mathcal{E}$ satisfies $A_i\subseteq \Omega_i\times X$ for some $\Omega_i\in \mathcal{F}$, and observe that
\begin{equation} \label{1208032321}
(\Omega_i\times X)\cap \mathcal{E}\in \mathcal{F}_\mathcal{E}\subseteq \mathcal{P}^{\mathcal{F}_\mathcal{E}} (\mathcal{E}, (\mathcal{F}\times \mathcal{B}_X)\cap \mathcal{E}, \mu, G).
\end{equation}
Thus
\begin{eqnarray*}
\lambda_n^{\mathcal{F}_\mathcal{E}} (\mu) (\prod_{i= 1}^n A_i) &= & \int_\mathcal{E} \prod_{i= 1}^n \mu (A_i| \mathcal{P}^{\mathcal{F}_\mathcal{E}} (\mathcal{E}, (\mathcal{F}\times \mathcal{B}_X)\cap \mathcal{E}, \mu, G)) d \mu \\
&\le & \int_\mathcal{E} \prod_{i= 1}^n \mu ((\Omega_i\times X)\cap \mathcal{E}| \mathcal{P}^{\mathcal{F}_\mathcal{E}} (\mathcal{E}, (\mathcal{F}\times \mathcal{B}_X)\cap \mathcal{E}, \mu, G)) d \mu \\
&= & \int_\mathcal{E} \prod_{i= 1}^n 1_{(\Omega_i\times X)\cap \mathcal{E}} d \mu\ (\text{using \eqref{1208032321}}) \\
&= & \mu ((\bigcap_{i= 1}^n \Omega_i\times X)\cap \mathcal{E})= \mathbb{P} (\bigcap_{i= 1}^n \Omega_i).
\end{eqnarray*}
In particular, $\lambda_n^{\mathcal{F}_\mathcal{E}} (\mu) (\prod\limits_{i= 1}^n A_i)= 0$ once $\mathbb{P} (\bigcap\limits_{i= 1}^n \Omega_i)= 0$.

Now assume that $((\omega_1, x_1), \cdots, (\omega_n, x_n))\in \mathcal{E}^n$ such that $\omega_i\neq \omega_j$ for some $1\le i< j\le n$. Obviously there exist open neighborhoods $\Omega_i$ and $\Omega_j$ of $x_i$ and $x_j$, respectively, such that $\Omega_i\cap \Omega_j= \emptyset$. Thus, by the above discussions,
\begin{equation*}
\lambda_n^{\mathcal{F}_\mathcal{E}} (\mu) \left(\prod_{k\in \{1, \cdots, n\}\setminus \{i, j\}} (\Omega\times X)\cap \mathcal{E}\times \prod_{p= i, j} (\Omega_p\times X)\cap \mathcal{E}\right)= 0,
\end{equation*}
which implies $((\omega_1, x_1), \cdots, (\omega_n, x_n))\notin \text{supp} (\lambda_n^{\mathcal{F}_\mathcal{E}} (\mu))$. This finishes our proof.
\end{proof}

Hence one has:

\begin{thm} \label{1008081608}
Let $\mu\in \mathcal{P}_\mathbb{P} (\mathcal{E}, G)$ and $(x_1, \cdots, x_n)\in X^n\setminus \Delta_n (X), n\in \mathbb{N}\setminus \{1\}$. Assume that $\Omega$ is a topological space with $\mathcal{F}= \mathcal{B}_\Omega$. Then
\begin{enumerate}

\item \label{1008082244} $(a)\Longleftrightarrow (b)\Longleftarrow (c)$.


\item \label{1008082246} If, in addition, $\Omega$ is a compact metric space then $(a)\Longleftrightarrow (b)\Longleftrightarrow (c)$.
\end{enumerate}
Where:
\begin{enumerate}

\item[(a)] $(x_1, \cdots, x_n)\in E^{(r)}_{n, \mu} (\mathcal{E}, G)$.

\item[(b)] If $V_i$ is a Borel neighborhood of $x_i$ for each $i= 1, \cdots, n$ then
$$\lambda_n^{\mathcal{F}_\mathcal{E}} (\mu) (\prod\limits_{i= 1}^n (\Omega\times V_i)\cap \mathcal{E}^n)> 0.$$

\item[(c)] There exists $\omega\in \Omega$ such that $((\omega, x_1), \cdots, (\omega, x_n))\in \text{supp} (\lambda_n^{\mathcal{F}_\mathcal{E}} (\mu))$.
\end{enumerate}
\end{thm}
\begin{proof}
\eqref{1008082244}
Observing that $\lambda_n^{\mathcal{F}_\mathcal{E}} (\mu) (\mathcal{E}^n)= 1$,
 $(c)$ implies $(b)$. Reminder that $(\Omega, \mathcal{F}, \mathbb{P})$ is a Lebesgue space by Standard Assumption 3, and so
$(a)\Longleftrightarrow (b)$ follows from the definitions and Theorem \ref{1006301434}.

\eqref{1008082246} Now, in addition, we assume that $\Omega$ is a compact metric space. By \eqref{1008082244}, it remains to show $(b)\Longrightarrow (c)$.

 For each $\omega\in \Omega$ and $r> 0$ denote by $B (\omega, r)$ the open ball of $\Omega$ with center $\omega$ and radius $r$.
For any $m\in \mathbb{N}$, let $V_i^m$ be a Borel neighborhood of $x_i$ with diameter at most $\frac{1}{m}$ for each $i= 1, \cdots, n$. Our assumption is that
$$\lambda_n^{\mathcal{F}_\mathcal{E}} (\mu) (\prod\limits_{i= 1}^n (\Omega\times V_i^m)\cap \mathcal{E}^n)> 0$$
and $\Omega$ is a compact metric space. One has $\Omega_m\neq \emptyset$, where
\begin{equation} \label{1208032354}
\Omega_m= \{(\omega_1, \cdots, \omega_n)\in \Omega^n: \lambda_n^{\mathcal{F}_\mathcal{E}} (\mu) \left(\prod_{i= 1}^n \left(B (\omega_i, \frac{1}{m})\times V_i^m\right)\cap \mathcal{E}^n\right)> 0\}.
\end{equation}

Set $\Omega^*= \bigcap\limits_{m\in \mathbb{N}} \overline{\Omega_m}$.
From the definition of \eqref{1208032354}, for each $m_1\in \N$ there exists $M\in \N$ such that if $m\ge M$ then $\Omega_m\subseteq \Omega_{m_1}$. Now $\Omega^n$ is also a compact metric space, and so $\Omega^*\subseteq \Omega^n$ is a non-empty subset.

   Now let $(\omega_1, \cdots, \omega_n)\in \Omega^*$ and let $V$ be a Borel neighborhood of $((\omega_1, x_1), \cdots,$ $(\omega_n, x_n))$. By the construction of $\Omega^*$,
for $m\in \mathbb{N}$ sufficiently large, there exists $(\omega_1^m, \cdots, \omega_n^m)\in \Omega_m$ such that, if $V_i$ is the closed ball in $X$ with center $x_i$ and radius $\frac{1}{m}$ for each $i= 1, \cdots, n$ then $\prod\limits_{i= 1}^n B (\omega_i^m, \frac{1}{m})\times V_i \subseteq V$. Hence
$$\lambda_n^{\mathcal{F}_\mathcal{E}} (\mu) (V)\ge \lambda_n^{\mathcal{F}_\mathcal{E}} (\mu) \left(\prod_{i= 1}^n \left(B (\omega_i^m, \frac{1}{m})\times V_i\right)\cap \mathcal{E}^n\right) > 0.$$
   Since $V$ is arbitrary, one has $((\omega_1, x_1), \cdots, (\omega_n, x_n))\in \text{supp} (\lambda_n^{\mathcal{F}_\mathcal{E}} (\mu))$, and, in addition, $\omega_1= \cdots= \omega_n$ by Lemma \ref{1008081528}. This finishes the proof.
\end{proof}

As a direct corollary of Theorem \ref{1007232313} and Theorem \ref{1008081608}, one has:

\begin{thm} \label{conclusion}
Let $\mu\in \mathcal{P}_\mathbb{P} (\mathcal{E}, G)$ and $n\in \mathbb{N}\setminus \{1\}$ with $\pi_n: (\Omega\times X)^n\rightarrow X^n$ the natural projection. Assume that $\Omega$ is a compact metric space with $\mathcal{F}= \mathcal{B}_\Omega$. Then
$$E^{(r)}_{n, \mu} (\mathcal{E}, G)= \pi_n (\text{supp} (\lambda_n^{\mathcal{F}_\mathcal{E}} (\mu))) \setminus \Delta_n (X),$$
$$_\mathbb{P}E_n^{(r)} (\mathcal{E}, G)= \pi_n \left(\bigcup_{\nu\in \mathcal{P}_\mathbb{P} (\mathcal{E}, G)} \text{supp} (\lambda_n^{\mathcal{F}_\mathcal{E}} (\nu))\right) \setminus \Delta_n (X).$$
\end{thm}



 \section{Applications to topological dynamical systems} \label{factor}

 In this section, we apply results obtained in the previous sections to the case of a topological dynamical system.
  We recover many recent results in the local entropy theory of $\Z$-actions (cf \cite{B2, BHMMR, G, G1, GY, HY, HYZ1, HYZ2}) and of infinite countable discrete amenable group actions (cf \cite{HYZ}).
  We also prove new results, some of which are novel even in the case of infinite countable discrete amenable groups, for example Theorem \ref{1010201200} and Theorem \ref{1008301614}.

 \subsection{Preparations on topological dynamical systems}\label{1208042310}\

 \medskip

 Recall from \S \ref{third} that if $Y$ is a compact metric space and $G$ is a group of homeomorphisms of $Y$ with $e_G$ acting as the identity, then we say $(Y, G)$ is a TDS.

 We denote by $\mathcal{P} (Y, G)$ the set of all $G$-invariant elements of $\mathcal{P} (Y)$, which we suppose equipped with the weak* topology. Then $\mathcal{P} (Y, G)$ is a non-empty compact metric space. For each $\nu\in \mathcal{P} (Y)$, clearly $(Y, \mathcal{B}_Y^\nu, \nu)$ (also denoted by $(Y, \mathcal{B}_Y, \nu)$ where there is no ambiguity) is a Lebesgue space, where $\mathcal{B}_Y$ is the Borel $\sigma$-algebra of $Y$ and $\mathcal{B}_Y^\nu$ is the $\nu$-completion of $\mathcal{B}_Y$.

Recall that $\pi: (Y_1, G)\rightarrow (Y_2, G)$ is a \emph{factor map from TDS $(Y_1, G)$ to TDS $(Y_2, G)$} if $\pi: Y_1\rightarrow Y_2$ is a continuous surjection which intertwines the actions of $G$ (i.e. $\pi\circ g (y_1)= g\circ \pi (y_1)$ for each $g\in G$ and any $y_1\in Y_1$).

\medskip

Let $\pi: (Y_1, G)\rightarrow (Y_2, G)$ be a factor map between TDS's and $\mathcal{W}\in \mathbf{C}_{Y_1}, \nu_1\in \mathcal{P} (Y_1, G)$.
Observe that the sub-$\sigma$-algebra $\pi^{- 1} \mathcal{B}_{Y_2}\subseteq \mathcal{B}_{Y_1}$ is $G$-invariant, so
we may introduce
the {\it measure-theoretic
$\nu_1$-entropy of $\mathcal{W}$ relative to $\pi$} by
\begin{equation} \label{1208050105}
h_{\nu_1} (G, \mathcal{W}| \pi)= h_{\nu_1} (G, \mathcal{W}| \pi^{- 1} \mathcal{B}_{Y_2})= h_{\nu_1, +} (G, \mathcal{W}| \pi^{- 1} \mathcal{B}_{Y_2}).
\end{equation}
The second equality follows from Theorem \ref{1006272300}, since $(Y_1, \mathcal{B}_{Y_1}, \nu_1)$ is a Lebesgue space.
Finally, the \emph{measure-theoretic $\nu_1$-entropy of $(Y_1, G)$ relative to $\pi$} is defined by
\begin{equation*}
h_{\nu_1} (G, Y_1| \pi)= h_{\nu_1} (G, Y_1| \pi^{- 1} \mathcal{B}_{Y_2}).
\end{equation*}

Now assume that $\mathcal{W}\in \mathbf{C}_{Y_1}^o$. For each $y_2\in Y_2$ let $N (\mathcal{W}, \pi^{- 1} y_2)$ be the minimal cardinality of a sub-family of $\mathcal{W}$ covering $\pi^{- 1} (y_2)$ and put
$$N (\mathcal{W}| \pi)= \sup\limits_{y_2\in Y_2} N (\mathcal{W}, \pi^{- 1} y_2).$$
It is easy to see that
$$\log N (\mathcal{W}_\bullet| \pi): \mathcal{F}_G\rightarrow \R, F\mapsto \log N (\mathcal{W}_F| \pi)$$ is a monotone non-negative
$G$-invariant sub-additive function, and so by Proposition \ref{1006122129}
we may define the \emph{topological entropy of $\mathcal{W}$ relative to $\pi$} as
\begin{equation*}
h_{\text{top}} (G, \mathcal{W}| \pi)= \lim_{n\rightarrow \infty} \frac{1}{|F_n|} \log N (\mathcal{W}_{F_n}| \pi).
\end{equation*}
Lastly, the \emph{topological entropy of $(Y_1, G)$ relative to $\pi$} is defined by:
\begin{equation*}
h_{\text{top}} (G, Y_1| \pi)= \sup_{\mathcal{W}\in \mathbf{C}_{Y_1}^o} h_{\text{top}} (G, \mathcal{W}| \pi).
\end{equation*}

In fact, the concepts of monotonicity, sub-additivity and $G$-invariance can be introduced as above for functions in the space $C (Y_1)$ of all real-valued continuous functions on $Y_1$. Then, following \S \ref{fourth}, let $\mathbf{D}= \{d_F: F\in \mathcal{F}_G\}\subseteq C (Y_1)$ be a monotone sub-additive $G$-invariant family. For each $\nu_1\in \mathcal{P} (Y_1, G)$ and all $\mathcal{W}\in \mathbf{C}_{Y_1}^o, y_2\in Y_2, F\in \mathcal{F}_G$, we define
$$\nu_1 (\mathbf{D})= \lim_{n\rightarrow \infty} \frac{1}{|F_n|} \int_{Y_1} d_{F_n} (y_1) d \nu_1 (y_1),$$
$$P_\pi (y_2, \mathbf{D}, F, \mathcal{W})= \inf \left\{\sum_{A\in \alpha} \sup_{x\in A\cap \pi^{- 1} (y_2)} e^{d_F (x)}: \alpha\in
\mathbf{P}_{Y_1}, \alpha\succeq
\mathcal{W}_F\right\}.$$

With the above notation, we have the following easy observation.

\begin{prop} \label{1208042200}
Let $y_2\in Y_2, g\in G$ and $E, F\in \mathcal{F}_G$ with $E\cap F= \emptyset$. Then
$$P_\pi (y_2, \mathbf{D}, F g, \mathcal{W})= P_\pi (g y_2, \mathbf{D}, F, \mathcal{W}), \,\, \mathrm{and}$$
$$P_\pi (y_2, \mathbf{D}, E\cup F, \mathcal{W})\le P_\pi (y_2, \mathbf{D}, E, \mathcal{W})\cdot P_\pi (y_2, \mathbf{D}, F, \mathcal{W}).$$
\end{prop}

 By Proposition \ref{1208042200}, one readily deduces that
 $$\mathcal{F}_G\rightarrow \R, F\mapsto \sup_{y_2\in Y_2} \log P_\pi (y_2, \mathbf{D}, F, \mathcal{W})$$
 is a monotone non-negative $G$-invariant sub-additive function. Hence
we can define
$$P_\pi (\mathbf{D}, \mathcal{W})= \lim_{n\rightarrow \infty} \frac{1}{F_n} \sup_{y_2\in Y_2} \log P_\pi (y_2, \mathbf{D}, F_n, \mathcal{W}).$$
We may further define
$$P_\pi (\mathbf{D})= \sup_{\mathcal{U}\in \mathbf{C}_{Y_1}^o} P_\pi (\mathbf{D}, \mathcal{U}).$$

\medskip

Let $\pi: (Y_1, G)\rightarrow (Y_2, G)$ be a factor map between TDS's, $\nu_1\in \mathcal{P} (Y_1, G)$ and $(x_1, \cdots, x_n)\in Y_1^n\setminus \Delta_n (Y_1), n\in \mathbb{N}\setminus \{1\}$. $(x_1, \cdots, x_n)$ is called a:
\begin{enumerate}

\item \emph{relative topological entropy $n$-tuple relevant to $\pi$} if: For any $m\in \mathbb{N}$, there exists a closed neighborhood $V_i$ of $x_i$ with diameter at most $\frac{1}{m}$ for each $i= 1, \cdots, n$, such that $\mathcal{V}\doteq \{V_1^c, \cdots, V_n^c\}\in \mathbf{C}^o_{Y_1}$ and $h_{\text{top}} (G, \mathcal{V}| \pi)> 0$.

\item \emph{relative measure-theoretic $\nu_1$-entropy $n$-tuple relevant to $\pi$} if: For any $m\in \mathbb{N}$, there exists a closed neighborhood $V_i$ of $x_i$ with diameter at most $\frac{1}{m}$ for each $i= 1, \cdots, n$, such that $\mathcal{V}\doteq \{V_1^c, \cdots, V_n^c\}\in \mathbf{C}^o_{Y_1}$ and $h_{\nu_1} (G, \mathcal{V}| \pi)> 0$.
\end{enumerate}
Denote by $E_n (Y_1, G| \pi)$ and $E_n^{\nu_1} (Y_1, G| \pi)$ the set of all relative topological entropy $n$-tuples relevant to $\pi$ and relative measure-theoretic $\nu_1$-entropy $n$-tuples relevant to $\pi$, respectively.

Notice that, when the factor map is trivial in the sense that $Y_2$ is a singleton, these notions of entropy tuples in both settings cover the standard definitions for $\Z$-actions and more
generally for actions of an infinite countable discrete amenable group (cf \cite{B2, BHMMR, HY, HYZ2, HYZ}).

\subsection{Equivalence of a topological dynamical system with a particular continuous bundle random dynamical system}\label{1208042024}\

\medskip

In this subsection we show that the above definition of a topological dynamical system is a special case of a particular continuous bundle RDS.

\medskip

To do this,  suppose that $\pi: (Y_1, G)\rightarrow (Y_2, G)$ be a factor map of TDS's, let $\nu_2\in \mathcal{P} (Y_2, G), \mathcal{V}\in \mathbf{C}_{Y_1}$ and let $\mathbf{D}= \{d_F: F\in \mathcal{F}_G\}\subseteq C (Y_1)$ be a monotone sub-additive $G$-invariant family. For each $g\in G$ and for any $y_2\in Y_2$, set
$$F^\pi_{g, y_2}: \{y_2\}\times \pi^{- 1} (y_2)\rightarrow \{g y_2\}\times \pi^{- 1} (g y_2), (y_2, y_1)\mapsto (g y_2, g y_1)$$
 and
 $$\mathcal{E}_\pi= \{(y_2, y_1)\in Y_2\times Y_1: \pi (y_1)= y_2\}.$$
It is easy to see that $\mathcal{E}_\pi$ is a non-empty compact subset of $Y_2\times Y_1$, and $G$ acts naturally on $\mathcal{E}_\pi$. One checks that the family
 $$\mathbf{F}^\pi\doteq \{F^\pi_{g, y_2}: \{y_2\}\times \pi^{- 1} (y_2)\rightarrow \{g y_2\}\times \pi^{- 1} (g y_2)| g\in G, y_2\in Y_2\}$$
  forms a continuous bundle RDS over MDS $(Y_2, \mathcal{B}_{Y_2}, \nu_2, G)$ with $(Y_2, \mathcal{B}_{Y_2}, \nu_2)$ a Lebesgue space, and the family $\mathbf{D}$ may be viewed as a monotone sub-additive $G$-invariant family $\mathbf{D}^\pi= \{d_F^\pi: F\in \mathcal{F}_G\}\subseteq \mathbf{L}^1_{\mathcal{E}_\pi} (Y_2, C (Y_1))$ by the natural map
   \begin{equation*}
   d_F^\pi (y_2, y_1)= d_F (y_1)\ \text{for any}\ (y_2, y_1)\in \mathcal{E}_\pi.
   \end{equation*}

For each $V\in \mathcal{V}$, we can define
\begin{equation} \label{defn1}
V^\pi= \{(\pi y_1, y_1): y_1\in V\}= (Y_2\times V)\cap \mathcal{E}_\pi,
\end{equation}
and it follows immediately that
\begin{equation} \label{defn2}
  \mathcal{V}^\pi\doteq \{V^\pi: V\in \mathcal{V}\}\in \mathbf{C}_{\mathcal{E}_\pi}.
\end{equation}
In fact, if $\mathcal{V}\in \mathbf{C}_{Y_1}^o$ then it is simple to see that $\mathcal{V}^\pi\in \mathbf{C}_{\mathcal{E}_\pi}^o$.
 Henceforth, for the state space $(Y_2, \mathcal{B}_{Y_2}, \nu_2, G)$ we take
 \begin{equation*}
  _{\nu_2} h_{\text{top}}^{(r)} (\mathbf{F}^\pi), _{\nu_2} h_{\text{top}}^{(r)} (\mathbf{F}^\pi, \mathcal{V}^\pi), _{\nu_2} P_{\mathcal{E}_\pi} (\mathbf{D}^\pi, \mathbf{F}^\pi), _{\nu_2} P_{\mathcal{E}_\pi} (\mathbf{D}^\pi, \mathcal{V}^\pi, \mathbf{F}^\pi)
 \end{equation*}
  to be the fiber topological entropy of $\mathbf{F}^\pi$ (with respect to $\mathcal{V}^\pi$) and the fiber topological $\mathbf{D}^\pi$-pressure of $\mathbf{F}^\pi$ (with respect to $\mathcal{V}^\pi$), respectively.

 Moreover, observing that $\mathcal{E}_\pi$ is identical to $Y_1$ by the natural homeomorphism $(y_2, y_1)\mapsto y_1$, there is a natural one-to-one map between $\mathcal{P}_{\nu_2} (\mathcal{E}_\pi, G)$ and
 \begin{equation} \label{1208042247}
 \{\nu_1\in \mathcal{P} (Y_1, G): \pi \nu_1= \nu_2\}\ \text{(denoted by $\mathcal{P}_{\nu_2} (Y_1, G)$)},
 \end{equation}
 a non-empty compact subset of $\mathcal{P} (Y_1, G)$.

 Similarly, there is a natural one-to-one map between $\mathcal{P}_{\nu_2} (\mathcal{E}_\pi)$ and
 \begin{equation*}
 \{\nu_1\in \mathcal{P} (Y_1): \pi \nu_1= \nu_2\}\ \text{(denoted by $\mathcal{P}_{\nu_2} (Y_1)$)},
 \end{equation*}
 which extends the one-to-one map between $\mathcal{P}_{\nu_2} (\mathcal{E}_\pi, G)$ and $\mathcal{P}_{\nu_2} (Y_1, G)$.
In fact, let $\{\nu_1^n: n\in \mathbb{N}\}\subseteq \mathcal{P}_{\nu_2} (\mathcal{E}_\pi)$ and $\nu_1\in \mathcal{P}_{\nu_2} (\mathcal{E}_\pi)$, and assume that $\mu_1^n, n\in \mathbb{N}$ and $\mu_1$ is the natural correspondence of $\nu_1^n, n\in \mathbb{N}$ and $\nu_1$ in $\mathcal{P}_{\nu_2} (Y_1)$, respectively. Then it is not hard to check that the following statements are equivalent:
\begin{enumerate}

\item \label{1008281117} the sequence $\{\nu_1^n: n\in \mathbb{N}\}$ converges to $\nu_1$;

\item \label{1008281118} the sequence $\{\int_{Y_2\times Y_1} f d \nu_1^n: n\in \mathbb{N}\}$ converges to $\int_{Y_2\times Y_1} f d \nu_1$ for any $f\in C (Y_2\times Y_1)$;

\item \label{1008281119} the sequence $\{\int_{\mathcal{E}_\pi} f d \nu_1^n: n\in \mathbb{N}\}$ converges to $\int_{\mathcal{E}_\pi} f d \nu_1$ for any $f\in C (\mathcal{E}_\pi)$;

\item \label{1008281120} the sequence $\{\mu_1^n: n\in \mathbb{N}\}$ converges to $\mu_1$ in the sense of the weak* topology on $\mathcal{P} (Y_1)$, i.e. the sequence $\{\int_{Y_1} f d \mu_1^n: n\in \mathbb{N}\}$ converges to $\int_{Y_1} f d \mu_1$ for any $f\in C (Y_1)$.
\end{enumerate}
In fact, the equivalence of \eqref{1008281117} and \eqref{1008281118} follows from the ideas in the proof of \cite[Lemma 2.1]{K1}; note that $\mathcal{E}_\pi$ is a non-empty compact subset of the compact metric space $Y_2\times Y_1$, the equivalence of \eqref{1008281118} and \eqref{1008281119} is obvious; and note that $\mathcal{E}_\pi$ is identical to $Y_1$ by homeomorphism $(y_2, y_1)\mapsto y_1$, the equivalence of \eqref{1008281119} and \eqref{1008281120} is natural.

From the above arguments, as topological spaces, $\mathcal{P}_{\nu_2} (\mathcal{E}_\pi)$ is identical to $\mathcal{P}_{\nu_2} (Y_1)$ by the natural homeomorphism, which is also a homeomorphism from $\mathcal{P}_{\nu_2} (\mathcal{E}_\pi, G)$ onto $\mathcal{P}_{\nu_2} (Y_1, G)$.
It is not hard to check the following observations:
\begin{enumerate}

\item If $\mathcal{V}\in \mathbf{C}_{Y_1}^o$ then,
by the constructions \eqref{defn1} and \eqref{defn2} of $\mathcal{V}^\pi$, and using Theorem \ref{1007212202}, we see that $\mathcal{V}^\pi\in \mathbf{C}_{\mathcal{E}_\pi}^o$ is factor excellent.

\item For each $\nu_1\in \mathcal{P} (Y_1, G)$, $\nu_1 (\mathbf{D}^\pi)= \nu_1 (\mathbf{D})$
 and
\begin{equation} \label{123}
h_{\nu_1}^{(r)} (\mathbf{F}^\pi, \mathcal{V}^\pi)= h_{\nu_1} (G, \mathcal{V}| \pi)\ \text{for any}\ \mathcal{V}\in \mathbf{C}_{Y_1}.
\end{equation}
Hence by Theorem \ref{1006122212}, one has
$h_{\nu_1}^{(r)} (\mathbf{F}^\pi)= h_{\nu_1} (G, Y_1| \pi)$.

\item For each $\nu_2\in \mathcal{P} (Y_2, G)$,
$$_{\nu_2} h_{\text{top}}^{(r)} (\mathbf{F}^\pi, \mathcal{V}^\pi)= \lim_{n\rightarrow \infty} \frac{1}{|F_n|} \int_{Y_2} \log N (\mathcal{V}_{F_n}, \pi^{- 1} (y_2)) d \nu_2 (y_2),$$
\begin{equation} \label{1234}
_{\nu_2} P_{\mathcal{E}_\pi} (\mathbf{D}^\pi, \mathcal{V}^\pi, \mathbf{F}^\pi)= \lim_{n\rightarrow \infty} \frac{1}{|F_n|} \int_{Y_2} \log P_\pi (y_2, \mathbf{D}, F_n, \mathcal{V}) d \nu_2 (y_2).
\end{equation}

\item For each $\nu_1\in \mathcal{P} (Y_1, G)$, $E_n^{\nu_1} (Y_1, G| \pi)$ $= E_{n, \nu_1}^{(r)} (\mathcal{E}_\pi, G)$ for any $n\in \N\setminus \{1\}$.
\end{enumerate}
The straightforward details of checking these observations are left to the reader.

\medskip

\subsection{The equations \eqref{1207291438} and \eqref{1207291439} imply main results of \cite{LW}}\label{1208042011} \

\medskip

In this subsection, we show how to obtain the main results of \cite{LW}, using equations \eqref{1207291438} and \eqref{1207291439} and the equivalence given by \S\S \ref{1208042024}.

\medskip

Let $\pi: (Y_1, G)\rightarrow (Y_2, G)$ be a factor map between TDS's and $\nu_2\in \mathcal{P} (Y_2, G), \mathcal{V}\in \mathbf{C}_{Y_1}, f\in C (Y_1)$.
As shown in \S\S \ref{1208042024}, we may view  $\pi: (Y_1, G)\rightarrow (Y_2, G)$ and $\nu_2, \mathcal{V}, f$ as a continuous bundle RDS with:
\begin{enumerate}

\item \label{1111} $(\Omega, \mathcal{F}, \mathbb{P}, G)= (Y_2, \mathcal{B}_{Y_2}^{\nu_2}, \nu_2, G)$, where $\mathcal{B}_{Y_2}$ is the Borel $\sigma$-algebra of $Y_2$ and $\mathcal{B}_{Y_2}^{\nu_2}$ is the $\nu_2$-completion of $\mathcal{B}_{Y_2}$,

\item \label{2222} $\mathcal{E}= \{(y_2, y_1)\in Y_2\times Y_1: \pi (y_1)= y_2\}\in \mathcal{B}_{Y_2}^{\nu_2}\times \mathcal{B}_{Y_1}$,

\item \label{3333}
 $\mathbf{F}= \{F_{g, y_2}: \{y_2\}\times \pi^{- 1} (y_2)\rightarrow \{g y_2\}\times \pi^{- 1} (g y_2)| g\in G, y_2\in Y_2\}$, where
 $F_{g, y_2}: (y_2, y_1)\mapsto (g y_2, g y_1)$ for each $g\in G$ and any $y_2\in Y_2, y_1\in \pi^{- 1} (y_2)$,

 \item \label{4444} $\mathcal{U}= \{(Y_2\times V)\cap \mathcal{E}: V\in \mathcal{V}\}\in \mathbf{C}_\mathcal{E}^o$ is factor excellent, and

\item \label{5555} $\mathbf{D}= \{d_F: F\in \mathcal{F}_G\}$, where $d_F (y_2, y_1)= \sum\limits_{g\in F} f (g y_1)$ for each $(y_2, y_1)\in \mathcal{E}$.
\end{enumerate}
It is easy to check that, $\mathbf{D}$ is a sub-additive $G$-invariant family satisfying $(\spadesuit)$, and if we take
 $C= \max\limits_{y_1\in Y_1} |f (y_1)|\in \R_+$ then $\mathbf{D}'= \{d_F': F\in \mathcal{F}_G\}$ is a monotone sub-additive $G$-invariant family satisfying $(\spadesuit)$, where $d_F'= d_F+ |F| C$ for each $F\in \mathcal{F}_G$. Thus equations \eqref{1207291438} and \eqref{1207291439} hold for this continuous bundle RDS.

\medskip

Let us first show: the equation \eqref{1207291439} can be used to obtain \cite[Proposition 3.5]{LW} in the setting of a topological dynamical system of an amenable group action.

 Using \eqref{1111}, \eqref{2222}, \eqref{3333}, \eqref{4444} and \eqref{5555}, one has
  \begin{eqnarray}
& & \sup_{\mu\in \mathcal{P}_\mathbb{P} (\mathcal{E}, G)} [h_\mu^{(r)} (\mathbf{F})+ \mu (\mathbf{D})]\nonumber \\
& & \hskip 26pt = \sup_{\mathcal{V}\in \mathbf{C}_{Y_1}^o} \lim_{n\rightarrow \infty} \frac{1}{|F_n|} \int_\Omega
\log P_\mathcal{E} (\omega, \mathbf{D}, F_n, (\Omega\times \mathcal{V})_\mathcal{E}, \mathbf{F}) d \mathbb{P} (\omega)\ (\text{using \eqref{1207291439}})\label{1208071703} \\
& & \hskip 26pt \le \sup_{\mathcal{V}\in \mathbf{C}_{Y_1}^o} \int_\Omega
\limsup_{n\rightarrow \infty} \frac{1}{|F_n|} \log P_\mathcal{E} (\omega, \mathbf{D}, F_n, (\Omega\times \mathcal{V})_\mathcal{E}, \mathbf{F}) d \mathbb{P} (\omega)\label{1208071702} \\
& & \hskip 26pt \le \int_\Omega
\sup_{\mathcal{V}\in \mathbf{C}_{Y_1}^o} \limsup_{n\rightarrow \infty} \frac{1}{|F_n|} \log P_\mathcal{E} (\omega, \mathbf{D}, F_n, (\Omega\times \mathcal{V})_\mathcal{E}, \mathbf{F}) d \mathbb{P} (\omega). \label{1208041636}
\end{eqnarray}
For each $\mathcal{V}\in \mathbf{C}_{Y_1}^o$, the function in \eqref{1208071702} is bounded above by $C+ \log |\mathcal{V}|$ , and thus we obtain \eqref{1208071702} from \eqref{1208071703} by the Fatou Lemma.

Furthermore, the function in \eqref{1208041636} is bounded by $- C$ from below and is measurable: thus the integral in \eqref{1208041636} is well defined. The argument is standard.
  \begin{enumerate}

  \item[(6)] If $\{\mathcal{V}_m: m\in \N\}\subseteq \mathbf{C}_{Y_1}^o$ is a sequence of sets whose diameters tend to zero, then
  \begin{eqnarray*}
& & \sup_{\mathcal{V}\in \mathbf{C}_{Y_1}^o} \limsup_{n\rightarrow \infty} \frac{1}{|F_n|} \log P_\mathcal{E} (\omega, \mathbf{D}, F_n, (\Omega\times \mathcal{V})_\mathcal{E}, \mathbf{F}) \\
& & \hskip 26pt = \sup_{m\in \N} \limsup_{n\rightarrow \infty} \frac{1}{|F_n|} \log P_\mathcal{E} (\omega, \mathbf{D}, F_n, (\Omega\times \mathcal{V}_m)_\mathcal{E}, \mathbf{F}).
  \end{eqnarray*}

  \item[(7)] Now let $\mathcal{V}\in \mathbf{C}_{Y_1}^o$. By the measurability of $\log P_\mathcal{E} (\omega, \mathbf{D}, F_n, (\Omega\times \mathcal{V})_\mathcal{E}, \mathbf{F})$ for each $n\in \N$, one easily deduces the measurability of $$\limsup\limits_{n\rightarrow \infty} \frac{1}{|F_n|} \log P_\mathcal{E} (\omega, \mathbf{D}, F_n, (\Omega\times \mathcal{V})_\mathcal{E}, \mathbf{F}).$$
  \end{enumerate}

For each $y_2\in Y_2$ denote by $P (f, \pi^{- 1} (y_2))$ the topological pressure of $f$ on $\pi^{- 1} (y_2)$, introduced in \cite{LW} in the case of $G= \Z$: in fact,  the definition of $P (f, \pi^{- 1} (y_2))$ in \cite{LW} works for an infinite countable discrete amenable group $G$.

 Recall \eqref{1208042247} that $\mathcal{P}_{\nu_2} (Y_1, G)= \{\nu_1\in \mathcal{P} (Y_1, G): \pi \nu_1= \nu_2\}$. Thus \eqref{1208041636} is an equivalent statement of the following inequality, using the above notation for a continuous bundle RDS:
\begin{equation} \label{1207290049}
\int_{Y_2} P (f, \pi^{- 1} (y_2)) d \nu_2 (y_2)\ge \sup_{\nu_1\in \mathcal{P}_{\nu_2} (Y_1, G)} [h_{\nu_1} (G, Y_1| \pi)+ \int_{Y_1} f (y_1) d \nu_1 (y_1)].
\end{equation}
In the special case
of $G= \Z$, \eqref{1207290049} is exactly \cite[Proposition 3.5]{LW}.

\medskip

Now, using similar arguments as above, we show that equation \eqref{1207291438} can be used to obtain \cite[Theorem 2.1]{LW}, the main result of \cite{LW}.

In the above setting we may use \eqref{1207291438} to see
  \begin{equation*}
\lim_{n\rightarrow \infty} \frac{1}{|F_n|} \int_\Omega
\log P_\mathcal{E} (\omega, \mathbf{D}, F_n, \mathcal{U}, \mathbf{F}) d \mathbb{P} (\omega)= \max_{\mu\in \mathcal{P}_\mathbb{P} (\mathcal{E}, G)} [h_\mu^{(r)} (\mathbf{F}, \mathcal{U})+ \mu (\mathbf{D})],
\end{equation*}
which is equivalent to the following equation:
\begin{eqnarray} \label{1208042322}
& & \lim_{n\rightarrow \infty} \frac{1}{|F_n|} \int_{Y_2} \log P_\pi (y_2, \mathbf{D}, F_n, \mathcal{V}) d \nu_2 (y_2)\nonumber \\
& & \hskip 26 pt = \max_{\nu_1\in \mathcal{P}_{\nu_2} (Y_1, G)} [h_{\nu_1} (G, \mathcal{V}| \pi)+ \int_{Y_1} f (y_1) d \nu_1 (y_1)],
\end{eqnarray}
where we use the notation of \eqref{123} and \eqref{1234}.

In order to deduce \cite[Theorem 2.1]{LW}, we restrict our setting to $G= \Z$ and $F_n= \{0, 1, \cdots, n- 1\}$ for each $n\in \N$.
In addition we assume that the action of $\Z$ on $Y_2$ is $\{S_2^m: m\in \Z\}$, where $S_2: Y_2\rightarrow Y_2$ is a homeomorphism.

Now, for any $y_2\in Y_2$ and for each $n\in \N$, let
$$l_{n, f, \mathcal{V}} (y_2)= \log P_\pi (y_2, \mathbf{D}, F_n, \mathcal{V}).$$
Then \eqref{1208042322} may be reformulated as:
\begin{eqnarray} \label{1208050014}
& & \lim_{n\rightarrow \infty} \frac{1}{n} \int_{Y_2} l_{n, f, \mathcal{V}} (y_2) d \nu_2 (y_2)\nonumber \\
& & \hskip 26pt = \max_{\nu_1\in \mathcal{P}_{\nu_2} (Y_1, G)} [h_{\nu_1} (G, \mathcal{V}| \pi)+ \int_{Y_1} f (y_1) d \nu_1 (y_1)].
\end{eqnarray}
 Here though, $\mathbf{D}$, given by \eqref{5555}, need not be a monotone sub-additive $G$-invariant family, it is easy to see that Proposition \ref{1208042200} still holds for such $\mathbf{D}$. In particular,
$$l_{n+ m, f, \mathcal{V}} (y_2)\le l_{n, f, \mathcal{V}} (y_2)+ l_{m, f, \mathcal{V}} (S_2^n y_2)\ \text{for each $n, m\in \N$ and any $y_2\in Y_2$}.$$
Reminder $\nu_2\in \mathcal{P} (Y_2, G)$. By the Kingman Sub-additive Ergodic Theorem (cf \cite{Kingman} or \cite[Theorem 10.1]{W}) one has:
\begin{equation} \label{1208050003}
\text{for $\nu_2$-a.e.}\ y_2\in Y_2, \lim_{n\rightarrow \infty} \frac{1}{n} l_{n, f, \mathcal{V}} (y_2)\ \text{exists (denoted by $p_{f, \mathcal{V}} (y_2)$)}.
\end{equation}
In particular,
\begin{equation} \label{1208051046}
\text{for $\nu_2$-a.e.}\ y_2\in Y_2, p_{f, \mathcal{V}} (y_2)= \limsup_{n\rightarrow \infty} \frac{1}{n} l_{n, f, \mathcal{V}} (y_2).
\end{equation}
In fact, since $C= \max\limits_{y_1\in Y_1} |f (y_1)|\in \R_+$, it is easy to show
$$- n C\le l_{n, f, \mathcal{V}} (y_2)\le n (C+ \log |\mathcal{V}|);$$
and then, by \eqref{1208050003}, we apply the Bounded Convergence Theorem to \eqref{1208050014} to obtain
\begin{equation} \label{1208050020}
\int_{Y_2} p_{f, \mathcal{V}} (y_2) d \nu_2 (y_2)= \max_{\nu_1\in \mathcal{P}_{\nu_2} (Y_1, G)} [h_{\nu_1} (G, \mathcal{V}| \pi)+ \int_{Y_1} f (y_1) d \nu_1 (y_1)].
\end{equation}
Moreover, as in (6), (7) and \eqref{1208050020} we deduce
\begin{equation} \label{1208050019}
\int_{Y_2} \sup_{\mathcal{V}\in \mathbf{C}_{Y_1}^o} p_{f, \mathcal{V}} (y_2) d \nu_2 (y_2)= \sup_{\nu_1\in \mathcal{P}_{\nu_2} (Y_1, G)} [h_{\nu_1} (G, Y_1| \pi)+ \int_{Y_1} f (y_1) d \nu_1 (y_1)].
\end{equation}
Equation \eqref{1208050019} is now just \cite[Theorem 2.1]{LW}; and, when $f$ is the constant zero function, equation \eqref{1208050020} is exactly \cite[Theorem 4.2.15]{Z-Thesis}.

\medskip

\subsection{Local variational principles for a topological dynamical system}\label{inner}\

\medskip

With the equivalence given by \S\S \ref{1208042024}, in previous subsection we have shown how to apply our results on continuous bundle RDS's to obtain results on general topological dynamical systems. In this subsection, we give a number of further applications.

\medskip

First, using preparations made in \S\S \ref{1208042310} and \S\S \ref{1208042024}, we have the following equivalent statement of \eqref{1207281620} in the setting of giving a factor map between TDS's. This is useful in building symbolic extension theory for amenable group actions \cite{DownZ}.

\begin{thm} \label{1010201200}
Let $\pi: (Y_1, G)\rightarrow (Y_2, G)$ be a factor map between TDS's and $\mathcal{V}\in \mathbf{C}^o_{Y_1}, \nu_2\in \mathcal{P} (Y_2, G)$. Then
\begin{equation*}
\lim_{n\rightarrow \infty} \frac{1}{|F_n|} \int_{Y_2} \log N (\mathcal{V}_{F_n}, \pi^{- 1} (y_2)) d \nu_2 (y_2)= \max_{\nu_1\in \mathcal{P}_{\nu_2} (Y_1, G)} h_{\nu_1} (G, \mathcal{V}| \pi).
\end{equation*}
\end{thm}

In fact, in the setting of a topological dynamical system of amenable group actions and a finite open cover of the space, Theorem \ref{1010201200} generalizes the Inner Variational Principle \cite[Theorem 4]{DS} in the following sense.

Let $\pi: (Y_1, G)\rightarrow (Y_2, G)$ be a factor map between TDS's and $\nu_2\in \mathcal{P} (Y_2, G)$. In the setting of $G= \Z$ and $F_n= \{0, 1, \cdots, n- 1\}$ for each $n\in \N$, Downarowicz and Serafin proved the Inner Variational Principle \cite[Theorem 4]{DS}, which may be stated equivalently as (cf \cite[Definition 5, Definition 7, Definition 8 and Theorem 4]{DS}):
\begin{eqnarray} \label{1208041924}
& & \sup_{\nu_1\in \mathcal{P}_{\nu_2} (Y_1, G)} h_{\nu_1} (G, Y_1| \pi)\nonumber \\
& & \hskip 26pt =
\sup_{\mathcal{U}\in \mathbf{C}_{Y_1}^o} \lim_{n\rightarrow \infty} \frac{1}{|F_n|} \int_{Y_2} \log N (\mathcal{U}_{F_n}, \pi^{- 1} (y_2)) d \nu_2 (y_2).
\end{eqnarray}
By Theorem \ref{1010201200} one sees that \eqref{1208041924} holds for any infinite countable discrete amenable group $G$ and any F\o lner sequence $\{F_n: n\in \mathbb{N}\}$.

\medskip

In the next subsection, we will need the following result.

\begin{thm} \label{1008301614}
Let $\pi: (Y_1, G)\rightarrow (Y_2, G)$ be a factor map between TDS's and $\mathcal{V}\in \mathbf{C}^o_{Y_1}$. Assume that $\mathbf{D}= \{d_F: F\in \mathcal{F}_G\}\subseteq C (Y_1)$ is a monotone sub-additive $G$-invariant family satisfying:

\medskip

\begin{center}
{\bf $(\heartsuit)$}
\begin{minipage}[]{110mm}
\emph{for any given sequence $\{\nu_n: n\in \mathbb{N}\}\subseteq \mathcal{P} (Y_1)$, set $\mu_n= \frac{1}{|F_n|} \sum\limits_{g\in F_n} g \nu_n$ for each $n\in \mathbb{N}$. There exists a subsequence $\{n_j: j\in \mathbb{N}\}\subseteq \mathbb{N}$ such that the sequence $\{\mu_{n_j}: j\in \mathbb{N}\}$ converges to  $\mu\in \mathcal{P} (Y_1)$ (and hence a fortiori $\mu\in \mathcal{P} (Y_1, G)$) such that
\begin{equation*}
\limsup_{j\rightarrow \infty} \frac{1}{|F_{n_j}|} \int_{Y_1} d_{F_{n_j}} (y_1) d \nu_{n_j} (y_1)\le \mu (\mathbf{D}).
\end{equation*}
}
\end{minipage}
\end{center}

\medskip

\noindent
 Then
\begin{equation*}
P_\pi (\mathbf{D}, \mathcal{V})= \max_{\nu_2\in \mathcal{P} (Y_2, G)}\ _{\nu_2} P_{\mathcal{E}_\pi} (\mathbf{D}^\pi, \mathcal{V}^\pi, \mathbf{F}^\pi)= \max_{\nu_1\in \mathcal{P} (Y_1, G)} [h_{\nu_1} (G, \mathcal{V}| \pi)+ \nu_1 (\mathbf{D})].
\end{equation*}
In particular,
\begin{equation} \label{1208051133}
h_{\text{top}} (G, \mathcal{V}| \pi)= \max_{\nu_2\in \mathcal{P} (Y_2, G)}\ _{\nu_2} h_{\text{top}}^{(r)} (\mathbf{F}^\pi, \mathcal{V}^\pi)= \max_{\nu_1\in \mathcal{P} (Y_1, G)} h_{\nu_1} (G, \mathcal{V}| \pi).
\end{equation}
Moreover,
\begin{equation*}
P_\pi (\mathbf{D})= \sup_{\nu_1\in \mathcal{P} (Y_1, G)} [h_{\nu_1} (G, Y_1| \pi)+ \nu_1 (\mathbf{D})],
\end{equation*}
\begin{equation*} \label{1208051134}
h_{\text{top}} (G, Y_1| \pi)= \sup_{\nu_1\in \mathcal{P} (Y_1, G)} h_{\nu_1} (G, Y_1| \pi).
\end{equation*}
\end{thm}
\begin{proof}
The proof follows the ideas of the proof of Theorem \ref{1007141414}.

As the process is similar to that in \S \ref{seventh}, we shall skip some details.

As $\mathbf{D}$ satisfies $(\heartsuit)$, it is not hard to check that $\mathbf{D}^\pi$ satisfies $(\spadesuit)$. Observe that $\mathcal{V}^\pi\in \mathbf{C}_{\mathcal{E}_\pi}^o$ is factor excellent. It follows that for each $\nu_2\in \mathcal{P} (Y_2, G)$, we can apply Theorem \ref{1007141414} to $\mathbf{F}^\pi, \mathcal{V}^\pi, \mathbf{D}^\pi$ and $(Y_2, \mathcal{B}_{Y_2}, \nu_2)$.

Thus, to complete our proof, we need only prove:
\begin{equation} \label{1010201909}
h_{\nu_1} (G, \mathcal{V}| \pi)+ \nu_1 (\mathbf{D})\ge P_\pi (\mathbf{D}, \mathcal{V})\ \text{for some $\nu_1\in \mathcal{P} (Y_1, G)$}.
\end{equation}

First, we assume that the compact metric space $Y_1$ is zero-dimensional. By Lemma \ref{1007242349}, the family $\mathbf{P}_c (\mathcal{V})$ is countable, and we let $\{\alpha_l: l\in \mathbb{N}\}$ denote an enumeration of this family. Then each $\alpha_l, l\in \mathbb{N}$ is finer than $\mathcal{V}$, and
\begin{equation}
   \label{add}
 h_{\nu_1} (G, \mathcal{V}| \pi)= \inf\limits_{l\in \mathbb{N}} h_{\nu_1} (G, \alpha_l| \pi)\ \text{for each $\nu_1\in \mathcal{P} (Y_1, G)$}\ (\text{by \eqref{1208050105}}).
\end{equation}

Let $n\in \mathbb{N}$ be fixed. Using the reasoning of Lemma \ref{1007141419}, one sees that
 there exist $x_n\in Y_2$ and a non-empty finite subset $B_n\subseteq \pi^{- 1} (x_n)$ such that
\begin{equation} \label{1010202058}
\sum_{y\in B_n} e^{d_{F_n} (y)}\ge \frac{1}{n} \left[\sup_{y_2\in Y_2} P_\pi (y_2, \mathbf{D}, F_n, \mathcal{V})- M\right]
\end{equation}
with
$$M= \frac{1}{2} e^{- \max\limits_{y_1\in Y_1} |d_{F_n} (y_1)|},$$
and each atom of $(\alpha_l)_{F_n}, l= 1, \cdots, n$ contains at most one point of $B_n$.

 Now let
\begin{equation} \label{1010202126}
\nu_n= \sum_{y\in B_n} \frac{e^{d_{F_n} (y)} \delta_y}{\sum\limits_{x\in B_n} e^{d_{F_n} (x)}}\in \mathcal{P} (Y_1)\ \text{and}\ \mu_n= \frac{1}{|F_n|} \sum_{g\in F_n} g \nu_n\in \mathcal{P} (Y_1).
\end{equation}
By $(\heartsuit)$, we can choose a subsequence $\{n_j: j\in \mathbb{N}\}\subseteq \mathbb{N}$ such that the sequence $\{\mu_{n_j}: j\in \mathbb{N}\}$ converges to $\mu\in \mathcal{P} (Y_1, G)$ and
\begin{equation} \label{1010202127}
\limsup_{j\rightarrow \infty} \frac{1}{|F_{n_j}|} \int_{Y_1} d_{F_{n_j}} (y_1) d \nu_{n_j} (y_1)\le \mu (\mathbf{D}).
\end{equation}

Now fix any $l\in \mathbb{N}$ and let $n> l$. By the construction of $B_n$ and $\nu_n$ one has
\begin{equation} \label{1010202128}
H_{\nu_n} ((\alpha_l)_{F_n}| \pi)= H_{\nu_n} ((\alpha_l)_{F_n})= - \sum_{y\in B_n} \frac{e^{d_{F_n} (y)}}{\sum\limits_{x\in B_n} e^{d_{F_n} (x)}} \log \frac{e^{d_{F_n} (y)}}{\sum\limits_{x\in B_n} e^{d_{F_n} (x)}},
\end{equation}
and so
\begin{eqnarray} \label{1010202153}
& &\hskip -26pt \log \sup_{y_2\in Y_2} P_\pi (y_2, \mathbf{D}, F_n, \mathcal{V})- \log (2 n)\nonumber \\
&\le & \log \left[\sup_{y_2\in Y_2} P_\pi (y_2, \mathbf{D}, F_n, \mathcal{V})- M\right]- \log n\nonumber \\
&\le & \log \sum_{y\in B_n} e^{d_{F_n} (y)}\ (\text{using \eqref{1010202058}})\nonumber \\
&= & H_{\nu_n} ((\alpha_l)_{F_n}| \pi)+ \sum_{y\in B_n} \frac{e^{d_{F_n} (y)} d_{F_n} (y)}{\sum\limits_{x\in B_n} e^{d_{F_n} (x)}}\ (\text{using \eqref{1010202128}})\nonumber \\
&= & H_{\nu_n} ((\alpha_l)_{F_n}| \pi)+ \int_{Y_1} d_{F_n} (y_1) d \nu_{n} (y_1)\ (\text{using \eqref{1010202126}}).
\end{eqnarray}
By Lemma \ref{1007221736} and Lemma \ref{1007222138} one has, for each $B\in \mathcal{F}_G$,
\begin{eqnarray} \label{1010202149}
& & H_{\nu_n} ((\alpha_l)_{F_n}| \pi)\nonumber \\
& & \hskip 26pt \le \sum_{g\in F_n} \frac{1}{|B|} H_{\nu_n} ((\alpha_l)_{B g}| \pi)+ |F_n\setminus \{g\in G: B^{- 1} g\subseteq F_n\}|\cdot \log |\alpha_l|\nonumber \\
& & \hskip 26pt = \sum_{g\in F_n} \frac{1}{|B|} H_{g \nu_n} ((\alpha_l)_B| \pi)+ |F_n\setminus \{g\in G: B^{- 1} g\subseteq F_n\}|\cdot \log |\alpha_l|\nonumber \\
& & \hskip 26pt \le |F_n| \frac{1}{|B|} H_{\mu_n} ((\alpha_l)_B| \pi)+ |F_n\setminus \{g\in G: B^{- 1} g\subseteq F_n\}|\cdot \log |\alpha_l|.
\end{eqnarray}
 Observe that the partition $\alpha_l$ is clopen and $|F_n|\ge n$ for each $n\in \mathbb{N}$ by Standard Assumption 2. Combining \eqref{1010202149} with \eqref{1010202127} and \eqref{1010202153} we obtain
\begin{equation*}
P_\pi (\mathbf{D}, \mathcal{V})\le \frac{1}{|B|} H_\mu ((\alpha_l)_B| \pi)+ \mu (\mathbf{D}).
\end{equation*}
Now, taking the infimum over all $B\in \mathcal{F}_G$ and using \eqref{1006272232}, we obtain
$$P_\pi (\mathbf{D}, \mathcal{V})\le h_\mu (G, \alpha_l| \pi)+ \mu (\mathbf{D}).$$
Finally, letting $l$ range over $\mathbb{N}$ and using \eqref{add}, we obtain \eqref{1010201909}.

Now consider the general case. Note that there exists a factor map $\phi: (X, G)\rightarrow (Y_1, G)$ between TDS's, where $X$ is a zero-dimensional space (cf the proof of Proposition \ref{1007211043} or \cite[Proof of Theorem 1]{BGH}). By the above discussions, there exists $\nu\in \mathcal{P} (X, G)$ such that
$$h_\nu (G, \phi^{- 1} \mathcal{V}| \pi\circ \phi)+ \nu (\mathbf{D}\circ \phi)\ge P_{\pi\circ \phi} (\mathbf{D}\circ \phi, \phi^{- 1} \mathcal{V}),$$
where the family $\mathbf{D}\circ \phi$ is defined naturally. Set $\eta= \phi \nu$. It is not hard to check that $\eta\in \mathcal{P} (Y_1, G)$ and $h_\eta (G, \mathcal{V}| \pi)+ \eta (\mathbf{D})\ge P_\pi (\mathbf{D}, \mathcal{V})$. This proves \eqref{1010201909} in the general case, which ends our proof.
\end{proof}

Remark that, when $G= \Z$, \eqref{1208051133} is exactly \cite[Theorem 2.5]{HYZ1}, the main result of \cite{HYZ1} by Huang, Ye and the second author of the paper.

We also remark that, almost all comments about Theorem \ref{1007141414} in Part \ref{skdj} work similarly for Theorem \ref{1008301614}.
Here, we mention only some of them.

Firstly, the discussion just before Remark \ref{1103031512} works for Theorem \ref{1008301614}. In particular, for $f\in C (Y_1)$, we can apply Theorem \ref{1008301614} to the family $\mathbf{D}= \{d_F: F\in \mathcal{F}_G\}$, where $d_F (y_1)= \sum\limits_{g\in F} f (g y_1)$ for each $y_1\in Y_1$, although $\mathbf{D}$ is not monotone unless the function $f$ is non-negative.

The second example is that, we can change assumption $(\heartsuit)$ as we did in \S \ref{assumption}.
Furthermore, following the ideas in \S \ref{special}, if we assume that $G$ admits a tiling F\o lner sequence, then we can alter Theorem \ref{1008301614} to deal with any sub-additive $G$-invariant family $\mathbf{D}\subseteq C (Y_1)$ satisfying $(\heartsuit)$: if, the group $G$ is abelian then each sub-additive $G$-invariant family $\mathbf{D}\subseteq C (Y_1)$ automatically satisfies $(\heartsuit)$.

\medskip

\subsection{Entropy tuples of a topological dynamical system}\

\medskip

In this subsection, we discuss the relative entropy tuples introduced in \S\S \ref{1208042310} with the equivalence given by \S\S \ref{1208042024}.

\medskip

Let $X_1, X_2$ be topological spaces. Recall that the map $\pi: X_1\rightarrow X_2$ is \emph{open} if $\pi (U)$ is an open subset of $X_2$ whenever $U$ is an open subset of $X_1$.

From the definitions, it is not hard to obtain:

\begin{prop} \label{1103271148}
Let $\pi: (Y_1, G)\rightarrow (Y_2, G)$ be a factor map between TDS's, $\nu_2\in \mathcal{P}( Y_2, G)$ and $n\in \N\setminus \{1\}$. Then
\begin{eqnarray} \label{1103271149}
& & _{\nu_2} E_n^{(r)} (\mathcal{E}_\pi) \subseteq \{(x_1, \cdots, x_n)\in Y_1^n\setminus \Delta_n (Y_1):\nonumber \\
 & & \hskip 66pt \pi (x_1)= \cdots= \pi (x_n)\in \text{supp} (\nu_2)\}.
\end{eqnarray}
If, additionally, $\pi$ is open, then equality holds.
\end{prop}
\begin{proof}
We first establish \eqref{1103271149}. Let $(x_1, \cdots, x_n)\in _{\nu_2} E_n^{(r)} (\mathcal{E}_\pi)$. By the definition, for each $m\in \N$ there exist $y_2^m\in Y_2$ and $(x_1^m, \cdots, x_n^m)\in Y_1^n$, such that $(y_2^m, x_i^m)\in \mathcal{E}_\pi$ and the distance between $x_i^m$ and $x_i$ is at most $\frac{1}{m}$ for each $i= 1, \cdots, n$. Without loss of generality (by selecting a subsequence if necessary), we may assume that the sequence $\{y_2^m: m\in \N\}$ converges to $y_2\in Y$, and so it is easy to check $\pi (x_1)= \cdots= \pi (x_n)= y_2$. Now we prove \eqref{1103271149} showing that $y_2\in \text{supp} (\nu_2)$. Assume the contrary, i.e. that $y_2\notin \text{supp} (\nu_2)$. If $m\in \N$ is large enough, and if $V_i$ is a closed neighborhood of $x_i$ with diameter at most $\frac{1}{m}$ for each $i= 1, \cdots, n$ such that $\mathcal{V}= \{V_1^c, \cdots, V_n^c\}\in \mathbf{C}_{Y_1}^o$, then $\bigcup\limits_{i= 1}^n V_i\subseteq \pi^{- 1} (Y_2\setminus \text{supp} (\nu_2))$. Hence because
$$\{y\in Y_2: \prod_{i= 1}^n \{y\}\times V_i\cap \mathcal{E}_\pi^n\neq \emptyset\}= \bigcap_{i= 1}^n \pi (V_i)\subseteq Y_2\setminus \text{supp} (\nu_2),$$
a contradiction to $(x_1, \cdots, x_n)\in _{\nu_2} E_n^{(r)} (\mathcal{E}_\pi)$, as $\nu_2 (Y_2\setminus \text{supp} (\nu_2))= 0$.

Now we assume that $\pi$ is open. Let $(x_1, \cdots, x_n)\in Y_1^n\setminus \Delta_n (Y_1)$ such that $\pi (x_1)= \cdots= \pi (x_n)\in \text{supp} (\nu_2)$. Observe that, once $V_i$ is a closed neighborhood of $x_i$ for each $i= 1, \cdots, n$, then $\bigcap\limits_{i= 1}^n \pi (V_i)$ is a closed neighborhood of $\pi (x_1)$ (using the openness of $\pi$), which implies $\nu_2 (\bigcap\limits_{i= 1}^n \pi (V_i))> 0$ (as $\pi (x_1)\in \text{supp} (\nu_2)$), and so, by
$$\{y\in Y_2: \prod_{i= 1}^n \{y\}\times V_i\cap \mathcal{E}_\pi^n\neq \emptyset\}= \bigcap_{i= 1}^n \pi (V_i),$$
 one has $(x_1, \cdots, x_n)\in _{\nu_2} E_n^{(r)} (\mathcal{E}_\pi)$. This finishes the proof.
\end{proof}

Let $(Y, G)$ be a TDS. Denote by $\text{supp} (Y, G)$, the \emph{support of $(Y, G)$}, i.e. the set $\bigcup\limits_{\mu\in \mathcal{P} (Y, G)} \text{supp} (\mu)$. Observe that $\text{supp} (Y, G)= \text{supp} (\nu)$ for some $\nu\in \mathcal{P} (Y, G)$.

Combining Proposition \ref{1103271148} with Proposition \ref{1008071740}, Proposition \ref{1102131645}, Proposition \ref{1008060011} and Theorem \ref{conclusion}, and using the equivalence given by \S\S \ref{1208042024} as we did in \S \S \ref{1208042011}, it is not hard to establish:

\begin{prop} \label{1102131617}
Let $\pi: (Y_1, G)\rightarrow (Y_2, G)$ be a factor map between TDS's and $\mu\in \mathcal{P} (Y_1, G), n\in \mathbb{N}\setminus \{1\}$. Then
\begin{enumerate}

\item Both $E_n (Y_1, G| \pi)\cup \Delta_n (Y_1)$ and $E_n^\mu (Y_1, G| \pi)\cup \Delta_n (Y_1)$ are closed $G$-invariant subsets of $Y_1^n$.


\item $E_n (Y_1, G| \pi)\neq \emptyset$ if and only if $h_\text{top} (G, Y_1| \pi)> 0$.

\item $E_n^\mu (Y_1, G| \pi)\neq \emptyset$ if and only if $h_\mu (G, Y_1| \pi)> 0$.

\item $E_n (Y_1, G| \pi)\subseteq \{(x_1, \cdots, x_n)\in \text{supp} (Y_1, G)^n: \pi (x_1)= \cdots= \pi (x_n)\}$.

\item $E_n^\mu (Y_1, G| \pi)= \text{supp} (\lambda_n^{\pi^{- 1} \mathcal{B}_{Y_2}} (\mu))\setminus \Delta_n (Y_1)$.
\end{enumerate}
\end{prop}

Similarly, using Proposition \ref{1007271514} one has:

\begin{prop} \label{1010202356}
Let $\pi_1: (Y_1, G)\rightarrow (Y_2, G)$ and $\pi_2: (Y_2, G)\rightarrow (Y_3, G)$ be factor maps between TDS's and $\nu_1\in \mathcal{P} (Y_1, G), \nu_2= \pi_1 \nu_1\in \mathcal{P} (Y_2, G), n\in \mathbb{N}\setminus \{1\}$. Then
\begin{enumerate}

\item $E_n^{\nu_2} (Y_2, G| \pi_2)\subseteq (\pi_1\times \cdots\times \pi_1) E_n^{\nu_1} (Y_1, G| \pi_2\circ \pi_1)\subseteq E_n^{\nu_2} (Y_2, G| \pi_2)\cup \Delta_n (Y_2)$.

\item $E_n (Y_2, G| \pi_2)\subseteq (\pi_1\times \cdots\times \pi_1) E_n (Y_1, G| \pi_2\circ \pi_1)\subseteq E_n (Y_2, G| \pi_2)\cup \Delta_n (Y_2)$.

\item $E_n^{\nu_1} (Y_1, G| \pi_1)\subseteq E_n^{\nu_1} (Y_1, G| \pi_2\circ \pi_1)$ and $E_n (Y_1, G| \pi_1)\subseteq E_n (Y_1, G| \pi_2\circ \pi_1)$.
\end{enumerate}
\end{prop}

Let $\pi: (Y_1, G)\rightarrow (Y_2, G)$ be a factor map between TDS's and $\nu_2\in \mathcal{P} (Y_2, G)$. Set $Y_1^{\nu_2}= \bigcup\limits_{\nu_1\in \mathcal{P}_{\nu_2} (Y_1, G)} \text{supp} (\nu_1)$, and recall the associated continuous bundle RDS
$$\mathbf{F}^\pi= \{F^\pi_{g, y_2}: \{y_2\}\times \pi^{- 1} (y_2)\rightarrow \{g y_2\}\times \pi^{- 1} (g y_2)| g\in G, y_2\in Y_2\}$$
 with $\mathcal{E}_\pi= \{(y_2, y_1)\in Y_2\times Y_1: \pi (y_1)= y_2\}$ from \S\S \ref{1208042024}.

Then, with the help of Theorem \ref{1007232313}, Theorem \ref{1008072226}, Lemma \ref{small}, Theorem \ref{conclusion} and Proposition \ref{1103271148}, using Theorem \ref{1008301614} we can prove:

\begin{thm} \label{1010200058}
Let $\pi: (Y_1, G)\rightarrow (Y_2, G)$ be a factor map between TDS's and $\nu\in \mathcal{P} (Y_1, G), \nu_2\in \mathcal{P} (Y_2, G), n\in \mathbb{N}\setminus \{1\}$. Then
\begin{eqnarray*}
& & E_{n, \nu}^{(r)} (\mathcal{E}_\pi, G)= E_n^\nu (Y_1, G| \pi) \\
& & \hskip 66pt \subseteq \{(x_1, \cdots, x_n)\in \text{supp} (\nu)^n\setminus \Delta_n (Y_1): \pi (x_1)= \cdots= \pi (x_n)\}, \\
& & _{\nu_2} E_n^{(r)} (\mathcal{E}_\pi, G)= \bigcup_{\nu_1\in \mathcal{P}_{\nu_2} (Y_1, G)} E_n^{\nu_1} (Y_1, G| \pi) \\
& & \hskip 66pt \subseteq \{(x_1, \cdots, x_n)\in (Y_1^{\nu_2})^n\setminus \Delta_n (Y_1): \pi (x_1)= \cdots= \pi (x_n)\}, \\
& & E_n (Y_1, G| \pi)= \bigcup_{\eta\in \mathcal{P} (Y_2, G)}\ _\eta E_n^{(r)} (\mathcal{E}_\pi, G)= \bigcup_{\mu\in \mathcal{P} (Y_1, G)} E_n^\mu (Y_1, G| \pi).
\end{eqnarray*}
In particular, there exists $\mu\in \mathcal{P} (Y_1, G)$ such that
$$E_n (Y_1, G| \pi)= _{\pi \mu} E_n^{(r)} (\mathcal{E}_\pi, G)= E_n^\mu (Y_1, G| \pi).$$
\end{thm}

Following the ideas of local entropy theory (cf \cite{GY} and the references therein), the proof of Theorem \ref{1010200058} is quite standard, and we omit it here.

\newpage

\bibliographystyle{amsplain} 


\end{document}